\documentclass[leqno]{amsart}
\usepackage[left=1.2in, right=1.2in, top=1in, bottom=1in]{geometry}
\usepackage{boilerplate}
\usepackage{mathrsfs}
\usepackage{tikz-cd}
\usepackage{mathdots}
\usepackage[pageanchor]{hyperref}
\hypersetup{
    colorlinks,
    citecolor=black,
    filecolor=black,
    linkcolor=black,
    urlcolor=black
}
\usepackage[mathscr]{euscript}
\usepackage[refpage, notocbasic]{nomencl}

\makenomenclature

\usepackage{cleveref}
\usepackage{mathtools}
\usepackage{enumitem}

\renewcommand{\Top}{\categ{Spc}}

\title[Parametrized higher category theory]{Parametrized higher category theory}

\author{Jay Shah}
\address{Fachbereich Mathematik und Informatik, WWU Münster, 48149 M\"{u}nster, Germany}
\email{jayhshah@gmail.com}

\begin{document}

\begin{abstract} We develop foundations for the theory of $\infty$-categories parametrized by a base $\infty$-category. Our main contribution is a theory of indexed homotopy limits and colimits, which specializes to a theory of $G$-colimits for $G$ a finite group when the base is chosen to be the orbit category of $G$. We apply this theory to show that the $G$-$\infty$-category of $G$-spaces is freely generated under $G$-colimits by the contractible $G$-space, thereby affirming a conjecture of Mike Hill.
\end{abstract}

\maketitle

\tableofcontents


\section{Introduction}

\noindent \textbf{Motivation from equivariant homotopy theory}: This paper lays foundations for a theory of $\infty$-categories parametrized by a base $\infty$-category $S$. Our interest in this project originates in attempting to locate the core homotopy theories of interest in equivariant homotopy theory - those of $G$-spaces and $G$-spectra - within the appropriate $\infty$-categorical framework. To explain, let $G$ be a finite group and let us review the definitions of the $\infty$-categories of $G$-spaces and $G$-spectra, with a view towards endowing them with universal properties.
\\

Consider a category $\categ{Top}_G$ of (nice) topological spaces equipped with $G$-action, with morphisms given by the $G$-equivariant continuous maps. There are various homotopy theories that derive from this category, depending on the class of weak equivalences that one chooses to invert. At one end, we can invert the class $\sW_1$ of $G$-equivariant maps which induce a weak homotopy equivalence of underlying topological spaces, forgetting the $G$-action. If we let $\Top$ denote the $\infty$-category of spaces (i.e., $\infty$-groupoids), then inverting $\sW_1$ obtains the \emph{$\infty$-category of spaces with $G$-action}
\[ \categ{Top}_G[\sW_1^{-1}] \simeq \Fun(BG, \Top). \]
For many purposes, $\Fun(BG, \Top)$ is the homotopy theory that one wishes to contemplate, but here we instead highlight its main deficiency. Namely, passing to this homotopy theory blurs the distinction between homotopy fixed points and actual fixed points, in that the functor $\categ{Top}_G \to \Fun(BG, \Top)$ forgets the homotopy types of the various spaces $X^H$ for $H$ a nontrivial subgroup of $G$. Because many arguments in equivariant homotopy theory involve comparing $X^H$ with the homotopy fixed points $X^{h H}$, we want to retain this data. To this end, we can instead let $\sW$ be the class of $G$-equivariant maps which induce an equivalence on $H$-fixed points for every subgroup $H$ of $G$. Let $\Top_G \coloneq \categ{Top}_G[\sW^{-1}]$; this is the \emph{$\infty$-category of $G$-spaces}.

\nomenclature[Spaces]{$\Top$}{$\infty$-category of spaces}

Like with $\categ{Top}_G[\sW_1^{-1}]$, we would like a description of $\Top_G$ which eliminates any reference to topological spaces with $G$-action, for the purpose of comprehending its universal property. Elmendorf's theorem grants such a description: we have 
\[ \Top_G \simeq \Fun(\OO_G^\op, \Top), \]
where $\OO_G$ is the category of orbits of the group $G$. Thus, as an $\infty$-category, $\Top_G$ is the \emph{free cocompletion} of $\OO_G$.

It is a more subtle matter to define the homotopy theory of $G$-spectra. There are at least three possibilities:

\begin{enumerate}[leftmargin=6ex]
	\item The $\infty$-category of \emph{Borel $G$-spectra}, i.e. spectra with $G$-action:  This is $$\Sp^{h G} \coloneq \Fun(B G, \Sp),$$ which is the stabilization of $\Fun(B G, \Top)$.
	\item The $\infty$-category of \emph{`naive' $G$-spectra}, i.e. spectral presheaves on $\OO_G$: This is $$\Sp_G \coloneq \Fun(\OO_G^\op, \Sp),$$ which is the stabilization of $\Top_G$.\footnote{The usage of a subscript $G$ to indicate presheaves on $\OO_G$ (whether valued in spaces or spectra) is consistent with our later notation for the $S$-category of $S$-objects in an arbitrary $\infty$-category -- see Construction~\ref{constr:Sobjects}.}
	\item The $\infty$-category of \emph{`genuine' $G$-spectra}, i.e. spectral Mackey functors on the category $\FF_G$ of finite $G$-sets: Let $A^{\eff}(\FF_G)$ be the effective Burnside $(2,1)$-category of $G$, given by taking as objects finite $G$-sets, as morphisms spans of finite $G$-sets, and as $2$-morphisms isomorphisms between spans. Then, the $\infty$-category of genuine $G$-spectra is defined to be
	$$\Sp^G \coloneq \Fun^{\oplus}(A^{\eff}(\FF_G),\Sp),$$
	the $\infty$-category of direct-sum preserving functors from $A^{\eff}(\FF_G)$ to $\Sp$.\footnote{This is not the definition which first appeared in the literature for $G$-spectra, but it is equivalent to e.g. the homotopy theory of orthogonal $G$-spectra by the pioneering work of Guillou-May \cite{guillou2}. For an $\infty$-categorical treatment, see \cite{M1}.}
\end{enumerate}

The third possibility incorporates essential examples of cohomology theories for $G$-spaces, such as equivariant $K$-theory, because $G$-spectra in this sense possess transfers along maps of finite $G$-sets, encoded by the covariant maps in $A^{\eff}(\FF_G)$. It is thus what homotopy theorists customarily mean by $G$-spectra. However, from a categorical perspective it is a more mysterious object than the $\infty$-category of naive $G$-spectra, since it is \emph{not} the stabilization of $G$-spaces. We are led to ask:
\\

\textbf{Question}: What is the universal property of $\Sp^G$? More precisely, we have an adjunction
\[ \adjunct{\Sigma^\infty_+}{\Top_G}{\Sp^G}{\Omega^\infty} \]
with the right adjoint given by taking $\Omega^{\infty}: \Sp \to \Spc$ objectwise and restricting along the evident map $\OO_G^\op \to A^{\eff}(\FF_G)$, and we would like a universal property for $\Sigma^\infty_+$ or $\Omega^\infty$.

Put another way, what is the categorical procedure which manufactures $\Sp^G$ from $\Top_G$?
\\

The key idea is that for this procedure of `$G$-stabilization' one needs to enforce `$G$-additivity' over and above the usual additivity satisfied by a stable $\infty$-category: that is, one wants the coincidence of coproducts and products indexed not just by finite sets but by finite sets with $G$-action. Reflecting upon the possible homotopical meaning of such a $G$-(co)product, we see that for a transitive $G$-set $G/H$, $\coprod_{G/H}$ and $\prod_{G/H}$ should be interpreted to mean the left and right adjoints to the restriction functor $\Sp^G \to \Sp^H$, i.e. the induction and coinduction functors, and $G$-additivity then becomes the Wirthm\"{u}ller isomorphism. In particular, we see that $G$-additivity is not a property that $\Sp^G$ can be said to enjoy in isolation, but rather one satisfied by the \emph{presheaf} $\underline{\Sp}^G$ of $\infty$-categories indexed by $\OO_G$; here, for every $G$-orbit $U$, a choice of basepoint specifying an isomorphism $U \cong G/H$ yields an equivalence $\underline{\Sp}^G(U) \simeq \Sp^{H}$, and the functoriality in maps of orbits is that of conjugation and restriction (in particular, recording the residual actions of the Weyl groups on $\Sp^H$). Correspondingly, we must rephrase our question so as to inquire after the universal property of the \emph{morphism of $\OO_G$-presheaves}
$$\underline{\Sigma}^\infty_+: \underline{\Top}_G \to \underline{\Sp}^G,$$
where $\underline{\Sigma}^\infty_+$ is objectwise given by genuine $H$-suspension ranging over all subgroups $H \leq G$.

We now pause to observe that for the purpose of this analysis the group $G$ is of secondary importance as compared to its associated category of orbits $\OO_G$. Indeed, we focused on $G$-additivity as the distinguishing feature of genuine vs. naive $G$-spectra, as opposed to the invertibility of representation spheres, in order to evade representation theoretic aspects of equivariant stable homotopy theory. In order to frame our situation in its proper generality, let us now dispense with the group $G$ and replace $\OO_G$ by an arbitrary $\infty$-category $T$. Call a presheaf of $\infty$-categories on $T$ a \emph{$T$-category}. The \emph{$T$-category of $T$-spaces} $\underline{\Top}_T$ is given by the functor $T^\op \to \Cat_\infty$, $t \mapsto \Fun((T^{/t})^\op, \Top)$. Note that this specializes to $\underline{\Top}_G$ when $T = \OO_G$ because $\OO_H \simeq (\OO_G)^{/(G/H)}$; slice categories stand in for subgroups in our theory. With the theory of $T$-colimits advanced in this paper, we can then supply a universal property for $\underline{\Top}_T$ \emph{as a $T$-category}. Write $\underline{\Fun}_T$ for the internal hom in the $\infty$-category of $T$-categories, which is cartesian closed. 

\begin{thm} \label{thm:UniversalPropertyOfTSpaces} Suppose $T$ is any $\infty$-category. Then $\underline{\Top}_T$ is $T$-cocomplete, and for any $T$-category $E$ which is $T$-cocomplete, the $T$-functor of evaluation at the $T$-final object\footnote{We define $\ast_T$ to be the constant $T$-presheaf valued at $\ast$, which is the final object in the $\infty$-category of $T$-categories.}
\[ \underline{\Fun}^L_T(\underline{\Top}_T,E) \to \underline{\Fun}_T(\ast_T, E) \simeq E \]
induces an equivalence from the $T$-category of $T$-functors $\underline{\Top}_T \to E$ which strongly preserve $T$-colimits to $E$. In other words, $\underline{\Top}_T$ is freely generated under $T$-colimits by the final $T$-category.
\end{thm}

\begin{rem} The notion of $T$-cocompleteness needed for the theorem is slightly more elaborate than one might naively expect. Namely, we say that a $T$-category $C$ is \emph{$T$-cocomplete} if for all $t \in T$, the pullback of $C$ to a $T^{/t}$-category $C_{\underline{t}}$ (Notation~\ref{ntn:parametrizedFibers}) admits all (small) $T^{/t}$-colimits (Definition~\ref{dfn:cocomplete}). Correspondingly, we say that a $T$-functor $F: C \to D$ strongly preserves $T$-colimits if for all $t \in T$, the pulled-back $T^{/t}$-functor $F_{\underline{t}}: C_{\underline{t}} \to D_{\underline{t}}$ preserves all $T^{/t}$-colimits (Definition~\ref{def:strong-preservation}).
\end{rem}

When $T = \OO_G$, this result was originally conjectured by Mike Hill.

To go further and define $T$-spectra, we need a condition on $T$ so that it supports a theory of spectral Mackey functors. We say that $T$ is \emph{orbital} if $T$ admits multipullbacks, by which we mean that its finite coproduct completion $\FF_T$ admits pullbacks. The purpose of the orbitality assumption is to ensure that the effective Burnside category $A^{\eff}(\FF_T)$ is well-defined. Note that the slice categories $T_{/t}$ are orbital if $T$ is. We define the \emph{$T$-category of $T$-spectra} $\underline{\Sp}^T$ to be the functor $T^\op \to \Cat_\infty$ given by $t \mapsto \Fun^\oplus(A^{\eff}(\FF_{T_{/t}}), \Sp)$. We then have the following theorem of Denis Nardin concerning $\underline{\Sp}^T$ from \cite{Nardin}, which resolves our question:

\begin{thm}[{\cite[Theorem~7.4]{Nardin}}] Suppose $T$ is an atomic\footnote{This is an additional technical hypothesis which we do not explain here. It will not concern us in the body of the paper.} orbital $\infty$-category. Then $\underline{\Sp}^T$ is $T$-stable, and for any pointed $T$-category $C$ which has all finite $T$-colimits, the functor of postcomposition by $\Omega^\infty$
\[ (\Omega^\infty)_\ast: \Fun_T^{T-\rex}(C, \underline{\Sp}^T) \to \Lin^T(C, \underline{\Top}_T) \]
induces an equivalence from the $\infty$-category of $T$-functors $C \to \underline{\Sp}^T$ which preserve finite $T$-colimits to the $\infty$-category of $T$-linear functors $C \to \underline{\Top}_T$, i.e. those $T$-functors which are fiberwise linear and send finite $T$-coproducts to $T$-products.
\end{thm}

We hope that the two aforementioned theorems will serve to impress upon the reader the utility of the purely $\infty$-categorical work that we undertake in this paper.

\begin{wrn} \label{warning:op} In contrast to this introduction thus far and the conventions adopted elsewhere (e.g. in \cite{Nardin}), we will henceforth speak of $S$-categories, $S$-colimits, etc. for $S = T^\op$.
\end{wrn}

\subsection*{What is parametrized \texorpdfstring{$\infty$}{infinity}-category theory?} Roughly speaking, parametrized $\infty$-category theory is an interpretation of the familiar notions of ordinary or `absolute' $\infty$-category theory within the $(\infty,2)$-category of functors $\Fun(S, \Cat_\infty)$, done relative to a fixed `base' $\infty$-category $S$. By `interpretation', we mean something along the lines of the program of Emily Riehl and Dominic Verity \cite{riehl_verity_2022}, which axiomatizes the essential properties of an $(\infty,2)$-category that one needs to do formal category theory into the notion of an \emph{$\infty$-cosmos}, of which $\Fun(S, \Cat_\infty)$ is an example. In an $\infty$-cosmos, one can write down in a formal way notions of limits and colimits, adjunctions, Kan extensions, and so forth. Working out what this means in the example of $\Cat_\infty$-valued functors is the goal of this paper. In the classical $2$-categorical setting, such limits and colimits are referred to as ``indexed'' limits and colimits, so another perspective on this paper is that it extends indexed category theory to the $\infty$-categorical setting.

In contrast to Riehl--Verity, we will work within the model of \emph{quasi-categories} and not hesitate to use special aspects of our model (e.g., combinatorial arguments involving simplicial sets). We are motivated in this respect by the existence of a highly developed theory of \emph{cocartesian fibrations} due to Jacob Lurie, which we review in \S 2. Cocartesian fibrations are our preferred way to model $\Cat_\infty$-valued functors, for two reasons:

\begin{enumerate}[leftmargin=6ex]
	\item The data of a functor $F: S \to \Cat_\infty$ is overdetermined vs. that of a cocartesian fibration over $S$, in the sense that to define $F$ one must prescribe an infinite hierarchy of coherence data, which under the functor-fibration correspondence amounts to prescribing an infinite sequence of compatible horn fillings.\footnote{It is for this reason that one speaks of \emph{straightening} a cocartesian fibration to a functor.} Because of this, specifying a cocartesian fibration (which one ultimately needs to do in order to connect our theory to applications) is typically an easier task than specifying the corresponding functor to $\Cat_{\infty}$.
	\item The Grothendieck construction on a functor $S \to \Cat_\infty$ is made visible in the cocartesian fibration setup, as the total category of the cocartesian fibration. Many of our arguments involve direct manipulation of the Grothendieck construction, in order to relate or reduce notions of parametrized $\infty$-category theory to absolute $\infty$-category theory.
\end{enumerate}

\noindent We have therefore tailored our exposition to the reader familiar with the first five chapters of \cite{HTT}; the only additional major prerequisite is the part of \cite[App.~B]{HA} dealing with variants of the cocartesian model structure of \cite[\S 3]{HTT} and functoriality in the base.

\subsection*{Linear overview} 
Let us now give a section-by-section summary of the contents of this paper.

\begin{enumerate}[leftmargin=6ex]
\item[(\S 2)] We define an \emph{$S$-category} as a cocartesian fibration over $S$, and then collect some necessary preliminaries on cocartesian fibrations and model structures on categories of marked simplicial sets. In particular, we recapitulate Lurie's theorem that establishes conditions under which change-of-base adjunctions are Quillen (Theorem~\ref{thm:FunctorialityOfCocartesianModelStructure}); this theorem will allow us to efficiently verify the fibrancy of many of the simplicial set constructions introduced in this paper.

\item[(\S 3)] We first define and study the internal hom $\underline{\Fun}_S(-,-)$ of $S$-categories (Definition~\ref{dfn:FunctorInternalHom}). We then recall the \emph{$S$-category of $S$-objects} $\underline{E}_{S}$ in an $\infty$-category $E$ from \cite{BDGNS} (Construction~\ref{constr:Sobjects}), which computes the right adjoint to the forgetful functor $[C \to S] \mapsto C$. When $S = \OO_G^{\op}$ and $E = \Spc$, this recovers the $G$-category of $G$-spaces $\underline{\Spc}_G$.

\item[(\S 4)] We first introduce the $S$-join $(- \star_S -)$ (Definition~\ref{dfn:join}), which in terms of presheaves computes the fiberwise join. We then define and study two (canonically equivalent) $S$-slice constructions: for a $S$-functor $p: K \to C$, we have $S$-undercategories $C_{(p,S)/}$ and $C^{(p,S)/}$ and $S$-overcategories $C_{/(p,S)}$ and $C^{/(p,S)}$. The `lower' construction (Definition~\ref{dfn:lowerSlice}) is a direct generalization of Joyal's slice construction (cf. \cite[Proposition~1.2.9.2]{HTT}) and participates in a Quillen adjunction with the $S$-join. The `upper' construction (Definition~\ref{dfn:AltSlice}) proceeds by taking an $S$-fiber of the relevant map of $S$-functor categories. In practice, the upper $S$-slice is far easier to work with as its definition is less bound up with the intricate combinatorics of the $S$-join (which need to be thoroughly understood to even establish the fibrancy of the lower $S$-slice; cf. Proposition~\ref{joinprp}). However, it is easier to establish the universal mapping property of the $S$-slice using its lower incarnation (Proposition~\ref{prp:SliceCompare}).

\item[(\S 5)] We initiate our study of $S$-colimits and $S$-limits by giving the basic definition \ref{dfn:colimit}, and then discuss a few special cases: $S$-(co)limits in an $S$-category of $S$-objects, $S$-colimits indexed by constant $S$-diagrams, and $S$-colimits indexed by $S$-points (i.e., $S$-coproducts). We then explain how to deduce results about $S$-limits from $S$-colimits (or vice-versa) by means of the vertical opposite construction (Corollary~\ref{cor:LimitToColimit}).

\item[(\S 6)] Our main goal in this section is to establish an $S$-analogue of Joyal's cofinality theorem \cite[Theorem~4.1.3.1]{HTT}: an $S$-functor $C \to D$ is \emph{$S$-final} if and only if it is fiberwise final\footnote{We write final and initial for what Lurie calls (left) cofinal and right cofinal, respectively.} (Theorem~\ref{thm:cofinality}). Our strategy is to control the functoriality encoded by the $S$-slice category in terms of a construction, the \emph{twisted slice} (Definition~\ref{dfn:twistedSlice}), fibered over the twisted arrow category $\twa(S)$; the right Kan extension of the latter will then obtain the former (Theorem~\ref{thm:ordinarySliceToParamSlice}). In fact, we first do the same for the internal hom $\underline{\Fun}_S$ itself (Equation~\ref{eqn:FunctorAsEndComparisonMap}). This may be thought of as a refinement of the end formula for an $\infty$-category of natural transformations (cf. Remark~\ref{rem:EndFormula}).

\item[(\S 7)] In this brief section, we introduce the notions of $S$-fibration, $S$-(co)cartesian fibration, and $S$-bifibration (Definition~\ref{dfn:ScocartesianFibration} and Definition~\ref{dfn:Sbifibration}). We also introduce the free $S$-(co)cartesian fibration as an example (Definition~\ref{dfn:freeCocartesian}).

\item[(\S 8)] We recall Lurie's definition of a relative adjunction and specialize it to the notion of an $S$-adjunction (Definition~\ref{dfn:sAdjunction}). We then prove a number of fundamental results about $S$-adjunctions \textemdash most notably, the fact that a left $S$-adjoint preserve $S$-colimits (Corollary~\ref{cor:leftAdjointPreservesColimits}).

\item[(\S 9)] Given an $S$-cocartesian fibration $\phi: C \to D$ and an $S$-functor $F: C \to E$, we construct the left $S$-Kan extension $\phi_! F: D \to E$, which will also call the \emph{$D$-parametrized $S$-colimit} of $F$. With our assumption on $\phi$, we have that for every object $x \in D_s$, $(\phi_! F)(x)$ is computed as the $S^{s/}$-colimit of the restriction of $F$ to the $S^{s/}$-fiber $C_{\underline{x}}$; this is precisely analogous to the situation where the left Kan extension along a cocartesian fibration is computed by taking colimits fiberwise. In order to construct $\phi_! F$, we need to solve the coherence problem of assembling the individual $S^{s/}$-colimits of $F_{\underline{s}}: C_{\underline{x}} \to E_{\underline{s}}$ (ranging over all $x \in D_s$) into a single $S$-functor out of $D$. We introduce the $S$-pairing construction \ref{paramPair}, and subsequently the $D$-parametrized slice (Construction~\ref{cnstr:pairingSlice}), to facilitate this. The problem of constructing $\phi_! F$ then ultimately reduces to choosing a section of a certain trivial Kan fibration defined in terms of the $D$-parametrized slice (Theorem~\ref{thm:ExistenceAndUqnessOfParamColimit}).

\item[(\S 10)] We define left $S$-Kan extensions in general (Definition~\ref{dfn:SLKE}) and prove the basic existence and uniqueness result about them (Theorem~\ref{thm:existenceOfKanExtensions}). In contrast to the brutal simplex-by-simplex approach taken in \cite[\S 4.3.2]{HTT} to the construction of Kan extensions (cf. \cite[Lemma~4.3.2.13]{HTT}), we instead reduce to the solved coherence problem for $D$-parametrized $S$-colimits via factoring the $S$-functor $\phi: C \to D$ to be extended along through the free $S$-cocartesian fibration on it.
We remark that, to our knowledge, the approach of \S 9 and \S 10 gave a novel\footnote{All these results date to 2017.} and more conceptual construction of Kan extensions even in the context of ordinary $\infty$-category theory. Lurie has since independently written up a treatment of (relative) Kan extensions along these lines in Kerodon \cite[\href{https://kerodon.net/tag/02Y1}{Tag 02Y1}]{kerodon}.

\item[(\S 11)] We recall the $S$-category of presheaves $\PP_S(-)$, prove the $S$-Yoneda lemma \ref{lm:Yoneda}, discuss $S$-mapping spaces, and establish the universal property of $\PP_S(-)$ as free $S$-cocompletion (Theorem~\ref{thm:UniversalPropertyOfPresheaves}), thereby proving Theorem~\ref{thm:UniversalPropertyOfTSpaces}.

\item[(\S 12)] We prove two Bousfield-Kan style\footnote{By this, we mean to refer to generalizations of the classical formula for writing a colimit as a coequalizer of coproducts, which were studied by Bousfield and Kan in the context of homotopy colimits with coequalizers replaced by geometric realization.} decomposition results that express an arbitrary $S$-colimit as a geometric realization of either $S$-coproducts or $S$-space-indexed $S$-colimits (Theorem~\ref{thm:relativeBKwithCoproducts} and Theorem~\ref{thm:relativeBKwithSpaces}). The essential content behind such formulas lies in replacing a given diagram $C$ with one fibered over $\Delta^\op \times S$ that possesses an $S$-final map to $C$. As a warmup, we first explain how this goes when $S$ is a point (Corollary~\ref{cor:ordinaryBKformula} and Corollary~\ref{cor:BousfieldKanFormulaHomotopyInvariantVersion}); the resulting formula appears to be new in the case of coproducts, whereas the case of spaces was first obtained by Aaron Mazel-Gee in \cite{MAZELGEE20194602}. We then apply the $S$-Bousfield-Kan formula to show that, supposing $S^\op$ admits multipullbacks, an $S$-category is $S$-cocomplete if and only if it admits all $S$-(co)products and geometric realizations (Corollary~\ref{cor:DecomposingColimits}).
\end{enumerate}

\subsection*{Notation and conventions}
Let $C$ be an $\infty$-category. We write
$$\sO(C) \coloneq \Fun(\Delta^1, C)$$
for the $\infty$-category of arrows in $C$. In this paper, we will frequently encounter fiber products of the form
\[ A \times_{F,C,\ev_0} \sO(C) \times_{\ev_1,C, G} B \]
where $F:A \to C$ and $G: B \to C$ are functors. To avoid notational clutter, we adopt the global convention that, unless otherwise decorated, fiber products with the source functor $\ev_0$ are to be written on the left, and fiber products with the target functor $\ev_1$ are to written on the right. Moreover, we will drop $F$ and $G$ from the notation if they are understood from context. For instance, we would write the preceding expression as $A \times_{C} \sO(C) \times_{C} B$.

\subsection*{Acknowledgements} This paper is a lightly revised version of my thesis, which was originally part of a joint project with my advisor Clark Barwick, Emanuele Dotto, Saul Glasman, and Denis Nardin. I would like to thank them and the other participants of the Bourbon seminar - Lukas Brantner, Peter Haine, Marc Hoyois, Akhil Mathew, and Sune Precht Reeh - for innumerable conversations and mathematical inspiration, without which this work would not have been possible. I would also like to thank the referee for writing an extremely detailed report that has helped to improve the readability of this paper.

\section{Cocartesian fibrations and model categories of marked simplicial sets}

Let $S$ be an $\infty$-category. In this section, we give a rapid review of the theory of cocartesian fibrations and the surrounding apparatus of marked simplicial sets. This primarily serves to fix some of our notation and conventions for the remainder of the paper; for a more detailed exposition of these concepts, we refer the reader to \cite{BarShah}. In particular, the reader should be aware of our special notation (Notation~\ref{ntn:parametrizedFibers}) for the $S$-fibers of a $S$-functor.

\subsection*{Cocartesian fibrations} We begin with the basic definitions:

\begin{dfn} \label{dfn:CocartesianFibration} Let $\pi: X \to S$ be a map of simplicial sets. Then $\pi$ is a \emph{cocartesian fibration} if 
\begin{enumerate} \item $\pi$ is an \emph{inner fibration}: for every $n>1$, $0<k<n$ and commutative square
\[ \begin{tikzcd}[row sep=2em, column sep=2em]
\Lambda^n_k \ar{r} \ar{d} & X \ar{d}{\pi} \\
\Delta^n \ar{r} \ar[dotted]{ru} & S,
\end{tikzcd} \]
the dotted lift exists.
\item For every edge $\alpha: s_0 \rightarrow s_1$ in $S$ and $x_0 \in X$ with $\pi(x_0) = s_0$, there exists an edge $e: x_0 \rightarrow x_1$ in $X$ with $\pi(e) = \alpha$, such that $e$ is \emph{$\pi$-cocartesian}: for every $n>1$ and commutative square
\[ \begin{tikzcd}[row sep=2em, column sep=2em]
\Lambda^n_0 \ar{r}{f} \ar{d} & X \ar{d}{\pi} \\
\Delta^n \ar{r} \ar[dotted]{ru} & S
\end{tikzcd} \]
with $f|_{\Delta^{\{ 0,1 \}} } = e$, the dotted lift exists.
\end{enumerate}

Dually, $\pi$ is a \emph{cartesian fibration} if $\pi^\op$ is a cocartesian fibration.

A cocartesian resp. cartesian fibration $\pi: X \to S$ is said to be a \emph{left} resp. \emph{right} fibration if for every object $s \in S$ the fiber $X_s$ is a Kan complex.

Now suppose $\pi: X \to S$ and $\rho: Y \to S$ are (co)cartesian fibrations. Then a \emph{map of (co)cartesian fibrations} $f: X \to Y$ is a map of simplicial sets such that $\rho \circ f = \pi$ and $f$ carries $\pi$-(co)cartesian edges to $\rho$-(co)cartesian edges. The collection of cocartesian fibrations over $S$ and maps thereof organize into a subcategory $\Cat_{\infty/S}^\cocart$ of the overcategory $\Cat_{\infty/S}$.
\end{dfn}

In this paper, owing to the importance of these notions we see fit to introduce more concise and suggestive terminology for cocartesian fibrations and left fibrations over $S$.

\begin{dfn} An \emph{$S$-category} resp. \emph{$S$-space} $C$ is a cocartesian resp. left fibration $\pi: C \to S$. An \emph{$S$-functor} $F: C \to D$ between $S$-categories $C$ and $D$ is a map of cocartesian fibrations over $S$.

Given an $S$-category $\pi: C \to S$, an \emph{$S$-subcategory} $D \subset C$ is a subcategory such that the restriction $\pi|_D$ is a cocartesian fibration and an edge in $D$ is $\pi|_D$-cocartesian if and only if it is $\pi$-cocartesian. The inclusion functor then necessarily preserves cocartesian edges, so is an $S$-functor. We further say that $D$ is a \emph{full $S$-subcategory} if $D \subset C$ is in addition a full subcategory, or equivalently, for every $s \in S$, $D_s \subset C_s$ is a full subcategory.
\end{dfn}

\begin{exm}[Arrow $\infty$-categories] The arrow $\infty$-category $\sO(S)$ of $S$ is cocartesian over $S$ via the target morphism $\ev_1$, and cartesian over $S$ via the source morphism $\ev_0$. An edge
\[ e:[s_0 \rightarrow t_0] \to [s_1 \rightarrow t_1] \]
in $\sO(S)$ is $\ev_1$-cocartesian resp. $\ev_0$-cartesian if and only if $\ev_0(e)$ resp. $\ev_1(e)$ is an equivalence in $S$.

The fiber of $\ev_0: \sO(S) \to S$ over $s$ is isomorphic to Lurie's `alternative' slice $\infty$-category $S^{s/}$. Using our knowledge of the $\ev_1$-cocartesian edges, we see that $\ev_1$ restricts to a left fibration $S^{s/} \to S$. In the terminology of \cite[Proposition~4.4.4.5]{HTT}, this is a \emph{corepresentable} left fibration. We will refer to the corepresentable left fibrations as \emph{$S$-points}. Further emphasizing this viewpoint, we will often let $\underline{s}$ denote $S^{s/}$.
\end{exm}

\nomenclature[ArrowCategory]{$\sO(S)$}{Arrow $\infty$-category of $S$}
\nomenclature[ArrowSlice]{$S^{s/}$}{Slice $\infty$-category of $S$ under the object $s$, Lurie's ``alternative'' version \cite[\S 4.2.1]{HTT}}

To a beginner, the lifting conditions of Definition~\ref{dfn:CocartesianFibration} can seem opaque. Under our standing assumption that $S$ is an $\infty$-category, we have a reformulation of the definition of cocartesian edge, and hence that of cocartesian fibration, which serves to illuminate its homotopical meaning.

\begin{prp} \label{prp:MappingSpacesOfCocartesianFibration} Let $\pi: X \to S$ be an inner fibration (so $X$ is an $\infty$-category). Then an edge $e: x_0 \rightarrow x_1$ in $X$ is $\pi$-cocartesian if and only if for every $x_2 \in X$, the commutative square of mapping spaces
\[ \begin{tikzcd}[row sep=2em, column sep=2em]
\Map_X(x_1,x_2) \ar{r}{e^\ast} \ar{d}{\pi} & \Map_X(x_0,x_2) \ar{d}{\pi} \\
\Map_S(\pi(x_1), \pi(x_2)) \ar{r}{\pi(e)^\ast} & \Map_S(\pi(x_0), \pi(x_2))
\end{tikzcd} \]
is homotopy cartesian.
\end{prp}

With some work, Proposition~\ref{prp:MappingSpacesOfCocartesianFibration} can be used to supply an alternative, model-independent definition of a cocartesian fibration: we refer to Mazel-Gee's paper \cite{mazelgeeCocartesianFibration} for an exposition along these lines.

\nomenclature[catscocart]{$\Cat_{\infty/S}^\cocart$}{$\infty$-category of cocartesian fibrations over $S$}
\nomenclature[cats]{$\Cat_{\infty/S}$}{$\infty$-category of $\infty$-categories over $S$}

\begin{exm}[{\cite[\S 3.2.2]{HTT}}] \label{exm:UniversalCocartesianFibration} Let $\Cat_\infty$ denote the (large) $\infty$-category of (small) $\infty$-categories. Then there exists a \emph{universal cocartesian fibration} $\sU \to \Cat_\infty$, which is characterized up to contractible choice by the requirement that any cocartesian fibration $\pi: X \to S$ (with essentially small fibers) fits into a homotopy pullback square
\[ \begin{tikzcd}[row sep=2em, column sep=2em]
X \ar{r} \ar{d}{\pi} & \sU \ar{d} \\
S \ar{r}{F_\pi} & \Cat_\infty.
\end{tikzcd} \]
Concretely, one can take $\sU$ to be the subcategory of the arrow category $\sO(\Cat_\infty)$ spanned by the representable right fibrations and morphisms thereof.
\end{exm}

As suggested by Example~\ref{exm:UniversalCocartesianFibration}, the functor
\[ \Fun(S,\Cat_\infty) \to \Cat_{\infty/S}^\cocart \]
given by pulling back $\sU \to \Cat_\infty$ is an equivalence. The composition
\[ \Gr: \Fun(S,\Cat_\infty) \xrightarrow{\simeq} \Cat_{\infty/S}^\cocart \subset \Cat_{\infty/S}  \] 
is the \emph{Grothendieck construction} functor. Since equivalences in $\Fun(S, \Cat_\infty)$ are detected objectwise, $\Gr$ is conservative. Moreover, one can check that $\Gr$ preserves limit and colimits, so by the adjoint functor theorem $\Gr$ admits both a left and a right adjoint.

\begin{ntn} \label{ntn:groth} Let \[ \Fr \dashv \Gr \dashv H \]
denote the left and right adjoints of $\Gr$.
\end{ntn}

We call $\Fr$ the \emph{free cocartesian fibration} functor (see also \cite{GHN}): concretely, one has
$$\Fr(X \to S) = X \times_S \sO(S) \xrightarrow{ev_1} S,$$
or as a functor $s \mapsto X \times_S S_{/s}$ with functoriality obtained from $S_{/(-)}$. The functor $H$ can also be concretely described using its universal mapping property: since $\Fr(\{s\} \subset S) = S_{s/}$, the fiber $H(X)_s$ is given by $\Fun_{/S}(S_{s/},X)$, and the functoriality in $S$ is obtained from that of $S_{(-)/}$.

\subsection*{A model structure for cocartesian fibrations} We want a model structure which has as its fibrant objects the cocartesian fibrations over a fixed simplicial set. However, it is clear that to define it we need some way to remember the data of the cocartesian edges. This leads us to introduce \emph{marked simplicial sets}.

\begin{dfn} A marked simplicial set $(X, \cE)$ is the data of a simplicial set $X$ and a subset $\cE \subset X_1$ of the edges of $X$, such that $\cE$ contains all of the degenerate edges. We call $\cE$ the set of \emph{marked edges} of $X$. A map of marked simplicial sets $f: (X,\cE) \to (Y,\cF)$ is a map of simplicial sets $f: X \to Y$ such that $f(\cE) \subset \cF$.
\end{dfn}

\nomenclature[MarkedSimplicialSet]{$(X, \cE)$}{Marked simplicial set}

\begin{ntn} We introduce notation for certain classes of marked simplicial sets. Let $X$ be a simplicial set.
\begin{itemize} \item $X^\flat$ is $X$ with only the degenerate edges marked. To avoid notational clutter, we will sometimes suppress this notation and simply write $X$ for $X^{\flat}$.
\item $X^\sharp$ is $X$ with all of its edges marked.
\item Suppose that $X$ is an $\infty$-category. Then $X^\sim$ is $X$ with its equivalences marked.
\item Suppose that $\pi: X \to S$ is an inner fibration. Then $\leftnat{X}$ is $X$ with its $\pi$-cocartesian edges marked, and $\rightnat{X}$ is $X$ with its $\pi$-cartesian edges marked.
\item Let $n > 0$. Let $\leftnat{\Delta^n}$ resp. $\leftnat{\Lambda^n_0}$ denote $\Delta^n$ resp. $\Lambda^n_0$ with the edge $\{0,1\}$ marked (if it exists, i.e. excluding $\Delta^0$ and $\Lambda^1_0 = \{ 0 \}$) along with the degenerate edges. Dually, let $\rightnat{\Delta^n}$ resp. $\rightnat{\Lambda^n_n}$ denote $\Delta^n$ resp. $\Lambda^n_n$ with the edge $\{n-1,n\}$ marked.
\end{itemize}

Note that our choice of notation $\leftnat{\Delta^n}$ and $\leftnat{\Lambda^n_0}$ is not meant to be interpreted as a special instance of marking cocartesian edges (though the map $\Delta^n \to \Delta^1$ given by $0 \mapsto 0$ and $1,..., n \mapsto 1$ renders it as such for the former); rather, we mean to indicate that the relevant lifting problem for a cocartesian fibration as a marked simplicial set is to lift along the marked horn inclusion $\leftnat{\Lambda^n_0} \to \leftnat{\Delta^n}$ (cf. Definition~\ref{dfn:fiberedObject} below), and vice-versa for cartesian fibrations and $\rightnat{\Lambda^n_n} \to \rightnat{\Delta^n}$.
\end{ntn}

\nomenclature[Xflat]{$X^{\flat}$}{Simplicial set $X$ with its degenerate edges marked}
\nomenclature[Xsharp]{$X^{\sharp}$}{Simplicial set $X$ with all its edges marked}
\nomenclature[Xequiv]{$X^{\sim}$}{$\infty$-category $X$ with its equivalences marked}
\nomenclature[Xleftnat]{$\leftnat{X}$}{Inner fibration $\pi: X \to S$ with its $\pi$-cocartesian edges marked}
\nomenclature[Xrightnat]{$\rightnat{X}$}{Inner fibration $\pi: X \to S$ with its $\pi$-cartesian edges marked}
\nomenclature[Dleftnat]{${\leftnat{\Delta^n}}$}{$\Delta^n$ with the edge $\{0,1\}$ marked}
\nomenclature[Drightnat]{${\rightnat{\Delta^n}}$}{$\Delta^n$ with the edge $\{n-1,n\}$ marked}
\nomenclature[Dhornleftnat]{${\leftnat{\Lambda^n_0}}$}{${\leftnat{\Lambda^n_0}}$ with the edge $\{0,1\}$ marked}
\nomenclature[Dhornrightnat]{${\rightnat{\Lambda^n_0}}$}{${\rightnat{\Lambda^n_0}}$ with the edge $\{n-1,n\}$ marked}

For the rest of this section, fix a marked simplicial set $(Z,\cE)$ where $Z$ is an $\infty$-category and $\cE$ contains all of the equivalences in $Z$; in our applications, $Z$ will generally be some type of fibration over $S$. Let $s\Set^+_{/(Z,\cE)}$ denote the category of marked simplicial sets over $(Z,\cE)$; following Lurie \cite[Notation~3.1.0.2]{HTT}, we will also denote $s\Set^+_{/Z^\sharp}$ more simply as $s\Set^+_{/Z}$. We will frequently abuse notation by referring an object $\pi: (X, \cF) \to (Z, \cE)$ of $s\Set^+_{/(Z,\cE)}$ by its domain $(X,\cF)$, or even just by $X$.

\nomenclature[markedsset]{$s\Set^+_{/(Z,\cE)}$}{The category of marked simplicial sets over $(Z,\cE)$}
\nomenclature[markedsset2]{$s\Set^+_{/Z}$}{The category of marked simplicial sets over $Z^\sharp$}

\begin{dfn} \label{dfn:fiberedObject} An object $(X,\cF)$ in $s\Set^+_{/(Z,\cE)}$ is \emph{$(Z,\cE)$-fibered}\footnote{This differs from \cite[Definition~B.0.19]{HA}, but nonetheless defines the correct class of anodyne morphisms \cite[Definition~B.1.1]{HA}.} if 
\begin{enumerate} \item $\pi: X \to Z$ is an inner fibration.
\item For every $n>0$ and commutative square 
\[ \begin{tikzcd}[row sep=2em, column sep=2em]
\leftnat{\Lambda^n_0} \ar{r} \ar{d} & (X,\cF) \ar{d} \\
\leftnat{\Delta^n} \ar{r} \ar[dotted]{ru} & (Z,\cE),
\end{tikzcd} \]
a dotted lift exists. In other words, letting $n=1$, $\pi$-cocartesian lifts exist over marked edges in $Z$, and letting $n>1$, marked edges in $X$ are $\pi$-cocartesian.\footnote{Note that condition (2) already guarantees that $X \to Z$ is a cocartesian fibration if $\sE = Z_1$; however, one additionally needs condition (4) to ensure that \emph{all} of the $\pi$-cocartesian edges are marked in $X$.}
\item For every commutative square
\[ \begin{tikzcd}[row sep=2em, column sep=2em]
(\Lambda^2_1)^\sharp \cup_{(\Lambda^2_1)^\flat} (\Delta^2)^\flat \ar{r} \ar{d} & (X,\cF) \ar{d} \\
(\Delta^2)^\sharp \ar{r} \ar[dotted]{ru} & (Z,\cE),
\end{tikzcd} \]
a dotted lift exists. In other words, marked edges are closed under composition.\footnote{Strictly speaking, condition (3) by itself only guarantees that for any pair of composable marked edges, there exists a composite that is again marked. One additionally needs condition (4) to ensure that \emph{all} compositions of marked edges are again marked.}
\item Let $Q = \Delta^0 \coprod_{\Delta^{ \{0,2\} }}  \Delta^3  \coprod_{\Delta^{ \{1,3\} }} \Delta^0$. For every commutative square
\[ \begin{tikzcd}[row sep=2em, column sep=2em]
Q^\flat \ar{r} \ar{d} & (X,\cF) \ar{d} \\
Q^\sharp \ar{r} \ar[dotted]{ru} & (Z,\cE),
\end{tikzcd} \]
a dotted lift exists. Since we assumed that $\cE$ contains all equivalences in $Z$, this implies that all equivalences in $X$ are marked.
\end{enumerate}
\end{dfn}

\begin{exm} Let $\pi: X \to Z$ be an inner fibration. Comparing with Definition~\ref{dfn:CocartesianFibration}, it is clear that $(X,\cF)$ is $Z^\sharp$-fibered if and only if $\pi$ is a cocartesian fibration and $(X,\cF) = \leftnat{X}$. At the other extreme, $(X,\cF)$ is $Z^{\sim}$-fibered if and only if $\pi$ is a categorical fibration and $(X,\cF) = X^{\sim}$.
\end{exm}

Recall that a model structure, if it exists, is determined by its cofibrations and fibrant objects. Collecting results of Lurie from \cite[App.~B]{HA}, we now define a model structure on $s\Set^+_{/(Z,\cE)}$ with cofibrations the monomorphisms and fibrant objects given by the $(Z,\cE)$-fibered objects. 

\begin{dfn} Define functors\footnote{In \cite[App.~B]{HA}, these functors are denoted as $\Map_Z^\sharp$ and $\Map_Z^\flat$ respectively.}
\begin{align*} \Map_Z(-,-): & {s\Set^+_{/(Z,\cE)}}^\op \times s\Set^+_{/(Z,\cE)} \to s\Set  \\
\Fun_Z(-,-): & {s\Set^+_{/(Z,\cE)}}^\op \times s\Set^+_{/(Z,\cE)} \to s\Set
\end{align*} 
by
\begin{align*} \Hom(A,\Map_Z(X,Y)) & = \Hom_{/(Z,\cE)}(A^\sharp \times X,Y), \\
\Hom(A,\Fun_Z(X,Y)) & = \Hom_{/(Z,\cE)}(A^\flat \times X,Y).
\end{align*}
\end{dfn}

\nomenclature[Map]{$\Map_{(-)}(-,-)$}{Mapping simplicial set relative to marked simplicial set, excludes non-invertible morphisms, $\infty$-groupoid when fibrant}

\nomenclature[Fun]{$\Fun_{(-)}(-,-)$}{Mapping simplicial set relative to marked simplicial set, includes non-invertible morphisms, $\infty$-category when fibrant}

\begin{dfn} A map $f: A \to B$ in $s\Set^+_{/(Z,\cE)}$ is a \emph{cocartesian equivalence} (with respect to $(Z,\cE)$) if the following equivalent conditions obtain:
\begin{enumerate} \item For all $(Z,\cE)$-fibered $X$, $f^\ast: \Map_Z(B,X) \to \Map_Z(A,X)$ is an equivalence of Kan complexes.
\item For all $(Z,\cE)$-fibered $X$, $f^\ast: \Fun_Z(B,X) \to \Fun_Z(A,X)$ is an equivalence of $\infty$-categories.
\end{enumerate}
\end{dfn}

\begin{thm}[{\cite[Theorem~B.0.20]{HA}}] There exists a left proper combinatorial model structure on the category $s\Set^+_{/(Z,\cE)}$, which we call the \textbf{cocartesian model structure}, such that:
\begin{enumerate} \item The cofibrations are the monomorphisms.
\item The weak equivalences are the cocartesian equivalences. 
\item The fibrant objects are the $(Z,\cE)$-fibered objects.
\end{enumerate}
Dually, we define the \textbf{cartesian model structure} on $s\Set^+_{/(Z,\cE)}$ to be the cocartesian model structure on $s\Set^+_{/(Z,\cE)^\op}$ under the isomorphism given by taking opposites.
\end{thm}

\begin{rem}
The underlying $\infty$-category of $s\Set^+_{/(Z,\cE)}$ identifies as the subcategory of $\Cat_{\infty/Z}$ on those isofibrations\footnote{Note that with this choice, the resulting subcategory is not stable under equivalence. One could alternatively appeal to a homotopy-invariant notion of cocartesian fibration and instead replace isofibrations with functors -- cf. \cite{mazelgeeCocartesianFibration}, which admits an obvious generalization to this setting.} $X \to Z$ that admit cocartesian lifts over $\cE$, and with morphisms preserving cocartesian edges. In particular, passing to the closure of $\cE$ under composition does not change the underlying $\infty$-category.
\end{rem}

We have the following characterization of the cocartesian equivalences between fibrant objects (which is unsurprising, in light of the equivalence $\Cat_{\infty/Z}^\cocart \simeq \Fun(Z,\Cat_\infty)$).

\begin{prp}[{\cite[Lemma~B.2.4]{HA}}] Let $X$ and $Y$ be fibrant objects in $s\Set^+_{/(Z,\cE)}$ equipped with the cocartesian model structure, and let $f: X \to Y$ be a map in $s\Set^+_{/(Z,\cE)}$. Then the following are equivalent:
\begin{enumerate}
	\item $f$ is a cocartesian equivalence.
	\item $f$ is a homotopy equivalence, i.e. $f$ admits a homotopy inverse: there exists a map $g: Y \to X$ and homotopies $h: (\Delta^1)^\sharp \times X \to X$, $h': (\Delta^1)^\sharp \times Y \to Y$ in $s\Set^+_{/(Z,\cE)}$ connecting $g \circ f$ to $\id_X$ and $f \circ g$ to $\id_Y$, respectively.
	\item $f$ is a categorical equivalence.
	\item For every (not necessarily marked) edge $\alpha: \Delta^1 \to Z$, $f_\alpha: \Delta^1 \times_Z X \to \Delta^1 \times_Z Y$ is a categorical equivalence.	
\end{enumerate}
If every edge of $Z$ is marked, then (4) can be replaced by the following apparently weaker condition:
\begin{enumerate}
	\item[(4\textquotesingle)] For every object $z \in Z$, $f_z: X_z \to Y_z$ is a categorical equivalence.
\end{enumerate}
\end{prp}

We also have the following characterization of the fibrations between fibrant objects.

\begin{prp}[{\cite[Proposition~B.2.7]{HA}}] \label{prp:FibrationBetweenFibrantObjects} Let $Y = (Y,\cF)$ be a fibrant object in $s\Set^+_{/(Z,\cE)}$ equipped with the cocartesian model structure, and let $f: X \to Y$ be a map in $s\Set^+_{/(Z,\cE)}$. Then the following are equivalent:
\begin{enumerate}
	\item $f$ is a fibration.
	\item $X$ is fibrant, and $f$ is a categorical fibration.	
	\item $f$ is fibrant in $s\Set^+_{/(Y,\cF)}$.
\end{enumerate}
\end{prp}

\begin{cor} \label{cor:CocartesianModelStructureIsSliceModelStructure} Suppose $Z \to S$ is a cocartesian fibration. Then the cocartesian model structure $s\Set^+_{/\leftnat{Z}}$ coincides with the `slice' model structure on $(s\Set^+_{/S})_{/\leftnat{Z}}$ created by the forgetful functor to $s\Set^+_{/S}$ equipped with its cocartesian model structure.
\end{cor}
\begin{proof}
This immediately follows from Proposition~\ref{prp:FibrationBetweenFibrantObjects}.
\end{proof}

\begin{exm} Suppose that $Z$ is a Kan complex. Then the cocartesian and cartesian model structures on $s\Set^+_{/Z}$ coincide. In particular, taking $Z = \Delta^0$, we will also refer to the cocartesian model structure on $s\Set^+$ as the \emph{marked model structure}. Since this model structure on $s\Set^+$ is unambiguous, we will always regard $s\Set^+$ as equipped with it. Then the fibrant objects of $s\Set^+$ are precisely the $\infty$-categories with their equivalences marked.
\end{exm}

\begin{exm} \label{exm:flatQuillenEquivalence} Suppose that $(Z,\cE) = Z^\sim$. Then the cocartesian and cartesian model structures on $s\Set^+_{/Z^\sim}$ coincide. Moreover, we have a Quillen equivalence
\[ \adjunct{(-)^\flat}{(s\Set_{\text{Joyal}})_{/Z}}{s\Set^+_{/Z^\sim}}{U} \]
where the functor $U$ forgets the marking. In particular, $(-)^\flat$ sends categorical equivalences to marked equivalences.
\end{exm}

\begin{exm} \label{exm:MarkedQuillenAdjunctions} The inclusion functor $\Top \subset \Cat_\infty$ admits left and right adjoints $B$ and $\iota$, where $B$ is the classifying space functor that inverts all edges and $\iota$ is the `core' functor that takes the maximal sub-$\infty$-groupoid. These two adjunctions are modeled by the two Quillen adjunctions
\[ \adjunct{U}{s\Set^+}{s\Set_{\text{Quillen}}}{(-)^\sharp}, \]
\[ \adjunct{(-)^\sharp}{s\Set_{\text{Quillen}}}{s\Set^+}{M}. \]
Here $M(X,\sE)$ is the maximal sub-simplicial set of $X$ such that all of its edges are marked. In particular, $(-)^\sharp$ sends weak homotopy equivalences to marked equivalences.
\end{exm}

\begin{prp}[{\cite[Remark~B.2.5]{HA}}] \label{prp:SimplicialEnrichmentOfCocartesianModelStructure} The bifunctor
\[ - \times -: s\Set^+_{/(Z_1,\cE_1)} \times s\Set^+_{/(Z_2,\cE_2)} \to s\Set^+_{/(Z_1 \times Z_2,\cE_1 \times \cE_2)} \]
is left Quillen. Consequently, the bifunctors
\begin{align*} \Map_Z(-,-): & {s\Set^+_{/(Z,\cE)}}^\op \times s\Set^+_{/(Z,\cE)} \to s\Set_{\text{Quillen}}  \\
\Fun_Z(-,-): & {s\Set^+_{/(Z,\cE)}}^\op \times s\Set^+_{/(Z,\cE)} \to s\Set_{\text{Joyal}}
\end{align*} 
are right Quillen, so $s\Set^+_{/(Z,E)}$ is both a $s\Set_{\text{Quillen}}$-enriched model category (with respect to $\Map_Z$) and $s\Set_{\text{Joyal}}$-enriched model category (with respect to $\Fun_Z$).
\end{prp}

\begin{rem} As explained in \cite[Digression~1.2.13]{riehl_verity_2022}, by Proposition~\ref{prp:SimplicialEnrichmentOfCocartesianModelStructure} the full subcategory of $s\Set^+_{/(Z,\cE)}$ spanned by the fibrant objects is an example of an $\infty$-cosmos \cite[Definition~1.2.1]{riehl_verity_2022}.
\end{rem}

Finally, we explain how the formalism of marked simplicial sets can be used to extract the pushforward functors implicitly defined by a cocartesian fibration. First, we need a lemma.

\begin{lem} \label{lm:generalizedCocartesianPushforward} For $n>0$, the inclusion $i_n: \Delta^{n-1} \cong \Delta^{ \{0\} } \star \Delta^{ \{2,...,n\} } \to \leftnat{\Delta^n}$ is left marked anodyne. Consequently, for a cocartesian fibration $C \to S$, the map
\[ \Fun(\leftnat{\Delta^n},\leftnat{C}) \to \Fun(\Delta^{n-1}, C) \times_{\Fun(\Delta^{n-1},C)} \Fun(\Delta^n,S) \]
induced by $i_n$ is a trivial Kan fibration.
\end{lem}
\begin{proof} We proceed by induction on $n$, the base case $n=1$ being the left marked anodyne map $\Delta^{ \{0\}} \to \leftnat{\Delta^1} = (\Delta^1)^\sharp$. Consider the commutative diagram
\[ \begin{tikzcd}[row sep=2em, column sep=2em]
\Delta^{\{0\}} \star \partial \Delta^{n-2} \ar{d}{\bigcup i_{n-1}} \ar{r} & \Delta^{ \{0\} } \star \Delta^{ \{2,...,n\} } \ar{d} \ar[bend left]{dd}{i_n} \\
(\Delta^{\{0\}} \star \Lambda^{n-1}_0, \sE) \ar{r} & \leftnat{\Lambda^n_0} \ar{d} \\
& \leftnat{\Delta^n}
\end{tikzcd} \]
where $\sE$ is the collection of edges $\{ 0,i\}$, $0<i \leq n$ (and the degenerate edges). The square is a pushout, and by the inductive hypothesis, the lefthand vertical map is left marked anodyne. We deduce that $i_n$ is left marked anodyne. The second statement now follows because the lifting problem
\[ \begin{tikzcd}[row sep=2em, column sep=2em]
A \ar{r} \ar{d} & \Fun(\leftnat{\Delta^n},\leftnat{C}) \ar{d} \\
B \ar{r} \ar[dotted]{ru} & \Fun(\Delta^{n-1}, C) \times_{\Fun(\Delta^{n-1},C)} \Fun(\Delta^n,S)
\end{tikzcd} \]
transposes to
\[ \begin{tikzcd}[row sep=2em, column sep=2em]
A \times \leftnat{\Delta^n} \bigcup\limits_{A \times \Delta^{n-1}} B \times \Delta^{n-1} \ar{r} \ar{d} & \leftnat{C} \ar{d} \\
B \times \leftnat{\Delta^n} \ar{r} \ar[dotted]{ru} & S
\end{tikzcd} \]
and the lefthand vertical map is left marked anodyne for any cofibration $A \to B$ by \cite[Proposition~3.1.2.3]{HTT}.
\end{proof}

The main case of interest in Lemma~\ref{lm:generalizedCocartesianPushforward} is when $n=1$, which shows that $$\sO^\cocart(C) \to C \times_S \sO(S)$$ is a trivial Kan fibration. Let $$P: C \times_S \sO(S) \to \sO^\cocart(C)$$ be a section that fixes the inclusion $C \subset \sO^\cocart(C)$ (for this, note that $C \subset C \times_S \sO(S) $ is a cofibration as it is a monomorphism of simplicial sets). Then we say that $P$ or the further composite $P' = \ev_1 \circ P$ is a \emph{cocartesian pushforward} for $C \to S$. Given an edge $\alpha$ of $S$, $P'_{\alpha}: C_s \to C_t$ is the pushforward functor $\alpha_!$ determined under the equivalence $\Cat^\cocart_{\infty/S} \simeq \Fun(S, \Cat_\infty)$.

\subsection*{Functoriality in the base} Let $\pi: X \to Z$ be a map of simplicial sets. Then the pullback functor $\pi^\ast: s\Set_{/Z} \to s\Set_{/X}$ admits a left adjoint $\pi_!$, given by postcomposing with $\pi$. In addition, since $s\Set$ is a topos, $\pi^\ast$ also admits a right adjoint $\pi_\ast$, which may be thought of as the functor of relative sections because $\Hom_{/X}(A,\pi_\ast(B)) \cong \Hom_{/Z}(A \times_X Z, B)$.

Now supposing that $\pi$ is a map of marked simplicial sets, $\pi^\ast$, $\pi_!$, and $\pi_\ast$ extend to functors of marked simplicial sets over $X$ or $Y$ in an evident manner. We then seek conditions under which the adjunctions $\pi_! \dashv \pi^\ast$ and $\pi^\ast \dashv \pi_\ast$ are Quillen with respect to the cocartesian model structures. To this end, we have the following theorem of Lurie:

\begin{thm}[{\cite[Theorem~B.4.2]{HA}}] \label{thm:FunctorialityOfCocartesianModelStructure} Let
\[ (Z,\cE) \xleftarrow{\pi} (X,\cF) \xrightarrow{\rho} (X',\cF') \]
be a span of marked simplicial sets such that $Z, X, X'$ are $\infty$-categories and the collections of markings contain all the equivalences.
\begin{enumerate}[label=(\roman*)]
\item The adjunction
\[ \adjunct{\rho_!}{s\Set^+_{/(X,\cF)}}{s\Set^+_{/(X',\cF')}}{\rho^\ast} \]
is Quillen with respect to the cocartesian model structures.
\item Further suppose that
\begin{enumerate}[label=(\arabic*)]
	\item For every object $x \in X$ and marked edge $f: z \rightarrow \pi(x)$ in $Z$, there exists a locally $\pi$-cartesian edge $x_0 \rightarrow x$ in $X$ lifting $f$.
	\item $\pi$ is a flat categorical fibration.
	\item $\cE$ and $\cF$ are closed under composition.
	\item Suppose given a commutative diagram 
\[ \begin{tikzcd}[row sep=2em, column sep=2em]
& x_1 \ar{rd}{g} & \\
x_0 \ar{ru}{f} \ar{rr}{h} & & x_2
\end{tikzcd} \]
in $X$ where $g$ is locally $\pi$-cartesian, $\pi(g)$ is marked, and $\pi(f)$ is an equivalence. Then $f$ is marked if and only if $h$ is marked. (Note in particular that, taking $f$ to be an identity morphism, every locally $\pi$-cartesian edge lying over a marked edge is itself marked.)
\end{enumerate}
Then the adjunction
\[ \adjunct{\pi^\ast}{s\Set^+_{/(X,\cF)}}{s\Set^+_{/(Z,\cE)}}{\pi_\ast} \]
is Quillen with respect to the cocartesian model structures.
\end{enumerate}
\end{thm}

We formulated Theorem~\ref{thm:FunctorialityOfCocartesianModelStructure} as a theorem concerning a span $Z \xleftarrow{\pi} X \xrightarrow{\rho} X'$ because in applications we will typically be interested in the composite Quillen adjunction
\[ \adjunct{\rho_! \pi^\ast}{s\Set^+_{/(Z,\cE)}}{s\Set^+_{/(X',\cF')}}{\pi_\ast \rho^\ast}. \]

Here are two examples.

\begin{exm}[Pairing cartesian and cocartesian fibrations] \label{exm:pairing} Let $\pi: X \to Z$ be a cartesian fibration. Then the span
\[ Z^\sharp \xleftarrow{\pi} \rightnat{X} \xrightarrow{\pi} Z^\sharp \]
satisfies the hypotheses of Theorem~\ref{thm:FunctorialityOfCocartesianModelStructure}. Now given a cocartesian fibration $Y \to Z$, define
\[ \widetilde{\Fun}_Z(X,Y) \coloneq (\pi_\ast \pi^\ast)(\leftnat{Y} \rightarrow Z^\sharp). \]
Then the fiber of $\widetilde{\Fun}_Z(X,Y)$ over an object $z \in Z$ is $\Fun(X_z,Y_z)$, and given a morphism $\alpha: z_0 \to z_1$, the pushforward functor
$$\alpha_!: \Fun(X_{z_0},Y_{z_0}) \to \Fun(X_{z_1},Y_{z_1})$$
is given by precomposition in the source and postcomposition in the target. Note how this example highlights the relevance of condition (1) in Theorem~\ref{thm:FunctorialityOfCocartesianModelStructure}(ii).
\end{exm}

\nomenclature[pairingAbs]{$\widetilde{\Fun}_D(C,E)$}{Pairing construction}

\begin{exm}[Right Kan extension] \label{exm:RKEcocartesian} Let $f: Y \to Z$ be a functor. We can apply Theorem~\ref{thm:FunctorialityOfCocartesianModelStructure} to perform the operation of right Kan extension at the level of cocartesian fibrations. Consider the span
\[ Z^\sharp \xleftarrow{\ev_0} (\sO(Z) \times_{Z,f} Y)^\sharp \xrightarrow{\pr_Y} Y^\sharp. \]
Then the conditions of Theorem~\ref{thm:FunctorialityOfCocartesianModelStructure} are satisfied, so we obtain a Quillen adjunction
\[ \adjunct{(\pr_Y)_! (\ev_0)^\ast}{s\Set^+_{/Z}}{s\Set^+_{/Y}}{(\ev_0)_\ast (\pr_Y)^\ast}. \]
In addition, the map $C \times_Z Y^\sharp \to C \times_Z \sO(Z)^\sharp \times_Z Y^\sharp$ induced by the identity section $\iota: Z \to \sO(Z)$ is a cocartesian equivalence in $s\Set^+_{/Y}$ for $C \to Z$ \emph{fibrant} in $s\Set^+_{/Z}$, by \cite[Lemma~9.8]{BDGNS1}. Consequently, the induced adjunction of $\infty$-categories
\[ \adjunct{(\pr_Y)_! (\ev_0)^\ast}{\Cat_{\infty/Z}^\cocart}{\Cat_{\infty/Y}^\cocart}{(\ev_0)_\ast (\pr_Y)^\ast} \]
is equivalent to
\[ \adjunct{f^\ast}{\Fun(Z,\Cat_\infty)}{\Fun(Y,\Cat_\infty)}{f_\ast} \]
under the straightening/unstraightening equivalence (which is natural with respect to pullback).

Note that as a special case, if $Z = \Delta^0$ we recover the formula $\Fun_Y(Y^\sharp,\leftnat{C}) \simeq \lim\limits_{\ot} F_C$ of \cite[Corollary~3.3.3.2]{HTT} (where $C \to Y$ is a cocartesian fibration and $F_C: Y \to \Cat_\infty$ the corresponding functor). Indeed, this construction of the right Kan extension of a cocartesian fibration is suggested by that result and the pointwise formula for a right Kan extension.
\end{exm}

Finally, we will use the following two observations concerning the interaction of Theorem~\ref{thm:FunctorialityOfCocartesianModelStructure} with compositions and homotopy equivalences of spans (which we also recorded in \cite{BarShah}).

\begin{lem} \label{lm:SpanComposition} Suppose we have spans of marked simplicial sets
\[ X_0 \xleftarrow{\pi_0} Z_0 \xrightarrow{\rho_0} X_1 \]
and
\[ X_1 \xleftarrow{\pi_1} Z_1 \xrightarrow{\rho_1} X_2 \]
which each satisfy the hypotheses of Theorem~\ref{thm:FunctorialityOfCocartesianModelStructure}. Then the span
\[ Z_0 \xleftarrow{\pr_0} Z_0 \times_{X_1} Z_1 \xrightarrow{\pr_1} Z_1 \]
also satisfies the hypothesis of Theorem~\ref{thm:FunctorialityOfCocartesianModelStructure}.\footnote{However, one should beware that the ``long'' span
\[ X_0 \ot Z_0 \times_{X_1} Z_1 \to X_2 \]
may fail to satisfy the hypotheses of Theorem~\ref{thm:FunctorialityOfCocartesianModelStructure}, because the composition of locally cartesian fibrations may fail to again be locally cartesian; this explains the roundabout formulation of the statement.} Consequently, we obtain a Quillen adjunction
\[ \adjunct{(\rho_1 \circ \pr_1)_! (\pi_0 \circ \pr_0)^\ast}{s\Set^+_{/ X_0}}{s\Set^+_{X_2}}{(\pi_0 \circ \pr_0)_\ast (\rho_1 \circ \pr_1)^\ast}, \]
which is the composite of the Quillen adjunction from $s\Set^+_{/X_0}$ to $s\Set^+_{/X_1}$ with the one from $s\Set^+_{/X_1}$ to $s\Set^+_{/X_2}$.
\end{lem}
\begin{proof} The assertion that the span satisfies the hypotheses of Theorem~\ref{thm:FunctorialityOfCocartesianModelStructure} is by inspection. The other assertion that the Quillen adjunction factors as a composite follows from the base-change isomorphism $\rho_0^\ast \pi_{1,\ast} \cong \pr_{0,\ast} \circ \pr_1^\ast$.
\end{proof}

\begin{lem} \label{lm:SpanHomotopyInvariance} Suppose a morphism of spans of marked simplicial sets
\[ \begin{tikzcd}[row sep=2em, column sep=2em]
& Z \ar{ld}[swap]{\pi} \ar{rd}{\rho} \ar{d}{f} & \\
X & Z' \ar{l}{\pi'} \ar{r}[swap]{\rho'} & X' 
\end{tikzcd} \]

where $\rho_! \pi^\ast$ and $(\rho')_! (\pi')^\ast$ are left Quillen with respect to the cocartesian model structures on $X$ and $X'$. Suppose moreover that $f$ is a homotopy equivalence in $s\Set^+_{/X'}$, so that there exists a homotopy inverse $g$ and homotopies
\[h\colon \id \simeq g \circ f\text{\quad and\quad}k\colon \id \simeq f \circ g.\]
Then the natural transformation $\rho_! \pi^\ast \to (\rho')_! (\pi')^\ast$ induced by $f$ is a cocartesian equivalence on all objects, and, consequently, the adjoint natural transformation $(\pi')_\ast (\rho')^\ast \to \pi_\ast \rho^\ast$ is a cocartesian equivalence on all fibrant objects.
\begin{proof} The homotopies $h$ and $k$ pull back to show that for all $X \to C$, the map
\[\id_X \times_C f\colon X \times_C K \to X \times_C L\]
is a homotopy equivalence with inverse $\id_X \times_C g$. The last statement now follows from \cite[Corollary~1.4.4(b)]{Hovey}.
\end{proof}
\end{lem}

\subsection*{Parametrized fibers} In this brief subsection, we record notation for the $S$-fibers of an $S$-functor.

\begin{ntn} \label{ntn:parametrizedFibers} Given an $S$-category $\pi: D \to S$ and an object $x \in D$, define
\[ \sO_{x \rightarrow}(D) \coloneq \{x\} \times_D \sO(D). \]
For the full subcategory of cocartesian edges $\sO^\cocart(D) \subset \sO(D)$, also define
\[ \underline{x} \coloneq \{ x \} \times_D \sO^\cocart(D). \]
Given an $S$-functor $\phi: C \to D$, define
\[ C_{\underline{x}} \coloneq \underline{x} \times_{D,\phi} C. \]
\end{ntn}

\nomenclature[ParamPoint]{$\underline{x}$}{Parametrized point}
\nomenclature[ParamFiber]{$C_{\underline{x}}$}{Parametrized fiber}

Note that by definition, the objects of $\underline{x}$ are $\pi$-cocartesian edges in $D$ with source $x$. Then by the right cancellative property of $\pi$-cocartesian edges \cite[Lemma~2.4.2.7]{HTT}, the morphisms in $\underline{x}$ are $2$-simplices of cocartesian edges with source $x$, hence $\underline{x}$ is an $S$-space (via the map $\ev_1: \underline{x} \to S$). In fact, by Lemma~\ref{lm:fattenedSlice}, $\ev_1: \underline{x} \to S^{\pi x/}$ is a trivial fibration, so we may think of $\underline{x}$ as an `$S$-point' of $D$.

In view of this, we will also regard $C_{\underline{x}}$ as a $S^{\pi x/}$-category (and we will sometimes be cavalier about the distinction between $\underline{x}$ and $S^{\pi x/}$). Note however, that the functor $\underline{x} \to D$ is canonical in our setup, whereas we need to make a choice of cocartesian pushforward to choose a $S$-functor $S^{\pi x/} \to D$ that selects $x \in D$.

\section{Functor categories}
Let $S$ be an $\infty$-category. Then $\Fun(S,\Cat_\infty)$ is cartesian closed, so it possesses an internal hom. As a basic application of Theorem~\ref{thm:FunctorialityOfCocartesianModelStructure}, we will define this internal hom at the level of cocartesian fibrations over $S$.

\begin{prp} \label{func} Let $C \to S$ be a cocartesian fibration. Let $\ev_0, \ev_1: \sO(S) \times_S C \to S$ denote the source and target maps. Then the functor
\[ (\ev_1)_! (\ev_0)^\ast: s\Set^+_{/S} \to s\Set^+_{/ \sO(S)^\sharp \times_S \leftnat{C}} \to s\Set^+_{/S} \]
is left Quillen with respect to the cocartesian model structures.
\end{prp} 
\begin{proof} We verify the hypotheses of Theorem~\ref{thm:FunctorialityOfCocartesianModelStructure} as applied to the span $S^\sharp \xleftarrow{\ev_0} \sO(S)^\sharp \times_S \leftnat{C} \xrightarrow{\ev_1} S^\sharp$. By \cite[Corollary~2.4.7.12]{HTT}, $\ev_0$ is a cartesian fibration and an edge $e$ in $\sO(S) \times_S C$ is $\ev_0$-cartesian if and only if its projection to $C$ is an equivalence. (1) thus holds. (2) holds since cartesian fibrations are flat categorical fibrations. (3) is obvious. (4) follows from the stability of cocartesian edges under equivalence.
\end{proof}

\begin{dfn} \label{dfn:FunctorInternalHom}
In the statement of Proposition~\ref{func}, let
\[ \underline{\Fun}_S(C,-) \coloneq (\ev_0)_\ast (\ev_1)^\ast: s\Set^+_{/S} \to s\Set^+_{/S}. \]
We will also write this as $\underline{\Fun}_S(\leftnat{C},-)$ if we wish to emphasize the marking.
\end{dfn}

Proposition~\ref{func} implies that if $D \to S$ is a cocartesian fibration, $\underline{\Fun}_S(C,D) \to S$ is a cocartesian fibration. Unwinding the definitions, we see that an object of $\underline{\Fun}_S(C,D)$ over $s \in S$ is a $S^{s/}$-functor
$$S^{s/} \times_S C \to S^{s/} \times_S D,$$ 
and a cocartesian edge of $\underline{\Fun}_S(C,D)$ over an edge $e: \Delta^1 \to S$ is a $\Delta^1 \times_S \mathscr{O}(S)$-functor
$$\Delta^1 \times_S \mathscr{O}(S) \times_S C \to \Delta^1 \times_S \mathscr{O}(S) \times_S D.$$
\nomenclature[FunctorCategoryParameterized]{$\underline{\Fun}_S(-,-)$}{$S$-category of $S$-functors}

Our first goal is to prove that the construction $\underline{\Fun}_S(C,-)$ implements the internal hom at the level of underlying $\infty$-categories. To this end, we have the following lemma and proposition.

\begin{lem} \label{funclem} Let $\iota: S \to \mathscr{O}(S)$ be the identity section and regard $\mathscr{O}(S)^\sharp$ as a marked simplicial set over $S$ via the target map. Then
\begin{enumerate}
\item[(1)] For every marked simplicial set $X \to S$ and cartesian fibration $C \to S$,
\[ \id_X \times \iota \times \id_C: X \times_S \rightnat{C} \to X \times_S \mathscr{O}(S)^\sharp \times_S \rightnat{C} \]
is a cocartesian equivalence in $s\Set^+_{/S}$.

\item[(1\textquotesingle)] For every marked simplicial set $X \to S$ and cartesian fibration $C \to S$,
\[ \iota \times \id_C: X \times_S \rightnat{C} \to \Fun((\Delta^1)^\sharp,X) \times_S \rightnat{C} \]
is a cocartesian equivalence in $s\Set^+_{/S}$, where the marked edges in $\Fun((\Delta^1)^\sharp,X)$ are the marked squares in $X$.

\item[(2)] For every marked simplicial set $X \to S$ and cocartesian fibration $C \to S$,
\[  \id_C \times \iota \times \id_X : \leftnat{C} \times_S X \to \leftnat{C} \times_S \mathscr{O}(S)^\sharp \times_S X \]
is a homotopy equivalence in $s\Set^+_{/S}$.
\end{enumerate}
\end{lem}
\begin{proof}
\begin{enumerate}[leftmargin=*]
\item[(1)] Because $- \times_S \rightnat{C}$ preserves cocartesian equivalences, we reduce to the case where $C = S$. By definition, $X \to X \times_S \mathscr{O}(S)^\sharp$ is a cocartesian equivalence if and only if for every cocartesian fibration $Z \to S$, $\Map^\sharp_S(X \times_S \mathscr{O}(S)^\sharp, \leftnat{Z}) \to \Map^\sharp_S(X,\leftnat{Z})$ is a trivial Kan fibration. In other words, for every monomorphism of simplicial sets $A \to B$ and cocartesian fibration $Z \to S$, we need to provide a lift in the following commutative square
\[ \begin{tikzcd}[row sep=2em, column sep=2em]
B^\sharp \times X \bigsqcup_{A^\sharp \times X} (A^\sharp \times X) \times_S \mathscr{O}(S)^\sharp \ar{r}{\phi} \ar{d}  & \leftnat{Z} \ar{d} \\
(B^\sharp \times X) \times_S \mathscr{O}(S)^\sharp \ar{r} \ar[dotted]{ru} & S^\sharp
\end{tikzcd} \]

Define $h_0: \mathscr{O}(S)^\sharp \times (\Delta^1)^\sharp \to \mathscr{O}(S)^\sharp$ to be the adjoint to the map $\mathscr{O}(S)^\sharp \to \mathscr{O}(\mathscr{O}(S))^\sharp$ obtained by precomposing by the map of posets $\Delta^1 \times \Delta^1 \to \Delta^1$ which sends $(1,1)$ to $1$ and the other vertices to $0$. Precomposing $\phi$ by $\id_{A^\sharp \times X} \times h_0$, define a homotopy
\[ h: (A^\sharp \times X) \times_S \mathscr{O}(S)^\sharp \times (\Delta^1)^\sharp \to \leftnat{Z} \]
from $\phi|_{A^\sharp \times X} \circ \pr_{A^\sharp \times X}$ to $\phi|_{(A^\sharp \times X) \times_S \mathscr{O}(S)^\sharp}$. Using $h$ and $\phi|_{B^\sharp \times X}$, define a map
\[ \psi: B^\sharp \times X \bigsqcup_{A^\sharp \times X} (A^\sharp \times X) \times_S \mathscr{O}(S)^\sharp \to \Fun((\Delta^1)^\sharp, \leftnat{Z}) \]
such that $\psi|_{B^\sharp \times X}$ is adjoint to $\phi|_{B^\sharp \times X} \circ \pr_{B^\sharp \times X}$ and $\psi|_{(A^\sharp \times X) \times_S \mathscr{O}(S)^\sharp}$ is adjoint to $h$. Then we may factor the above square through the trivial fibration $\Fun((\Delta^1)^\sharp, \leftnat{Z}) \to \leftnat{Z} \times_S \mathscr{O}(S)^\sharp$ to obtain the commutative rectangle
\[ \begin{tikzcd}[row sep=2em, column sep=2em]
B^\sharp \times X \bigsqcup_{A^\sharp \times X} (A^\sharp \times X) \times_S \mathscr{O}(S)^\sharp \ar{r}{\psi} \ar{d} & \Fun((\Delta^1)^\sharp, \leftnat{Z}) \ar{r}{e_1} \ar[->>]{d}{\simeq} &  \leftnat{Z} \ar{d} \\
(B^\sharp \times X) \times_S \mathscr{O}(S)^\sharp \ar[dotted]{ru}[description]{\widetilde{\psi}} \ar{r}[swap]{\phi|_{B^\sharp \times X} \times \id} & \leftnat{Z} \times_S \mathscr{O}(S)^\sharp \ar{r}{e_1} & S^\sharp.
\end{tikzcd} \]
The dotted lift $\widetilde{\psi}$ exists, and $e_1 \circ \widetilde{\psi}$ is our desired lift.

\item[(1\textquotesingle)] Repeat the argument of (1) with $\Fun((\Delta^1)^\sharp,X)$ in place of $\sO(S)^\sharp$.

\item[(2)] Let $p: C \to S$ denote the structure map and let $P$ be a lift in the commutative square
\[ \begin{tikzcd}[row sep=2em, column sep=2em]
\leftnat{C} \ar{r}{\iota_C} \ar{d} & \Fun((\Delta^1)^\sharp, \leftnat{C}) \ar[->>]{d}{(e_0,\sO(p))}[swap]{\simeq} \\
\leftnat{C} \times_S \sO(S)^\sharp \ar{r}{=} \ar[dotted]{ru}{P} & \leftnat{C} \times_S \sO(S)^\sharp.
\end{tikzcd} \]
Let 
\[ g = (e_1 \times \id_X) \circ (P \times \id_X): \leftnat{C} \times_S \sO(S)^\sharp \times_S X \to \leftnat{C} \times_S X \]
and note that $g$ is a map over $S$. We claim that $g$ is a marked homotopy inverse of $f = \id_C \times \iota \times \id_X$. By construction, $g \circ f = \id$. For the other direction, define
$$h_0: \Fun((\Delta^1)^\sharp, \leftnat{C}) \times (\Delta^1)^\sharp \to \Fun((\Delta^1)^\sharp, \leftnat{C})$$
as the adjoint of the map $\Fun((\Delta^1)^\sharp, \leftnat{C}) \to \Fun((\Delta^1 \times \Delta^1)^\sharp, \leftnat{C})$ obtained by precomposing by the map of posets $\Delta^1 \times \Delta^1 \to \Delta^1$ which sends $(0,0)$ to $0$ and the other vertices to $1$. Define
\[ h: \leftnat{C} \times_S \sO(S)^\sharp \times_S X \times (\Delta^1)^\sharp \to \leftnat{C} \times_S \sO(S)^\sharp \times_S X \]
as the composite $((e_0, \mathscr{O}(p)) \times X) \circ (h_0 \times X) \circ (P \times \id_{X \times (\Delta^1)^\sharp})$. Then $h$ is a homotopy over $S$ from $\id$ to $f \circ g$.
\end{enumerate}
\end{proof}

\begin{prp} \label{funsect} Let $C, C', D \to S$ be cocartesian fibrations and let $F: C \to C'$ be a monomorphism of cocartesian fibrations over $S$ (so preserving cocartesian edges). For all marked simplicial sets $Y$ over $S$, the map
\[ \Fun_S(\leftnat{D}, \underline{\Fun}_S (\leftnat{C'},Y) ) \to \Fun_S(\leftnat{D} \times_S \leftnat{C'},Y) \times_{\Fun_S(\leftnat{D} \times_S \leftnat{C},Y)} \Fun_S(\leftnat{D},\underline{\Fun}_S(\leftnat{C},Y)) \]
which precomposes by $F$ is a trivial Kan fibration.
\end{prp}
\begin{proof} From the defining adjunction, for all $X,Y \in s\Set^+_{/S}$ we have a natural isomorphism
\[ \Fun_S(X,\underline{\Fun}_S(\leftnat{C},Y)) \cong \Fun_S(X \times_S \mathscr{O}(S)^\sharp \times_S \leftnat{C}, Y) \]
of simplicial sets. Since $\Fun_S(-,-)$ is a right Quillen bifunctor, the assertion reduces to showing that
\[ \leftnat{D} \times_S \leftnat{C'} \bigsqcup_{\leftnat{D} \times_S \leftnat{C}} \leftnat{D} \times_S \sO(S)^\sharp \times_S \leftnat{C} \to \leftnat{D} \times_S \sO(S)^\sharp \times_S \leftnat{C'} \]
is a trivial cofibration in $s\Set^+_{/S}$, which follows from Lemma~\ref{funclem}(2).
\end{proof}

In Proposition~\ref{funsect}, letting $C = \emptyset$ and $Y = \leftnat{E}$ for another cocartesian fibration $E \to S$, we deduce that $\underline{\Fun}_S(C',-)$ is right adjoint to $C' \times_S -$ as an endofunctor of $\Fun(S,\Cat_\infty)$. Further setting $D=S$, we deduce that the category of cocartesian sections of $\underline{\Fun}_S(\leftnat{C},\leftnat{E})$ is equivalent to $\Fun_S(\leftnat{C},\leftnat{E})$. We will employ the following notation to explicitly track objects under this correspondence.

\begin{ntn} \label{ntn:cocartesianSection}
Given a map $f: \leftnat{C} \to \leftnat{E}$, let $\sigma_f$ denote the cocartesian section $S^\sharp \to \underline{\Fun}_S( \leftnat{C},\leftnat{E})$ given by adjointing the map $\sO(S)^\sharp \times_S \leftnat{C} \xrightarrow{\pr_C} \leftnat{C} \xrightarrow{f} \leftnat{E}$.
\end{ntn}

\nomenclature[sectionOfFunctorCategory]{$\sigma_f$}{Cocartesian section $S \to \underline{\Fun}_S(C,E)$ classifying $S$-functor $f: C \to E$}

We next study varying the second variable in the construction $\underline{\Fun}_S(-,-)$.

\begin{lem} \label{funclem2} Let $C \to D$ be a fibration of marked simplicial sets over $S$.
\begin{enumerate} \item Let $K \to S$ be a cocartesian fibration. Then
\[ \underline{\Fun}_S(\leftnat{K},C) \to \underline{\Fun}_S(\leftnat{K},D) \times_D C \]
is a fibration in $s\Set^+_{/S}$.

\item The map
\[ \underline{\Fun}_S(S^\sharp,C) \to \underline{\Fun}_S(S^\sharp,D) \times_D C \]
is a trivial fibration in $s\Set^+_{/S}$.
\end{enumerate}
\end{lem}
\begin{proof} Let $i: A \to B$ be a map of marked simplicial sets. For (1), we use that if $i$ is a trivial cofibration, then
\[ B \bigsqcup_A A \times_S \sO(S)^\sharp \times_S \leftnat{K} \to B \times_S \sO(S) \times_S \leftnat{K} \]
is a trivial cofibration, which follows from Proposition~\ref{func}. For (2), we use that if $i$ is a cofibration, then
\[ B \bigsqcup_A A \times_S \sO(S)^\sharp \to B \times_S \sO(S) \]
is a trivial cofibration, which follows from Lemma~\ref{funclem}(1).
\end{proof}

The following proposition indicates that we can promote the conclusion $\underline{\Fun}_S(S,-) \simeq \mathrm{id}$ (as an endofunctor of $\Fun(S, \Cat_{\infty})$) of Proposition~\ref{funsect} to the level of cocartesian model structures. It will not be used in the sequel and is included only for illustrative purposes.

\begin{prp} \label{funceqv} The Quillen adjunction 
\[ \adjunct{- \times_S \mathscr{O}(S)^\sharp}{s\Set^+_{/S}}{s\Set^+_{/S}}{\underline{\Fun}_S(S^\sharp,-)}  \]
is a Quillen equivalence.
\end{prp}
\begin{proof} We first check that for every cocartesian fibration $C \to S$, the counit map
\[ \underline{\Fun}_S(S^\sharp,\leftnat{C}) \times_S \sO(S)^\sharp \to \leftnat{C} \]
is a cocartesian equivalence. By Lemma~\ref{funclem}(1), it suffices to show that
\[ \underline{\Fun}_S(S^\sharp,\leftnat{C}) \to \leftnat{C} \]
is a trivial marked fibration, which follows from Lemma~\ref{funclem2}(2) (taking $D=S$). We now complete the proof by checking that $- \times_S \mathscr{O}(S)^\sharp$ reflects cocartesian equivalences: i.e., given the commutative square
\begin{equation*}
\begin{tikzpicture}[baseline]
\matrix(m)[matrix of math nodes,
row sep=4ex, column sep=4ex,
text height=1.5ex, text depth=0.25ex]
 { A  & B \\
 A \times_S \mathscr{O}(S)^\sharp & B \times_S \mathscr{O}(S)^\sharp. \\ };
\path[>=stealth,->,font=\scriptsize]
(m-1-1) edge (m-1-2)
edge (m-2-1)
(m-1-2) edge (m-2-2)
(m-2-1) edge (m-2-2);
\end{tikzpicture}
\end{equation*}
if the lower horizontal map is a cocartesian equivalence over $S$ (with respect to the target map) then the upper horizontal map is a cocartesian equivalence over $S$. But the vertical maps are cocartesian equivalences by Lemma~\ref{funclem}(1).
\end{proof}

The construction $\underline{\Fun}_S(-,-)$ does not make homotopical sense when the first variable is not fibrant, so it does not yield a Quillen bifunctor. Nevertheless, we can say the following about varying the first variable.

\begin{prp} \label{prp:FunctorFirstVariable} Let $K$, $L$, and $C$ be fibrant marked simplicial sets over $S$, let $f: K \to L$ be a map and let
\[ f^\ast: \underline{\Fun}_S(L,C) \to \underline{\Fun}_S(K,C) \]
denote the induced map.
\begin{enumerate}
\item Suppose that $f$ is a cocartesian equivalence over $S$. Then $f^\ast$ is a cocartesian equivalence over $S$.
\item Suppose that $f$ is a cofibration. Then $f^\ast$ is a fibration in $s\Set^+_{/S}$.
\end{enumerate}
\end{prp}

\begin{proof} (1): It suffices to check that for all $s \in S$, $f^\ast$ induces a categorical equivalence between the fibers over $s$, i.e. that
\[ \Fun_S((S^{s/})^\sharp \times_S L,C) \to \Fun_S((S^{s/})^\sharp \times_S K, C) \]
is a categorical equivalence. Our assumption implies that $(S^{s/})^\sharp \times_S K \to (S^{s/})^\sharp \times_S L$ is a cocartesian equivalence over $S$, so this holds.

(2): For any trivial cofibration $A \to B$ in $s\Set^+_S$, we need to check that
\[ A \times_S \sO(S) \times_S L \bigsqcup_{A \times_S \sO(S) \times_S K} B \times_S \sO(S) \times_S K \to B \times_S \sO(S) \times_S L \]
is a trivial cofibration in $s\Set^+_{/S}$. By Proposition~\ref{func}, $- \times_S \sO(S) \times_S K$ preserves trivial cofibrations and ditto for $L$. The result then follows.
\end{proof}

A final word on notation: since $\underline{\Fun}_S(-,-)$ is only homotopically meaningful (and fibrant) when both variables are fibrant, we will henceforth cease to denote the markings on the variables.

\subsection*{\texorpdfstring{$S$}{S}-categories of \texorpdfstring{$S$}{S}-objects}
For the convenience of the reader, we briefly review the construction and basic properties of the $S$-category of $S$-objects in an $\infty$-category $C$. This is a construction, at the level of marked simplicial sets, of the right adjoint to the Grothendieck construction functor\footnote{We write $\Gr_U$ to distinguish from Notation~\ref{ntn:groth}.}
\[ \Gr_U: \Cat^\cocart_{\infty/S} \to \Cat_\infty, \quad (C \to S) \mapsto C. \]
This material is originally due to Denis Nardin in \cite[\S 7]{BDGNS1}.

\begin{cnstr}[{\cite[Definition~7.4]{BDGNS1}}] \label{constr:Sobjects} The span
\[ \begin{tikzcd}[row sep = 2em, column sep = 2em]
S^\sharp & \rightnat{\sO(S)} \ar{l}[swap]{\ev_0} \ar{r}{\rho} & \Delta^0 
\end{tikzcd} \]
defines a right Quillen functor $(\ev_0)_* \rho^* : s\Set^+ \to s\Set^+_{/S}$, which sends an $\infty$-category $E$ to $\widetilde{\Fun}_S(\sO(S),E \times S)$ (cf. Example~\ref{exm:pairing}). This is the \emph{$S$-category of objects in $E$}, which we will denote by $\underline{E}_S$.
\end{cnstr}

\nomenclature[categoryOfObjects]{$\underline{C}_S$}{$S$-category of objects in an $\infty$-category $C$}

The next proposition shows that the functor $E \mapsto \underline{E}_S$ indeed implements the right adjoint to $\Gr_U$.

\begin{prp} \label{prp:UMPofSCatOfObjects} Suppose $C$ a $S$-category and $E$ an $\infty$-category. Then we have an equivalence
\begin{align*} \psi: \Fun_S(C, \underline{E}_S) &\xrightarrow{\simeq} \Fun(C,E)
\end{align*}
\end{prp}
\begin{proof} Consider the commutative diagram
\[ \begin{tikzcd}[row sep=2em, column sep=2em]
C^\sim \ar{r} \ar{d} & \rightnat{\sO(S)} \ar{r} \ar{d} & \Delta^0 \:. \\
\leftnat{C} \ar{r} \ar{d} & S^\sharp \\
\Delta^0
\end{tikzcd} \]
Given an $\infty$-category $E$, travelling along the outer span (i.e., pulling back and then pushing forward) yields $\Fun(C,E)$, travelling along the two inner spans yields $\Fun_S(C, \underline{E}_S)$, and the comparison functor $\psi$ is induced by the map $\iota: C^\sim \to \leftnat{C} \times_S \rightnat{\sO(S)}$. By \cite[Proposition~6.2]{BDGNS1}, $\iota$ is a homotopy equivalence in $s\Set^+_{/S}$. Therefore, combining Lemma~\ref{lm:SpanComposition} and Lemma~\ref{lm:SpanHomotopyInvariance}, we deduce the claim.
\end{proof}

\begin{ntn} \label{ntn:CatOfObjectsCorresp} Given a $S$-functor $p: C \to \underline{E}_S$, let $p^{\dagger}: C \to E$ denote the corresponding functor under the equivalence of Proposition~\ref{prp:UMPofSCatOfObjects}.
\end{ntn}

\nomenclature[dagger]{$p^\dagger$}{Corresponding functor under universal mapping property of $\underline{C}_S$}

\begin{exm} \label{exm:Gobjects} Let $E = \Top$ or $\Cat_\infty$. Then $\underline{\Top}_S$ resp. $\underline{\Cat}_{\infty,S}$ is the $S$-category of $S$-spaces resp. $S$-categories. In particular, suppose $E = \Top$ and $S = \OO_G^\op$. Then we also call $\underline{\Top}_{\OO_G^\op}$ the $G$-$\infty$-category of $G$-spaces. Note that the fiber of this cocartesian fibration over a transitive $G$-set $G/H$ is equivalent to the $\infty$-category of $H$-spaces $\Fun(\OO_H^\op, \Top)$, and the pushforward functors are given by restriction along a subgroup and conjugation.
\end{exm}

\begin{rem}
Let $C$ be an $S$-category and $\pi: X \to C$ a left fibration. Then $\pi$ straightens to a functor $F: C \to \Spc$, which under the equivalence of Proposition~\ref{prp:UMPofSCatOfObjects} corresponds to a $S$-functor $F': C \to \underline{\Spc}_S$. We will say that $\pi$ \emph{$S$-straightens} to $F'$. Similarly, if $\pi$ is a cocartesian fibration, then $\pi$ $S$-straightens to a $S$-functor valued in $\underline{\Cat}_{\infty,S}$.
\end{rem}

\section{Join and slice}

The join and slice constructions are at the heart of the $\infty$-categorical approach to limits and colimits. In this section, we introduce relative join and slice constructions and explore their homotopical properties.

\subsection*{The \texorpdfstring{$S$}{S}-join}

\begin{dfn} \label{dfn:join} Let $\iota: S \times \partial \Delta^1 \tohook S \times \Delta^1$ be the inclusion. Define the $S$-\emph{join} to be the functor 
\begin{equation*} (- \star_S -) \coloneq \iota_\ast: s\Set_{/ S \times \partial \Delta^1} \to s\Set_{/ S \times \Delta^1}.
\end{equation*}

Define the \emph{marked} $S$-\emph{join} to be the functor
\begin{equation*} (- \star_S -) \coloneq \iota_\ast: s\Set^+_{/ S^\sharp \times (\partial \Delta^1)^\flat} \to s\Set^+_{/ S^\sharp \times (\Delta^1)^\flat}.
\end{equation*}

\end{dfn}

\nomenclature[join]{$X \star_S Y$}{$S$-join}

\begin{ntn} Given $X,Y$ marked simplicial sets over $S$, we will usually refer to the structure maps to $S$ by $\pi_1: X \to S$, $\pi_2: Y \to S$, and $\pi: X \star_S Y \to S$. Explicitly, a $(i+j+1)$-simplex $\lambda$ of $X \star_S Y$ is the data of simplices $\sigma: \Delta^i \to X$, $\tau: \Delta^j \to Y$, and $\lambda': \Delta^i \star \Delta^j \to S$ such that the diagram

\begin{equation*}
\begin{tikzpicture}[baseline]
\matrix(m)[matrix of math nodes,
row sep=4ex, column sep=4ex,
text height=1.5ex, text depth=0.25ex]
 { \Delta^i & \Delta^i \star \Delta^j & \Delta^j \\
 X & S & Y \\ };
\path[>=stealth,->,font=\scriptsize]
(m-1-1) edge (m-1-2)
edge node[right]{$\sigma$} (m-2-1)
(m-1-2) edge node[right]{$\lambda'$} (m-2-2)
(m-1-3) edge node[right]{$\tau$} (m-2-3)
edge (m-1-2)
(m-2-1) edge node[above]{$\pi_1$} (m-2-2)
(m-2-3) edge node[above]{$\pi_2$} (m-2-2);
\end{tikzpicture}
\end{equation*}

commutes; we then have that $\lambda' = \pi \circ \lambda$. We will sometimes write $\lambda = (\sigma,\tau)$ so as to remember the data of the $i$-simplex of $X$ and the $j$-simplex of $Y$ in the notation. If given an $n$-simplex of $X \star_S Y$, we will indicate the decomposition of $\Delta^n$ given by the structure map to $\Delta^1$ as $\Delta^{n_0} \star \Delta^{n_1}$ (with either side possibly empty).
\end{ntn}

\begin{prp} \label{joinrt} Let $\iota: S \times \partial \Delta^1 \tohook S \times \Delta^1$ be the inclusion. Then
\begin{enumerate}
\item[(a)] $\iota_\ast : s\Set_{/ S \times \partial \Delta^1} \to s\Set_{/ S \times \Delta^1}$ is a right Quillen functor.
\item[(b)] $\iota_\ast : s\Set^+_{/ S^\sharp \times (\partial \Delta^1)^\flat} \to s\Set^+_{/ S^\sharp \times (\Delta^1)^\flat}$ is a right Quillen functor.
\end{enumerate}
Consequently, if $X$ and $Y$ are categorical resp. cocartesian fibrations over $S$, then $X \star_S Y$ is a categorical resp. cocartesian fibration over $S$, with the cocartesian edges given by those in $X$ and $Y$.
\end{prp}
\begin{proof} For (b), we verify the hypotheses of Theorem~\ref{thm:FunctorialityOfCocartesianModelStructure}(ii). All of the requirements are immediate except for (1) and (2).

(1): Let $(s,i)$ be a vertex of $S^\sharp \times (\partial \Delta^1)^\flat$, $i=0$ or $1$. Let $f: (s',i') \to (s,i)$ be a marked edge in $S^\sharp \times (\Delta^1)^\flat$. Then $i' = i$ and $f$ viewed as an edge in $S^\sharp \times (\partial \Delta^1)^\flat$ is locally $\iota$-cartesian.

(2): It is obvious that $\partial \Delta^1 \tohook \Delta^1$ is a flat categorical fibration, so by stability of flat categorical fibrations under base change, $S \times \partial \Delta^1 \tohook S \times \Delta^1$ is a flat categorical fibration.

(a) also follows from (2) by \cite[Proposition~B.4.5]{HA}. By (a), if $X$ and $Y$ are categorical fibrations over $S$, $X \star_S Y$ is a categorical fibration over $S \times \Delta^1$. The projection map $S \times \Delta^1 \to S$ is a categorical fibration, so $X \star_S Y$ is also a categorical fibration over $S$. By (b), if $X$ and $Y$ are cocartesian fibrations over $S$, $\leftnat{X} \star_S \leftnat{Y}$ is fibrant in $s\Set^+_{/ S^\sharp \times (\Delta^1)^\flat}$. Since $S^\sharp \times (\Delta^1)^\flat$ is marked as a cocartesian fibration over $S$, $\leftnat{X} \star_S \leftnat{Y}$ is marked as a cocartesian fibration over $S$.
\end{proof}

We have the compatibility of the relative join with base change.

\begin{lem} Let $f: T \to S$ be a functor and let $X$ and $Y$ be (marked) simplicial sets over $S$. Then we have a canonical isomorphism
\[ (X \star_S Y) \times_S T \cong (X \times_S T) \star_T (Y \times_S T). \]
\end{lem}
\begin{proof} From the pullback square
\[ \begin{tikzcd}[row sep=2em, column sep=2em]
T \times \partial \Delta^1 \ar{r}{\iota_T} \ar{d}{f \times \id} & T \times \Delta^1 \ar{d}{f \times \id} \\
S \times \partial \Delta^1 \ar{r}{\iota_S} & S \times \Delta^1
\end{tikzcd} \]
we obtain the base-change isomorphism $f^\ast (\iota_S)_\ast \cong (\iota_T)_\ast f^\ast$.
\end{proof}

In \cite[\S 4.2.2]{HTT}, Lurie introduces the relative `diamond' join operation $\diamond_S$, which we now recall. Given $X$ and $Y$ marked simplicial sets over $S$, define
\[ X \diamond_S Y = X \sqcup_{X \times_S Y \times \{0\}} X \times_S Y \times (\Delta^1)^\flat \sqcup_{X \times_S Y \times \{1\}} Y.\] 
There is a comparison map
$$\psi_{(X,Y)}: X \diamond_S Y \to X \star_S Y = \iota_\ast(X,Y),$$
adjoint to the isomorphism $\iota^\ast(X \star_S Y) \cong (X,Y)$.

\begin{lem} \label{lm:DiamondJoinComparison} Let $X$ be a marked simplicial set. Then $\psi_{(X,S)}: X \diamond_S S^\sharp \to X \star_S S^\sharp$ is a cocartesian equivalence in $s\Set^+_{/S}$. Dually, if $X$ is in addition fibrant, then $\psi_{(S,X)}: S^{\sharp} \diamond_S X \to S^{\sharp} \star_S X$ is a cocartesian equivalence in $s\Set^+_{/S}$.
\end{lem}
\begin{proof} We first address the map $\psi_{(X,S)}$. By left properness of the cocartesian model structure, the defining pushout for $X \diamond_S S^{\sharp}$ is a homotopy pushout. By Theorem~\ref{thm:JoinFirstVariable},\footnote{Note there is no circularity since Lemma~\ref{lm:DiamondJoinComparison} is only later referenced in this paper at the beginning of \S \ref{sec:paramcolimits}.} $- \star_S S^{\sharp}$ preserves cocartesian equivalences. Therefore, choosing a fibrant replacement for $X$ and using naturality of the comparison map $\psi_{(X,S)}$, we may reduce to the case that $X$ is fibrant. Then we have to check that
\[ \begin{tikzcd}[row sep=2em, column sep=2em]
X \times \{1\} \ar{r} \ar{d} & X \times (\Delta^1)^\flat \ar{d} \\
S^\sharp \ar{r} & X \star_S S^\sharp
\end{tikzcd} \]
is a homotopy pushout square. Since this is a square of fibrant objects, this assertion can be checked fiberwise, in which case it reduces to the equivalence $X_s \diamond \Delta^0 \xrightarrow{\simeq} X^\rhd$ of \cite[Proposition~4.2.1.2]{HTT}.

The second statement concerning $\psi_{(S,X)}$ follows by the same type of argument, but without the reduction step.
\end{proof}

\begin{wrn} In general, $\psi_{(X,Y)}$ is not a cocartesian equivalence. As a counterexample, consider $S = \Delta^1$, $X = \{0\}$, and $Y = \{1\}$. Then $\psi_{(X,Y)}$ is the inclusion of $X \diamond_S Y \cong \Delta^{\{0\}} \sqcup \Delta^{\{1\}}$ into $X \star_S Y \cong \Delta^1$, which is not a cocartesian equivalence over $\Delta^1$.
\end{wrn}

We will later need the following strengthening of the conclusion of Proposition~\ref{joinrt}.
\begin{prp} \label{joinIsSCategory}
	\begin{enumerate}[leftmargin=*]
		\item Let $C,C',D \to S$ be inner fibrations and let $C,C' \to D$ be functors over $S$.

		\hspace{.15ex} Then $C \star_D C' \to S$ is an inner fibration.
	\end{enumerate}
	\begin{enumerate} \setcounter{enumi}{1}
		\item Let $C,C',D \to S$ be $S$-categories and let $C,C' \to D$ be $S$-functors. Then $C \star_D C' \to S$ is a $S$-category with cocartesian edges given by those in $C$ or $C'$, and $C \star_D C' \to D$ is a $S$-functor.
\end{enumerate}
\end{prp}
\begin{proof}
\begin{enumerate}[leftmargin=*]
\item Let $0<k<n$. We need to solve the lifting problem
\[ \begin{tikzpicture}[baseline]
\matrix(m)[matrix of math nodes,
row sep=6ex, column sep=4ex,
text height=1.5ex, text depth=0.25ex]
 { \Lambda^n_k & C \star_D C'  \\
  \Delta^n & S. \\ };
\path[>=stealth,->,font=\scriptsize]
(m-1-1) edge node[above]{$\lambda_0$} (m-1-2)
edge (m-2-1)
(m-1-2) edge (m-2-2)
(m-2-1) edge (m-2-2)
edge[dotted] node[above]{$\lambda$} (m-1-2);
\end{tikzpicture} \]
Let $\overline{\lambda}: \Delta^n \to D$ be a lift in the commutative square
\[ \begin{tikzpicture}[baseline]
\matrix(m)[matrix of math nodes,
row sep=6ex, column sep=4ex,
text height=1.5ex, text depth=0.25ex]
 { \Lambda^n_k & D  \\
  \Delta^n & S. \\ };
\path[>=stealth,->,font=\scriptsize]
(m-1-1) edge (m-1-2)
edge (m-2-1)
(m-1-2) edge (m-2-2)
(m-2-1) edge (m-2-2)
edge node[above]{$\overline{\lambda}$} (m-1-2);
\end{tikzpicture} \]
Define $\lambda$ using the data $(\lambda_0|_{\Delta^{n_0}}, \lambda_0|_{\Delta^{n_1}}, \overline{\lambda})$. Then $\lambda$ is a valid lift.

\item Consider $C \star_D C'$ as a marked simplicial set with marked edges those in $\leftnat{C}$ or in $\leftnat{C'}$. We need to solve the lifting problem
\[ \begin{tikzpicture}[baseline]
\matrix(m)[matrix of math nodes,
row sep=6ex, column sep=4ex,
text height=1.5ex, text depth=0.25ex]
 { \leftnat{\Lambda^n_0} & C \star_D C'  \\
  \leftnat{\Delta^n} & S. \\ };
\path[>=stealth,->,font=\scriptsize]
(m-1-1) edge node[above]{$\lambda_0$} (m-1-2)
edge (m-2-1)
(m-1-2) edge (m-2-2)
(m-2-1) edge (m-2-2)
edge[dotted] node[above]{$\lambda$} (m-1-2);
\end{tikzpicture} \]
Let $\overline{\lambda}: \Delta^n \to D$ be a lift in the commutative square
\[ \begin{tikzpicture}[baseline]
\matrix(m)[matrix of math nodes,
row sep=6ex, column sep=4ex,
text height=1.5ex, text depth=0.25ex]
 { \leftnat{\Lambda^n_0} & \leftnat{D}  \\
  \leftnat{\Delta^n} & S. \\ };
\path[>=stealth,->,font=\scriptsize]
(m-1-1) edge (m-1-2)
edge (m-2-1)
(m-1-2) edge (m-2-2)
(m-2-1) edge (m-2-2)
edge node[above]{$\overline{\lambda}$} (m-1-2);
\end{tikzpicture} \]
Define $\lambda$ using the data $(\lambda_0|_{\Delta^{n_0}}, \lambda_0|_{\Delta^{n_1}}, \overline{\lambda})$. Then $\lambda$ is a valid lift. Finally, note that we may obviously lift against classes (3) and (4) of \cite[Definition~3.1.1.1]{HTT}. We conclude that $C \star_D C' \to S$ is fibrant in $s\Set^+_{/S}$, hence an $S$-category with cocartesian edges as marked.
\end{enumerate}
\end{proof}

Since the $S$-join is defined as a right Kan extension, it is simple to map into. In the other direction, we can offer the following lemma.
\begin{lem} \label{bifibration} Let $C$, $C'$, $D$, and $E$ be $S$-categories and let $C,C' \to D$ be $S$-functors. Then
\[ \Fun_S(C \star_D C', E) \to \Fun_S(C,E) \times \Fun_S(C',E) \]
is a bifibration \cite[Definition~2.4.7.2]{HTT}. Consequently,
\[ \Fun_S(C \star_D C', E) \to \Fun_S(C,E) \]
 is a cartesian fibration with cartesian edges those sent to equivalences in $\Fun_S(C',E)$, and
 \[ \Fun_S(C \star_D C', E) \to \Fun_S(C',E) \]
 is a cocartesian fibration with cocartesian edges those sent to equivalences in $\Fun_S(C',E)$.
\end{lem}
\begin{proof} By inspection, the span
\[ (\Delta^1)^\flat \xleftarrow{\pi} \leftnat{(C \star_D C')} \xrightarrow{\pi'} S^\sharp \]
satisfies the hypotheses of Theorem~\ref{thm:FunctorialityOfCocartesianModelStructure}. Therefore, $\pi_\ast \pi'^\ast (\leftnat{E} \to S)$ is a categorical fibration over $\Delta^1$. The claim now follows from \cite[Proposition~2.4.7.10]{HTT}, and the consequence from \cite[Lemma~2.4.7.5]{HTT} and its opposite.
\end{proof}

\subsection*{The Quillen adjunction between \texorpdfstring{$S$}{S}-join and \texorpdfstring{$S$}{S}-slice}

Our next goal is to obtain a relative join and slice Quillen adjunction. To this end, we need a good understanding of the combinatorics of the relative join (Proposition~\ref{joinprp}). We prepare for the proof of that proposition with a few lemmas.

\begin{lem}
\label{lem1}
Let $i,l \geq -1$ and $j,k \geq 0$. Then the map
\begin{equation*} \Delta^i \star \Delta^j \star \partial \Delta^k \star \Delta^l \bigsqcup_{\Delta^j \star \partial \Delta^k \star \Delta^l} \Delta^{j+k+l+2} \tohook \Delta^{i+j+k+l+3} 
\end{equation*}
is inner anodyne.
\end{lem}
\begin{proof} Let $f: \Delta^{j-1} \tohook \Delta^i \star \Delta^{j-1}$ and $g: \Lambda^{k+1}_0 \tohook \Delta^{k+1}$. The map in question is $f \star g \star \Delta^l$, so is inner anodyne by \cite[Lemma~2.1.2.3]{HTT}.
\end{proof}

By \cite[Lemma~2.1.2.4]{HTT}, the join of a left anodyne map and an inclusion is left anodyne. We need a slight refinement of this result:

\begin{lem}
\label{markedAnodyneBasicJoinLemma}
Let $f: A_0 \tohook A$ be a cofibration of simplicial sets.

\begin{enumerate}
\item Let $g: B_0 \tohook B$ be a right marked anodyne map between marked simplicial sets. Then
\begin{equation*}
f^\flat \star g: A_0^\flat \star B \bigsqcup_{A_0^\flat \star B_0} A^\flat \star B_0 \tohook A^\flat \star B
\end{equation*}
is a right marked anodyne map. 

\item Let $g: B_0 \tohook B$ be a left marked anodyne map between marked simplicial sets. Then
\begin{equation*}
g \star f^\flat: B \star A_0^\flat \bigsqcup_{B_0 \star A_0^\flat} B_0 \star A^\flat \tohook B \star A^\flat
\end{equation*}
is a left marked anodyne map. 
\end{enumerate}
\end{lem}
\begin{proof} We prove (1); the dual assertion (2) is proven by a similar argument. $f$ lies in the weakly saturated closure of the inclusions $i_m: \partial \Delta^m \tohook \Delta^m$, so it suffices to check that $i_m^\flat \star g$ is right marked anodyne for the four classes of morphisms enumerated in \cite[Definition~3.1.1.1]{HTT}. For $g: (\Lambda^n_i)^\flat \tohook (\Delta^n)^\flat$, $0<i<n$, $i_m^\flat \star g$ obtained from an inner anodyne map by marking common edges, so is marked right anodyne. For $g: \rightnat{\Lambda^n_n} \tohook \rightnat{\Delta^n}$, $i_m^\flat \star g$ is $\rightnat{\Lambda^{n+m+1}_{n+m+1}} \tohook \rightnat{\Delta^{n+m+1}}$, so $i_m^\flat \star g$ is marked right anodyne. For the remaining two classes, $i_m^\flat \star g$ is the identity because no markings are introduced when joining two marked simplicial sets.
\end{proof}

The following proposition reveals a basic asymmetry of the relative join, which is related to our choice of \emph{cocartesian} fibrations to model functors.

\begin{prp} Let $K$ be a marked simplicial set over $S$. 
\label{joinprp}
\begin{itemize}

\item [(1)] For every marked left horn inclusion $\leftnat{\Lambda^n_0} \tohook \leftnat{\Delta^n}$ over $S$, the induced map
\[ K \star_S (\leftnat{\Lambda^n_0} \times_S \rightnat{\mathscr{O}(S)}) \tohook K \star_S (\leftnat{\Delta^n} \times_S \rightnat{\mathscr{O}(S)}) \]
is left marked anodyne, where the pullbacks $\leftnat{\Lambda^n_0} \times_S \rightnat{\mathscr{O}(S)}$ and $\leftnat{\Delta^n} \times_S \rightnat{\mathscr{O}(S)}$ are formed with respect to the source map $e_0$ and are regarded as marked simplicial sets over $S$ via the target map $e_1$.

\item [(1\textquotesingle)] For every left horn inclusion $\Lambda^n_0 \tohook \Delta^n$ over $S$, the induced map
\[ \Delta^n \times_S \mathscr{O}(S) \bigsqcup_{\Lambda^n_0 \times_S \mathscr{O}(S)} K \star_S (\Lambda^n_0 \times_S \mathscr{O}(S)) \tohook K \star_S (\Delta^n \times_S \mathscr{O}(S)) \]
is an inner anodyne map.

\item [(2)] Let $e_0: C \to S$ be a cartesian fibration over $S$ and let $e_1: C \to S$ be any map of simplicial sets. For every inner horn inclusion $\Lambda^n_k \tohook \Delta^n$, $0<k<n$ over $S$, the induced map
\[ K \star_S (\Lambda^n_k \times_S C) \tohook K \star_S (\Delta^n \times_S C) \]
is inner anodyne, where the pullbacks $\Lambda^n_k \times_S C$ and $\Delta^n \times_S C$ are formed with respect to $e_0$ and are regarded as simplicial sets over $S$ via $e_1$.

\item [(3)] For every marked right horn inclusion $\rightnat{\Lambda^n_n} \tohook \rightnat{\Delta^n}$ over $S$, the induced map
\[ K \star_S \rightnat{\Lambda^n_n} \tohook K \star_S \rightnat{\Delta^n} \]
is right marked anodyne.
\end{itemize}
\end{prp}

\begin{proof}
Let $I$ be the set of simplices of $K$ endowed with a total order such that $\sigma < \sigma'$ if the dimension of $\sigma$ is less than that of $\sigma'$, where we view the empty set as a simplex of dimension $-1$. Let $J$ be the set of epimorphisms $\chi: \onto{\Delta^j}{\Delta^{n-1}}$ endowed with a total order such that $\chi < \chi'$ if the dimension of $\chi$ is less than that of $\chi'$. Order $I \times J$ by $(\sigma,\chi)<(\sigma',\chi')$ if $\sigma < \sigma'$ or $\sigma = \sigma'$ and $\chi < \chi'$. For any simplex $\tau: \Delta^j \to \Delta^n$, we let $r_k(\tau)$ be the pullback

\begin{equation*}
\begin{tikzpicture}[baseline]
\matrix(m)[matrix of math nodes,
row sep=4ex, column sep=4ex,
text height=1.5ex, text depth=0.25ex]
{ \Delta^{r_k(\tau)_0} & \Delta^{n-1} \\
\Delta^j & \Delta^n \\ };
\path[>=stealth,->,font=\scriptsize]
(m-1-1) edge node[above]{$r_k(\tau)$} (m-1-2)
edge (m-2-1)
(m-1-2) edge node[right]{$d_k$} (m-2-2)
(m-2-1) edge node[above]{$\tau$} (m-2-2);
\end{tikzpicture}
\end{equation*}

We will let $\iota$ denote the map under consideration. We first prove (1). Given $\sigma \in I$ and $\chi \in J$, let $X_{\sigma,\chi}$ be the sub-marked simplicial set of $K \star_S (\leftnat{\Delta^n} \times_S \rightnat{\mathscr{O}(S)})$ on $K \star_S (\leftnat{\Lambda^n_0} \times_S \rightnat{\mathscr{O}(S)})$ and simplices $(\sigma',\tau'): \Delta^i \star \Delta^j \to K \star_S (\Delta^n \times_S \mathscr{O}(S))$ not in $K \star_S (\Lambda^n_0 \times_S \mathscr{O}(S))$ with $(\sigma',r_0(e_0 \circ \tau')) \leq (\sigma,\chi)$. If $(\sigma,\chi)<(\sigma',\chi')$, then we have an obvious inclusion $\into{X_{\sigma,\chi}}{X_{\sigma',\chi'}}$, and we let
\begin{equation*}
X_{<(\sigma,\chi)} = (\leftnat{\Lambda^n_0} \times_S \rightnat{\mathscr{O}(S)}) \bigcup (\cup_{(\sigma',\chi')<(\sigma,\chi)} X_{\sigma',\chi'}).
\end{equation*}

Since $K \star_S (\leftnat{\Delta^n} \times_S \rightnat{\mathscr{O}(S)}) = \colim_{(\sigma,\chi) \in I \times J} X_{\sigma,\chi}$, in order to show that $\iota$ is left marked anodyne it suffices to show that $\into{X_{<(\sigma,\chi)}}{X_{\sigma,\chi}}$ is left marked anodyne for all $(\sigma,\chi) \in I \times J$. We will say that a simplex of $X_{\sigma,\chi}$ is \emph{new} if it does not belong to $X_{<(\sigma,\chi)}$.

Let $\sigma: \Delta^i \to K$ be an element of $I$ and $\chi: \onto{\Delta^j}{\Delta^{n-1}}$ an element of $J$. Let $\lambda = (\sigma,\tau): \Delta^i \star \Delta^j \to K \star_S (\Delta^n \times_S \mathscr{O}(S))$ be any nondegenerate new simplex of $X_{\sigma,\chi}$, so $r_0(e_0 \circ \tau) = \chi$. Let $\bar{\chi}: \onto{\Delta^{j+1}}{\Delta^n}$ be the unique epimorphism with $r_0(\bar{\chi}) = \chi$ and let $e: \Delta^1 \to \Delta^n \times_S \mathscr{O}(S)$ be a cartesian edge over $\{0, 1\}$ with $e(1) = \tau(0)$. The inclusion $\into{(\Delta^1)^\sharp \bigsqcup_{\Delta^0} \Delta^j}{\leftnat{\Delta^{j+1}}}$ is right marked anodyne, so we have a lift $\bar{\tau}$ in the following diagram

\begin{equation*}
\begin{tikzpicture}[baseline]
\matrix(m)[matrix of math nodes,
row sep=4ex, column sep=4ex,
text height=1.5ex, text depth=0.25ex]
 { \Delta^1 \bigsqcup_{\Delta^0} \Delta^j & \Delta^n \times_S \mathscr{O}(S) \\
 \Delta^{j+1} & \Delta^n. \\ };
\path[>=stealth,->,font=\scriptsize]
(m-1-1) edge node[above]{$e \cup \tau$} (m-1-2)
edge (m-2-1)
(m-1-2) edge (m-2-2)
(m-2-1) edge node[above]{$\bar{\chi}$} (m-2-2)
edge[dotted] node[above]{$\bar{\tau}$} (m-1-2);
\end{tikzpicture}
\end{equation*}

By Lemma~\ref{markedAnodyneBasicJoinLemma},
\begin{equation*}
\into{\Delta^i \star \Delta^j \bigsqcup_{\Delta^j} \leftnat{\Delta^{j+1}}}{\Delta^i \star \leftnat{\Delta^{j+1}}}
\end{equation*}
is right marked anodyne. Using that $(e_1 \circ \bar{\tau})(e)$ is an equivalence, we obtain a lift
\begin{equation*}
\begin{tikzpicture}[baseline]
\matrix(m)[matrix of math nodes,
row sep=4ex, column sep=6ex,
text height=1.5ex, text depth=0.25ex]
 { \Delta^i \star \Delta^j \bigsqcup_{\Delta^j} \leftnat{\Delta^{j+1}} & S^\sim \\
 \Delta^i \star \leftnat{\Delta^{j+1}} \\ };
\path[>=stealth,->,font=\scriptsize]
(m-1-1) edge node[above]{$\pi \lambda \cup e_1 \bar{\tau}$} (m-1-2)
edge (m-2-1)
(m-2-1) edge[dotted] (m-1-2);
\end{tikzpicture}
\end{equation*}
which allows us to define $\bar{\lambda}: \Delta^i \star \Delta^{j+1} \to K \star_S (\Delta^n \times_S \mathscr{O}(S))$ extending $\lambda$ and $\bar{\tau}$. Then $\bar{\lambda}$ is a nondegenerate new simplex of $X_{\sigma,\chi}$ and every face of $\bar{\lambda}$ except for $\lambda = d_{i+1}(\bar{\lambda})$ lies in $X_{<(\sigma,\chi)}$. We may thus form the pushout
\begin{equation*}
\begin{tikzpicture}[baseline]
\matrix(m)[matrix of math nodes,
row sep=4ex, column sep=4ex,
text height=1.5ex, text depth=0.25ex]
 { \bigsqcup_{\lambda} (\Lambda^{i+j+2}_{i+1}, \{i+1,i+2\}) & X_{<(\sigma,\chi)} \\
 \bigsqcup_{\lambda} (\Delta^{i+j+2}, \{i+1,i+2\})  & X_{<(\sigma,\chi),1} \\ };
\path[>=stealth,->,font=\scriptsize]
(m-1-1) edge (m-1-2)
edge (m-2-1)
(m-1-2) edge (m-2-2)
(m-2-1) edge (m-2-2);
\end{tikzpicture}
\end{equation*}
which factors the inclusion $\into{X_{<(\sigma,\chi)}}{X_{(\sigma,\chi)}}$ as the composition of a left marked anodyne map and an inclusion (there is one further complication involving markings: in the special case $n=1$, $\sigma = \emptyset$, $j=1$, we may have that $\lambda = \tau$ is a marked edge, i.e. an equivalence over $1$. Then the edges of $\bar{\tau}$ are all marked, so we should form the pushout via maps $\into{(\Lambda^2_0)^\sharp}{(\Delta^2)^\sharp}$, which are left marked anodyne by \cite[Corollary~3.1.1.7]{HTT}).

Now for the inductive step suppose that we have defined a sequence of left marked anodyne maps
\begin{equation*} X_{<(\sigma,\chi)} \tohook \ldots \tohook X_{<(\sigma,\chi),m} \subset X_{(\sigma,\chi)}
\end{equation*}
such that for all $0<l \leq m$ all new nondegenerate simplices in $X_{(\sigma,\chi)}$ of dimension $i+l+j$ lie in $X_{<(\sigma,\chi),l}$ and admit an extension to a $i+l+j+1$-simplex with the edge $\{i+l,i+l+1\}$ marked in $X_{<(\sigma,\chi),l}$, and no new nondegenerate simplices of dimension $>i+l+j+1$ lie in $X_{<(\sigma,\chi),l}$. Let $\lambda = (\sigma,\tau)$ be any new nondegenerate $i+m+j+1$-simplex not in $X_{<(\sigma,\chi),m}$. For $0 \leq l < m$ let $\lambda_l = (\sigma,\tau_l)$ be a nondegenerate $i+m+j+1$-simplex in $X_{<(\sigma,\chi),m}$ with $d_{i+m}(\lambda_l)=d_{i+l+1}(\lambda)$ and edge $\{ i+m,i+m+1\}$ marked. $\tau$ and $\tau_0,...,\tau_{m-1}$ together define a map
\begin{equation*}
\tau': \Lambda^{m+1}_{m+1} \star \Delta^{j-1} \to \Delta^n \times_S \mathscr{O}(S)
\end{equation*}
where the domain of $\tau$ is the subset $\{0,...,m-1,m+1,...,m+j+1\}$ and the domain of $\tau_l$ is the subset $\{0,...,\widehat{l},...,m+j+1\}$. Observe that the map $\into{\rightnat{\Lambda^{m+1}_{m+1}} \star \Delta^{j-1}}{\rightnat{\Delta^{m+1}} \star \Delta^{j-1}}$ is right marked anodyne, since it factors as
\begin{equation*} \rightnat{\Lambda^{m+1}_{m+1}} \star \Delta^{j-1} \tohook \rightnat{\Delta^{m+1}} \bigsqcup_{\rightnat{\Lambda^{m+1}_{m+1}}} \rightnat{\Lambda^{m+1}_{m+1}} \star \Delta^{j-1} \tohook \rightnat{\Delta^{m+1}} \star \Delta^{j-1}
\end{equation*}
where the first map is obtained as the pushout of the right marked anodyne map $\rightnat{\Lambda^{m+1}_{m+1}} \tohook \rightnat{\Delta^{m+1}}$ along the inclusion $\rightnat{\Lambda^{m+1}_{m+1}} \tohook \rightnat{\Lambda^{m+1}_{m+1}} \star \Delta^{j-1}$ and the second map is obtained by marking a common edge of an inner anodyne map. Let $\bar{\chi}: \Delta^{m+j+1} \toepi \Delta^n$ be the unique epimorphism with $r_0(\bar{\chi})=\chi$. Then we have a lift $\bar{\tau}$ in the following commutative diagram

\begin{equation*}
\begin{tikzpicture}[baseline]
\matrix(m)[matrix of math nodes,
row sep=4ex, column sep=4ex,
text height=1.5ex, text depth=0.25ex]
 { \Lambda^{m+1}_{m+1} \star \Delta^{j-1} & \Delta^n \times_S \mathscr{O}(S) \\
 \Delta^{m+1} \star \Delta^{j-1} & \Delta^n. \\ };
\path[>=stealth,->,font=\scriptsize]
(m-1-1) edge node[above]{$\tau'$} (m-1-2)
edge (m-2-1)
(m-1-2) edge (m-2-2)
(m-2-1) edge node[above]{$\bar{\chi}$} (m-2-2)
edge[dotted] node[above]{$\bar{\tau}$} (m-1-2);
\end{tikzpicture}
\end{equation*}

By Lemma~\ref{markedAnodyneBasicJoinLemma}, the map
\begin{equation*} \Delta^i \star \rightnat{\Lambda^{m+1}_{m+1}} \star \Delta^{j-1} \bigsqcup_{\rightnat{\Lambda^{m+1}_{m+1}} \star \Delta^{j-1}} \rightnat{\Delta^{m+1}} \star \Delta^{j-1} \tohook \Delta^i \star \rightnat{\Delta^{m+1}} \star \Delta^{j-1}
\end{equation*}
is right marked anodyne. Since $(e_1 \circ \bar{\tau})(\{m,m+1\})$ is an equivalence, we may extend $(\cup_l \pi \lambda_l) \cup \pi \lambda \cup e_1 \bar{\tau}$ to a map $\Delta^{i+m+j+2} \to S$, which defines a nondegenerate $(i+m+j+2)$-simplex $\bar{\lambda}$ with $\lambda$ as its $(i+m+1)$th face and which extends $\bar{\tau}$. By construction every other face of $\bar{\lambda}$ lies in $X_{<(\sigma,\chi),m}$. Thus we may form the pushout
\begin{equation*}
\begin{tikzpicture}[baseline]
\matrix(m)[matrix of math nodes,
row sep=4ex, column sep=4ex,
text height=1.5ex, text depth=0.25ex]
 { \bigsqcup_{\lambda} (\Lambda^{i+m+j+2}_{i+m+1}, \{i+m+1,i+m+2\}) & X_{<(\sigma,\chi),m} \\
 \bigsqcup_{\lambda} (\Delta^{i+m+j+2}, \{i+m+1,i+m+2\}) & X_{<(\sigma,\chi),m+1} \\ };
\path[>=stealth,->,font=\scriptsize]
(m-1-1) edge (m-1-2)
edge (m-2-1)
(m-1-2) edge (m-2-2)
(m-2-1) edge (m-2-2);
\end{tikzpicture}
\end{equation*}
and complete the inductive step (again, there is one further complication involving markings: in the special case $i=-1$, $n=1$, $j=0$, $m=1$, we may have that $\lambda$ is marked. Then every edge of $\bar{\lambda}$ is marked since $(\Lambda^2_2)^\sharp \tohook (\Delta^2)^\sharp$ is right marked anodyne, and we form the pushout along maps $(\Lambda^2_1)^\sharp \tohook (\Delta^2)^\sharp$). Passing to the colimit, we deduce that $X_{<(\sigma,\chi)} \tohook X_{\sigma,\chi}$ is marked left anodyne, which completes the proof.

For (1\textquotesingle), simply observe that if $i > -1$ we are attaching along inner horns.

We now modify the above proof to prove (2). Let $X_{\sigma,\chi}$ be the sub-simplicial set of $K \star_S (\Delta^n \times_S C)$ on $K \star_S (\Lambda^n_k \times_S C)$ and simplices $(\sigma',\tau'): \Delta^i \star \Delta^j \to K \star_S (\Delta^n \times_S C)$ not in $K \star_S (\Lambda^n_k \times_S C)$ with $(\sigma',r_k(e_0 \circ \tau')) \leq (\sigma,\chi)$. Let $X_{<(\sigma,\chi)} = (K \star (\Lambda^n_k \times_S C)) \bigcup (\cup_{(\sigma',\chi')<(\sigma,\chi)} X_{\sigma',\chi'})$. We will show that $X_{<(\sigma,\chi)} \tohook X_{\sigma,\chi}$ is inner anodyne for all $(\sigma,\chi) \in I \times J$.

Let $\sigma: \Delta^i \to K$ be an element of $I$, $\chi: \Delta^j \toepi \Delta^{n-1}$ an element of $J$, and let $k'$ be the first vertex of $\chi$ with $\chi(k') = k$. Let $\lambda = (\sigma,\tau): \Delta^i \star \Delta^j \to K \star_S (\Delta^n \times_S C)$ be any nondegenerate new simplex of $X_{\sigma,\chi}$, so $r_k(e_0 \circ \tau) = \chi$. Let $\bar{\chi}: \Delta^{j+1} \toepi \Delta^n$ be the unique epimorphism with $r_k(\bar{\chi}) = \chi$. Combining \cite[Lemma~2.1.2.3]{HTT} and Lemma~\ref{markedAnodyneBasicJoinLemma}, we see that the inclusion
\begin{equation*} d_{k'}: \Delta^j = \Delta^{k'-1} \star \Delta^{j-k'} \tohook \Delta^{k'-1} \star \leftnat{\Delta^{j-k'+1}}
\end{equation*}
is right marked anodyne, so we have a lift $\bar{\tau}$ in the following diagram
\begin{equation*}
\begin{tikzpicture}[baseline]
\matrix(m)[matrix of math nodes,
row sep=4ex, column sep=4ex,
text height=1.5ex, text depth=0.25ex]
 { \Delta^j & \Delta^n \times_S C \\
 \Delta^{j+1} & \Delta^n \\ };
\path[>=stealth,->,font=\scriptsize]
(m-1-1) edge node[above]{$\tau$} (m-1-2)
edge (m-2-1)
(m-1-2) edge (m-2-2)
(m-2-1) edge node[above]{$\bar{\chi}$} (m-2-2)
edge[dotted] node[above]{$\bar{\tau}$} (m-1-2) ;
\end{tikzpicture}
\end{equation*}
where $\bar{\tau}(\{k',k'+1\})$ is a cartesian edge. By Lemma~\ref{lem1}, $\Delta^i \star \Delta^j \bigsqcup_{\Delta^j} \Delta^{j+1} \tohook \Delta^i \star \Delta^{j+1}$ is inner anodyne. We thus obtain an extension
\begin{equation*}
\begin{tikzpicture}[baseline]
\matrix(m)[matrix of math nodes,
row sep=4ex, column sep=6ex,
text height=1.5ex, text depth=0.25ex]
 { \Delta^i \star \Delta^j \bigsqcup_{\Delta^j} \Delta^{j+1} & S \\
 \Delta^i \star \Delta^{j+1} \\ };
\path[>=stealth,->,font=\scriptsize]
(m-1-1) edge node[above]{$\pi \lambda \cup e_1 \bar{\tau}$} (m-1-2)
edge (m-2-1)
(m-2-1) edge[dotted] (m-1-2);
\end{tikzpicture}
\end{equation*}
which allows us to define $\bar{\lambda}: \Delta^i \star \Delta^{j+1} \to K \star_S (\Delta^n \times_S C)$ extending $\lambda$ and $\bar{\tau}$. Then $\bar{\lambda}$ is nondegenerate and every face of $\bar{\lambda}$ except for $\lambda = d_{i+k'+1}(\bar{\lambda})$ lies in $X_{<(\sigma,\chi)}$. We may thus form the pushout
\begin{equation*}
\begin{tikzpicture}[baseline]
\matrix(m)[matrix of math nodes,
row sep=4ex, column sep=4ex,
text height=1.5ex, text depth=0.25ex]
 { \bigsqcup_{\lambda} \Lambda^{i+j+2}_{i+k'+1} & X_{<(\sigma,\chi)} \\
 \bigsqcup_{\lambda} \Delta^{i+j+2}  & X_{<(\sigma,\chi),1} \\ };
\path[>=stealth,->,font=\scriptsize]
(m-1-1) edge (m-1-2)
edge (m-2-1)
(m-1-2) edge (m-2-2)
(m-2-1) edge (m-2-2);
\end{tikzpicture}
\end{equation*}
which factors the inclusion $X_{<(\sigma,\chi)} \tohook X_{(\sigma,\chi)}$ as the composition of an inner anodyne map and an inclusion.

Now for the inductive step suppose that we have defined a sequence of inner anodyne maps
\begin{equation*} X_{<(\sigma,\chi)} \tohook ... \tohook X_{<(\sigma,\chi),m} \subset X_{(\sigma,\chi)}
\end{equation*}
such that for all $0<l \leq m$ all new nondegenerate simplices in $X_{(\sigma,\chi)}$ of dimension $i+l+j$ lie in $X_{<(\sigma,\chi),l}$ and admit an extension to a $i+l+j+1$-simplex such that the edge $\{ i+k'+l,i+k'+l+1\}$ is sent to a cartesian edge of $\Delta^n \times_S C$, and no new nondegenerate simplices of dimension $>i+l+j+1$ lie in $X_{<(\sigma,\chi),l}$. Let $\lambda = (\sigma,\tau)$ be any new nondegenerate $i+m+j+1$-simplex not in $X_{<(\sigma,\chi),m}$. For $0 \leq l < m$ let $\lambda_l = (\sigma,\tau_l)$ be a nondegenerate $i+m+j+1$-simplex in $X_{<(\sigma,\chi),m}$ with $d_{i+m+k'}(\lambda_l)=d_{i+l+k'+1}(\lambda)$. $\tau$ and $\tau_0,...,\tau_{m-1}$ together define a map
\begin{equation*}
\tau': \Delta^{k'-1} \star \Lambda^{m+1}_{m+1} \star \Delta^{j-k'-1} \to \Delta^n \times_S C
\end{equation*}
where the domain of $\tau$ is the subset $\{0,...,k'+m-1,k'+m+1,...,m+j+1\}$ and the domain of $\tau_l$ is the subset $\{0,...,\widehat{k'+l},...,m+j+1\}$. The map
\begin{equation*}
\Delta^{k'-1} \star \rightnat{\Lambda^{m+1}_{m+1}} \star \Delta^{j-k'-1} \tohook \Delta^{k'-1} \star \rightnat{\Delta^{m+1}} \star \Delta^{j-k'-1}
\end{equation*}
is $\Delta^{k'-1}$ joined with a right marked anodyne map, so is right marked anodyne by Lemma~\ref{markedAnodyneBasicJoinLemma}. Let $\bar{\chi}: \Delta^{m+j+1} \toepi \Delta^n$ be the unique epimorphism with $r_k(\bar{\chi}) = \chi$. Then we have a lift $\bar{\tau}$ in the following commutative diagram
\begin{equation*}
\begin{tikzpicture}[baseline]
\matrix(m)[matrix of math nodes,
row sep=4ex, column sep=4ex,
text height=1.5ex, text depth=0.25ex]
 { \Delta^{k'-1} \star \Lambda^{m+1}_{m+1} \star \Delta^{j-k'-1} & \Delta^n \times_S C \\
 \Delta^{m+j+1} & \Delta^n \\ };
\path[>=stealth,->,font=\scriptsize]
(m-1-1) edge node[above]{$\tau'$} (m-1-2)
edge (m-2-1)
(m-1-2) edge (m-2-2)
(m-2-1) edge node[above]{$\bar{\chi}$} (m-2-2)
edge[dotted] node[above]{$\bar{\tau}$} (m-1-2);
\end{tikzpicture}
\end{equation*}
such that $\bar{\tau}(\{k'+m,k'+m+1\})$ is a cartesian edge. By Lemma \ref{lem1}, the map
\begin{equation*} \Delta^i \star \Delta^{k'-1} \star \partial \Delta^m \star \Delta^{j-k'} \bigsqcup_{\Delta^{k'-1} \star \partial \Delta^m \star \Delta^{j-k'}} \Delta^{m+j+1} \tohook \Delta^{i+m+j+2}
\end{equation*}
is inner anodyne. Therefore, we may extend $(\cup_l \pi \lambda_l) \cup \pi \lambda \cup e_1 \bar{\tau}$ to a map $\Delta^{i+m+j+2} \to S$, which defines a nondegenerate $(i+m+j+2)$-simplex $\bar{\lambda}$ with $\lambda$ as its $(i+k'+m+1)$th face and which extends $\bar{\tau}$. By construction every other face of $\bar{\lambda}$ lies in $X_{<(\sigma,\chi),m}$. Thus we may form the pushout
\begin{equation*}
\begin{tikzpicture}[baseline]
\matrix(m)[matrix of math nodes,
row sep=4ex, column sep=4ex,
text height=1.5ex, text depth=0.25ex]
 { \bigsqcup_{\lambda} \Lambda^{i+m+j+2}_{i+k'+m+1} & X_{<(\sigma,\chi),m} \\
 \bigsqcup_{\lambda} \Delta^{i+m+j+2} & X_{<(\sigma,\chi),m+1} \\ };
\path[>=stealth,->,font=\scriptsize]
(m-1-1) edge (m-1-2)
edge (m-2-1)
(m-1-2) edge (m-2-2)
(m-2-1) edge (m-2-2);
\end{tikzpicture}
\end{equation*}
and complete the inductive step. Passing to the colimit, we deduce that $X_{<(\sigma,\chi)} \tohook X_{\sigma,\chi}$ is inner anodyne, which completes the proof.

We finally modify the above proof to prove (3). Given $\sigma \in I$ and $\chi \in J$, let $X_{\sigma,\chi}$ be the sub-marked simplicial set of $K \star_S \rightnat{\Delta^n}$ on $K \star_S \rightnat{\Lambda^n_n}$ and simplices $(\sigma',\tau'): \Delta^i \star \Delta^j \to K \star_S \rightnat{\Delta^n}$ not in $K \star_S \rightnat{\Lambda^n_n}$ with $(\sigma',r_n(\tau')) \leq (\sigma,\chi)$. Let $X_{<(\sigma,\chi)} = (K \star_S \rightnat{\Lambda^n_n}) \bigcup (\cup_{(\sigma',\chi')<(\sigma,\chi)} X_{\sigma',\chi'})$. We will show that $X_{<(\sigma,\chi)} \tohook X_{\sigma,\chi}$ is right marked anodyne for all $(\sigma,\chi) \in I \times J$.

Let $\sigma: \Delta^i \to K$ be an element of $I$ and $\chi: \Delta^j \toepi \Delta^{n-1}$ an element of $J$. Let $\lambda = (\sigma,\tau): \Delta^i \star \Delta^j \to K \star_S \rightnat{\Delta^n}$ be any nondegenerate new simplex of $X_{\sigma,\chi}$, so $r_n(\tau) = \chi$. Let $\bar{\chi}: \Delta^{j+1} \toepi \Delta^n$ be the unique epimorphism with $r_n(\bar{\chi}) = \chi$. By Lemma~\ref{lem1}, the inclusion
\begin{equation*} \Delta^i \star \Delta^j \bigsqcup_{\Delta^j} \Delta^{j+1} \tohook \Delta^i \star \Delta^{j+1}
\end{equation*}
is inner anodyne, so we have an extension in the following diagram
\begin{equation*}
\begin{tikzpicture}[baseline]
\matrix(m)[matrix of math nodes,
row sep=4ex, column sep=6ex,
text height=1.5ex, text depth=0.25ex]
 { \Delta^i \star \Delta^j \bigsqcup_{\Delta^j} \Delta^{j+1} & S \\
 \Delta^i \star \Delta^{j+1} \\ };
\path[>=stealth,->,font=\scriptsize]
(m-1-1) edge node[above]{$\pi \lambda \cup \pi_2 \bar{\chi}$} (m-1-2)
edge (m-2-1)
(m-2-1) edge[dotted] (m-1-2);
\end{tikzpicture}
\end{equation*}
which allows us to define $\bar{\lambda}: \Delta^i \star \Delta^{j+1} \to K \star_S \rightnat{\Delta^n}$ extending $\lambda$ and $\bar{\chi}$. Then $\bar{\lambda}$ is nondegenerate and every face of $\bar{\lambda}$ except for $\lambda = d_{i+j+2}(\bar{\lambda})$ lies in $X_{<(\sigma,\chi)}$. We may thus form the pushout
\begin{equation*}
\begin{tikzpicture}[baseline]
\matrix(m)[matrix of math nodes,
row sep=4ex, column sep=4ex,
text height=1.5ex, text depth=0.25ex]
 { \bigsqcup_{\lambda} \rightnat{\Lambda^{i+j+2}_{i+j+2}} & X_{<(\sigma,\chi)} \\
 \bigsqcup_{\lambda} \rightnat{\Delta^{i+j+2}}  & X_{<(\sigma,\chi),1} \\ };
\path[>=stealth,->,font=\scriptsize]
(m-1-1) edge (m-1-2)
edge (m-2-1)
(m-1-2) edge (m-2-2)
(m-2-1) edge (m-2-2);
\end{tikzpicture}
\end{equation*}
which factors the inclusion $X_{<(\sigma,\chi)} \to X_{(\sigma,\chi)}$ as the composition of a right marked anodyne map and an inclusion.

Now for the inductive step suppose that we have defined a sequence of right marked anodyne maps
\begin{equation*} X_{<(\sigma,\chi)} \tohook ... \tohook X_{<(\sigma,\chi),m} \subset X_{(\sigma,\chi)}
\end{equation*}
such that for all $0<l \leq m$ all new nondegenerate simplices in $X_{(\sigma,\chi)}$ of dimension $i+l+j$ lie in $X_{<(\sigma,\chi),l}$ and admit an extension to a $i+l+j+1$-simplex, and no new nondegenerate simplices of dimension $>i+l+j+1$ lie in $X_{<(\sigma,\chi),l}$. Let $\lambda = (\sigma,\tau)$ be any new nondegenerate $i+m+j+1$-simplex not in $X_{<(\sigma,\chi),m}$. For $0 < l \leq m$ let $\lambda_l = (\sigma,\tau_l)$ be a nondegenerate $i+m+j+1$-simplex in $X_{<(\sigma,\chi),m}$ with $d_{i+m+j+1}(\lambda_l)=d_{i+j+l+1}(\lambda)$ (note that $\tau_l = \tau$). By Lemma~\ref{lem1}, the map

\begin{equation*} \Delta^i \star \Delta^j \star \partial \Delta^m \bigsqcup_{\Delta^j \star \partial \Delta^m} \Delta^j \star \Delta^m \tohook \Delta^i \star \Delta^j \star \Delta^m 
\end{equation*}
is inner anodyne. Therefore, we may extend $\pi \lambda \cup (\cup_l \pi \lambda_l)$ to a map $\Delta^{i+j+m+2} \to S$ and define a $(i+j+m+2)$-simplex $\bar{\lambda}$ of $K \star \rightnat{\Delta^n}$ with $d_{i+j+m+2} \bar{\lambda} = \lambda$ and $d_{i+j+l+1} \bar{\lambda} = \lambda+l$. By construction every face of $\bar{\lambda}$ except for $\lambda$ lies in $X_{<(\sigma,\chi),m}$. Thus we may form the pushout
\begin{equation*}
\begin{tikzpicture}[baseline]
\matrix(m)[matrix of math nodes,
row sep=4ex, column sep=4ex,
text height=1.5ex, text depth=0.25ex]
 { \bigsqcup_{\lambda} \rightnat{\Lambda^{i+j+m+2}_{i+j+m+2}} & X_{<(\sigma,\chi),m} \\
 \bigsqcup_{\lambda} \rightnat{\Delta^{i+j+m+2}} & X_{<(\sigma,\chi),m+1} \\ };
\path[>=stealth,->,font=\scriptsize]
(m-1-1) edge (m-1-2)
edge (m-2-1)
(m-1-2) edge (m-2-2)
(m-2-1) edge (m-2-2);
\end{tikzpicture}
\end{equation*}
and complete the inductive step. Passing to the colimit, we deduce that $X_{<(\sigma,\chi)} \tohook X_{\sigma,\chi}$ is right marked anodyne, which completes the proof.
\end{proof}

\begin{rem} \label{pullingBackAnodyneMaps} The proof of Proposition \ref{joinprp} can be adapted to show that for any cartesian fibration $C \to S$, $\leftnat{\Lambda^n_0} \times_S \rightnat{C} \tohook \leftnat{\Delta^n} \times_S \rightnat{C}$ is marked left anodyne (in the $\sigma = \emptyset$ case, we only use that $e_0: \mathscr{O}(S) \to S$ is a cartesian fibration). As well, letting $K = \emptyset$, part (2) of Proposition \ref{joinprp} shows that $\Lambda^n_k \times_S C \tohook \Delta^n \times_S C$ is inner anodyne. This refines the theorem that marked left anodyne maps resp. inner anodyne maps pullback to cocartesian equivalences resp. categorical equivalences along cartesian fibrations.
\end{rem}

For later use, we state a criterion for showing that a functor is left Quillen.

\begin{lem}\label{lm:showingFunctorLeftQuillen} Let $\mathscr{M}$ and $\mathscr{N}$ be model categories and let $F: \mathscr{M} \to \mathscr{N}$ be a functor which preserves cofibrations. Let $I$ be a weakly saturated \cite[Definition A.1.2.2]{HTT} subset of the trivial cofibrations in $\mathscr{M}$ such that for every object $A \in \mathscr{M}$, we have a map $f: A \to A'$ where $f \in I$ and $A'$ is fibrant. Then $F$ preserves trivial cofibrations if and only if

\begin{enumerate}
\item For every $f \in I$, $F(f)$ is a trivial cofibration.

\item $F$ preserves trivial cofibrations between fibrant objects.
\end{enumerate}
\end{lem}

\begin{proof} The `only if' direction is obvious. For the other direction, let $A \to B$ be a trivial cofibration in $\mathscr{M}$. We may form the diagram

\begin{equation*}
\begin{tikzpicture}[baseline]
\matrix(m)[matrix of math nodes,
row sep=4ex, column sep=4ex,
text height=1.5ex, text depth=0.25ex]
 { A & B \\
 A' & A' \bigsqcup_A B & (A' \bigsqcup_A B)' \\ };
\path[>=stealth,->,font=\scriptsize]
(m-1-1) edge (m-1-2)
edge (m-2-1)
(m-1-2) edge (m-2-2)
edge (m-2-3)
(m-2-1) edge (m-2-2)
(m-2-2) edge (m-2-3);
\end{tikzpicture}
\end{equation*}
where the vertical and lower right horizontal arrows are in $I$. Then our two assumptions along with the two-out-of-three property of the weak equivalences shows that $F(A) \to F(B)$ is a trivial cofibration.
\end{proof}

\begin{lem} Let $K$ be a simplicial set over $S$. Then
\begin{equation*}
K \star_S -, - \star_S K : s\Set_{/S} \to s\Set_{K/ /S}
\end{equation*}
are left adjoints. Similarly, for $K$ a marked simplicial set over $S$,
\begin{equation*}
K \star_S -, - \star_S K : s\Set^+_{/S} \to s\Set^+_{K/ /S}
\end{equation*}
are left adjoints.
\end{lem}
\begin{proof}
We will prove that $K \star_S -$ is a left adjoint in the unmarked case and leave the other cases to the reader. Let $F$ denote $K \star_S -$ and define a functor $G: s\Set_{K/ /S} \to s\Set_{/S}$ by letting $G(K \to C)$ be the simplicial set over $S$ which satisfies 
\begin{equation*} \Hom_{/S}(\Delta^n,G(K \to C)) = \Hom_{K/ /S}(K \star_S \Delta^n, C);
\end{equation*}
this is evidently natural in $K \to C$. Define a unit map $\eta: \id \to G F$ on objects $X$ by sending $\sigma: \Delta^n \to X$ to $K \star_S \sigma: K \star_S \Delta^n \to K \star_S X$, which corresponds to $\Delta^n \to G(K \star_S X)$. Define a counit map $\eta: F G \to \id$ on objects $K \to C$ by sending $\lambda = (\sigma,\tau): \Delta^i \star \Delta^j \to K \star_S G(K \to C)$ to $\Delta^i \star \Delta^j \xrightarrow{(\sigma,\id)} K \star_S \Delta^j \xrightarrow{\tau'} C$, where $\tau'$ corresponds to $\tau: \Delta^j \to G(K \to C)$. Then it is straightforward to verify the triangle identities, so $F$ is adjoint to $G$.
\end{proof}

For the following pair of results, endow $s\Set^+_{/S}$ with the cocartesian model structure and $s\Set^+_{K/ /S} = (s\Set^+_{/S})_{K/}$ with the model structure created by the forgetful functor to $s\Set^+_{/S}$.

\begin{thm} \label{jointhm} Let $K$ be a marked simplicial set over $S$. The functor
\[ K \star_S (- \times_S \mathscr{O}(S)^\sharp) : s\Set^+_{/S} \to s\Set^+_{K/ /S} \]
is left Quillen.
\end{thm}
\begin{proof} 
We will denote the functor in question by $F$. First observe that $F$ is the composite of the three left adjoints $e_0^\ast$, ${e_1}_!$, and $K \star_S -$, so $F$ is a left adjoint. $F$ evidently preserve cofibrations, so it only remains to check that $F$ preserves the trivial cofibrations. We first verify that $F$ preserves the left marked anodyne maps. Since $F$ preserves colimits it suffices to check that $F$ preserves a collection of morphisms which generate the left marked anodyne maps as a weakly saturated class. We verify that $F$ preserves the four classes of maps enumerated in \cite[Definition~3.1.1.1]{HTT}.

(1): For $\iota: (\Lambda^n_k)^\flat \to (\Delta^n)^\flat$, $0<k<n$, the underlying map of simplicial sets of $F(\iota)$ is inner anodyne by Proposition \ref{joinprp}. $F(\iota)$ is obtained by marking common edges of an inner anodyne map, so is left marked anodyne.

(2): For $\iota: \leftnat{\Lambda^n_0} \to \leftnat{\Delta^n}$, we observe that the map

\begin{equation*} K \star_S (\leftnat{\Lambda^n_0} \times_S \mathscr{O}(S)^\sharp ) \bigsqcup_{K \star_S (\leftnat{\Lambda^n_0} \times_S \rightnat{\mathscr{O}(S)})} K \star_S (\leftnat{\Delta^n} \times_S \rightnat{\mathscr{O}(S)}) \to K \star_S (\leftnat{\Delta^n} \times_S \mathscr{O}(S)^\sharp )
\end{equation*}
in the case $n=1$ is marked left anodyne, since every marked edge in the codomain factors as a composite of two marked edges in the domain, and is the identity if $n>1$. It thus suffices to show that $K \star_S (\leftnat{\Lambda^n_0} \times_S \rightnat{\mathscr{O}(S)}) \to K \star_S (\leftnat{\Delta^n} \times_S \rightnat{\mathscr{O}(S)})$ is left marked anodyne, which is the content of part 1 of \ref{joinprp}.

(3) and (4): In both of these cases one has a map of marked simplicial sets $A \to B$ whose underlying map is an isomorphism of simplicial sets. Then

\begin{equation*}
\begin{tikzpicture}[baseline]
\matrix(m)[matrix of math nodes,
row sep=4ex, column sep=4ex,
text height=1.5ex, text depth=0.25ex]
 { A & F(A) \\
 B & F(B) \\ };
\path[>=stealth,->,font=\scriptsize]
(m-1-1) edge (m-1-2)
edge (m-2-1)
(m-1-2) edge (m-2-2)
(m-2-1) edge (m-2-2);
\end{tikzpicture}
\end{equation*}
is a pushout square, so $F(A) \to F(B)$ is left marked anodyne if $A \to B$ is.

Next, let $f: \leftnat{C} \to \leftnat{D}$ be a cocartesian equivalence between cocartesian fibrations over $S$. Let $g: \leftnat{D} \to \leftnat{C}$ be a homotopy inverse of $f$, so that there exists a homotopy $h: \leftnat{C} \times (\Delta^1)^\sharp \to \leftnat{C}$ over $S$ from $\id_C$ to $g \circ f$. Define a map
\begin{equation*}
\phi: (K \star_S (\leftnat{C} \times_S \mathscr{O}(S)^\sharp)) \times (\Delta^1)^\sharp \to K \star_S ((\leftnat{C} \times_S \mathscr{O}(S)^\sharp) \times (\Delta^1)^\sharp)
\end{equation*}
by sending a $(i+j+1)$-simplex $(\lambda,\alpha)$ given by the data $\sigma:\Delta^i \to K$, $\tau:\Delta^j \to \leftnat{C} \times_S \mathscr{O}(S)^\sharp$, $\pi \circ \lambda: \Delta^{i+j+1} \to \Delta^1$, $\alpha: \Delta^{i+j+1} \to \Delta^1$ to a $i+j+1$-simplex $\lambda'$ given by the data $\sigma$, $(\tau,\alpha \circ \iota)$, $\pi \circ \lambda$ where $\iota: \Delta^j \to \Delta^i \star \Delta^j$ is the inclusion. It is easy to see that $\phi$ restricts to an isomorphism on $(K \star_S (\leftnat{C} \times_S \mathscr{O}(S)^\sharp)) \times \partial \Delta^1$. We deduce that $F(h) \circ \phi$ is a homotopy from $F(g \circ f)$ to the identity. A similar argument concerning a chosen homotopy from $f \circ g$ to $\id_D$ shows that $F(f)$ is a cocartesian equivalence.

Finally, invoking Lemma~\ref{lm:showingFunctorLeftQuillen} completes the proof.
\end{proof}

\begin{thm} \label{thm:JoinFirstVariable} Let $K$ be a marked simplicial set over $S$. The functor
\[ - \star_S K: s\Set^+_{/S} \to s\Set^+_{K/ /S} \]
is left Quillen.
\end{thm}
\begin{proof} As with the proof of Theorem~\ref{jointhm}, the proof will be an application of Lemma~\ref{lm:showingFunctorLeftQuillen}. We first verify that $- \star_S K$ preserves the four classes of left marked anodyne maps enumerated in \cite[Definition~3.1.1.1]{HTT}. (1) is handled by the dual of part (2) of Proposition~\ref{joinprp}. (2) is handled by the dual of part (3) of Proposition~\ref{joinprp}. (3) and (4) are handled as in the proof of Theorem~\ref{jointhm}. Finally, the case of $A \to B$ a cocartesian equivalence between fibrant objects is also handled as in the proof of Theorem~\ref{jointhm}.
\end{proof}

\begin{dfn} \label{dfn:lowerSlice} Let $K, C \to S$ be marked simplicial sets over $S$ and let $p: K \to C$ be a map over $S$. Define the marked simplicial set $C_{(p,S)/} \to S$ as the value of the right adjoint to $K \star_S (- \times_S \sO(S)^\sharp)$ on $K \to C \to S$ in $s\Set^+_{K/ /S}$. By Theorem~\ref{jointhm}, if $C \to S$ is a $S$-category, then $C_{(p,S)/} \to S$ is a $S$-category. We will refer to $C_{(p,S)/}$ as a $S$-\emph{undercategory} of $C$.

Dually, define the marked simplicial set $C_{/(p,S)} \to S$ as the value of the right adjoint to $- \star_S (K \times_S \sO(S)^\sharp)$ on $K \to C \to S$ in $s\Set^+_{K/ /S}$. By Theorem~\ref{thm:JoinFirstVariable} applied to $K \times_S \sO(S)^\sharp$, if $C \to S$ is a $S$-category, then $C_{/(p,S)} \to S$ is a $S$-category. We will refer to $C_{/(p,S)}$ as a $S$-\emph{overcategory} of $C$.
\end{dfn}

\nomenclature[sliceParameterizedUnder]{$C_{(p,S)/}$}{$S$-undercategory of $S$-category $C$ with respect to $p: K \to C$}
\nomenclature[sliceParameterizedOver]{$C_{/(p,S)}$}{$S$-overcategory of $S$-category $C$ with respect to $p: K \to C$}

In the sequel, we will focus our attention on the $S$-undercategory and leave proofs of the evident dual assertions to the reader.

\subsection*{Functoriality in the diagram} We now study the functoriality of the $S$-undercategory with respect to the diagram category. Given maps $f: K \to L$ and $p: L \to X$ of marked simplicial sets over $S$, we have an induced map $X_{(p,S)/} \to X_{(p f,S)/}$, which in terms of the functors that $X_{(p,S)/}$ and $X_{(p f,S)/}$ represent is given by precomposing $L \star_S (A \times_S \mathscr{O}(S)^\sharp) \to X$ by $f \star_S \id$.

Recall that for a category $\mathscr{M}$ admitting pushouts and a map $f: K \to L$, we have an adjunction
\[ \adjunct{f_!}{\mathscr{M}_{K/}}{\mathscr{M}_{L/}}{f^\ast} \]
where $f_!(K \to X) = X \bigsqcup_K L$ and $f^\ast(L \xrightarrow{p} X) = p \circ f$. If $\mathscr{M}$ is a model category and $\mathscr{M}_{K/}$, $\mathscr{M}_{L/}$ are provided with the model structures induced from $\mathscr{M}$, then $(f_!,f^\ast)$ is a Quillen adjunction. Moreover, if $\mathscr{M}$ is a left proper model category and $f$ is a weak equivalence, then $(f_!,f^\ast)$ is a Quillen equivalence.

\begin{prp} \label{slicevar2} Let $f: K \to L$ be a cocartesian equivalence in $s\Set^+_{/S}$. Let $C$ be a $S$-category and let $p: L \to \leftnat{C}$ be a map. Then $\leftnat{C_{(p,S)/}} \to \leftnat{C_{(p f,S)/}}$ is a cocartesian equivalence in $s\Set^+_{/S}$.
\end{prp}
\begin{proof} Let $F = f_! \circ (K \star_S (- \times_S \mathscr{O}(S)^\sharp))$ and let $F' = L \star_S (- \times_S \mathscr{O}(S)^\sharp)$. Let $G$ and $G'$ be the right adjoints to $F$ and $F'$, respectively. Let $\alpha: F \to F'$ be the evident natural transformation and let $\beta: G' \to G$ be the dual natural transformation, defined by $G' \xrightarrow{\eta_{G'}} G F G' \xrightarrow{G \alpha G'} G F' G' \xrightarrow{G \epsilon'} G$. Then $\beta_C: \leftnat{C_{(p,S)/}} \to \leftnat{C_{(p f,S)/}}$ is the map under consideration. By Theorem~\ref{thm:JoinFirstVariable}, $\alpha_X$ is a cocartesian equivalence for all $X \in s\Set^+_{/S}$. Therefore, by \cite[Corollary~1.4.4(b)]{Hovey}, $\beta_C$ is a cocartesian equivalence.
\end{proof} 

\begin{prp} \label{prp:SliceFunctorialityInDiagram} Consider a commutative diagram of marked simplicial sets
\[ \begin{tikzpicture}[baseline]
\matrix(m)[matrix of math nodes,
row sep=4ex, column sep=4ex,
text height=1.5ex, text depth=0.25ex]
 { K & C \\
 L & D \\ };
\path[>=stealth,->,font=\scriptsize]
(m-1-1) edge (m-1-2)
edge node[left]{$i$} (m-2-1)
(m-1-2) edge node[right]{$q$} (m-2-2)
(m-2-1) edge (m-2-2)
edge node[above]{$p$} (m-1-2);
\end{tikzpicture} \]
where $i$ is a cofibration and $q$ is a fibration.

\begin{enumerate} \item The map
\[ C_{(p,S)/} \to C_{(p i,S)/} \times_{D_{(q p i,S)/}} D_{(q p,S)/} \]
is a fibration.

\item Let $K = \emptyset$ and $D = S^\sharp$. Then the map
\[ C_{(p,S)/} \to C_{(p i,S)/} \cong \underline{\Fun}_S(S^\sharp,C) \]
is a left fibration (of the underlying simplicial sets).
\end{enumerate}
\end{prp}
\begin{proof} 
\begin{enumerate}[leftmargin=*]
\item Given a trivial cofibration $A \to B$, we need to solve lifting problems of the form
\[ \begin{tikzpicture}[baseline]
\matrix(m)[matrix of math nodes,
row sep=4ex, column sep=4ex,
text height=1.5ex, text depth=0.25ex]
 { L \star_S (A \times_S \sO(S)^\sharp) \bigsqcup_{K \star_S (A \times_S \sO(S)^\sharp)} K \star_S (B \times_S \sO(S)^\sharp) & C \\
 L \star_S (B \times_S \sO(S)^\sharp) & D. \\ };
\path[>=stealth,->,font=\scriptsize]
(m-1-1) edge (m-1-2)
edge (m-2-1)
(m-1-2) edge (m-2-2)
(m-2-1) edge (m-2-2)
edge[dotted] (m-1-2);
\end{tikzpicture} \]
But the lefthand map is a trivial cofibration by Theorem~\ref{jointhm}.

\item We need to solve lifting problems of the form 
\[ \begin{tikzpicture}[baseline]
\matrix(m)[matrix of math nodes,
row sep=4ex, column sep=4ex,
text height=1.5ex, text depth=0.25ex]
 { (\Delta^n)^\flat \times_S \sO(S)^\sharp \bigsqcup_{(\Lambda^n_i)^\flat} K \star_S ((\Lambda^n_i)^\flat \times_S \sO(S)^\sharp) & C \\
 K \star_S ( (\Delta^n)^\flat \times_S \sO(S)^\sharp) & S \\ };
\path[>=stealth,->,font=\scriptsize]
(m-1-1) edge (m-1-2)
edge (m-2-1)
(m-1-2) edge (m-2-2)
(m-2-1) edge (m-2-2)
edge[dotted] (m-1-2);
\end{tikzpicture} \]
where $0 \leq i < n$. But the lefthand map is a trivial cofibration by Proposition~\ref{joinprp} (1\textquotesingle) and (2).
\end{enumerate}
\end{proof}

Combining (2) of the above proposition with Lemma~\ref{funclem2} (2) (which supplies a trivial marked fibration $\underline{\Fun}_S(S^\sharp, C) \to C$), we obtain a map $C_{(p,S)/} \to C$ which is a marked fibration and a left fibration, and such that for any $f: K \to L$, the triangle
\[ \begin{tikzpicture}[baseline]
\matrix(m)[matrix of math nodes,
row sep=4ex, column sep=4ex,
text height=1.5ex, text depth=0.25ex]
 { C_{(p,S)/} & & C_{(p f,S)/}   \\
 & C & \\ };
\path[>=stealth,->,font=\scriptsize]
(m-1-1) edge (m-1-3)
edge (m-2-2)
(m-1-3) edge (m-2-2);
\end{tikzpicture} \]
commutes.

\subsection*{The universal mapping property of the \texorpdfstring{$S$}{S}-slice}
Because the $S$-join and slice Quillen adjunction is not simplicial, we do not immediately obtain a universal mapping property characterizing the $S$-slice. Our goal in this subsection is to supply such a universal mapping property (Proposition~\ref{prp:SliceCompare}). We first recall how to slice Quillen bifunctors. Suppose $\sV$ is a closed symmetric monoidal category and $\sM$ is enriched, tensored, and cotensored over $\sV$. Denote the internal hom by
\[ \underline{\Hom}(-,-): \mathscr{M}^\op \times \mathscr{M} \to \mathscr{V}. \]
Define bifunctors
\begin{align*}  \underline{\Hom}_{x/}(-,-) &: \sM_{x/}^\op \times \sM_{x/} \to \sV \\
\underline{\Hom}_{/x}(-,-) &: \sM_{/x}^\op \times \sM_{/x} \to \sV 
\end{align*}
on objects $f: x \to a, g: x \to b$ and $f': a \to x, g': b \to x$ to be pullbacks
\begin{equation*}
\begin{tikzpicture}[baseline]
\matrix(m)[matrix of math nodes,
row sep=4ex, column sep=4ex,
text height=1.5ex, text depth=0.25ex]
 { \underline{\Hom}_{x/}(f,g) & \underline{\Hom}(a,b) & \underline{\Hom}_{/x}(f',g') & \underline{\Hom}(a,b) \\
 1 & \underline{\Hom}(x,b) & 1 & \underline{\Hom}(a,x) \\ };
\path[>=stealth,->,font=\scriptsize]
(m-1-1) edge (m-1-2)
edge (m-2-1)
(m-1-2) edge node[right]{$f^\ast$} (m-2-2)
(m-2-1) edge node[above]{$g$} (m-2-2)
(m-1-3) edge (m-1-4)
edge (m-2-3)
(m-1-4) edge node[right]{$g'_\ast$} (m-2-4)
(m-2-3) edge node[above]{$f'$} (m-2-4);
\end{tikzpicture}
\end{equation*}
and on morphisms in the obvious way (we abusively denote by $g: 1 \to \underline{\Hom}(x,b)$ the map corresponding to $g$ under the natural isomorphisms $\underline{\Hom}(1,\underline{\Hom}(x,b)) \cong \underline{\Hom}(1 \otimes x,b) \cong \underline{\Hom}(x,b),$ and likewise for $f'$). It is easy to see that $\underline{\Hom}_{x/}$ and $\underline{\Hom}_{/x}$ preserve limits separately in each variable.

\begin{lem} \label{sliceQB} In the above situation let $\mathscr{M}$ be a model category and $\mathscr{P}$ be a monoidal model category. If $\underline{\Hom}(-,-)$ is a right Quillen bifunctor, then $\underline{\Hom}_{x/}(-,-)$ and $\underline{\Hom}_{/x}(-,-)$ are right Quillen bifunctors, where we endow $\mathscr{M}_{x/}$ and $\mathscr{M}_{/x}$ with the model structures created by the forgetful functor to $\mathscr{M}$.
\end{lem}
\begin{proof} We prove the assertion for $\underline{\Hom}_{x/}(-,-)$, the proof for $\underline{\Hom}_{/x}(-,-)$ being identical. Let $i: a \to b$ and $f: c \to d$ be morphisms in $\sM_{x/}$ (so they are compatible with the structure maps $\pi_a,...,\pi_d$). In the commutative diagram
\begin{equation*}
\begin{tikzpicture}[baseline]
\matrix(m)[matrix of math nodes,
row sep=4ex, column sep=4ex,
text height=1.5ex, text depth=0.25ex]
 { \underline{\Hom}_{x/}(\pi_b,\pi_c) & \underline{\Hom}(b,c) \\
 \underline{\Hom}_{x/}(\pi_a,\pi_c) \times_{\underline{\Hom}_{x/}(\pi_a,\pi_d)} \underline{\Hom}_{x/}(\pi_b,\pi_d) & \underline{\Hom}(a,c) \times_{\underline{\Hom}(a,d)} \underline{\Hom}(b,d) \\
1 & \underline{\Hom}(x,c) \\};
\path[>=stealth,->,font=\scriptsize]
(m-1-1) edge (m-1-2)
edge (m-2-1)
(m-1-2) edge (m-2-2)
(m-2-1) edge (m-2-2)
edge (m-3-1)
(m-2-2) edge (m-3-2)
(m-3-1) edge (m-3-2);
\end{tikzpicture}
\end{equation*}
it is easy to see that the lower square and the rectangle are pullback squares, so the upper square is a pullback square. It is now clear that if $\underline{\Hom}(-,-)$ is a right Quillen bifunctor, then $\underline{\Hom}_{x/}(-,-)$ is as well.
\end{proof}

We apply Lemma~\ref{sliceQB} to the bifunctors
\begin{align*} \Map_{K/ /S}(-,-): & {s\Set^+_{K//S}}^\op \times s\Set^+_{K//S} \to s\Set_{\text{Quillen}}  \\
\Fun_{K/ /S}(-,-): & {s\Set^+_{K//S}}^\op \times s\Set^+_{K//S} \to s\Set_{\text{Joyal}}
\end{align*} 
induced by $\Map_S(-,-)$ and $\Fun_S(-,-)$.

\begin{lem} \label{slicecomparelem} Let $K$, $A$, and $B$ be simplicial
 sets and define a map
\[ A \times (K \star B) \to K \star (A \times B) \]
by sending the data $(\Delta^n \to A, \Delta^k \to K, \Delta^{n-k-1} \to B)$ of a $n$-simplex of $A \times (K \star B)$ to the data $(\Delta^k \to K, \Delta^{n-k-1} \to A \times B)$ of a $n$-simplex of $K \star (A \times B)$. Then
\[ \phi: A \times (K \star B) \bigsqcup_{A \times K} K \to K \star (A \times B) \]
is a categorical equivalence.
\end{lem}
\begin{proof} Recall \cite[Proposition~4.2.1.2]{HTT} that there is a map
\[ \eta_{X,Y}: X \diamond Y = X \bigsqcup_{X \times Y \times \{0 \} } X \times Y \times \Delta^1 \bigsqcup_{X \times Y \times \{1 \} } Y \to X \star Y \]
natural in $X$ and $Y$ which is always a categorical equivalence. Thus
\[ f = (A \times \eta_{K,B}) \sqcup \id_K: A \times (K \diamond B) \bigsqcup_{A \times K} K \to A \times (K \star B) \bigsqcup_{A \times K} K \]
is a categorical equivalence. The domain is isomorphic to $K \diamond (A \times B)$, and it is easy to check that the map $\eta_{K,A \times B}$ is the composite
\[ K \diamond (A \times B) \xrightarrow{f} A \times (K \star B) \bigsqcup_{A \times K} K \xrightarrow{\phi} K \star (A \times B). \]
Using the 2 out of 3 property of the categorical equivalences, we deduce that $\phi$ is a categorical equivalence.
\end{proof}

\begin{lem} \label{lm:JoinSliceAlmostSimplicialAdjunction} For all $L \in s\Set^+_{/S}$, we have a natural equivalence
\[ \phi: \Fun_S(L,\leftnat{C_{(p,S)/}}) \xrightarrow{\simeq} \Fun_{K/ /S}(K \star_S (L \times_S \sO(S)^\sharp),\leftnat{C}). \]
\end{lem}
\begin{proof} Define bisimplicial sets $X, Y: \Delta^\op \to s\Set$ by
\begin{align*} X_n & = \Map_{K/ /S}(K \star_S ( (\Delta^n)^\flat \times L \times_S \sO(S)^\sharp), \leftnat{C}) \\
Y_n & = \Map(\Delta^n,\Fun_{K/ /S}(K \star_S (L \times_S \sO(S)^\sharp),\leftnat{C})) \\
 & \cong \Map_{K/ /S}( (\Delta^n)^\flat \times (K \star_S (L \times_S \sO(S)^\sharp) \bigsqcup_{(\Delta^n)^\flat \times K} K,\leftnat{C}).
\end{align*}
and define a map of bisimplicial sets $\Phi: X \to Y$ by precomposing levelwise by the map
\[ g_{L,n}: (\Delta^n)^\flat \times (K \star_S (L \times_S \sO(S)^\sharp)) \bigsqcup_{(\Delta^n)^\flat \times K} K \to K \star_S ((\Delta^n)^\flat \times L \times_S \sO(S)^\sharp) \]
adjoint as a map over $S \times \Delta^1$ to the identity over $S \times \partial \Delta^1$. Taking levelwise zero simplices then defines the map $\phi$, which is clearly natural in $L$, $K$, and $C$. By Theorem~\ref{thm:JoinFirstVariable}, taking a fibrant replacement of $K$ we may suppose that $K$ is fibrant. We first check that $X$ and $Y$ are complete Segal spaces. By \cite[Theorem~4.12]{JT}, $Y$ is a complete Segal space as it arises from a $\infty$-category. For $X$, since $\Map_{K/ /S}(-,-)$ is a right Quillen bifunctor, we only have to observe that:
\begin{itemize}
\item Every monomorphism $A \to B$ of simplicial sets induces a cofibration
\[ K \star_S ( A^\flat \times L \times_S \sO(S)^\sharp) \to K \star_S ( B^\flat \times L \times_S \sO(S)^\sharp ) \]
so $X$ is Reedy fibrant.

\item The spine inclusion $\iota_n: \text{Sp}(n) \to \Delta^n$ induces a trivial cofibration
\[ K \star_S ( \text{Sp}(n)^\flat \times L \times_S \sO(S)^\sharp) \to K \star_S ( (\Delta^n)^\flat \times L \times_S \sO(S)^\sharp); \]
$\iota_n$ is inner anodyne, so this follows from Theorem~\ref{jointhm} and \cite[Proposition~3.1.4.2]{HTT}.

\item The map $\pi: E \to \Delta^0$ where $E$ is the nerve of the contractible groupoid with two elements induces a cocartesian equivalence
\[ K \star_S ( E^\flat \times L \times_S \sO(S)^\sharp) \to K \star_S (L \times_S \sO(S)^\sharp); \]
$\pi^\flat$ is a cocartesian equivalence (as the composite of $E^\flat \to E^\sharp$ and $E^\sharp \to \Delta^0$), so this also follows from Theorem~\ref{jointhm} and \cite[Proposition~3.1.4.2]{HTT}.
\end{itemize}
We next prove that $\Phi$ is an equivalence in the complete Segal model structure. For this, we will prove that each map $g_{L,n}$ is a cocartesian equivalence in $s\Set^+_{/S}$. Both sides preserves colimits as a functor of $L$ (valued in $s\Set^+_{K//S}$), so by left properness and the stability of cocartesian equivalences under filtered colimits we reduce to the case $L$ is an $m$-simplex with some marking. In particular, $(\Delta^m)^\flat \times_S \sO(S)^\sharp \to S$ is fibrant in $s\Set^+_{/S}$. By \cite[Theorem~4.2.4.1]{HTT} we may check that the square of fibrant objects
\[ \begin{tikzcd}[row sep=2em, column sep=2em]
(\Delta^n)^\flat \times K \ar{r} \ar{d} & K \ar{d} \\
(\Delta^n)^\flat \times (K \star_S ((\Delta^m)^\flat \star_S \sO(S)^\sharp)) \ar{r} & K \star_S ((\Delta^n)^\flat \times (\Delta^m)^\flat \times_S \sO(S)^\sharp)
\end{tikzcd} \]
is a homotopy pushout square in the underlying $\infty$-category $\Cat_{\infty,S}^\cocart \simeq \Fun(S, \Cat_{\infty})$, where colimits are computed objectwise. In other words, we may check that for every $s \in S$, the fiber of the square over $s$ is a homotopy pushout square in $s\Set$, which holds by Lemma~\ref{slicecomparelem}. Pushing out along the cofibration $(\Delta^m)^\flat \times_S \sO(S)^\sharp \to L \times_S \sO(S)^\sharp$ and using left properness, we deduce that $g_{L,m}$ is a cocartesian equivalence. Finally, we invoke \cite[Theorem~4.11]{JT} to deduce that $\phi$ is a categorical equivalence.
\end{proof}

\begin{lem} \label{lm:JoinWithFibrantVariable} Let $L \to S$ be a cocartesian fibration. Then $\id_K \star \iota_L: K \star_S \leftnat{L} \to K \star_S (\leftnat{L} \times_S \sO(S)^\sharp)$ is a cocartesian equivalence in $s\Set^+_{/S}$.
\end{lem}
\begin{proof} By Theorem~\ref{thm:JoinFirstVariable}, taking a fibrant replacement of $K$ we may suppose that $K$ is fibrant. By Proposition~\ref{prp:fiberwiseFibrantReplacement}, it suffices to show that for every $s \in S$, $K_s^\sim \star L_s^\sim \to K_s^\sim \star (\leftnat{L} \times_S (S^{/s})^\sharp)$ is a marked equivalence in $s\Set^+$. Observe that the \emph{cartesian} equivalence $\{s \} \to (S^{/s})^\sharp$ pulls back by the cocartesian fibration $\leftnat{L} \to S^\sharp$ to a marked equivalence $L_s^\sim \to \leftnat{L} \times_S (S^{/s})^\sharp$. Then by Theorem~\ref{jointhm} for $S = \Delta^0$, $K_s^\sim \star -$ preserves marked equivalences, which concludes the proof.
\end{proof}

\begin{ntn} \label{ntn:FunctorSliceUnderOver} Suppose we have a commutative square of $S$-categories and $S$-functors
\[ \begin{tikzcd}[row sep=4ex, column sep=4ex, text height=1.5ex, text depth=0.25ex]
K \ar{r}{G} \ar{d}{F} & D \ar{d}{\pi} \\
C \ar{r}{\rho} & M. 
\end{tikzcd} \]
Define $\underline{\Fun}_{K//M,S}(C,D)$ to be the pullback
\[ \begin{tikzcd}[row sep=4ex, column sep=4ex, text height=1.5ex, text depth=0.25ex]
\underline{\Fun}_{K//M,S}(C,D) \ar{r} \ar{d} & \underline{\Fun}_S(C,D) \ar{d}{(F^\ast,\pi_{\ast})} \\
S \ar{r}{\sigma_{\pi G}} & \underline{\Fun}_S(K,M).
\end{tikzcd} \]
If $K=\emptyset$, we will also denote $\underline{\Fun}_{K//M,S}(C,D)$ as $\underline{\Fun}_{/M,S}(C,D)$. If $M = S$, we will write $\underline{\Fun}_{K//S}(C,D)$ in place of $\underline{\Fun}_{K//S,S}(C,D)$.
\end{ntn}
\nomenclature[functorParameterizedUnderOver]{$\underline{\Fun}_{K//M,S}(C,D)$}{$S$-category of $S$-functors, relative variant}

Note that by Proposition~\ref{prp:FunctorFirstVariable} and Proposition~\ref{prp:FibrationBetweenFibrantObjects}, the defining pullback square is a homotopy pullback square if $F$ is a monomorphism and $\pi$ is a categorical fibration.

\begin{prp} \label{prp:SliceCompare} Let $K, L, C$ be $S$-categories and let $p: K \to C$, $q: L \to C$ be $S$-functors.
\begin{enumerate}
\item We have an equivalence
\[ \psi: \underline{\Fun}_S(L,C_{(p,S)/}) \xrightarrow{\simeq} \underline{\Fun}_{K/ /S}(K \star_S L, C). \]

\item We have an equivalence
\[ \psi': \underline{\Fun}_S(L,C_{/(q,S)}) \xrightarrow{\simeq} \underline{\Fun}_{L/ /S}(K \star_S L, C) \]

\item We have equivalences
\[ \begin{tikzcd}[row sep=1em, column sep=2em]
\underline{\Fun}_{/C,S}(L,C_{(p,S)/}) \ar{r}{\psi_q}[swap]{\simeq} & \underline{\Fun}_{K \sqcup L/ / S}(K \star_S L, C)  & \ar{l}[swap]{\psi'_p}{\simeq} \underline{\Fun}_{/C,S}(K,C_{/(q,S)}).
\end{tikzcd} \]
\end{enumerate}
\end{prp}
\begin{proof}
\begin{enumerate}[leftmargin=*]
\item Define the $S$-functor $\psi$ as follows: suppose given a marked simplicial set $A$ and a map $A \to \underline{\Fun}_S(L,C_{(p,S)/})$ over $S$. This is equivalently given by the datum of a map
\[ f_A: \leftnat{K} \star_S ((A \times_S \sO(S)^\sharp \times_S \leftnat{L}) \times_S \sO(S)^\sharp) \to \leftnat{C} \]
under $K$ and over $S$. Let
\[ \leftnat{K} \bigsqcup\limits_{A \times_S \sO(S)^\sharp \times_S \leftnat{K}} (A \times_S \sO(S)^\sharp) \times_S (\leftnat{K} \star_S (\leftnat{L} \times_S \sO(S)^\sharp)) \to K \star_S (A \times_S \sO(S)^\sharp \times_S \leftnat{L} \times_S \sO(S)^\sharp) \]
be the map over $S \times \Delta^1$ adjoint to the identity over $S \times \partial \Delta^1$. Precomposing $f_A$ by this and $\iota_L: \leftnat{L} \to \leftnat{L} \times_S \sO(S)^\sharp$ on that factor defines the desired map $A \to \underline{\Fun}_{K/ /S}(K \star_S L, C)$.

Now to check that $\psi$ is an equivalence, we may work fiberwise and combine Lemma~\ref{lm:JoinSliceAlmostSimplicialAdjunction} and Lemma~\ref{lm:JoinWithFibrantVariable}.

\item This follows by a parallel argument to the proof of (1).

\item We prove that $\psi_q$ is an equivalence; a parallel argument will work for $\psi'_p$. $\underline{\Fun}_{K \sqcup L/ / S}(K \star_S L, C)$ fits into a diagram
\[ \begin{tikzcd}[row sep=2em, column sep=2em]
\underline{\Fun}_{K \sqcup L/ / S}(K \star_S L, C) \ar{r} \ar{d} & \underline{\Fun}_{K/ / S}(K \star_S L, C) \ar{r} \ar{d} & \underline{\Fun}_{S}(K \star_S L, C) \ar{d} \\
S \ar{r} \ar[bend left=10]{rr}{\sigma_{p \sqcup q}} & \underline{\Fun}_{K/ / S}(K \bigsqcup L, C) \ar{r} \ar{d} & \underline{\Fun}_{S}(K \bigsqcup L, C) \ar{d} \\
& S \ar{r}{\sigma_p} & \underline{\Fun}_{S}(K, C)
\end{tikzcd} \]
in which every square is a pullback square. The map $\psi_q$ is then defined to be the pullback of the map of spans
\[ \begin{tikzcd}[row sep=2em, column sep=2em]
\underline{\Fun}_S(L, C_{(p,S)/}) \ar{r} \ar{d}{\psi} & \underline{\Fun}_S(L, C) \ar{d}{p \sqcup -} & S \ar{l}[swap]{\sigma_q} \ar{d}{=} \\
\underline{\Fun}_{K/ / S}(K \star_S L, C) \ar{r} & \underline{\Fun}_{K/ / S}(K \bigsqcup L, C) & \ar{l} S
\end{tikzcd} \]
in which the vertical arrows are equivalences. By Proposition~\ref{prp:SliceFunctorialityInDiagram} and $\underline{\Fun}_S(L,-)$ being right Quillen, the top left horizontal arrow is a $S$-fibration, and by Proposition~\ref{prp:FunctorFirstVariable}, the bottom left horizontal arrow is a $S$-fibration. It follows that $\psi_q$ is an equivalence.
\end{enumerate}
\end{proof}

In light of Proposition~\ref{prp:SliceCompare}, we have evident `alternative' $S$-slice $S$-categories, whose definition more closely adheres to the intuition that a slice category is a category of extensions.

\begin{dfn} \label{dfn:AltSlice} Let $p: K \to C$ be a $S$-functor. We define the \emph{alternative} $S$-\emph{undercategory} 
\[ C^{(p,S)/} \coloneq \underline{\Fun}_{K/ / S}(K \star_S S,C). \]
Similarly, we define the \emph{alternative} $S$-\emph{overcategory}
\[ C^{/(p,S)} \coloneq \underline{\Fun}_{K/ / S}(S \star_S K,C). \]
\end{dfn}

\nomenclature[sliceParameterizedUnderAlternative]{$C^{(p,S)/}$}{$S$-undercategory of $S$-category $C$ with respect to $p: K \to C$, alternative version}
\nomenclature[sliceParameterizedOverAlternative]{$C^{/(p,S)}$}{$S$-overcategory of $S$-category $C$ with respect to $p: K \to C$, alternative version}

\begin{cor} \label{cor:SliceCompare} Let $p: K \to C$ and $q: L \to C$ be $S$-functors.
\begin{enumerate}
\item We have equivalences $C_{(p,S)/} \xrightarrow{\simeq} C^{(p,S)/}$ and $C_{/(q,S)} \xrightarrow{\simeq} C^{/(q,S)}$.
\item We have an equivalence $\underline{\Fun}_{/C,S}(L,C^{(p,S)/}) \simeq \underline{\Fun}_{/C,S}(K, C^{/(q,S)})$ through a natural zig-zag.
\end{enumerate}
\end{cor}
\begin{proof} For (1), let $L=S$ and $K=S$ in Proposition~\ref{prp:SliceCompare}(1) and (2), respectively. For (2), combine the preceding (1) and Proposition~\ref{prp:SliceCompare}(3).
\end{proof}

\begin{wrn} When $S = \Delta^0$, the alternative $S$-undercategory $C^{(p,S)/} \cong \{p\} \times_{\Fun(K,C)} \Fun(K^\rhd,C)$ differs from Lurie's alternative undercategory $C^{p/}$. However, we have a comparison functor
\[ \{p\} \times_{\Fun(K,C)} \Fun(K^\rhd,C) \to C^{p/} \]
 which is a categorical equivalence and which factors through the categorical equivalence $C_{p/} \to C^{p/}$ of \cite[Proposition~4.2.1.5]{HTT}.
\end{wrn}

\subsection*{Slicing over and under \texorpdfstring{$S$}{S}-points} We give a smaller model for slicing over and under $S$-points in an $S$-category $C$.

\begin{ntn} \label{ntn:fiberwiseArrowCategory} Suppose $C$ an $S$-category. Let
$$\sO_S(C) \coloneq \widetilde{\Fun}_S(S \times \Delta^1,C) \cong S \times_{\sO(S)} \sO(C)$$
denote the fiberwise arrow $S$-category of $C$. Given an object $x \in C$, let
\[ C^{/\underline{x}} \coloneq \sO_S(C) \times_C \underline{x} \:, \quad C^{\underline{x}/} \coloneq \underline{x} \times_C \sO_S(C).\]
\end{ntn}

\nomenclature[ArrowCategoryRelative]{$\sO_S(C)$}{Fiberwise arrow $S$-category of $C$}
\nomenclature[sliceParameterizedUnderObject]{$C^{\underline{x}/}$}{Slice $S$-category under a point $x \in C$}
\nomenclature[sliceParameterizedOverObject]{$C^{/\underline{x}}$}{Slice $S$-category over a point $x \in C$}

\begin{prp} Let $x \in C$ be an object and denote by $i_x: \underline{x} \to C_{\underline{x}}$ the $\underline{x}$-functor defined by $x$. We have natural equivalences of $\underline{x}$-categories
\begin{align*}
{C_{\underline{x}}}^{/ (\underline{x}, i_x)} &\simeq C^{/\underline{x}} \\
{C_{\underline{x}}}^{/ (i_x, \underline{x})} &\simeq C^{\underline{x}/}.
\end{align*}
\end{prp}
\begin{proof} For any functor $S' \to S$ and $S$-category $C$, $\sO_S(C) \times_S S' \cong \sO_{S'}(C \times_S S')$. Therefore, $\sO_S(C) \times_C \underline{x} \cong \sO_{\underline{x}}(C_{\underline{x}}) \times_{C_{\underline{x}}} \underline{x}$ and likewise for $\underline{x} \times_C \sO_S(C)$. Changing base to $\underline{x}$, we may suppose $S = \underline{x}$ and $i_x = i: S \to C$ is any $S$-functor. The identity section $S \to \sO(S)$ induces a morphism of spans
\[ \begin{tikzcd}[row sep=2em,column sep=3em]
S \ar{r}{\sigma_i} \ar{d}{=} & \underline{\Fun}_S(S,C) \ar{d} & \underline{\Fun}_S(S \times \Delta^1,C) \ar{l} \ar{d} \\
S \ar{r}{i} & C & \widetilde{\Fun}_S(S \times \Delta^1,C) \ar{l}
\end{tikzcd} \]
with the vertical maps equivalences. Taking pullbacks now yields the claim (where we use the isomorphism $S \star_S S \cong S \times \Delta^1$ to identify the upper pullback with the $S$-slice category in question).
\end{proof}

\begin{prp} \label{SlicingUnderObject} We have a natural equivalence $C^{\underline{x}/} \simeq C^{x/}$ of left fibrations over $C$.
\end{prp}
\begin{proof} Using the marked left anodyne map $\leftnat{\Lambda^2_1} \to \leftnat{\Delta^2}$ and the map of Lemma~\ref{lm:generalizedCocartesianPushforward} for $n=2$, we obtain a span
\[ \begin{tikzcd}[row sep=1em, column sep=1em]
& \Fun(\leftnat{\Delta^2},\leftnat{C}) \ar[->>]{ld}[swap]{\simeq} \ar[->>]{rd}{\simeq} & \\
\Fun((\Delta^{\{0,1\}})^\sharp,\leftnat{C}) \times_{C^{\{1\}}} \Fun(\Delta^{\{1,2\}},C) & & \Fun(\Delta^{\{0,2\}},C) \times_{S^{\{0,2\}}} \Fun(\Delta^2,S).
\end{tikzcd} \]
Pulling back via $\{x\} \times_{C^{\{0\}}} -$ on the left and $- \times_{S^{\{1,2\}}} S$ on the right, and using that the inclusion $\Delta^{\{0,2\}} \to \Delta^2 \cup_{\Delta^{\{1,2\}}} \Delta^0$ is a categorical equivalence, we get
\[ \begin{tikzcd}[row sep=1em, column sep=1em]
& \{x\} \times_{C^{\{0\}}} \Fun(\leftnat{\Delta^2},\leftnat{C}) \times_{S^{\{1,2\}}} S \ar[->>]{ld}[swap]{\simeq} \ar[->>]{rd}{\simeq} & \\
C^{\underline{x}/} & & C^{x/}.
\end{tikzcd} \]
\end{proof}


\section{Limits and colimits}

In this section, we introduce $S$-colimits and study their basic properties. We then study the correspondence between $S$-colimits and $S$-limits through the vertical opposite construction of \cite{BGN}.

\begin{dfn} Let $C$ be a $S$-category and $\sigma: S \to C$ be a cocartesian section. We say that $\sigma$ is a \emph{$S$-initial object} if $\sigma(s)$ is an initial object for all objects $s \in S$. Dually, $\sigma$ is a \emph{$S$-final object} if $\sigma(s)$ is a final object for all $s \in S$.
\end{dfn}

\begin{dfn} \label{dfn:colimit} Let $K$ and $C$ be $S$-categories. Let $\overline{p}: K \star_S S \to C$ be an extension of a $S$-functor $p: K \to C$. From the commutativity of the diagram
\[ \begin{tikzcd}[row sep=2em, column sep=2em]
S \ar{r}{\sigma_{\overline{p}}} \ar{d}{=} & \underline{\Fun}_S(K \star_S S, C) \ar{d} \\
S \ar{r}{\sigma_p} & \underline{\Fun}_S(K,C)
\end{tikzcd} \]
(recall Notation~\ref{ntn:cocartesianSection} for $\sigma_{(-)}$) we see that $\sigma_{\overline{p}}$ defines a cocartesian section of $C^{(p,S)/}$ (Definition~\ref{dfn:AltSlice}), which we also denote by $\sigma_{\overline{p}}$. We say that $\overline{p}$ is a \emph{$S$-colimit diagram} if $\sigma_{\overline{p}}$ is a $S$-initial object. If $\overline{p}$ is a $S$-colimit diagram, then $\overline{p}|_S: S \to C$ is said to be a \emph{$S$-colimit} of $p$. If $S$ admits an initial object $s$, we will also identify the $S$-colimit with its value on $s$.

Dually, substituting $S \star_S K$ for $K \star_S S$ leads in a parallel way to the definition of an \emph{$S$-limit diagram} and an \emph{$S$-limit}.
\end{dfn}

\begin{rem} In view of the comparison result Corollary~\ref{cor:SliceCompare}, we could also use the $S$-slice category $C_{(p,S)/}$ to make the definition of a $S$-colimit diagram. This would yield some additional generality, in that $C_{(p,S)/}$ is defined for an arbitrary marked simplicial set $K$. However, the construction $C^{(p,S)/}$ is easier to relate to functor categories, which we need to do to show that the left adjoint to the restriction along $K \subset K \star_S S$ computes colimits (a special case of Corollary~\ref{cor:leftAdjointToRestrictionIsColimitFunctor}).
\end{rem}

\begin{rem}
Suppose $K$ and $C$ are $\infty$-categories, and write $\pi: K \to \ast$ for the map to a point. One may define the $K$-indexed colimit `globally' as the (partially defined) left adjoint $\pi_!$ to the restriction functor $\pi^*: C \to \Fun(K,C)$. Given a diagram $p: K \to C$ that admits an extension to a colimit diagram $\overline{p}: K^{\rhd} \to C$ with cone point $\{v\}$, one then has $\overline{p}|_{\{v\}} \simeq \pi_!(p)$.

To establish a parallel picture for $S$-colimits, we will first need to introduce the concept of $S$-adjunctions (Definition~\ref{dfn:sAdjunction}). If we now let $K$ and $C$ be $S$-categories and $\pi: K \to S$ denote the structure map, we will show that if for all $s \in S$, $C_{\underline{s}}$ admits $K_{\underline{s}}$-indexed $S^{s/}$-colimits, then the restriction $S$-functor $\pi^*: C \to \underline{\Fun}_S(K,C)$ admits a left $S$-adjoint $\pi_!$ such that
$$(\pi_!)_s: \Fun_{S^{s/}}(K_{\underline{s}}, C_{\underline{s}}) \to C_{s}$$ computes the $S^{s/}$-colimit (Theorem~\ref{thm:leftKanExtensionIsLeftAdjointToRestriction} in the special case $\phi = \pi$). Furthermore, taking cocartesian sections of this $S$-adjunction then yields an adjunction, which we may abusively denote as
\[ \adjunct{\pi_!}{\Fun_S(K,C)}{\Fun_S(S,C)}{\pi_*}, \]
in which $\pi_!$ computes the $S$-colimit.

In proving some of the assertions in this subsection (Corollary~\ref{cor:ConstantCocompleteness}, Proposition~\ref{colimitOverCorepresentableDiagramIsLeftAdjointToRestriction}, and Proposition~\ref{prp:BeckChevalley}), it will be convenient to have this relationship between $S$-colimits and $S$-adjunctions established. We note that there is no danger of circularity here since the proof of Theorem~\ref{thm:leftKanExtensionIsLeftAdjointToRestriction} (or its simpler predecessor Theorem~\ref{thm:ExistenceAndUqnessOfParamColimit}) doesn't use any of the remainder of this subsection (which, apart from $S$-(co)limits in an $S$-category of $S$-objects, is only devoted to working out special classes of diagrams in the theory).
\end{rem}

There are a couple instances where the notion of $S$-colimit specializes to a notion of ordinary category theory. For example, we have the following pair of propositions computing $S$-colimits and $S$-limits in an $S$-category of objects $\underline{C}_S$ as left or right Kan extensions in $C$; the asymmetry in their formulations arises due to working with \emph{cocartesian} fibrations instead of cartesian fibrations to model $S$-categories. In the statements, recall Notation~\ref{ntn:CatOfObjectsCorresp} for the meaning of $(-)^{\dagger}$.

\begin{prp} \label{prp:identifyingColimitsInCatOfObjects} Let $\overline{p}: K \star_S S \to \underline{C}_S$ be a $S$-functor extending $p: K \to \underline{C}_S$. Suppose further that a left Kan extension of $p^\dagger: K \to C$ to a functor $K \star_S S \to C$ exists. Then the following are equivalent:
\begin{enumerate}
	\item $\overline{p}$ is a $S$-colimit diagram.
	\item $\overline{p}^\dagger$ is a left Kan extension of $p^\dagger$.
	\item $\overline{p}^\dagger|_{K_s^\rhd}$ is a colimit diagram for all $s \in S$.
\end{enumerate}
\end{prp}
\begin{proof} (2) and (3) are equivalent because left Kan extensions along cocartesian fibrations are computed fiberwise. Suppose (3). To prove (1), we want to show that for every $s \in S$, $\overline{p}_{\underline{s}}$ is an initial object in $((\underline{C}_S)^{(p,S)/})_s$. But $((\underline{C}_S)^{(p,S)/})_s$ is equivalent to the fiber of $\Fun(K_{\underline{s}} \star_{\underline{s}} \underline{s},C) \to \Fun(K_{\underline{s}},C)$ over $p^\dagger|_{K_{\underline{s}}}$, so to prove the claim it suffices to show that the functor $\overline{p}^\dagger|_{K_{\underline{s}}}$ is a left Kan extension of $p|_{K_{\underline{s}}}$. This holds by the equivalence of (2) and (3) for $S^{s/}$. 

Conversely, suppose (1). Since we supposed that a left Kan extension of $p^\dagger$ exists, left Kan extensions of $p^\dagger|_{K_s}$ all exist and \emph{any} initial object in the fiber of $\Fun(K_{\underline{s}} \star_{\underline{s}} \underline{s},C) \to \Fun(K_{\underline{s}},C)$ over $p^\dagger|_{K_{\underline{s}}}$ is a left Kan extension of $p^\dagger|_{K_{\underline{s}}}$, necessarily a fiberwise colimit diagram (we need this hypothesis because Kan extensions as defined in \cite[\S 4.3.2]{HTT} are always \emph{pointwise} Kan extensions). This implies (3).
\end{proof}

\begin{prp} \label{prp:LimitsInCatOfObjects} Let $\overline{p}: S \star_S K \to \underline{C}_S$ be a $S$-functor extending $p: K \to \underline{C}_S$. Suppose further that a right Kan extension of $p^\dagger: K \to C$ to a functor $S \star_S K \to C$ exists. Then the following are equivalent:
\begin{enumerate}
	\item $\overline{p}$ is a $S$-limit diagram.
	\item $\overline{p}^\dagger$ is a right Kan extension of $p^\dagger$.
	\item[(2\textquotesingle)] $\overline{p}^\dagger|_{\underline{s} \star_{\underline{s}} K_{\underline{s}}}$ is a right Kan extension of $p^\dagger|_{K_{\underline{s}}}$ for all $s \in S$.
	\item $\overline{p}^\dagger|_{ K_{\underline{s}}^\lhd}$ is a limit diagram for all $s \in S$.
\end{enumerate}
\end{prp}
\begin{proof} We first observe that because the inclusion $S \to S \star_S K$ is left adjoint to the structure map $S \star_S K \to S$ of the cocartesian fibration, 
\[ (S \star_S K)^{s/} \simeq S^{s/} \times_S (S \star_S K) \cong \underline{s} \star_{\underline{s}} K_{\underline{s}}. \]
The equivalence of (2) and (2\textquotesingle) now follows from the formula for a right Kan extension. Also, if we view $K_{\underline{s}}^\lhd$ as mapping to $S \star_S K$ via $\{ s\} \star K_{\underline{s}} \to \underline{s} \star_{\underline{s}} K_{\underline{s}} \to S \star_S K$ where the first map is adjoint to $(\{s\} \to \underline{s}, \id)$, then (2) and (3) are also equivalent by the same argument. Finally, (2\textquotesingle) implies (1) by definition, and (1) implies (2\textquotesingle) under our additional assumption that a right Kan extension of $p^\dagger$ exists (for the same reason as given in the proof of Proposition~\ref{prp:identifyingColimitsInCatOfObjects}).
\end{proof}

If $S$ is a Kan complex, then the notion of $S$-colimit reduces to the usual notion of colimit. 

\begin{prp} Let $S$ be a Kan complex. Then a $S$-functor $\overline{p}: K \star_S S \to C$ is a $S$-colimit diagram if and only if for every object $s \in S$, $\overline{p}|_s: (K_s)^\rhd \to C_s$ is a colimit diagram.
\end{prp}
\begin{proof} If $S$ is a Kan complex, then for every $s \in S$, $S^{s/}$ is a contractible Kan complex. Therefore,
for all $s \in S$ we have $(C^{(p,S)/})_s \simeq \{ p_s\} \times_{\Fun(K_s,C_s)} \Fun(K_s^\rhd, C_s)$, which proves the claim.
\end{proof}

We say that $K$ is a \emph{constant} $S$-category if it is equivalent to $S \times L$ for $L$ an $\infty$-category. We have an isomorphism $L^\rhd \times S \to (L \times S) \star_S S$ (defined as a map over $S \times \Delta^1$ to be the adjoint to the identity on $(L \times S, S)$).

\begin{prp} \label{prp:ColimitsOfConstantDiagrams} A $S$-functor $\overline{p}: L^\rhd \times S \to C$ is a $S$-colimit diagram if and only if for every object $s \in S$, $\overline{p}_s: L^\rhd \to C_s$ is a colimit diagram.
\end{prp}
\begin{proof} Observe that
\[ (C^{(p,S)/})_s = \{ p_{\underline{s}} \} \times_{\Fun_{S^{s/}}(L \times S^{s/}, C_{\underline{s}})} \Fun_{S^{s/}}(L^\rhd \times S^{s/}, C_{\underline{s}}) \simeq \{ p_s \} \times_{\Fun(L,C_s)} \Fun(L^\rhd,C_s). \]
Therefore, $\sigma_{\overline{p}}: S \to C^{(p,S)/}$ is $S$-initial if and only if for all $s \in S$, $\{ \overline{p}_s \} \in \{ p_s \} \times_{\Fun(L,C_s)} \Fun(L^\rhd,C_s)$ is an initial object, which is the claim.
\end{proof}

\begin{cor} \label{cor:ConstantCocompleteness} Suppose $C$ is a $S$-category such that $C_s$ admits all colimits for every object $s \in S$ and the pushforward functors $\alpha_!: C_s \to C_t$ preserve all colimits for every morphism $\alpha: s \rightarrow t$ in $S$. Then $C$ admits all $S$-colimits indexed by constant diagrams.
\end{cor}
\begin{proof} First suppose that $S$ has an initial object $s$. Suppose that $p: L \times S \to C$ is a $S$-functor. Let $\overline{p_s}: L^\rhd \to C_s$ be a colimit diagram extending $p_s$. Let $\overline{p}: L^\rhd \times S \to C$ be a $S$-functor corresponding to $\overline{p_s}$ under the equivalence $\Fun_S(L^\rhd \times S, C) \simeq \Fun(L^\rhd, C_s)$, which we may suppose extends $p$. By Proposition~\ref{prp:ColimitsOfConstantDiagrams}, $\overline{p}$ is a $S$-colimit diagram.

The general case now follows from Theorem~\ref{thm:ExistenceAndUqnessOfParamColimit}, taking $\phi: C \to D$ to be $L \times S \to S$.
\end{proof}

We now turn to the example of corepresentable fibrations.

\begin{dfn} Let $s \in S$ be an object and let $K$ be an $S^{s/}$-category which is equivalent to a coproduct of corepresentable fibrations $$\coprod_{i \in I} S^{\alpha_i /} \simeq \coprod_{i \in I} S^{t_i /} \xrightarrow{\coprod \alpha_i^\ast} S^{s/}$$ for $\{ \alpha_i: s \rightarrow t_i \}_{i \in I}$ a collection of morphisms in $S$. Let $p: K \to C \times_S S^{s/}$ be a $S^{s/}$-functor, so $p$ is precisely the data of objects $\{ x_i \in C_{t_i} \}_{i \in I}$. Let $\overline{p}: K \star_{S^{s/}} S^{s/} \to C \times_S S^{s/}$ be a $S^{s/}$-colimit diagram extending $p$, and let $y = \overline{p}(v) \in C_s$ for $v = \id_s$ the cone point. Then we say that $y$ is the \emph{$S$-coproduct} of $\{ x_i\}_{i \in I}$ along $\{ \alpha_i \}_{i \in I}$, and we adopt the notation $y = \coprod_{\alpha_i} x_i$.
\end{dfn}

\nomenclature[coproduct]{$\coprod_{\alpha_i} x_i$}{Indexed coproduct}

Our choice of terminology is guided by the following result, which shows that a $S^{s/}$-colimit of a $S^{s/}$-functor $p: S^{\alpha/} \simeq S^{t/} \to C$ obtains the value of a left adjoint to the pushforward functor $\alpha_!$ on $p(t)$. In the case of $S = \OO_G^\op$, $C = \underline{\Top}_G$ or ${\underline{\Sp}}^G$, and $K= \OO_H^\op$, this is the induction or indexed coproduct functor from $H$ to $G$.

\begin{prp} \label{colimitOverCorepresentableDiagramIsLeftAdjointToRestriction} Let $C$ be a $S$-category, let $\alpha: s \rightarrow t$ be a morphism in $C$, and let $\pi: M \to \Delta^1$ be a \textbf{cartesian} fibration classified by the pushforward functor $\alpha_!: C_s \to C_t$. Let $p: S^{t/} \to C \times_S S^{s/}$ be a $S^{s/}$-functor and let $x = p(\id_t) \in C_t$. Then the data of a $S^{s/}$-colimit diagram extending $p$ yields a $\pi$-cocartesian edge $e$ in $M$ with $d_0(e) =x$ and lifting $0 \rightarrow 1$.
\end{prp}
\begin{proof} Let $\overline{p}: S^{t/} \star_{S^{s/}} S^{s/} \to C \times_S S^{s/}$ be a $S^{s/}$-colimit diagram extending $p$. Let $y = \overline{p}(\id_s)$ and let $f': \Delta^1 \to S^{t/} \star_{S^{s/}} S^{s/}$ be the edge connecting $\id_t$ to $\alpha$. We may suppose that $M$ is given by the relative nerve of $\alpha_!$, so that edges in $M$ over $\Delta^1$ are given by commutative squares
\[ \begin{tikzcd}[row sep=2em, column sep=2em]
\{ 1\} \ar{r} \ar{d} & C_s \ar{d}{\alpha_!} \\
\Delta^1 \ar{r} & C_t.
\end{tikzcd} \]
Then let $e$ be the edge in $M$ determined by $y$ and $f = \overline{p} \circ f': x \rightarrow \alpha_! y$. By definition, $d_0(e) = x$.

We claim that $e$ is $\pi$-cocartesian. This holds if and only if for every $y' \in C_s$ the map
\[ \Map_{C_s}(y,y') \to \Map_{C_t}(x, \alpha_! y') \]
induced by $f$ is an equivalence. But the local variant of the adjunction of Theorem~\ref{thm:leftKanExtensionIsLeftAdjointToRestriction} implies this (passing to global sections).
\end{proof}

$S$-coproducts also satisfy a base-change condition. This is awkward to articulate in general, because the pullback of a corepresentable fibration along another need not be corepresentable. However, if we impose the additional hypothesis that $T = S^{\op}$ admits multipullbacks, then a pullback of a corepresentable fibration decomposes as a finite coproduct of corepresentable fibrations. In this case, we have the following useful reformulation of the base-change condition. Recall from the introduction that we let $\FF_T$ denote the finite coproduct completion of $T$. Let $X \subset \sO(\FF_T)$ be the full subcategory on those arrows whose source lies in $T$ and consider the span
\[ (\FF_T)^\sharp \xleftarrow{\ev_1} \leftnat{X} \xrightarrow{\ev_0} T^\sharp. \]
This satisfies the dual of the hypotheses of Theorem~\ref{thm:FunctorialityOfCocartesianModelStructure}, so
$$C^{\times} \coloneq (\ev_0)_\ast (\ev_1)^\ast (\rightnat{(C^\vee)})$$
is a cartesian fibration over $\FF_T$ (with the cartesian edges marked), where $C^\vee \to T$ is the dual cartesian fibration of \cite{BGN}. Unwinding the definitions, given a finite $T$-set $U = \coprod_i s_i$, we have that the fiber
$$(C^{\times})_{U} \simeq \Fun_{T}(\coprod_i T^{/s_i},C^\vee) \simeq \prod_i C_{s_i}$$
(where $\Fun_T(-,-)$ denotes those functors over $T$ that preserve cartesian edges), and given a morphism of $T$-sets $\alpha: U \to V$, the pullback functor $\alpha^\ast: (C^{\times})_U \to (C^{\times})_V$ is induced by restriction.

\begin{prp} \label{prp:BeckChevalley} $C$ admits finite $S$-coproducts if and only if $\pi: C^{\times} \to \FF_{T}$ is a \textbf{Beck-Chevalley fibration}, i.e. $\pi$ is both cocartesian and cartesian, and for every pullback square
\[ \begin{tikzcd}[row sep=2em, column sep=2em]
W \ar{r}{\alpha'} \ar{d}{\beta'} & V' \ar{d}{\beta} \\
U \ar{r}{\alpha} & V
\end{tikzcd} \]
in $\FF_T$, the natural transformation 
\begin{equation} \label{eqn:basechange} (\alpha')_! (\beta')^\ast \to \beta^\ast \alpha_! \tag{$\ast$} \end{equation}
adjoint to the equivalence $(\beta')^\ast \alpha^\ast \simeq (\alpha')^\ast \beta^\ast$ is itself an equivalence.
\end{prp}
\begin{proof} By Theorem~\ref{thm:leftKanExtensionIsLeftAdjointToRestriction}, $C$ admits finite $S$-coproducts if and only if for every finite collection of morphisms $ \{ \alpha_i: s \rightarrow t_i \}$, the restriction functor
\[ (\coprod \alpha_i)^\ast: \underline{\Fun}_S(S^{s/},C) \to \underline{\Fun}_S(\coprod_i S^{t_i/},C)  \]
admits a left $S$-adjoint, in which case that left $S$-adjoint is computed by the $S$-coproduct along the $\alpha_i$. This in turn is immediately equivalent to $\pi$ being additionally cocartesian and (\ref{eqn:basechange}) being an equivalence for $\alpha = \coprod \alpha_i: \coprod t_i \rightarrow s$ and all morphisms $\beta: s' \rightarrow s$ in $T$. Finally, note that the apparently more general case of (\ref{eqn:basechange}) being an equivalence for any pullback square is actually determined by this, because any map $\alpha: U = \coprod t_i \rightarrow V = \coprod s_j$ is the data of $f: I \rightarrow J$ and $\{ \alpha_{i j}: s_j \rightarrow t_i \}_{i \in f^{-1}(j)}$, whence $\alpha^\ast = (\alpha_{i j})^\ast: \prod_j C_{s_j} \to \prod_i C_{t_i}$, etc. yields a decomposition of the map (\ref{eqn:basechange}) in terms of the `basic' squares that we already handled.
\end{proof}

We conclude this subsection by introducing a bit of useful terminology.

\begin{dfn} \label{dfn:cocomplete} Let $C$ be a $S$-category. We say that $C$ is \emph{$S$-cocomplete} if, for every object $s \in S$ and $S^{s/}$-diagram $p: K \to C_{\underline{s}}$ (with $K$ fiberwise small), $p$ admits a $S^{s/}$-colimit.
\end{dfn}

\begin{rem} Suppose that $E$ is $S$-cocomplete. Then taking $D=S$ in Theorem~\ref{thm:ExistenceAndUqnessOfParamColimit}, $E$ admits all (small) $S$-colimits. However, the converse may fail: if we suppose that $E$ admits all $S$-colimits, then any $S^{s/}$-diagram $K_{\underline{s}} \to E_{\underline{s}}$ pulled back from a $S$-diagram $K \to E$ admits a $S^{s/}$-colimit; however, not every $S^{s/}$-diagram need be of this form.
\end{rem}

\subsection*{Vertical opposites} In this subsection we study the vertical opposite construction of \cite{BGN}, with the goal of justifying our intuition that the theory of $S$-limits can be recovered from that of $S$-colimits, and vice-versa (Corollary~\ref{cor:LimitToColimit}). We first recall the definition of the twisted arrow $\infty$-category from \cite[\S 2]{M1}.

\begin{dfn} Given a simplicial set $X$, we define $\widetilde{\sO}(X)$ to be the simplicial set whose $n$-simplices are given by the formula
\[ \widetilde{\sO}(X)_n \coloneq \Hom((\Delta^n)^{\op} \star \Delta^n, X). \]
If $X$ is an $\infty$-category, then $\widetilde{\sO}(X)$ is the twisted arrow $\infty$-category of $X$.
\end{dfn}

\begin{wrn}
By definition, $\twa(X)$ comes equipped with a source functor $\ev_0: \twa(X) \to X^\op$ and a target functor $\ev_1: \twa(X) \to X$. In other words, twisted arrows are \emph{contravariant} in the source and \emph{covariant} in the target. This convention is opposite to that in \cite{HA}, but agrees with \cite{BGN}.
\end{wrn}

\nomenclature[twistedArrows]{$\widetilde{\sO}(S)$}{Twisted arrow $\infty$-category}

\begin{rec} Suppose $X \to T$ a cocartesian fibration. Then the simplicial set $X^\vop$ is defined to have $n$-simplices
\[ \begin{tikzcd}[row sep=2em, column sep=2em]
\leftnat{\widetilde{\sO}(\Delta^n)} \ar{r} \ar{d}[swap]{\ev_1} & \leftnat{X} \ar{d} \\
(\Delta^n)^\sharp \ar{r} & T^\sharp.
\end{tikzcd} \]
The forgetful map $X^\vop \to T$ is a cocartesian fibration with cocartesian edges given by $\widetilde{\sO}(\Delta^1)^\sharp \to \leftnat{X}$. For every $t \in T$, we have an equivalence $(X_t)^\op \xrightarrow{\simeq} (X^\vop)_t$ implemented by the map which precomposes by $\ev_0: \leftnat{\widetilde{\sO}(\Delta^n)} \to ((\Delta^n)^\op)^\flat$, which is an equivalence in $s\Set^+$.

Dually, suppose $Y \to T$ a cartesian fibration. Then the simplicial set $Y^\vop$ is defined to have $n$-simplices
\[ \begin{tikzcd}[row sep=2em, column sep=2em]
\rightnat{(\widetilde{\sO}(\Delta^n)^\op)} \ar{r} \ar{d}[swap]{\ev_0^\op} & \rightnat{Y} \ar{d} \\
(\Delta^n)^\sharp \ar{r} & T^\sharp.
\end{tikzcd} \]
and similarly the forgetful map $Y^\vop \to T$ is a cartesian fibration with fibers $(Y^\vop)_t \xleftarrow{\simeq} (Y_t)^\op$. As a warning, note that the definition of the underlying simplicial set of $(-)^\vop$ changes depending on whether the input is a cocartesian or cartesian fibration; in particular, the notation is potentially ambiguous for a bicartesian fibration. We will not apply $(-)^{\vop}$ to bicartesian fibrations in this paper.
\end{rec}

\nomenclature[vop]{$X^{\vop}$}{Vertical opposite}

Define a functor $\widetilde{\sO}'(-): s\Set^+_{/S} \to s\Set^+_{/S}$ by
\[ \widetilde{\sO}'(A \xrightarrow{\pi} S) = (\widetilde{\sO}(A),\sE_A) \xrightarrow{\pi \circ \ev_1} S \]
where an edge $e$ is in $\sE_A$ just in case $\ev_0(e)$ is marked in $A^\op$. Note that $\widetilde{\sO}(-)$ preserves colimits since it is defined as precomposition by $\Delta^\op \xrightarrow{(\rev \star \id)^\op} \Delta^\op$, and from this it easily follows that $\widetilde{\sO}'(-)$ also preserves colimits. By the adjoint functor theorem, $\widetilde{\sO}'(-)$ admits a right adjoint, which we label $(-)^\vop$; this agrees with the previously defined $(-)^\vop$ for cocartesian fibrations $\leftnat{X} \to S^\sharp$.

\begin{prp} \label{prp:VOPisQuillenRightAdjoint} The adjunction
\[ \adjunct{\widetilde{\sO}'(-)}{s\Set^+_{/S}}{s\Set^+_{/S}}{(-)^\vop} \]
is a Quillen equivalence with respect to the cocartesian model structure on $s\Set^+_{/S}$.
\end{prp}
\begin{proof} We first prove the adjunction is Quillen by employing the criterion of Lemma~\ref{lm:showingFunctorLeftQuillen}. Consider the four classes of maps which generate the left marked anodyne maps:
\begin{enumerate} \item $i: \Lambda^n_k \tohook \Delta^n$, $0<k<n$: By \cite[Lemma~12.15]{M1}, $\widetilde{\sO}(\Lambda^n_k) \tohook \widetilde{\sO}(\Delta^n)$ is inner anodyne, so $\widetilde{\sO}'(i)$ is left marked anodyne.
\item $i: \leftnat{\Lambda^n_0} \tohook \leftnat{\Delta^n}$: We can adapt the proof of \cite[Lemma~12.16]{M1} to show that $\widetilde{\sO}'(i)$ is a cocartesian equivalence in $s\Set^+_{/S}$ (even though it fails to be left marked anodyne). The basic fact underlying this is that a \emph{right} marked anodyne map is an equivalence in $s\Set^+$, so in $s\Set^+_{/S}$ if it lies entirely over an object; details are left to the reader.
\item $i: K^\flat \tohook K^\sharp$ for $K$ a Kan complex: Because $\widetilde{\sO}(K) \to K^\op \times K$ is a left fibration, $\widetilde{\sO}(K)$ is then again a Kan complex. It follows that $\widetilde{\sO}'(i)$ is left marked anodyne.
\item $(\Lambda^2_1)^\sharp \cup_{\Lambda^2_1} (\Delta^2)^\flat \tohook (\Delta^2)^\sharp$: Obvious from the definitions.
\end{enumerate}
It remains to show that for a trivial cofibration $f: \leftnat{X} \tohook \leftnat{Y}$ between fibrant objects, $\widetilde{\sO}'(f)$ is again a trivial cofibration. Since $\widetilde{\sO}(X) \to \widetilde{\sO}(Y)$ is a map of cocartesian fibrations over $S$ and the marking on $\widetilde{\sO}'(-)$ contains these cocartesian edges, by Proposition~\ref{prp:fiberwiseFibrantReplacement} it suffices to show that for every object $s \in S$, $\widetilde{\sO}'(X)_s \to \widetilde{\sO}'(Y)_s$ is an equivalence in $s\Set^+$. We have a commutative square
\[ \begin{tikzcd}[row sep=2em, column sep=2em]
\widetilde{\sO}'(X)_s \ar{r} \ar{d} & \widetilde{\sO}'(Y)_s \ar{d} \\
X_s^\sharp \ar{r}{f_s} & Y_s^\sharp
\end{tikzcd} \]
where the vertical maps are left fibrations and the bottom map is an equivalence in $s\Set^+$. Therefore, the map $X_s^\sharp \times_{Y_s^\sharp} \widetilde{\sO}'(Y)_s \to \widetilde{\sO}'(Y)_s$ is an equivalence in $s\Set^+$. Applying Proposition~\ref{prp:fiberwiseFibrantReplacement} once more, we reduce to showing that for every object $x_1 \in X$, $\widetilde{\sO}'(X)_{x_1} \to \widetilde{\sO}'(Y)_{f(x_1)}$ is an equivalence in $s\Set^+$.

Now employing the source maps, we have a commutative square
\[ \begin{tikzcd}[row sep=2em, column sep=2em]
\widetilde{\sO}'(X)_{x_1} \ar{r} \ar{d} & \widetilde{\sO}'(Y)_{f(x_1)} \ar{d} \\
\rightnat{X^\op} \ar{r}{f^\op} & \rightnat{Y^\op}
\end{tikzcd} \]
where the vertical maps are left fibrations and the bottom horizontal map is a \emph{cartesian} equivalence in $s\Set^+_{/S^\op}$. Therefore, the map $X^\op \times_{Y^\op} \widetilde{\sO}'(Y)_s \to \widetilde{\sO}'(Y)_s$ is a cartesian equivalence. By a third application of Proposition~\ref{prp:fiberwiseFibrantReplacement}, we reduce to showing that for every object $x_0 \in X$, $\widetilde{\sO}'(X)_{(x_0,x_1)} \to \widetilde{\sO}'(Y)_{(f(x_0),f(x_1))}$ is an equivalence. But now both sides are endowed with the maximal marking and the map is equivalent to $\Map_X(x_0,x_1) \xrightarrow{f_\ast} \Map_Y(f(x_0),f(x_1))$, which is an equivalence by assumption.

The fact that this Quillen adjunction is an equivalence follows immediately from \cite[Theorem~1.4]{BGN}.	
\end{proof}

\begin{lem} \label{lm:functorialityofVOP} Let $C \to S$ be a cocartesian fibration.
\begin{enumerate}
	\item Let $f: S' \to S$ be a functor. Then we have an isomorphism $f^\ast(C^\vop) \cong f^\ast(C)^\vop$.
	\item Let $g: S \to T$ be a cartesian fibration and let $C$ be a $S$-category. Then there is a $T$-functor $\chi: g_\ast(C)^\vop \to g_\ast(C^\vop)$ natural in $C$ which is an equivalence. 
\end{enumerate}
\end{lem}
\begin{proof} (1) is obvious from the definitions. For (2), the map $\chi$ is defined as follows: an $n$-simplex of $g_\ast(C)^\vop$ over $\sigma \in T_n$ is given by the data of a commutative diagram
\[ \begin{tikzcd}[row sep=2em, column sep=2em]
\leftnat{\widetilde{\sO}(\Delta^n)} \times_{T^\sharp} S^\sharp \ar{r} \ar{d} & \leftnat{C} \ar{d} \\
(\Delta^n \times_T S)^\sharp \ar{r}{g^\ast \sigma} & S^\sharp
\end{tikzcd} \]
and precomposition by the obvious map $\widetilde{\sO}(\Delta^n \times_T S) \to \widetilde{\sO}(\Delta^n) \times_T S$ yields an $n$-simplex of $g_\ast(C^\vop)$.

We now show that for all $t \in T$, $\chi_t$ is a categorical equivalence. Because $\chi_t$ is obtained by taking levelwise $0$-simplices of the map of complete Segal spaces
\[ \Map_S(\leftnat{\widetilde{\sO}(\Delta^\bullet}) \times S_t^\sharp, \leftnat{C}) \to \Map_S(\leftnat{\widetilde{\sO}}(\Delta^\bullet) \times \widetilde{\sO}(S_t)^\sharp, \leftnat{C}), \]
it suffices to show that for all $n$, $\leftnat{\widetilde{\sO}(\Delta^n)} \times \widetilde{\sO}(S_t)^\sharp \to \leftnat{\widetilde{\sO}(\Delta^n)} \times S_t^\sharp$ is a cocartesian equivalence in $s\Set^+_{/S}$. As a special case of Proposition~\ref{prp:endFormula}, $\widetilde{\sO}(S_t)^\sharp \to S_t^\sharp$ is a cocartesian equivalence in $s\Set^+_{/ S_t}$, so the claim follows.
\end{proof}

\begin{lem} \label{lm:StructureMapofTwistedArrowCatofSimplexIsLMA} The map $\ev^\op: \rightnat{(\widetilde{\sO}(\Delta^n)^\op)} \to (\Delta^n)^\sharp \times ((\Delta^n)^\op)^\flat$ is left marked anodyne.
\end{lem}
\begin{proof} For convenience, we will relabel $\widetilde{\sO}(\Delta^n)^\op$ as the nerve of the poset $I_n$ with objects $ij$, $0 \leq i \leq j \leq n$ and maps $ij \to kl$ for $i \leq k$ and $j \leq l$. Then an edge $ij \rightarrow kl$ is marked in $I_n$ just in case $j=l$, and the map $\ev^\op$ becomes the projection $\rho_n: I_n \to (\Delta^n)^\sharp \times (\Delta^n)^\flat$, $ij \mapsto (i,j)$. Let $f_n: (\Delta^n)^\flat \to I_n$ be the map which sends $i$ to $0i$. Then $\rho_n \circ f_n: \{0\} \times (\Delta^n)^\flat \to (\Delta^n)^\sharp \times (\Delta^n)^\flat$ is left marked anodyne, so by the right cancellativity of left marked anodyne maps it suffices to show that $i_n$ is left marked anodyne. For this, we factor $f_n$ as the composition
\[ (\Delta^n)^\flat = I_{n,-1} \to I_{n,0} \to \dots \to I_{n,n} = I_n \]
where $I_{n,k} \subset I_n$ is the subcategory on objects $ij$, $i=0$ or $j \leq k$ (and inherits the marking from $I_n$), and argue that each inclusion $g_k: I_{n,k} \subset I_{n,k+1}$ is left marked anodyne. For this, note that $g_k$ fits into a pushout square
\[ \begin{tikzcd}[row sep=2em, column sep=2em]
\{0\} \times (\Delta^{k+1})^\flat \bigcup_{\{0\} \times (\Delta^k)^\flat} (\Delta^{n-k-1})^\sharp \times (\Delta^k)^\flat  \ar{r} \ar{d} &  (\Delta^{n-k-1})^\sharp \times (\Delta^{k+1})^\flat \ar{d} \\
I_{n,k} \ar{r}{g_k} & I_{n,k+1}
\end{tikzcd} \]
with the upper horizontal map marked left anodyne.
\end{proof}

\begin{cnstr} \label{DualizingBifunctor}
Suppose $T$ an $\infty$-category, $X, Z \to T$ cocartesian fibrations, $Y \to T$ a cartesian fibration, and a map $\mu: \leftnat{X} \times_T \rightnat{Y} \to \leftnat{Z}$ of marked simplicial sets over $T$. We define a map
$$\mu^\vop: \leftnat{X^\vop} \times_T \rightnat{Y^\vop} \to \leftnat{Z^\vop}$$
by the following process:

Let $J_n$ be the nerve of the poset with objects $ij$, $0 \leq i \leq n$, $-n \leq j \leq n$ and $-j \leq i$ and maps $ij \rightarrow kl$ if $i \leq k$, $j \leq l$.  Mark edges $ij \to kl$ if $j=l$. Let $I_n \subset J_n$ be the subcategory on $ij$ with $j \geq 0$ and $I'_n \subset J_n$ be the subcategory on $ij$ with $j \leq 0$; also give $I_n$, $I'_n$ the induced markings. We have an inclusion $(\Delta^n)^\sharp \to J_n$ given by $i \mapsto i0$ which restricts to inclusions $(\Delta^n)^\sharp \to I_n$, $(\Delta^n)^\sharp \to I'_n$ and induces a map $\gamma_n: I_n \cup_{(\Delta^n)^\sharp} I'_n \subset J_n$.

Define auxiliary (unmarked) simplicial sets $Z' \to T$ by $\Hom_{/T}(\Delta^n,Z') = \Hom_{/T}(J_n,\leftnat{Z})$ and $Z'' \to T$ by $\Hom_{/T}(\Delta^n,Z'') = \Hom_{/T}(I_n \cup_{(\Delta^n)^\sharp} I'_n, \leftnat{Z})$, where $J_n \to \Delta^n$ via $ij \mapsto i$. We have a map $r: Z' \to Z''$ given by restriction along the $\gamma_n$, which we claim is a trivial fibration. By a standard reduction, for this it suffices to show that $\gamma_n$ is left marked anodyne. Indeed, this follows from Lemma~\ref{lm:StructureMapofTwistedArrowCatofSimplexIsLMA} applied to $I_n \to (\Delta^n)^\sharp \times \Delta^n$ and the observation that the map $\Delta^n \times \Delta^n \cup_{\Delta^n} I'_n \to J_n$ is inner anodyne, whose proof we leave to the reader.

Define also a map $Z' \to Z^{\vop}$ over $T$ by restriction along the map $\leftnat{\widetilde{\sO}(\Delta^n)} \to J_n$ which sends $ij$ to $j n$ if $i=0$ and $j (-i)$ otherwise. Finally, define a map $X^\vop \times_T Y^\vop \to Z''$ over $T$ as follows: a map $\Delta^n \to X^\vop \times_T Y^\vop$ is given by the data
\[ \begin{tikzcd}[row sep=2em, column sep=2em]
\leftnat{\widetilde{\sO}(\Delta^n)} \ar{r} \ar{d} & \leftnat{X} \ar{d} \\
(\Delta^n)^\sharp \ar{r} & T^\sharp
\end{tikzcd},
\begin{tikzcd}[row sep=2em, column sep=2em]
\rightnat{(\widetilde{\sO}(\Delta^n)^\op)} \ar{r} \ar{d} & \rightnat{Y} \ar{d} \\
(\Delta^n)^\sharp \ar{r} & T^\sharp.
\end{tikzcd} \]
We have isomorphisms $\leftnat{\widetilde{\sO}(\Delta^n)} \cong I'_n$ and $\rightnat{(\widetilde{\sO}(\Delta^n)^\op)} \cong I_n$, and obvious retractions $I_n \cup_{(\Delta^n)^\sharp} I'_n \to I_n, I_n'$ given by collapsing the complementary part onto $\Delta^n$. Using this, we may define
\[  I_n \cup_{(\Delta^n)^\sharp} I'_n \to \leftnat{X} \times_T \rightnat{Y} \to \leftnat{Z} \]
which is an $n$-simplex of $Z''$.

Choosing a section of $r$, we may compose these maps to define $\mu^\vop$, which is then easily checked to also preserve the indicated markings. For example, $\mu^\vop$ on edges is given by
\[ \left( \begin{tikzcd}[row sep=1em, column sep=2em]
 & x_{11} \ar{d} \\
x_{00} \ar{r} & x_{01}, \\
y_{01} \ar{r} \ar{d} & y_{11} \\
y_{00} &
\end{tikzcd} \right) \mapsto \left(
\begin{tikzcd}[row sep=2em, column sep=2em]
 & \mu(x_{11},y_{11}) \ar{d} \\
\mu(x_{00},y_{01}) \ar{r} \ar{d} & \mu(x_{01},y_{11}) \ar{d}  \\
\mu(x_{00},y_{00}) \ar{r} & \alpha_! \mu(x_{00},y_{00})
\end{tikzcd} \right) \mapsto \left(
\begin{tikzcd}[row sep=2em, column sep=2em]
& \mu(x_{11},y_{11}) \ar{d} \\
\mu(x_{00},y_{00}) \ar{r} & \alpha_! \mu(x_{00},y_{00})
\end{tikzcd} \right) \]
where $\alpha_! \mu(x_{00},y_{00})$ is a choice of pushforward for the edge $\alpha$ in $T$ that the diagrams are vertically over.
\end{cnstr}

\begin{lem} \label{lm:VOPcommuteswithPairing} Let $C \to T$ be a cartesian fibration and let $D \to T$ be a cocartesian fibration. There exists a $T$-equivalence $\psi: \widetilde{\Fun}_T(C,D)^\vop \to \widetilde{\Fun}_T(C^\vop,D^\vop)$.
\end{lem}
\begin{proof} We have a map $\mu: \widetilde{\Fun}_T(C,D) \times_T C \to D$ adjoint to the identity. Employing Construction~\ref{DualizingBifunctor} on $\mu$ and then adjointing, we obtain our desired $T$-functor $\psi$. A chase of the definitions then shows that for all objects $t \in T$, $\psi_t$ is homotopic to the known equivalence $\Fun(C_t,D_t)^\op \simeq \Fun(C_t^\op, D_t^\op)$.
\end{proof}

\begin{lem} \label{lm:VOPcommuteswithJoin} Let $K$ and $L$ be $S$-categories. Then there exists a $S$-equivalence
\[ \psi: (K \star_S L)^\vop \xrightarrow{\simeq} L^\vop \star_S K^\vop \]
over $S \times \Delta^1$.
\end{lem}
\begin{proof} Note that $(S \times \Delta^1)^\vop \cong S \times (\Delta^1)^\op$. View $(K \star_S L)^\vop$ as lying over $S \times \Delta^1$ via the isomorphism $(\Delta^1)^\op \cong \Delta^1$. Since $(K \star_S L)^\vop_0 \cong L^\vop$ and $(K \star_S L)^\vop_1 \cong K^\vop$, we have our $S$-functor $\psi$ as adjoint to the identity over $S \times \partial \Delta^1$. Fiberwise, $\psi_s$ is homotopic to the known isomorphism $(K_s \star L_s)^\op \cong L_s^\op \star K_s^\op$, so $\psi$ is an equivalence.
\end{proof}

\begin{prp} Suppose $S$-categories $K$ and $C$.
\begin{enumerate}
	\item The adjoint of the vertical opposite of the evaluation map induces a equivalence
\[ \underline{\Fun}_S(K,C)^\vop \xrightarrow{\simeq} \underline{\Fun}_S(K^\vop, C^\vop). \]
	\item Suppose a $S$-functor $p:K \to C$. We have equivalences
\[ (C^{(p,S)/})^{\vop} \simeq (C^\vop)^{/(p^\vop,S)}, \quad (C^{/(p,S)})^{\vop} \simeq (C^\vop)^{(p^\vop,S)/}. \]
\end{enumerate}
\end{prp}
\begin{proof}
\begin{enumerate}[leftmargin=*]
\item Recall from \ref{eqn:FunctorAsEndComparisonMap} the equivalence $\underline{\Fun}_S(K,C) \simeq \pi_\ast \pi'^\ast \{K,C\}_S$. By Lemma~\ref{lm:VOPcommuteswithPairing} and Lemma~\ref{lm:functorialityofVOP}(1), $$\{K,C\}_S^\vop \simeq \{K^\vop,C^\vop\}_S.$$ By Lemma~\ref{lm:functorialityofVOP}(1) and (2), $$\pi_\ast \pi'^\ast \{K,C\}_S^\vop \simeq (\pi_\ast \pi'^\ast \{K,C\}_S)^\vop.$$ Combining these equivalences supplies an equivalence $\underline{\Fun}_S(K,C)^\vop \simeq \underline{\Fun}_S(K^\vop, C^\vop)$. It is straightforward but tedious to verify that the adjoint of the vertical opposite of the evaluation map $\underline{\Fun}_S(K,C)^\vop \times_S K^\vop \to C^\vop$ is homotopic to this equivalence.

\item Combine (1), Lemma~\ref{lm:VOPcommuteswithJoin}, Proposition~\ref{prp:VOPisQuillenRightAdjoint} (which shows in particular that $(-)^\vop$ is right Quillen), and the definition of the $S$-slice category.
\end{enumerate}
\end{proof}

\begin{cor} \label{cor:LimitToColimit} Let $\overline{p}: S \star_S K \to C$ be a $S$-functor. Then $\overline{p}$ is a $S$-limit diagram if and only if $\overline{p}^\vop: K^\vop \star_S S \to C^\vop$ is a $S$-colimit diagram.
\end{cor}

This allows us to deduce statements about $S$-limits from statements about $S$-colimits, and vice-versa. For this reason, we will primarily concentrate our attention on proving statements concerning $S$-colimits (and eventually, $S$-left Kan extensions), leaving the formulation of the dual results to the reader.

\begin{wrn}
Even with Corollary~\ref{cor:LimitToColimit}, it seems difficult to deduce Proposition~\ref{prp:LimitsInCatOfObjects} concerning $S$-limits in an $S$-category of objects $\underline{C}_S$ directly from Proposition~\ref{prp:identifyingColimitsInCatOfObjects} on $S$-colimits in $\underline{C}_S$. This is because the formation of vertical opposites $\underline{C}_S \mapsto (\underline{C}_S)^{\vop}$ doesn't intertwine with any operation at the level of the $\infty$-category $C$.
\end{wrn}

\section{Assembling \texorpdfstring{$S$}{S}-slice categories from ordinary slice categories}

Suppose a $S$-functor $p:K \to C$. For every morphism $\alpha: s \rightarrow t$ in $S$, we have a functor $p_\alpha: K_s \to C_t$, and we may consider the collection of `absolute' slice categories $C_{p_{\alpha}/}$ and examine the functoriality that they satisfy. For this, we have the following basic observation: given a morphism $f: t \rightarrow t'$, \emph{covariant} functoriality of slice categories in the target yields a functor $C_{p_{\alpha}/} \to C_{p_{f \alpha}/}$, and given a morphism $g: s' \rightarrow s$, \emph{contravariant} functoriality in the source yields a functor $C_{p_{\alpha}/} \to C_{p_{\alpha g}/}$. Elaborating, we will show in this section that there exists a functor
$$F \coloneq F(p:K \to C): \twa(S) \to \Cat_\infty$$
out of the twisted arrow category $\twa(S)$ such that $F(\alpha) \simeq C_{p_{\alpha}/}$, which encodes all of this functoriality (Definition~\ref{dfn:twistedSlice}). Moreover, the right Kan extension of $F$ along the target functor $\twa(S) \to S$ is $C_{(p,S)/}$ (Theorem~\ref{thm:ordinarySliceToParamSlice}). We will end with some applications of this result to the theory of cofinality and presentability (Theorem~\ref{thm:cofinality} and Remark~\ref{rem:presentability}).

We first record a cofinality result which implies that the values of a right Kan extension along $\ev_1: \twa(S) \to S$ are computed as ends.

\begin{lem} The functor $\twa(S^{s/}) \to \twa(S) \times_S S^{s/}$ is initial.
\end{lem}
\begin{proof} Let $(\alpha: u \to t, \beta: s \to t)$ be an object of $\twa(S) \times_S S^{s/}$. We will prove that 
\[ C = \twa(S^{s/}) \times_{\twa(S) \times_S S^{s/}} (\twa(S) \times_S S^{s/})_{/(\alpha,\beta)} \]
is weakly contractible. An object of $C$ is the data of an edge
\[ \begin{tikzcd}[row sep=1em, column sep=1em]
& s \ar{ld}[swap]{f} \ar{rd}{g} & \\
x \ar{rr}{h} & & y
\end{tikzcd} \]
in $S^{s/}$, which we will abbreviate as $f \overset{h}{\to} g$, and an edge
\[ \left( \begin{tikzcd}[row sep=1em, column sep=1em]
x \ar{r}{h} & y \ar{d}{\gamma} \\
u \ar{u}{\delta} \ar{r}{\alpha} & t
\end{tikzcd},
 \begin{tikzcd}[row sep=1em, column sep=1em]
s \ar{r}{g} \ar{rd}[swap]{\beta} & y \ar{d}{\gamma} \\
 & t
\end{tikzcd} \right) \]
in $\twa(S) \times_S {S^{s/}}$, which we will abbreviate as $(h,g) \overset{(\delta,\gamma)}{\to} (\alpha,\beta)$.

Let $C_0 \subset C$ be the full subcategory on objects $c = ((f \overset{h}{\to} g), (h,g) \overset{(\delta,\gamma)}{\to} (\alpha,\beta))$ such that $\gamma$ is a degenerate edge in $S^{s/}$. We will first show that $C_0$ is a reflective subcategory of $C$ by verifying the first condition of \cite[Proposition~5.2.7.8]{HTT}. Given an object $c$ of $C$, define $c'$ to be $((f \overset{\gamma h}{\to} \beta), (\gamma h, \beta) \overset{(\delta,\id_t)}{\to} (\alpha,\beta))$ and let $e: c \to c'$ be the edge given by
 \[ \left( \begin{tikzcd}[row sep=1em, column sep=1em]
f \ar{r}{h} & g \ar{d}{\gamma} \\
f \ar{u}{\id_f} \ar{r}{\gamma h} & \beta
\end{tikzcd},
 \begin{tikzcd}[row sep=1em, column sep=1em]
(h,g) \ar{rr}{(\id_x,\gamma)} \ar{rd}[swap]{(\delta,\gamma))} & & (\gamma h, \beta) \ar{ld}{(\delta,\id_t))} \\
 & (\alpha,\beta) & 
\end{tikzcd} \right). \]
We need to show that for all $d = ((f' \overset{h'}{\to} \beta), (h',\beta) \overset{(\delta',\id)}{\to} (\alpha,\beta)) \in C_0$, $\Map_C(c',d) \overset{e^\ast}{\to} \Map_C(c,d)$ is a homotopy equivalence. The space $\Map_C(c,d)$ lies in a commutative diagram
\[ \begin{tikzcd}[row sep=2em, column sep=2em]
\Map_C(c,d) \ar{r} \ar{d} & \Map_{\twa(S^{s/})}(f \overset{h}{\to} g, f' \overset{h'}{\to} \beta) \ar{d} \\
\Map_{(\twa(S) \times_S S^{s/})_{/(\alpha,\beta)}} ((h,g),(h',\beta)) \ar{r} \ar{d} & \Map_{\twa(S) \times_S S^{s/}} ((h,g),(h',\beta)) \ar{d}{(\delta',\id)_\ast} \\
\Delta^0 \ar{r}{(\delta,\gamma)} & \Map_{\twa(S) \times_S S^{s/}}((h,g),(\alpha,\beta))
\end{tikzcd} \]
where the two squares are homotopy pullback squares. We also have the analogous diagram for $\Map_C(c',d)$, and the map $e^\ast$ is induced by a natural transformation of these diagrams. The assertion then reduces to checking that the upper square in the diagram
\[ \begin{tikzcd}[row sep=2em, column sep=2em]
 \Map_{\twa(S^{s/})}(f \overset{\gamma h}{\to} \beta, f' \overset{h'}{\to} \beta) \ar{r}{(\id_f,\gamma)^\ast} \ar{d} & \Map_{\twa(S^{s/})}(f \overset{h}{\to} g, f' \overset{h'}{\to} \beta) \ar{d}\\
\Map_{\twa(S) \times_S S^{s/}}((\gamma h,\beta),(\alpha,\beta)) \ar{r}{(\id_x,\gamma))^\ast} \ar{d} & \Map_{\twa(S) \times_S S^{s/}}((h,g),(\alpha,\beta)) \ar{d} \\
\Map_{S^{s/}}(\beta,\beta) \ar{r}{\gamma^\ast} & \Map_{S^{s/}}(g,\beta)
\end{tikzcd} \]
is a homotopy pullback square. Since $(\id_x,\gamma)$ and $(\id_f,\gamma)$ are $\ev_1$-cocartesian edges in $\twa(S)$ and $\twa(S^{s/})$ respectively, the lower and outer squares are homotopy pullback squares (where we implicitly use that the map $(\delta',\id)$ covers the identity in $S^{s/}$ to identify the long vertical maps with those induced by $\ev_1$), and the claim is proven.

To complete the proof, we will show that $c = (\beta = \beta,(\id_t, \beta) \overset{(\alpha,\id_t)}{\to} (\alpha,\beta))$ is an initial object in $C_0$. Let $d \in C_0$ be as above. In the diagram
\[ \begin{tikzcd}[row sep=2em, column sep=2em]
\Delta^0 \ar{r}{(h',\id_\beta)} \ar{d} & \Map_{\twa(S^{s/})}(\beta = \beta, f' \overset{h'}{\to} \beta) \ar{d} \\
\Delta^0 \ar{r}{(\alpha,\id_t)} \ar{d} & \Map_{\twa(S) \times_S S^{s/}}((\id_t,\beta),(\alpha,\beta)) \ar{r} \ar{d} & \Map_{\twa(S)}(\id_t,\alpha) \ar{d} \\
\Delta^0 \ar{r}{\id_\beta} & \Map_{S^{s/}}(\beta,\beta) \ar{r} & \Map_S(t,t)
\end{tikzcd} \]
we need to show that the upper square is a homotopy pullback square in order to prove that $\Map_C(c,d) \simeq \ast$. The fiber of $\twa(S)$ over $t \in S$ is equivalent to $(S_{/t})^\op$; in particular, $\id_t$ is an initial object in the fiber over $t$. Therefore, the two outer squares are both homotopy pullbacks. Since the lower right square is a homotopy pullback, this shows that all squares in the diagram are homotopy pullbacks, as desired.
\end{proof}

Let $K$ be an $S$-category. Let $J_n$ be the poset with objects $i j$ for $0 \leq i \leq j \leq 2n+1$ which has a unique morphism $i j \to k l$ if and only if $k \leq i \leq j \leq l$. Let $I_n \subset J_n$ be the full subcategory on objects $i j$ such that $i \leq n$. In view of the isomorphisms $$J_n \cong \twa(\Delta^{2n+1}) \cong \twa((\Delta^n)^\op \star \Delta^n),$$ the $I_n$ and $J_n$ extend to functors
$$I_\bullet \subset J_\bullet \cong \twa((\Delta^\bullet)^\op \star \Delta^\bullet): \Delta \to s\Set.$$
Viewing $I_n$ and $J_n$ as marked simplicial sets where $ij \to kl$ is marked just in case $k=i$, we moreover have functors to $s\Set^+$. Define the simplicial set $X: \Delta^\op \to \Set$ to be the functor
$$\Hom_{s\Set^+}(I_\bullet,\leftnat{K}) \times_{\Hom(I_\bullet,S)} \Hom((\Delta^\bullet)^\op \star \Delta^\bullet, S)$$ where $I_\bullet \subset J_\bullet \to (\Delta^\bullet)^\op \star \Delta^\bullet$ is given by the target map. An $n$-simplex of $X$ is thus the data of a diagram
\[ \begin{tikzcd}[row sep=1em, column sep=2em]
& & & k_{n n} \ar{d} \ar{r} & k_{n(n+1)} \ar{r} \ar{d} & \dots \ar{r} & k_{n(2n+1)} \ar{d} \\
& & \iddots & \vdots & \vdots & \dots & \vdots \\
& k_{11} \ar{d} \ar{r} & \dots \ar{r} & k_{1n} \ar{r} \ar{d} & k_{1(n+1)} \ar{r} \ar{d} & \dots \ar{r} & k_{1(2n+1)} \ar{d} \\
k_{00} \ar{r} & k_{01} \ar{r} & \dots \ar{r} &  k_{0 n} \ar{r} & k_{0 (n+1)} \ar{r} & \dots \ar{r} & k_{0 (2n+1)} 
\end{tikzcd} \]
where the horizontal edges are cocartesian in $K$ and the vertical edges lie over degeneracies in $S$.

Declare an edge $e$ in $X$ to be marked if the corresponding map $I_1 \to \leftnat{K}$ sends all edges to marked edges. We have a commutative square of marked simplicial sets
\[ \begin{tikzcd}[row sep=2em, column sep=2em]
X \ar{r} \ar{d} & \twa(S)^\sharp \ar{d}{\ev_0} \\
\rightnat{(K^{\vee})} \ar{r} & (S^{\op})^\sharp
\end{tikzcd} \]
where $K^{\vee} = (K^{\vop})^{\op} \to S^{\op}$ is the dual cartesian fibration and the map $X \to K^{\vee}$ is defined by restricting $I_n \to K$ to $I'_n \to K$ (where $I'_n$ is the full subcategory of $I_n$ on $ij$ with $j \leq n$). Let $\psi$ denote the resulting map from $X$ to the pullback.

\begin{lem} \label{lm:strongPushforwardTwistedVariant} $\psi: X \to \rightnat{(K^{\vee})} \times_{(S^\op)^\sharp} \twa(S)^\sharp$ is a trivial fibration of marked simplicial sets.
\end{lem}
\begin{proof} Since any lift of a marked edge in $\rightnat{(K^{\vee})} \times_{(S^\op)^\sharp} \twa(S)^\sharp$ to an edge in $X$ is marked, it suffices to prove that the underlying map of simplicial sets is a trivial fibration. 

We first show that $I_n' \subset I_n$ is left marked anodyne. Let $I_{n,k} \subset I_n$ be the full subcategory on objects $i j$ with $i \leq k$ and similarly for $I'_{n,k}$. For $0 \leq k < n$ we have a pushout decomposition
\[ \begin{tikzcd}[row sep=2em, column sep=2em]
((\Delta^{n-k})^\op)^\flat \times (\Delta^k)^\sharp \underset{((\Delta^{n-k-1})^\op)^\flat \times (\Delta^k)^\sharp}{\bigcup} ((\Delta^{n-k-1})^\op)^\flat \times (\Delta^{n+k+1})^\sharp \ar{r} \ar{d} & I'_{n,n-k} \underset{I'_{n,n-k-1}}{\bigcup} I_{n,n-k-1} \ar{d} \\
((\Delta^{n-k})^\op)^\flat \times (\Delta^{n+k+1})^\sharp \ar{r} & I_{n,n-k},
\end{tikzcd} \]
and the lefthand map is left marked anodyne by \cite[Proposition~3.1.2.3]{HTT}. It thus suffices to show that $$I'_{n,0} \cong (\Delta^n)^\sharp \to I_{n,0} \cong (\Delta^{2n+1})^\sharp$$ is left marked anodyne, and this is clear. 

We now explain how to solve the lifting problem
\[ \begin{tikzcd}[row sep=2em, column sep=2em]
\partial \Delta^n \ar{r} \ar{d} & X \ar{d} \\
\Delta^n \ar{r} \ar[dotted]{ur} & K^{\vee} \times_{S^\op} \twa(S).
\end{tikzcd} \]
To supply the dotted arrow we must provide a lift in the commutative square
\[ \begin{tikzcd}[row sep=2em, column sep=2em]
\partial I_n \cup_{\partial I'_n} I'_n \ar{r} \ar{d}{f} & \leftnat{K} \ar{d} \\
I_n \ar{r} \ar[dotted]{ur} & S^\sharp.
\end{tikzcd} \]
where $\partial I_n = \underset{[n-1] \subset [n]}{\cup} I_{n-1}$ as a simplicial subset of $I_n$ and likewise for $\partial I'_n$. Then since $I'_n \to \partial I_n \cup_{\partial I'_n} I'_n$ and $I'_n \to I_n$ are left marked anodyne, $f$ is a cocartesian equivalence in $s\Set^+_{/S}$, and the lift exists.
\end{proof}

For all $s \in S$, we have trivial cofibrations $i_s: K_s \overset{\simeq}{\to} (K^\vee)_s$, and thus commutative squares
\[ \begin{tikzcd}[row sep=2em, column sep=2em]
K_s \ar{r}{\id_s} \ar[hookrightarrow]{d} & \twa(S) \ar{d}{\ev_0} \\
K^{\vee} \ar{r} & S^\op.
\end{tikzcd} \]
From this we obtain a cofibration
$$\iota: \into{\bigsqcup_{s \in S} K_s}{K^{\vee} \times_{S^\op} \twa(S)}.$$
We have an explicit lift $\iota'$ of $\iota$ to $X$, where $K_s \to X$ is given by precomposition by $I_n \to \Delta^n$, $ij \mapsto n-i$.

By Lemma~\ref{lm:strongPushforwardTwistedVariant}, there exists a lift $\sigma$ in the commutative square
\[ \begin{tikzcd}[row sep=2em, column sep=2em]
\bigsqcup_{s \in S} K_s \ar{r}{\iota'} \ar[hookrightarrow]{d}{\iota} & X \ar{d}{\psi} \\
K^{\vee} \times_{S^\op} \twa(S) \ar{r}{=} \ar[dotted]{ru}{\sigma} &  K^{\vee} \times_{S^\op} \twa(S).
\end{tikzcd} \]
Let $\chi: X \to K$ be the functor induced by $\Delta^n \to I_n$, $i \mapsto (n-i)(n+i)$. Define the \emph{twisted pushforward}
$$\widetilde{P}: K^\vee \times_{S^\op} \twa(S) \to K$$
to be the map over $S$ given by the composite $\chi \circ \sigma$. Then for every object $\alpha: s \to t$ in $\twa(S)$, $\widetilde{P}_\alpha \circ i_s: K_s \to K_t$ is a choice of pushforward functor over $\alpha$, which is chosen to be the identity if $\alpha = \id_s$.

\begin{prp} \label{prp:endFormula} For all $A \in s\Set_{/S}$,
\[ \widetilde{P} \times_S \id_A: \rightnat{(K^{\vee})} \times_{(S^\op)^\sharp} \twa(S)^\sharp \times_S A^\sharp \to \leftnat{K} \times_S A^\sharp \]
is a cocartesian equivalence in $s\Set^+_{/A}$.
\end{prp}
\begin{proof} Let $(Z, E)$ denote the marked simplicial set $ \rightnat{(K^{\vee})} \times_{(S^\op)^\sharp} \twa(S)^\sharp$. Viewing $Z$ as $\twa(S) \times_{S^\op \times S} (K^\vee \times S)$, we see that $Z \to S$ is a cocartesian fibration with the cocartesian edges a subset of $E$. Moreover, every edge in $E$ factors as a cocartesian edge followed by an edge in $E$ in the fiber over $S$. By Proposition~\ref{prp:fiberwiseFibrantReplacement}, it suffices to verify that for all $s \in S$, $\widetilde{P}_s$ is a cocartesian equivalence in $s\Set^+$. Since $\id_s$ is an initial object in $\twa(S) \times_S \{s\}$, the inclusion of the fiber $(K^{\vee})_s^\sim \subset (Z_s,E_s)$ is a cocartesian equivalence in $s\Set^+$ by \cite[Lemma~3.3.4.1]{HTT}. We chose $\widetilde{P}$ so as to split the inclusion of $K_s$ in $Z$, so this completes the proof.
\end{proof}

Consider the commutative diagram
\[ \begin{tikzcd}[row sep=2em, column sep=2em]
\sO(S)^\sharp \times_S \leftnat{K} \ar[bend right]{rddd}[swap]{\ev_0} \ar[out=0,in=160]{rrrrd}[swap]{\ev_1}\\
& \sO(S)^\sharp \times_{\ev_1, S, \ev_1} (  \rightnat{(K^\vee)} \times_{S^\op} \twa(S)^\sharp ) \ar{r}{\pr} \ar{d}{\pr} \ar{ul}{\id_{\sO(S)} \times_S \widetilde{P}} &  \rightnat{(K^\vee)} \times_{S^\op} \twa(S)^\sharp \ar{r}{\id \times \ev_1} \ar{d}{\pr} & \rightnat{(K^\vee)} \times S^\sharp \ar{r}[swap]{\pr_S} \ar{d}[swap]{q^\vee \times \id} & S^\sharp \\
& \sO(S)^\sharp \times_{\ev_1,S,\ev_1} \twa(S)^\sharp \ar{r}{\pi'} \ar{d}{\pi} & \twa(S)^\sharp \ar{r}{\ev} & (S^\op)^\sharp \times S^\sharp \\
& S^\sharp.
\end{tikzcd} \]
Here, $\pi = \ev_0 \circ \pr_{\sO(S)}$ and $\pi' = \pr_{\twa(S)}$. Since $K^\vee \to S^\op$ is a cartesian fibration, by Theorem~\ref{thm:FunctorialityOfCocartesianModelStructure} $(q^\vee \times \id)_\ast$ is right Quillen. Therefore, given a $S$-category $C$, we obtain a $\twa(S)$-category
\[ \{K,C\}_S \coloneq (\ev^\ast \circ (q^\vee \times \id)_\ast \circ \pr_S^\ast)(\leftnat{C}). \]
Moreover, we saw in Example~\ref{exm:RKEcocartesian} that $\pi_\ast \pi'^\ast$ is right Quillen and computes right Kan extension along $\ev_1: \twa(S) \to S$. Finally, the map $\id_{\sO(S)} \times_S \widetilde{P}$ induces a $S$-functor
\begin{equation} \label{eqn:FunctorAsEndComparisonMap} \theta: \underline{\Fun}_S(K,C) \to \pi_\ast \pi'^\ast \{K,C\}_S,
\end{equation}

natural in $K$ and $C$. By Proposition~\ref{prp:endFormula} applied to $A = S^{s/}$ for all $s \in S$, $\theta$ is an equivalence.

\begin{rem} \label{rem:EndFormula} As a corollary, the global sections of $\{K,C\}_S$ are equivalent to $\Fun_S(K,C)$. If we knew that under the straightening functor $\text{St}$, $\{K,C\}_S$ was equivalent to the composite
\[ \twa(S) \to S^\op \times S \xrightarrow{\text{St}_S(K)^\op \times \text{St}_S(C) } \Cat_\infty^\op \times \Cat_\infty \xrightarrow{\Fun} \Cat_\infty, \]
then this would yield another proof of the end formula for the $\infty$-category of natural transformations, as proven in \cite[\S 6]{GHN}. As we manage to always stay within the environment of cocartesian fibrations, this identification is not necessary for our purposes.
\end{rem}

\begin{dfn} \label{dfn:twistedSlice} Given a $S$-functor $p: K \to C$ and a choice of twisted pushforward $\widetilde{P}$ for $K$, define the cocartesian section $\omega_p: \twa(S) \to \{ K,C\}_S$ to be the adjoint to
\[ p \circ \widetilde{P}:  \rightnat{K^\vee} \times_{S^\op} \twa(S)^\sharp \to \leftnat{K} \to \leftnat{C}. \]
For objects $[\alpha: s \rightarrow t]$ in $\twa(S)$, $\omega_p(\alpha) \in \Fun((K^\vee)_s,C_t)$ is the functor
\[ p_t \circ \widetilde{P}_\alpha: (K^\vee)_s \to K_t \to C_t. \]

Define the \emph{twisted slice $\twa(S)$-category} to be
\[ C^{\widetilde{(p,S)}/} \coloneq \twa(S) \times_{\{K,C\}_S} \{K \star_S S,C\}_S.\footnote{We omit the dependence on $\widetilde{P}$ from the notation.} \]
Note that the fiber of $C^{\widetilde{(p,S)}/}$ over an object $[\alpha: s \rightarrow t]$ is $C^{p_t \circ \widetilde{P}_\alpha /}$.
\end{dfn}

\nomenclature[twistedSlice]{$C^{\widetilde{(p,S)}/}$}{Twisted slice $\twa(S)$-category under a $S$-functor $p: K \to C$}

We now connect the constructions $C^{\widetilde{(p,S)}/}$ and $C^{(p,S)/}$. A check of the definitions reveals that $\theta \circ \sigma_p = \pi_\ast \pi'^\ast (\omega_p)$ for the canonical cocartesian section $\sigma_p: S \to \underline{\Fun}_S(K,C)$. We thus have a morphism of spans
\[ \begin{tikzcd}[row sep=2em, column sep=3em]
S \ar{r}[swap]{\sigma_p} \ar{d}{=} & \underline{\Fun}_S(K,C) \ar{d}{\simeq} & \underline{\Fun}_S(K \star_S S,C) \ar{d}{\simeq} \ar{l} \\
S \ar{r}[swap]{\pi_\ast \pi'^\ast (\omega_p)} & \pi_\ast \pi'^\ast \{K,C\}_S & \pi_\ast \pi'^\ast \{K \star_S S,C\}_S \ar{l}
\end{tikzcd} \]
with all objects fibrant and the right horizontal maps fibrations by a standard argument. Taking pullbacks, we deduce:

\begin{thm} \label{thm:ordinarySliceToParamSlice} We have an equivalence
\[ \pi_\ast \pi'^\ast (C^{\widetilde{(p,S)}/}) \overset{\simeq}{\to} C^{(p,S)/}. \]
In other words, the right Kan extension of $C^{\widetilde{(p,S)}/}$ along the target functor $\ev_1: \twa(S) \to S$ is equivalent to $C^{(p,S)/}$.
\end{thm}
\begin{proof} Our interpretation of this equivalence is by Example~\ref{exm:RKEcocartesian}.
\end{proof}

\subsection*{Relative cofinality}
Let us now apply Theorem~\ref{thm:ordinarySliceToParamSlice}. We have the $S$-analogue of the basic cofinality result \cite[Proposition~4.1.1.8]{HTT}.

\begin{thm} \label{thm:cofinality} Let $f: K \to L$ be a $S$-functor. The following conditions are equivalent:
\begin{enumerate} \item For every object $s \in S$, $f_s: K_s \to L_s$ is final.
\item For every $S$-functor $p: L \to C$, the functor $f^\ast: C^{(p,S)/} \to C^{(p f,S)/}$ is an equivalence.
\item For every $S$-colimit diagram $\overline{p}: L \star_S S \to C$, $\overline{p} \circ f^\rhd: K \star_S S \to C$ is a $S$-colimit diagram.
\end{enumerate}
\end{thm}
\begin{proof} (1) $\Rightarrow$ (2): Factoring $f$ as the composition of a cofibration and a trivial fibration, we may suppose that $f$ is a cofibration, in which case we may choose compatible twisted pushforward functors $\widetilde{P}_K$ and $\widetilde{P}_L$. Let $p: L \to C$ be a $S$-functor. Precomposition by $f$ yields a $\twa(S)$-functor $\widetilde{f^\ast}: C^{\widetilde{(p,S)}/} \to C^{\widetilde{(p f,S)}/}$. Passing to the fiber over an object $\alpha: s \to t$, the compatibility of $\widetilde{P}_K$ and $\widetilde{P}_L$ implies that the diagram
\[ \begin{tikzcd}[row sep=2em, column sep=3em]
(K^\vee)_s \ar{r}{(\widetilde{P}_K)_\alpha} \ar{d}[swap]{(f^\vee)_s} & K_t \ar{d}[swap]{f_t} \ar{rd}{(pf)_t} \\
(L^\vee)_s \ar{r}{(\widetilde{P}_L)_\alpha} & L_t \ar{r}{p_t} & C_t
\end{tikzcd} \]
commutes and that 
\[ (\widetilde{f^\ast})_\alpha = (f^\vee)_s^\ast: C^{p_t \circ (\widetilde{P_L})_\alpha /} \to C^{(pf)_t \circ (\widetilde{P_K})_\alpha /}. \]
By \cite[Corollary~4.1.1.10]{HTT}, $(f^\vee)_s$ is final, so by \cite[Proposition~4.1.1.8]{HTT}, $(f^\vee)_s^\ast$ is an equivalence. Consequently, $\widetilde{f^\ast}$ is an equivalence. Now by Theorem~\ref{thm:ordinarySliceToParamSlice}, $f^\ast$ is an equivalence.

(2) $\Rightarrow$ (3): Immediate from the definition.

(3) $\Rightarrow$ (1): Let $s \in S$ be any object and $\overline{p_s}: L_s^\rhd \to \Top$ a colimit diagram. Let $\overline{p}: (L \star_S S)_{\underline{s}} \to \Top$ be a left Kan extension of $\overline{p_s}$ along the full and faithful inclusion $L_s^\rhd \subset (L \star_S S)_{\underline{s}}$. By transitivity of left Kan extensions, $\overline{p}$ is a left Kan extension of its restriction to $L_{\underline{s}}$. By Proposition~\ref{prp:identifyingColimitsInCatOfObjects}, under the equivalence $\Fun(L,\Top) \simeq \Fun_S(L,\underline{\Top}_S)$, $\overline{p}$ is a $S^{s/}$-colimit diagram. By assumption, $\overline{p} \circ (f^\rhd)_{\underline{s}}$ is a $S^{s/}$-colimit diagram. By Proposition~\ref{prp:identifyingColimitsInCatOfObjects} again, $\overline{p_s} \circ f_s$ is a colimit diagram, as desired.
\end{proof}

\begin{dfn} Let $f: K \to L$ be a $S$-functor. We say that $f$ is \emph{$S$-final} if it satisfies the equivalent conditions of Theorem~\ref{thm:cofinality}. We say that $f$ is \emph{$S$-initial} if $f^\vop$ is $S$-final.
\end{dfn}

\begin{exm} Let $\adjunct{F}{C}{D}{G}$ be a $S$-adjunction (Definition~\ref{dfn:sAdjunction}). Then $F$ is $S$-initial and $G$ is $S$-final.
\end{exm}

\begin{rem} Let $C, D$ be $S$-categories and $F: C \to D$ an $S$-functor.
\begin{enumerate}
\item Suppose $F$ is fiberwise a weak homotopy equivalence. Then $F$ is a weak homotopy equivalence by \cite[Proposition~4.1.2.15]{HTT}, \cite[Proposition~4.1.2.18]{HTT}, and \cite[Proposition~3.1.5.7]{HTT}.
\item Suppose $F$ is $S$-final. Then $F$ is final. Indeed, for any diagram $p: D \to \Spc$, we have that
\[ \colim_{d \in D} p(d) \simeq \colim_{s \in S} \colim_{d \in D_s} p(d) \simeq \colim_{s \in S} \colim_{c \in C_s} p F (c) \simeq \colim_{c \in C} p F(c). \]
\item Suppose $F$ is $S$-initial. Then $F$ is initial. To show this, by (the dual of) \cite[Theorem~4.1.3.1]{HTT} it suffices to show that for every $d \in D$, $C \times_D D^{/d}$ is weakly contractible. Let $s$ be the image of $d$ in $S$. By Lemma~\ref{lm:miscCofinalityLemma}, the inclusion $C_s \times_{D_s} (D_s)^{/d} \to C \times_{D} D^{/d}$ is final, so in particular is a weak homotopy equivalence. Hence the desired conclusion follows by our assumption that $F$ is $S$-initial and \cite[Theorem~4.1.3.1]{HTT} again.
\end{enumerate}
\end{rem}

We conclude by using the twisted slice $\twa(S)$-category to give a criterion for the presentability of the $S$-slice.

\begin{rem}[Presentability of the parametrized slice] \label{rem:presentability}
Suppose the functor $S \to \Cat_\infty$ classifying the cocartesian fibration $C \to S$ factors through $\Pr^R$, i.e. $C \to S$ is a \emph{right presentable fibration}. For any $X$ a presentable $\infty$-category and diagram $f: A \to X$, $X^{f/}$ is again presentable and the forgetful functor $X^{f/} \to X$ creates limits and filtered colimits. Therefore, the twisted slice $\twa(S)$-category $C^{\widetilde{(p,S)}/}$ is a right presentable fibration. Since the forgetful functor $\Pr^R \to \Cat_{\infty}$ creates limits, by Theorem~\ref{thm:ordinarySliceToParamSlice} we deduce that $C^{(p,S)/}$ is a right presentable fibration. In particular, in every fiber there exists an initial object. However, these initial objects may fail to be preserved by the pushforward functors. In fact, even if we assume that $C \to S$ is both left and right presentable, $C$ may fail to be $S$-cocomplete.
\end{rem}

\section{Types of \texorpdfstring{$S$}{S}-fibrations}

In this section we introduce some additional classes of fibrations which are all defined relative to $S$.

\begin{dfn} \label{dfn:ScocartesianFibration} Let $\phi: C \to D$ be an $S$-functor. We say that $\phi$ is an \emph{$S$-fibration} if it is a categorical fibration. We then say that $\phi$ is an \emph{$S$-cocartesian fibration} if it is an $S$-fibration such that for every object $s \in S$, $\phi_s: C_s \to D_s$ is a cocartesian fibration, and for every square in $C$
\[ \begin{tikzpicture}[baseline]
\matrix(m)[matrix of math nodes,
row sep=4ex, column sep=4ex,
text height=1.5ex, text depth=0.25ex]
 { x_s & x_t \\
y_s  & y_t \\ };
\path[>=stealth,->,font=\scriptsize]
(m-1-1) edge node[above]{$h$} (m-1-2)
edge node[right]{$f$} (m-2-1)
(m-1-2) edge node[right]{$g$} (m-2-2)
(m-2-1) edge node[above]{$k$} (m-2-2);
\end{tikzpicture} \]

with $h$ and $k$ $\phi$-cocartesian edges over $\phi(h)=\phi(k): s \to t$, if $f$ is a $\phi_s$-cocartesian edge then $g$ is a $\phi_t$-cocartesian edge.

Dually, we say that $\phi$ is an \emph{$S$-cartesian fibration} if it is an $S$-fibration such that for every object $s \in S$, $\phi_s: C_s \to D_s$ is a cartesian fibration, and for every square in $C$ labeled as above, but now with $h$ and $k$ $\phi$-cartesian edges over $\phi(h)=\phi(k): s \to t$, if $f$ is a $\phi_s$-cartesian edge then $g$ is a $\phi_t$-cartesian edge.
\end{dfn}

Equivalently, $\phi: C \to D$ is $S$-(co)cartesian if it is a categorical fibration, fiberwise a (co)cartesian fibration, and for every edge in $S$, the cocartesian pushforward along that edge preserves (co)cartesian edges in the fibers. We formulate our definition as above so as to avoid having to make any `straightening' constructions such as choosing pushforward functors.

\begin{rem}
Declare a morphism of $S$-cocartesian fibrations $[C \xto{\phi} D] \to [C' \xto{\phi'} D']$ to be a commutative square of $S$-functors
\[ \begin{tikzcd}
C \ar{r}{F} \ar{d}{\phi} & C' \ar{d}{\phi'} \\
D \ar{r}{G} & D'
\end{tikzcd} \]
in which for all $s \in S$, $F_s$ sends $\phi_s$-cocartesian edges to $\phi'_s$ cocartesian edges. Let $\sO^{\mathrm{cocart.fib}}(\Cat^{\cocart}_{\infty/S})$ be the $\infty$-category of $S$-cocartesian fibrations and morphisms thereof. Then one has the straightening equivalence
\[ \sO^{\mathrm{cocart.fib}}(\Cat^{\cocart}_{\infty/S}) \simeq \Fun(S, \sO^{\mathrm{cocart.fib}}(\Cat_{\infty})). \]
\end{rem}

\begin{rem} $\phi: C \to D$ is a $S$-fibration if and only if $\phi: \leftnat{C} \to \leftnat{D}$ is a marked fibration.
\end{rem}

\begin{rem} \label{relCocartFibIsCocartFib} In view of \cite[Proposition~2.4.2.11]{HTT}, \cite[Lemma~2.4.2.7]{HTT}, and \cite[Proposition~2.4.2.8]{HTT}, $\phi: C \to D$ is an $S$-cocartesian fibration if and only if $\phi$ is a cocartesian fibration. However, there is no corresponding simplification of the definition of an $S$-cartesian fibration.
\end{rem}

\begin{lem} \label{fiberwiseCartEdgesAreCart} Let $\phi: C \to D$ be a $S$-cartesian fibration and let $f:x \to y$ be a $\phi_s$-cartesian edge in $C_s$. Then $f$ is a $\phi$-cartesian edge.
\end{lem}
\begin{proof} The property of being $\phi$-cartesian may be checked after base-change to the $2$-simplices of $D$. Consequently, we may suppose that $S = \Delta^1$ and $s = \{1\}$. We have to verify that for every object $w \in C$ we have a homotopy pullback square
\[ \begin{tikzpicture}[baseline]
\matrix(m)[matrix of math nodes,
row sep=4ex, column sep=4ex,
text height=1.5ex, text depth=0.25ex]
 { \Map_C(w,x) & \Map_C(w,y) \\
\Map_D(\phi w, \phi x)  & \Map_D(\phi w, \phi y). \\ };
\path[>=stealth,->,font=\scriptsize]
(m-1-1) edge node[above]{$f_\ast$} (m-1-2)
edge node[right]{$\phi_\ast$} (m-2-1)
(m-1-2) edge node[right]{$\phi_\ast$} (m-2-2)
(m-2-1) edge node[above]{$\phi(f)_\ast$} (m-2-2);
\end{tikzpicture} \]
If $w \in C_0$, for any choice of cocartesian edge $w \to w'$ over $0 \to 1$, the square is equivalent to 
\[ \begin{tikzpicture}[baseline]
\matrix(m)[matrix of math nodes,
row sep=4ex, column sep=4ex,
text height=1.5ex, text depth=0.25ex]
 { \Map_{C_1}(w',x) & \Map_{C_1}(w',y) \\
\Map_{D_1}(\phi w', \phi x)  & \Map_{D_1}(\phi w', \phi y). \\ };
\path[>=stealth,->,font=\scriptsize]
(m-1-1) edge node[above]{$f_\ast$} (m-1-2)
edge node[right]{$\phi_\ast$} (m-2-1)
(m-1-2) edge node[right]{$\phi_\ast$} (m-2-2)
(m-2-1) edge node[above]{$\phi(f)_\ast$} (m-2-2);
\end{tikzpicture} \]
Hence we may suppose that $w \in C_1$, in which case the square is a homotopy pullback square since $f$ is a $\phi_1$-cartesian edge.
\end{proof}

We next discuss an important example of $S$-(co)cartesian fibrations. Recall (Notation~\ref{ntn:fiberwiseArrowCategory}) the fiberwise arrow $S$-category $\sO_S(D)$. Fix $\phi: C \to D$ a $S$-functor.

\begin{dfn} \label{dfn:freeCocartesian} The \emph{free $S$-cocartesian} and \emph{free $S$-cartesian} fibrations on $\phi$ are the $S$-functors
\[ \text{Fr}^\cocart(\phi) \coloneq \ev_1 \circ \pr_2: C \times_D \sO_S(D) \to D, \]
 \[\text{Fr}^\cart(\phi) \coloneq \ev_0 \circ \pr_1: \sO_S(D) \times_D C \to D. \]
\end{dfn}

\nomenclature[freeCocart]{$\text{Fr}^\cocart(\phi)$}{Free $S$-cocartesian fibration on a $S$-functor $\phi$}
\nomenclature[freeCart]{$\text{Fr}^\cart(\phi)$}{Free $S$-cartesian fibration on a $S$-functor $\phi$}

\begin{prp} $\text{Fr}^\cocart(\phi)$ is a $S$-cocartesian fibration. Dually, $\text{Fr}^\cart(\phi)$ is a $S$-cartesian fibration.
\end{prp}
\begin{proof} We prove the second assertion, the proof of the first being similar but easier. First note that $\sO_S(D) \times_D C$ is a subcategory of $\sO(D) \times_D C$ stable under equivalences. Therefore, since $\ev_0: \sO(D) \times_D C \to D$ is a cartesian fibration, $\text{Fr}^\cart(\phi)$ is a categorical fibration. Moreover, for every object $s \in S$, $\text{Fr}^\cart(\phi)_s: \sO(D_s) \times_{D_s} C_s$ is the free cartesian fibration on $\phi_s: C_s \to D_s$. It remains to show that for every square
\[ \begin{tikzpicture}[baseline]
\matrix(m)[matrix of math nodes,
row sep=4ex, column sep=4ex,
text height=1.5ex, text depth=0.25ex]
 { (a \rightarrow \phi x, x) & (b \rightarrow \phi y, y) \\
(a' \rightarrow \phi x', x')  & (b' \rightarrow \phi y', y') \\ };
\path[>=stealth,->,font=\scriptsize]
(m-1-1) edge node[above]{$h$} (m-1-2)
edge node[right]{$f$} (m-2-1)
(m-1-2) edge node[right]{$g$} (m-2-2)
(m-2-1) edge node[above]{$k$} (m-2-2);
\end{tikzpicture} \]
in $\sO_S(D) \times_D C$ with the horizontal edges cocartesian over $S$ and the left vertical edge $\text{Fr}^\cart(\phi)_s$-cartesian, the right vertical edge is $\text{Fr}^\cart(\phi)_t$-cartesian. This amounts to verifying that $y \rightarrow y'$ is an equivalence in $C_t$. The above square yields a square
\[ \begin{tikzpicture}[baseline]
\matrix(m)[matrix of math nodes,
row sep=4ex, column sep=4ex,
text height=1.5ex, text depth=0.25ex]
 { x & y \\
x'  & y' \\ };
\path[>=stealth,->,font=\scriptsize]
(m-1-1) edge node[above]{$h$} (m-1-2)
edge node[right]{$f$} (m-2-1)
(m-1-2) edge node[right]{$g$} (m-2-2)
(m-2-1) edge node[above]{$k$} (m-2-2);
\end{tikzpicture} \]
in $C$ with $x \rightarrow x'$ an equivalence and the horizontal edges cocartesian over $S$, from which the claim follows.
\end{proof}

We conclude this section with an observation about the interaction between $S$-joins and $S$-cocartesian fibrations which will be used in the sequel.

\begin{lem} Let $C$, $C'$, and $D$ be $S$-categories and let $\phi,\phi':C, C' \to D$ be $S$-functors. If $\phi$ and $\phi'$ are $S$-(co)cartesian, then $\phi \star \phi': C \star_D C' \to D$ is $S$-(co)cartesian.
\end{lem}
\begin{proof} This is an easy corollary of Proposition~\ref{joinIsSCategory}.
\end{proof}

\begin{dfn} \label{dfn:Sbifibration} We say that a $S$-functor $F: C \to D \times_S E$ is a \emph{$S$-bifibration} if for all objects $s \in S$, $F_s$ is a bifibration. Observe it is then automatic that $\pr_D F$ is $S$-cartesian and $\pr_E F: C \to E$ is $S$-cocartesian.
\end{dfn}

\begin{exm} \label{exm:bifibrationFunctor} The $S$-functor
\[ \underline{\Fun}_S(K \star_S L, C) \to \underline{\Fun}_S(K,C) \times_S \underline{\Fun}_S(L,C) \]
is a $S$-bifibration by Lemma~\ref{bifibration}. In particular, for a $S$-functor $p: K \to C$, the $S$-functors $C^{(p,S)/} \to C$ and $C^{/(p,S)} \to C$ are $S$-cocartesian resp. $S$-cartesian.
\end{exm}

\section{Relative adjunctions}

In \cite[\S 7.3.2]{HA}, Lurie introduces the notion of a \emph{relative adjunction}.

\begin{dfn}[{\cite[Definition~7.3.2.2]{HA}}] \label{dfn:RelAdjLurie}
Suppose given categorical fibrations $q: C \to S$, $p: D \to S$ and functors $F: C \to D$, $G: D \to C$ over $S$. Suppose there exists a natural transformation $u: \id_C \to G F$ such that
\begin{enumerate}
\item $u$ exhibits $F$ as a left adjoint to $G$, and
\item $q(u)$ is the identity transformation from $q$ to itself.
\end{enumerate}
Then we say that the adjunction $F \dashv G$ is a \emph{relative adjunction} with respect to $S$.
\end{dfn}

\begin{rec}
By \cite[Proposition~7.3.2.5]{HA}, relative adjunctions are stable under base-change; in particular, they restrict to adjunctions over every fiber.
\end{rec}

\begin{dfn} \label{dfn:sAdjunction} Let $C$ and $D$ be $S$-categories. We call a relative adjunction (with respect to $S$)
$$\adjunct{F}{C}{D}{G}$$
an \emph{$S$-adjunction} if $F$ and $G$ are $S$-functors.
\end{dfn}

We prove some basic results about $S$-adjunctions in this section. Let us first reformulate the definition of a relative adjunction in terms of a correspondence. Let $F: C \to D$ be a $S$-functor. By the relative nerve construction, $F$ defines a cocartesian fibration $M \to \Delta^1$ by prescribing, for every $\Delta^n \cong \Delta^{n_0} \star \Delta^{n_1} \to \Delta^1$, the set $\Hom_{\Delta^1}(\Delta^n, M)$ to be the collection of commutative squares
\[ \begin{tikzpicture}[baseline]
\matrix(m)[matrix of math nodes,
row sep=4ex, column sep=4ex,
text height=1.5ex, text depth=0.25ex]
 { \Delta^{n_0} & C \\
\Delta^{n}  & D  \\ };
\path[>=stealth,->,font=\scriptsize]
(m-1-1) edge (m-1-2)
edge (m-2-1)
(m-1-2) edge node[right]{$F$} (m-2-2)
(m-2-1) edge (m-2-2);
\end{tikzpicture}
\]
for $n_1 \geq 0$, and setting $\Hom_{\Delta^1}(\Delta^{n}, M) = \Hom(\Delta^{n},C)$ for $n_1 = -1$. Moreover, the structure maps for $C$ and $D$ to $S$ define a functor $M \to S$ by sending $\Delta^n \to M$ to $\Delta^n \to D \to S$ if $n_1 \geq 0$, and $\Delta^n \to C \to S$ if $n_1<0$. Then $M$ is a $S$-category, $M \to S \times \Delta^1$ is a $S$-cocartesian fibration, and $F$ admits a right $S$-adjoint if and only if $M \to S \times \Delta^1$ is a $S$-cartesian fibration.

\begin{prp}
\label{adjFromRelAdj}
Let $\adjunct{F}{C}{D}{G}$ be a $S$-adjunction and let $I$ be a $S$-category. Then we have adjunctions
\[ \adjunct{F_\ast}{\Fun_S(I,C)}{\Fun_S(I,D)}{G_\ast} \]
\[ \adjunct{G^\ast}{\Fun_S(C,I)}{\Fun_S(D,I)}{F^\ast} \]
\end{prp}
\begin{proof} Let $M \to S \times \Delta^1$ be the $S$-functor obtained from $F$. We first produce the adjunction $F_\ast \dashv G_\ast$. Invoking Theorem~\ref{thm:FunctorialityOfCocartesianModelStructure} on the span
\[ (\Delta^1) \xleftarrow{\pi} \leftnat{I} \times (\Delta^1)^\sharp \xrightarrow{\pi'} S^\sharp \times (\Delta^1)^\sharp \]
we deduce that $\pi_\ast \pi'^\ast: s\Set^+_{/(S^\sharp \times (\Delta^1)^\sharp)} \to s\Set^+_{/(\Delta^1)^\sharp}$ is right Quillen. Let $N = \pi_\ast \pi'^\ast(M)$. Then $N \to \Delta^1$ is a cocartesian fibration classified by the functor 
\[ F_\ast: \Fun_S(I, C) \to \Fun_S(I,D). \]

Now invoking Theorem~\ref{thm:FunctorialityOfCocartesianModelStructure} on the span
\[ ((\Delta^1)^\sharp)^{\op} \xleftarrow{\rho} (I^\sim  \times (\Delta^1)^\sharp)^\op \xrightarrow{\rho'} (S^\sim \times (\Delta^1)^\sharp)^\op \]
we deduce that with respect to the cartesian model structures $\rho_\ast \rho'^\ast: s\Set^+_{/(S^\sim \times (\Delta^1)^\sharp)} \to s\Set^+_{/(\Delta^1)^\sharp}$ is right Quillen. Let $N' = \rho_\ast \rho'^\ast M$. Since $G$ is right $S$-adjoint to $F$, $N' \to \Delta^1$ is a cartesian fibration classified by the functor 
\[ G_\ast: \Fun_{/S}(I, D) \to \Fun_{/S}(I,C) \]
where we view $I,C,D$ as categorical fibrations over $S$. $N$ is a subcategory of $N'$, and the cartesian edges $e$ in $N'$ with $d_0(e) \in N$ are in $N$. Hence $N \to \Delta^1$ is also a cartesian fibration classified by the functor
\[ G_\ast: \Fun_{S}(I, D) \to \Fun_{S}(I,C). \]

We now produce the adjunction $G^\ast \dashv F^\ast$ by similar methods. Let $\sE_0$ be the collection of edges $e: x \to y$ in $M$ such that $e$ admits a factorization as a cocartesian edge over $S$ followed by a cartesian edge in the fiber. Note that since $M \to S \times \Delta^1$ is a $S$-cartesian fibration, $\sE_0$ is closed under composition of edges. Invoking Theorem~\ref{thm:FunctorialityOfCocartesianModelStructure} on the span
\[ (\Delta^1)^\sharp \xleftarrow{\mu} (M,\sE_0) \xrightarrow{\mu'} S^\sharp \times (\Delta^1)^\sharp \]
we deduce that $\mu_\ast \mu'^\ast: s\Set^+_{/(S^\sharp \times (\Delta^1)^\sharp)} \to s\Set^+_{/(\Delta^1)^\sharp}$ is right Quillen. Let $P = \mu_\ast \mu'^\ast (\leftnat{I} \times (\Delta^1)^\sharp)$. Then $P \to \Delta^1$ is a cocartesian fibration classified by the functor 
\[ G^\ast: \Fun_S(C, I) \to \Fun_S(D,I). \]

Let $\sE_1$ be the collection of edges $e: x \to y$ in $M$ such that $e$ is a cocartesian edge over an equivalence in $S$. Now invoking Theorem~\ref{thm:FunctorialityOfCocartesianModelStructure} on the span
\[ ((\Delta^1)^\sharp)^{\op} \xleftarrow{\nu} (M,\sE_1)^\op \xrightarrow{\nu'} (S^\sim \times (\Delta^1)^\sharp)^\op \]
we deduce that with respect to the cartesian model structures $\nu_\ast \nu'^\ast: s\Set^+_{/(S^\sim \times (\Delta^1)^\sharp)} \to s\Set^+_{/(\Delta^1)^\sharp}$ is right Quillen. Let $P' = \nu_\ast \nu'^\ast (I^\sim \times (\Delta^1)^\sharp)$. $P' \to \Delta^1$ is a cartesian fibration with $P$ as a subcategory. One may check that $P \to \Delta^1$ inherits the property of being a cartesian fibration, which is classified by the functor $F^\ast: \Fun_S(D, I) \to \Fun_S(C,I)$.
\end{proof}

\begin{cor} \label{cor:RelAdjFromRelAdj} Let $\adjunct{F}{C}{D}{G}$ be a $S$-adjunction and let $I$ be a $S$-category. Then we have $S$-adjunctions
\[ \adjunct{F_\ast}{\underline{\Fun}_S(I,C)}{\underline{\Fun}_S(I,D)}{G_\ast} \]
\[ \adjunct{G^\ast}{\underline{\Fun}_S(C,I)}{\underline{\Fun}_S(D,I)}{F^\ast} \]
\end{cor}
\begin{proof} By Proposition~\ref{adjFromRelAdj}, for every $s \in S$
\[ \adjunct{F_\ast}{\Fun_{S^{s/}}(I \times_S S^{s/},C \times_S S^{s/})}{\Fun_{S^{s/}}(I \times_S S^{s/},D \times_S S^{s/})}{G_\ast} \]
is an adjunction, and similarly for the contravariant case.
\end{proof}

To state the next corollary, it is convenient to introduce a definition.

\begin{dfn} \label{dfn:sectionsOfSFibration} Suppose $\pi: C \to D$ a $S$-fibration. Define the $\infty$-category $\Sect_{D/S}(\pi)$ of $S$-sections of $\pi$ to be the pullback
\[ \begin{tikzcd}[row sep=2em, column sep=2em]
\Sect_{D/S}(\pi) \ar{r} \ar{d} & \Fun_S(D,C) \ar{d}{\pi_\ast} \\
\Delta^0 \ar{r}{\id_D} & \Fun_S(D,D).
\end{tikzcd} \]
Define the $S$-category $\underline{\Sect}_{D/S}(\pi)$ to be the pullback
\[ \begin{tikzcd}[row sep=2em, column sep=2em]
\underline{\Sect}_{D/S}(\pi) \ar{r} \ar{d} & \underline{\Fun}_S(D,C) \ar{d}{\pi_\ast} \\
S \ar{r}{\sigma_{\id_D}} & \underline{\Fun}_S(D,D).
\end{tikzcd} \]
We will often denote $\Sect_{D/S}(\pi)$ by $\Sect_{D/S}(C)$, the $S$-functor $\pi$ being left implicit.
\end{dfn}

\nomenclature[sectionsParam]{$\underline{\Sect}_{D/S}(C)$}{$S$-category of sections}

Note that for any object $s \in S$, the fiber $\underline{\Sect}_{D/S}(\pi)_s$ is isomorphic to $\Sect_{D_{\underline{s}}/ \underline{s}}(\pi_{\underline{s}})$.

\begin{cor} \label{cor:sectionsOfRelativeAdjIsAdj} Let $p: C \to E$ and $q: D \to E$ be $S$-fibrations. Let $\adjunct{F}{C}{D}{G}$ be an adjunction relative to $E$ where $F$ and $G$ are $S$-functors. Then for any $S$-category $I$,
\[ \adjunct{F_\ast}{\Fun_S(I,C)}{\Fun_S(I,D)}{G_\ast} \]
is an adjunction relative to $\Fun_S(I,E)$. In particular, taking $I=E$ and the fiber over the identity, we deduce that 
\[ \adjunct{F_\ast}{\Sect_{E/S}(p)}{\Sect_{E/S}(q)}{G_\ast} \]
is an adjunction, and also that
\[ \adjunct{F_\ast}{\underline{\Sect}_{E/S}(p)}{\underline{\Sect}_{E/S}(q)}{G_\ast} \]
is a $S$-adjunction.
\end{cor}
\begin{proof} The proof of Proposition~\ref{adjFromRelAdj} shows that the unit for the adjunction $F_\ast \dashv G_\ast$ is sent by $p_\ast$ to a natural transformation through equivalences.
\end{proof}

\begin{lem} \label{lm:slicePullbackByAdjunction} Let $\adjunct{F}{C}{D}{G}$ be a $S$-adjunction. For every $S$-functor $p: K \to D$, we have a homotopy pullback square in $s\Set^+_{/S}$
\[ \begin{tikzcd}[row sep=2em, column sep=2em]
C^{/(G p,S)} \ar{r} \ar{d}{\ev^C_0} & D^{/(p,S)} \ar{d}{\ev^D_0} \\
C \ar{r}{F} & D
\end{tikzcd} \]
where the upper horizontal map is defined to be the composite $C^{/(G p,S)} \xrightarrow{F} C^{/(F G p,S)} \xrightarrow{\epsilon(p)_!} D^{/(p,S)}$. Dually, for every $S$-functor $p: K \to D$, we have a homotopy pullback square in $s\Set^+_{/S}$
\[ \begin{tikzcd}[row sep=2em, column sep=2em]
D^{(F p,S))/} \ar{r} \ar{d}{\ev^D_1} & C^{(p,S)/} \ar{d}{\ev^C_1} \\
D \ar{r}{G} & C.
\end{tikzcd} \]
where the upper horizontal map is defined to be the composite $D^{(F p,S))/} \xrightarrow{G} C^{(G F p,S)/} \xrightarrow{\eta(p)^\ast} C^{(p,S)/}$.
\end{lem}
\begin{proof} We prove the first assertion; the second then follows by taking vertical opposites. We first explain how to define the map $\epsilon(p)_!$. Choose a counit transformation $\epsilon: D \times \Delta^1 \to D$ for $F \dashv G$ such that $\pi_D \circ \epsilon$ is the identity natural transformation from $\pi_D$ to itself. Then $\epsilon \circ (p \times \id)$ is adjoint to a $S$-functor $\epsilon(p): S \times \Delta^1 \to \underline{\Fun}_S(K,D)$ with $\epsilon(p)_0 = \sigma_{F G p}$ and $\epsilon(p)_0 = \sigma_{p}$. Because $\underline{\Fun}_S(S \star_S K,D) \to D \times_S \underline{\Fun}_S(K,D)$ is an $S$-bifibration, from $\epsilon(p)$ we obtain a pushforward $S$-functor $\epsilon(p)_!: D^{/(F G p,S)} \to D^{/(p,S)}$ compatible with the source maps to $D$. 

We need to check that for every object $s \in S$, passage to the fiber over $s$ yields a homotopy pullback square of $\infty$-categories. Because $(D^{/(p,S)})_s \cong (D_{\underline{s}}^{/(p_{\underline{s}}, \underline{s})})_s$, we may replace $S$ by $S^{s/}$ and thereby suppose that $s$ is an initial object in $S$. 

Let $r: \{s\} \star S \to S$ be a left Kan extension of the identity $S \to S$. By the formula for a left Kan extension, $r(s)$ is an initial object in $S$, which without loss of generality we may suppose to be $s$. Using $r \circ (\id \star \pi_K)$ as the structure map for $\{s\} \star K$ over $S$, define $\phi': \{s\} \star \leftnat{K} \to \{s\} \star_S \leftnat{K}$ as adjoint to the identity over $S \times \partial \Delta^1$. It is easy to show that $\phi'$ is a trivial cofibration in $s\Set^+_{/S}$. Moreover, since the inclusion $\{s\} \to S^\sharp$ is a trivial cofibration,  $\{s\} \star_S \leftnat{K} \to S^\sharp \star_S \leftnat{K}$ is a trivial cofibration in $s\Set^+_{/S}$ by Theorem~\ref{thm:JoinFirstVariable}. Let $\phi$ be the composition of these two maps. Then because $\Fun_S(-,-)$ is a right Quillen bifunctor, $\phi^\ast : \Fun_S(S^\sharp \star_S \leftnat{K}, \leftnat{D}) \to \Fun_S(\{s\} \star \leftnat{K}, \leftnat{D})$ is a trivial Kan fibration.

We further claim that the inclusion $j: \Fun_S(\{s\} \star \leftnat{K}, \leftnat{D}) \to D_s \times_{D}  \Fun(\{s\} \star K, D) \times_{\Fun(K,D)} \Fun_S(\leftnat{K},\leftnat{D})$ is an equivalence. Indeed, we have the pullback square
\[ \begin{tikzcd}[row sep=2em, column sep=2em]
\Fun_S(\{s\} \star \leftnat{K}, \leftnat{D}) \ar{r} \ar{d} & D_s \times_{D}  \Fun(\{s\} \star K, D) \times_{\Fun(K,D)} \Fun_S(\leftnat{K},\leftnat{D}) \ar{d} \\
\Delta^0 \ar{r}{r \circ (\id \star \pi_K)} & \{s\} \times_S \Fun(\{s\} \star K, S) \times_{\Fun(K,S)} \{ \pi_K \}
\end{tikzcd} \]
and the term in the lower right is contractible since it is equivalent to the full subcategory $\Fun'(\{s\} \star K, S) \subset \Fun(\{s\} \star K, S)$ of functors which are left Kan extensions of $\pi_K$.

Now taking the pullback of the composition $j \circ \phi^\ast$ over $\{p\}$, we obtain an equivalence
\[ (D^{/(p,S)})_s \to D_s \times_D D^{/p}. \]
Similarly, we have an equivalence
\[ (C^{/(G p,S)})_s \to C_s \times_C C^{/G p}. \]
Since $F \dashv G$ is in particular an adjunction, by \cite[Lemma~5.2.5.5]{HTT} $C^{/G p} \to C \times_D D^{/p}$ is an equivalence. Taking the fiber over $s$, we deduce the claim.
\end{proof}

\begin{cor} \label{cor:leftAdjointPreservesColimits} Let $\adjunct{F}{C}{D}{G}$ be a $S$-adjunction. Then $F$ preserves $S$-colimits and $G$ preserves $S$-limits.
\end{cor}
\begin{proof} Let $\overline{p}: K \star_S S \to C$ be a $S$-colimit diagram. To show that $F \overline{p}$ is a $S$-colimit diagram, it suffices to prove that the restriction map $D^{(F \overline{p},S)/} \to D^{(F p,S)/}$ is an equivalence. We have the commutative square
\[ \begin{tikzcd}[row sep=2em, column sep=2em]
D^{(F \overline{p},S)/} \ar{r} \ar{d} & C^{(\overline{p},S)/} \times_C D \ar{d} \\
D^{(F p, S)/} \ar{r} & C^{(p,S)/} \times_C D
\end{tikzcd} \]
(here we suppress some details about the naturality of $\epsilon(-)_!$). The righthand vertical map is an equivalence by assumption, and the horizontal maps are equivalences by Lemma~\ref{lm:slicePullbackByAdjunction}. Thus the lefthand vertical map is an equivalence.
\end{proof}

\subsection*{Free \texorpdfstring{$S$}{S}-(co)cartesian fibrations revisited}

With the theory of $S$-adjunctions, we can now establish a key property of the free $S$-(co)cartesian fibration (Definition~\ref{dfn:freeCocartesian}). Let $\phi: C \to D$ be an $S$-functor and define $S$-functors
\[ \iota_0: C \to C \times_D \sO_S(D), \quad \iota_1: C \to \sO_S(D) \times_D C \]
via the commutative square
\[ \begin{tikzpicture}[baseline]
\matrix(m)[matrix of math nodes,
row sep=4ex, column sep=4ex,
text height=1.5ex, text depth=0.25ex]
 { C & \sO_S(D) \\
C  & D \\ };
\path[>=stealth,->,font=\scriptsize]
(m-1-1) edge (m-1-2)
edge node[right]{$=$} (m-2-1)
(m-1-2) edge node[right]{$\ev_i$} (m-2-2)
(m-2-1) edge node[above]{$\phi$} (m-2-2);
\end{tikzpicture} \]
where the upper horizontal map is the composite $C \xto{\iota} \sO_S(C) \to \sO_S(D)$.

\begin{prp} \label{prp:inclusionToFreeFibrationIsAdjoint} $\iota_0$ is left $S$-adjoint to $\pr_C$. Dually, $\iota_1$ is right $S$-adjoint to $\pr_C$.
\end{prp}
\begin{proof} We prove the first assertion, the proof of the second being similar. To prove that we have a relative $S$-adjunction $\iota_0 \dashv \pr_C$, we must prove that for each $s \in S$ we have an adjunction $(\iota_0)_s \dashv (\pr_C)_s$. So suppose that $S = \Delta^0$. Since $\pr_C \circ \iota_0 = \id$, it suffices by \cite[Proposition~5.2.2.8]{HTT} to check that the identity is a unit transformation: that is, for every $x \in C$ and $(y,\phi y \rightarrow a) \in C \times_D \sO(D)$,
\[ \pr_C: \Map_{C \times_D \sO(D)}((x,\id_{\phi x}),(y, \phi y \rightarrow a)) \to \Map_C(x,y) \]
is an equivalence. Under the fiber product decomposition
\[ \Map_{C \times_D \sO(D)}((x,\id_{\phi x}),(y, \phi y \rightarrow a)) \simeq \Map_C(x,y) \times_{\Map_D(\phi x, \phi y)} \Map_{\sO(D)}((\id_{\phi x}), (\phi y \rightarrow b)) \]
the map $\pr_C$ is projection onto the first factor. The adjunction $\adjunct{\iota}{D}{\sO(D)}{\ev_0}$ obtained by exponentiating the adjunction $\adjunct{i_0}{\{0\}}{\Delta^1}{p}$ implies that 
\[ \Map_{\sO(D)}((\id_{\phi x}), (\phi y \rightarrow b)) \to \Map_D(\phi x, \phi y)\]
is an equivalence, so the claim follows.
\end{proof}

\begin{rem}[Universal property of the free $S$-cocartesian fibration]
Let $\phi: C \to D$ be an $S$-functor and $\psi: E \to D$ be an $S$-cocartesian fibration. Then we would like to show that the restriction functor
\[ \Fun^{\cocart}_{/D}(C \times_D \sO_S(D),E) \to \Fun_{/D,S}(C,E) = S \times_{\sigma_{\phi}, \Fun_S(C,D), \psi_*} \Fun_{S}(C,E) \]
is an equivalence of $\infty$-categories.\footnote{We use Remark~\ref{relCocartFibIsCocartFib} to simplify the appearance of the lefthand side, which would otherwise be denoted as $\Fun^{\cocart}_{/D,S}(C \times_D \sO_S(D),E)$.} We prove this as \cite[Example~3.8]{Exp2b} as an application of the theory of parametrized factorization systems.
\end{rem}

\section{Parametrized colimits} \label{sec:paramcolimits}

In this section, we first introduce a parametrized generalization of Lurie's pairing construction \cite[Corollary~3.2.2.13]{HTT}. We then employ it to study \emph{$D$-parametrized $S$-(co)limits}. This material recovers and extends \cite[\S 4.2.2]{HTT} (in view of Lemma~\ref{lm:DiamondJoinComparison}). It is a precursor to our study of Kan extensions.

\subsection*{An \texorpdfstring{$S$}{S}-pairing construction}

\begin{cnstr} \label{paramPair} Let $p: C \to S$, $q: D \to S$ be $S$-categories and let $\phi: C \to D$ be a $S$-functor. Let $\pi, \pi': \sO^\cocart(D) \times_D C \to D$ be given by $\pi = \ev_0 \circ \pr_1$, $\pi' = \ev_1 \circ \pr_1$. Let $\sE$ denote the collection of edges $e$ in $\sO^\cocart(D) \times_{D} C$ such that $\pi(e)$ is $q$-cocartesian and $\pr_2(e)$ is $p$-cocartesian (so $\pi'(e)$ is $q$-cocartesian). Then the span
\[ \leftnat{D} \xleftarrow{\pi} (\sO^\cocart(D) \times_{D} C, \sE) \xrightarrow{\pi'} \leftnat{D} \]
defines a functor
\[ \pi_\ast \pi'^\ast: s\Set^+_{/\leftnat{D}} \to s\Set^+_{/\leftnat{D}}. \]
For a $S$-category $E$ and a $S$-functor $\psi: E \to D$, define
\[ (\widetilde{\Fun}_{D/S}(C,E) \to \leftnat{D}) \coloneq \pi_\ast \pi'^\ast (\leftnat{E} \xto{\psi} \leftnat{D}). \]
\end{cnstr}

\nomenclature[pairingParam]{$\widetilde{\Fun}_{D/S}(C,E)$}{Parametrized pairing construction}

\begin{lem} \label{sourceCocartArrowsProperties} Let $q: D \to S$ be a $S$-category. 
\begin{enumerate}
\item $\ev_0: \sO^\cocart(D) \to D$ is a cartesian fibration, and an edge $e$ in $\sO^\cocart(D)$ is $\ev_0$-cartesian if and only $(\ev_{S,1} \circ q)(e)$ is an equivalence in $S$. In particular, if $\ev_0(e)$ is $q$-cocartesian, then $e$ is $\ev_0$-cartesian if and only if $\ev_1(e)$ is an equivalence in $D$.
\item If $f:x \to y$ is an edge in $D$ such that $q(f)$ is an equivalence, then there exists a $\ev_0$-cocartesian edge $e$ over $f$. Moreover, an edge $e$ over $f$ is $\ev_0$-cocartesian if and only if it is $\ev_0$-cartesian.
\end{enumerate}
\end{lem}
\begin{proof} $\ev_0: \sO^\cocart(D) \to D$ factors as
\[ \sO^\cocart(D) \to D \times_S \sO(S) \to D \]
where the first functor is a trivial fibration and the second is a cartesian fibration, as the pullback of $\ev_{S,0}: \sO(S) \to S$. Thus $\ev_0$ is a cartesian fibration with cartesian edges as indicated. Moreover, since $\ev_{S,0}: \sO(S) \to S$ is a categorical fibration, the second claim follows from \cite[Proposition~B.2.9]{HA}.
\end{proof}

We have designed our construction so that for any object $x \in D$ and cocartesian section $S^{q x/} \to D$, the fiber of $\widetilde{\Fun}_{D/S}(C,E) \to D$ over $x$ is equivalent to $\Fun_{S^{q x/}}(C \times_D  S^{q x/}, E \times_D S^{q x/})$. For this reason, we think of $\widetilde{\Fun}_{D/S}(-,-)$ as the parametrized generalization of the pairing construction $\widetilde{\Fun}_D(-,-)$, to which it reduces when $S = \Delta^0$. 

\begin{thm} \label{paramPairMainThm} Notation as in Construction~\ref{paramPair}, $\widetilde{\Fun}_{D/S}(C,E)$ enjoys the following functoriality:
\begin{enumerate}
\item If $\phi$ is either a $S$-cartesian fibration or a $S$-cocartesian fibration and $\psi$ is a categorical fibration, then $\widetilde{\Fun}_{D/S}(C,E) \to S$ is a $S$-category with cocartesian edges marked as indicated in Construction~\ref{paramPair}, and $\widetilde{\Fun}_{D/S}(C,E) \to D$ is a categorical fibration.
\item If $\phi$ is a $S$-cartesian fibration and $\psi$ is a $S$-cocartesian fibration, then $\widetilde{\Fun}_{D/S}(C,E) \to D$ is a $S$-cocartesian fibration.
\item If $\phi$ is a $S$-cocartesian fibration and $\psi$ is a $S$-cartesian fibration, then $\widetilde{\Fun}_{D/S}(C,E) \to D$ is a $S$-cartesian fibration.
\end{enumerate}
\end{thm}
\begin{proof} \begin{enumerate}[leftmargin=*]
\item It suffices to check that Theorem~\ref{thm:FunctorialityOfCocartesianModelStructure} applies to the span
\[ \leftnat{D} \xleftarrow{\pi} (\sO^\cocart(D) \times_D C, \sE) \xrightarrow{\pi'} \leftnat{D}. \]
In the remainder of this proof we will verify that $\sO^\cocart(D) \times_D C \to D$ is a flat categorical fibration. For condition (4) we appeal to Lemma~\ref{sourceCocartArrowsProperties}. The rest of the conditions are easy verifications.
\item By Lemma~\ref{sourceCocartArrowsProperties} and \ref{fiberwiseCartEdgesAreCart}, $\pi: \sO^\cocart(D) \times_D C \to D$ is a cartesian fibration (hence flat) with an edge $e$ $\pi$-cartesian if and only if $\pr_1(e)$ is $\ev_0$-cartesian and $\pr_2(e)$ is $\phi$-cartesian. Let $\sE'$ be the collection of edges $e$ in $\sO^\cocart(D) \times_{\ev_1,D} C$ such that for any $\pi$-cartesian lift $e'$ of $\pi(e)$, the induced edge $d_1(e) \to d_1(e')$ is in $\sE$. Note that since $\phi$ is $S$-cartesian (and not just fiberwise cartesian), $\sE'$ is closed under composition. Invoking Theorem~\ref{thm:FunctorialityOfCocartesianModelStructure} on the span
\[ D^\sharp \xleftarrow{\pi} (\sO^\cocart(D) \times_D C, \sE') \xrightarrow{\pi'} D^\sharp \]
we deduce that
$$\pi_\ast \pi'^\ast: s\Set^+_{/D} \to s\Set^+_{/D}$$
is right Quillen. Note that there is no conflict of notation with the functor $\pi_\ast \pi'^\ast$
defined before on $s\Set^+_{/\leftnat{D}}$ since $\sE \subset \sE'$ and the two restrict to the same collections of marked edges in the fibers of $\pi$. Since $S$-cocartesian fibrations are cocartesian fibrations over $D$ (Remark~\ref{relCocartFibIsCocartFib}), we conclude.
\item First note that $\pi$ factors as a cocartesian fibration followed by a cartesian fibration, so is flat. Let $\sF$ be the collection of edges $f$ in $D$ such that $q(f)$ is an equivalence. By Lemma~\ref{sourceCocartArrowsProperties}, we have that $\pi: \sO^\cocart(D) \times_{\ev_1,D} C \to D$ admits cocartesian lifts of edges in $\sF$. Let $\sE''$ be the collection of those $\pi$-cocartesian edges. Invoking Theorem~\ref{thm:FunctorialityOfCocartesianModelStructure} on the span
\[ (D,\sF)^\op \xleftarrow{\rho} (\sO^\cocart(D) \times_D C, \sE'')^\op \xrightarrow{\rho'} (D,\sF)^\op, \]
where $\rho = \pi^\op$ and $\rho' = \pi'^\op$, we deduce that with respect to the cartesian model structures
$$\rho_\ast \rho'^\ast: s\Set^+_{/(D,\sF)} \to s\Set^+_{/(D,\sF)}$$
is right Quillen. We have that $\widetilde{\Fun}_{D/S}(C,E)$ is a full subcategory of $\rho_\ast \rho'^\ast (\psi)$. Moreover, the compatibility condition in the definition of a $S$-cartesian fibration ensures that $\widetilde{\Fun}_{D/S}(C,E) \to D$ inherits the property of being fibrant in $s\Set^+_{/(D,\sF)}$. Another routine verification shows that $\widetilde{\Fun}_{D/S}(C,E) \to D$ is indeed $S$-cartesian.
\end{enumerate}
\end{proof}

\begin{lem} \label{lm:paramPairFirstVariableCofibrationGivesCategoricalFibration} Let $C \to C'$ be a monomorphism between $S$-cartesian or $S$-cocartesian fibrations over $D$ and let $E \to D$ be a $S$-fibration. Then the induced functor
\[ \widetilde{\Fun}_{D/S}(C',E) \to \widetilde{\Fun}_{D/S}(C,E) \]
is a categorical fibration.
\end{lem}
\begin{proof} Given a trivial cofibration $A \to B$ in $s\Set_{\text{Joyal}}$, we need to solve the lifting problem
\[ \begin{tikzcd}[row sep=2em, column sep=2em]
A \ar{r} \ar{d} & \widetilde{\Fun}_{D/S}(C',E) \ar{d} \\
B \ar{r} \ar[dotted]{ru} & \widetilde{\Fun}_{D/S}(C,E).
\end{tikzcd} \]
This diagram transposes to
\[ \begin{tikzcd}[row sep=2em, column sep=2em]
A \times_{D} \sO^\cocart(D) \times_{D} C' \underset{A \times_{D} \sO^\cocart(D) \times_{D} C}{\bigcup} B \times_{D} \sO^\cocart(D) \times_{D} C \ar{r} \ar{d} & E \ar{d} \\
B \times_{D} \sO^\cocart(D) \times_{D} C' \ar{r} \ar[dotted]{ru}  & D.
\end{tikzcd} \]
By the proof of Theorem~\ref{paramPairMainThm}, $\sO^\cocart(D) \times_{D} C \to D$ is a flat categorical fibration. Therefore, by \cite[Proposition~B.4.5]{HA} the left vertical arrow is a trivial cofibration in $s\Set_{\text{Joyal}}$.
\end{proof}

For later use, we analyze some degenerate instances of the $S$-pairing construction.

\begin{lem} \label{lm:pairingFirstVariableTrivial} There is a natural equivalence $\widetilde{\Fun}_{D/S}(D,E) \xrightarrow{\simeq} E$ of $S$-categories over $D$.
\end{lem}
\begin{proof} The map is induced by the identity section $\iota_D: D \to \sO^\cocart(D)$ fitting into a morphism of spans
\[ \begin{tikzcd}[row sep=2em, column sep=2em]
 & \ar{ld}[swap]{=} \leftnat{D} \ar{rd}{=} \ar{d}{\iota_D} &  \\
\leftnat{D} & (\sO^\cocart(D),\sE) \ar{l} \ar{r} & \leftnat{D}.
\end{tikzcd} \]
By Lemma~\ref{funclem}(1\textquotesingle), $\iota_D$ is a cocartesian equivalence in $s\Set^+_{/S}$ via the target map. Since the cocartesian model structure on $s\Set^+_{/\leftnat{D}}$ is created by the forgetful functor to $s\Set^+_{/S}$, the assertion follows.
\end{proof}

\begin{lem} \label{lm:pairingProducts} Let $C' \to D'$ be a cartesian fibration of $\infty$-categories and let $E'$ be a $S$-category. For all $s \in S$, there is a natural equivalence
\[ \widetilde{\Fun}_{D' \times S/S}(C' \times S,D' \times E')_s \xrightarrow{\simeq} \widetilde{\Fun}_{D'}(C',D' \times E'_s) \]
of cartesian fibrations over $D'$.
\end{lem}
\begin{proof} The lefthand side is defined using the span
\[ \begin{tikzcd}
(D')^\sharp \times \{s\} & ((D')^\sharp \times \{s\}) \times_{D' \times S} (\sO^\cocart(D' \times S) \times_{D'} C',\sE') \ar{r} \ar{l} & S^\sharp
\end{tikzcd} \]
with $\sE'$ as in the proof of Theorem~\ref{paramPairMainThm}. Cocartesian edges (over $S$) in $D' \times S$ are precisely those edges which become equivalences when projected to $D'$, so $\sO^\cocart(D' \times S) \cong \Fun((\Delta^1)^\sharp,(D')^\sim) \times \sO(S)$, and the identity section $\iota_{D'}:D' \to \Fun((\Delta^1)^\sharp,(D')^{\sim})$ is a categorical equivalence. Therefore, the map
 \[ (D' \times S^{s/})^\sharp \to ((D')^\sharp \times \{s\}) \times_{D' \times S} (\sO^\cocart(D' \times S),\sE) \]
induced by $\iota_{D'}$ is a cocartesian equivalence in $s\Set^+_{/S}$. Since $C' \times S \to D' \times S$ is a cartesian fibration, it follows that
\[ \rightnat{(C')} \times (S^{s/})^\sharp \to ((D')^\sharp \times \{s\}) \times_{D' \times S} (\sO^\cocart(D' \times S) \times_{D'} C',\sE')\]
is also a cocartesian equivalence in $s\Set^+_{/S}$. Finally, using the inclusion $C' \times \{s\} \to C' \times S^{s/}$, we obtain a morphism from the span
\[ \begin{tikzcd}
(D')^\sharp & \rightnat{(C')} \ar{r} \ar{l} & \{s\} \subset S^\sharp
\end{tikzcd} \]
through a cocartesian equivalence in $s\Set^+_{/S}$. This yields the equivalence of the lemma.
\end{proof}

Directly from the definition, we have that for an object $x \in D$, the fiber $\widetilde{\Fun}_{D/S}(C,E)_x$ is isomorphic to $\Fun_{\underline{x}}(C_{\underline{x}}, E_{\underline{x}})$. We now proceed to identify the $S$-fiber $\widetilde{\Fun}_{D/S}(C,E)_{\underline{x}}$.

\begin{prp} \label{prp:identifyFibersOfPairing} There is a $\underline{x}$-functor
\[ \epsilon^\ast: \widetilde{\Fun}_{D/S}(C,E)_{\underline{x}} \to \underline{\Fun}_{\underline{x}} (C_{\underline{x}},E_{\underline{x}}) \]
which is a cocartesian equivalence in $s\Set^+_{/ \underline{x}}$.
\end{prp}
\begin{proof} We first define the $\underline{x}$-functor $\epsilon^\ast$. The data of maps of marked simplicial sets
\begin{align*} A &\to \leftnat{\widetilde{\Fun}_{D/S}(C,E)_{\underline{x}}} \\
A &\to \leftnat{\underline{\Fun}_{\underline{x}} (C_{\underline{x}},(E \times_S D)_{\underline{x}})}
\end{align*}
over $\underline{x}$ is identical to the data of maps
\begin{align*} A \times_{\underline{x}} \underline{x}^\sharp \times_{D} (\sO^\cocart(D),\sE) \times_D \leftnat{C} &\to \leftnat{E} \\
A \times_{\underline{x}} \sO(\underline{x})^\sharp \times_{\ev_1 \circ \ev_1, D} \leftnat{C} & \to \leftnat{E}
\end{align*}
over $\leftnat{D}$ (where $\sE$ is the collection of edges $e$ in $\sO^\cocart(D)$ such that $\ev_0(e)$ and $\ev_1(e)$ are cocartesian). We have a commutative square
\[ \begin{tikzpicture}[baseline]
\matrix(m)[matrix of math nodes,
row sep=6ex, column sep=4ex,
text height=1.5ex, text depth=0.25ex]
 {  \sO(\underline{x})^\sharp & \underline{x}^\sharp \\
 (\sO^\cocart(D),\sE) & \leftnat{D} \\ };
\path[>=stealth,->,font=\scriptsize]
(m-1-1) edge node[above]{$\ev_0$} (m-1-2)
edge node[right]{$\sO(\ev_1)$} (m-2-1)
(m-1-2) edge node[right]{$\ev_1$} (m-2-2)
(m-2-1) edge node[above]{$\ev_0$} (m-2-2);
\end{tikzpicture} \]
which defines the functor $\epsilon: \sO(\underline{x}) \to \underline{x} \times_D \sO^\cocart(D)$, and this in turn induces the functor $\epsilon^\ast$. To show that $\epsilon^\ast$ is a cocartesian equivalence, it will suffice to show that $\epsilon$ is a trivial fibration, for then a choice of section $\sigma$ and homotopy $\sigma \circ \epsilon \simeq \id$ will furnish a strong homotopy inverse to $\epsilon^\ast$ in the sense of \cite[Proposition~3.1.3.5]{HTT}. Since we have a pullback diagram
\[ \begin{tikzpicture}[baseline]
\matrix(m)[matrix of math nodes,
row sep=6ex, column sep=4ex,
text height=1.5ex, text depth=0.25ex]
 {  \sO(\underline{x}) & D \times_{\Fun(\Delta^1,D)} \Fun(\Delta^1 \times \Delta^1,D) \\
 \underline{x} \times_D \sO^\cocart(D) & \Fun(\Lambda^2_1,D) \\ };
\path[>=stealth,->,font=\scriptsize]
(m-1-1) edge (m-1-2)
edge node[right]{$\epsilon$} (m-2-1)
(m-1-2) edge node[right]{$\epsilon'$} (m-2-2)
(m-2-1) edge (m-2-2);
\end{tikzpicture} \]
it will further suffice to show that $\epsilon'$ is a trivial Kan fibration. $\epsilon'$ factors as the composition
\[ D \times_{\Fun(\Delta^1,D)} \Fun(\Delta^1 \times \Delta^1,D) \xrightarrow{\epsilon''} \Fun(\Delta^2,D) \xrightarrow{\epsilon'''} \Fun(\Lambda^2_1,D) \]
where $\epsilon''$ is defined by precomposing by the inclusion $i: \Delta^2 \to \Delta^1 \times \Delta^1$ which avoids the degenerate edge for objects in $D \times_{\Fun(\Delta^1,D)} \Fun(\Delta^1 \times \Delta^1,D)$, and $\epsilon'''$ is precomposition by $\Lambda^2_1 \to \Delta^2$. $\epsilon'''$ is a trivial fibration since $\Lambda^2_1 \to \Delta^2$ is inner anodyne. To argue that $\epsilon''$ is a trivial fibration, first note that $\epsilon''$ inherits the property of being a categorical fibration from $i^\ast: \Fun(\Delta^1 \times \Delta^1,D) \to \Fun(\Delta^2,D)$. Define an inverse $\sigma''$ by precomposing by the unique retraction $r: \Delta^1 \times \Delta^1 \to \Delta^2$ chosen so that $r \circ i = \id$. Then $\sigma''$ is a section of $\epsilon''$ and one can write down an explicit homotopy through equivalences of the identity functor on $D \times_{\Fun(\Delta^1,D)} \Fun(\Delta^1 \times \Delta^1,D)$ to $\sigma'' \circ \epsilon''$, so $\epsilon''$ is a trivial fibration.
\end{proof}

\subsection*{\texorpdfstring{$D$}{D}-parametrized slice}

We now study another slice construction defined using the $S$-pairing construction.

\begin{cnstr} \label{cnstr:pairingSlice} Let $\phi: C \to D$ be a $S$-cocartesian fibration, let $E \to D$ be a $S$-fibration, and let $F: C \to E$ be a $S$-functor over $D$. Then $F$ defines a section $S$-functor
\[ \tau_F:  D \to \widetilde{\Fun}_{D/S}(C,E) \]
as adjoint to the functor $\sO^\cocart(D) \times_{ev_1,D} C \to C \xrightarrow{F} E$. Define 
\[ E^{(\phi,F)/ S} \coloneq D \times_{\tau_F, \widetilde{\Fun}_{D/S}(C,E)} \widetilde{\Fun}_{D/S}(C \star_D D,E) \]
and let $\pi_{(\phi,F)}$ denote the projection $E^{(\phi,F)/ S} \to D$.
\end{cnstr}

\nomenclature[sliceParametrized]{$E^{(\phi,F)/ S}$}{$D$-parametrized slice for $S$-cocartesian fibration $\phi: C \to D$ and $S$-functor $F:C \to E$ over $D$}

Given an object $x \in D$, the functor $\tau_F: D \to \widetilde{\Fun}_{D/S}(C,E)$ induces via pullback a $\underline{x}$-functor
\[ \tau_{F_{\underline{x}}}: \underline{x} \to \widetilde{\Fun}_{D/S}(C,E)_{\underline{x}}. \]
We also have the $\underline{x}$-functor
\[ \sigma_{F_{\underline{x}}}: \underline{x} \to \underline{\Fun}_{\underline{x}}(C_{\underline{x}},E_{\underline{x}}) \]
adjoint to
\[ \sO(\underline{x}) \times_{\underline{x}} C_{\underline{x}} \xrightarrow{\pr_2} C_{\underline{x}} \xrightarrow{F_{\underline{x}}} E_{\underline{x}}. \]
An inspection of the definition of the comparison functor $\epsilon^\ast$ of \ref{prp:identifyFibersOfPairing} shows that the triangle
\[ \begin{tikzpicture}[baseline]
\matrix(m)[matrix of math nodes,
row sep=6ex, column sep=4ex,
text height=1.5ex, text depth=0.25ex]
 {  \underline{x} & \widetilde{\Fun}_{D/S}(C,E)_{\underline{x}} \\
  & \underline{\Fun}_{\underline{x}}(C_{\underline{x}},E_{\underline{x}}) \\ };
\path[>=stealth,->,font=\scriptsize]
(m-1-1) edge node[above]{$\tau_{F_{\underline{x}}}$} (m-1-2)
edge node[below]{$\sigma_{F_{\underline{x}}}$} (m-2-2)
(m-1-2) edge node[right]{$\epsilon^\ast$} (m-2-2);
\end{tikzpicture} \]
commutes. Recalling the definitions
\begin{align*} (E^{(\phi,F)/S})_{\underline{x}} &= \underline{x} \times_{\widetilde{\Fun}_{D/S}(C,E)_{\underline{x}}} \widetilde{\Fun}_{D/S}(C \star_D D,E)_{\underline{x}}, \\
(E_{\underline{x}})^{F_{\underline{x}} / \underline{x}} &= \underline{x} \times_{\underline{\Fun}_{\underline{x}} (C_{\underline{x}},E_{\underline{x}})} \underline{\Fun}_{\underline{x}} (C_{\underline{x}} \star_{\underline{x}} \underline{x},E_{\underline{x}}),
\end{align*}
\noindent we therefore obtain a comparison $\underline{x}$-functor
\[ \psi:(E^{(\phi,F)/S})_{\underline{x}} \to (E_{\underline{x}})^{F_{\underline{x}} / \underline{x}}. \]

\begin{cor} \label{cor:fibersOfParamSliceAreSliceCats} The functor $\psi$ is a cocartesian equivalence in $s\Set^+_{/ \underline{x}}$.
\end{cor}
\begin{proof} By \cite[Proposition~3.3.1.5]{HTT}, we have to verify that $\psi$ induces a categorical equivalence on the fibers. But after passage to the fiber over an object $e = [x \to y]$ in $\underline{x}$, by Lemma~\ref{bifibration} $\psi_e$ is a functor between two pullback squares in which one leg is a cartesian fibration. Therefore, by Proposition~\ref{prp:identifyFibersOfPairing} and \cite[Corollary~3.3.1.4]{HTT}, $\psi_e$ is a categorical equivalence.
\end{proof}

\begin{prp} \label{prp:paramSliceIsCartesianFibration} Setup as in Construction~\ref{cnstr:pairingSlice}, suppose in addition that $E \to D$ is a $S$-cartesian fibration. Then $\pi_{(\phi,F)}: E^{(\phi,F)/ S} \to D$ is a $S$-cartesian fibration.
\end{prp}
\begin{proof} By Lemma~\ref{lm:paramPairFirstVariableCofibrationGivesCategoricalFibration}, $\pi_{(\phi,F)}$ is a categorical fibration. By Theorem~\ref{paramPairMainThm}, Lemma~\ref{lm:paramPairFirstVariableCofibrationGivesCategoricalFibration}, and Lemma~\ref{bifibration}, the functor
\[ (\iota_C^\ast)_s: \widetilde{\Fun}_{D/S}(C \star_D D,E)_s \to \widetilde{\Fun}_{D/S}(C,E)_s  \]
over $D_s$ satisfies the hypotheses of \cite[Proposition~2.4.2.11]{HTT}, hence is a locally cartesian fibration.  To then show that $(\iota_C^\ast)_s$ is a cartesian fibration, it suffices to check that for every square
\[ \begin{tikzpicture}[baseline]
\matrix(m)[matrix of math nodes,
row sep=6ex, column sep=4ex,
text height=1.5ex, text depth=0.25ex]
 { \left[ G: C_{\underline{x}} \star_{\underline{x}} \underline{x} \to E_{\underline{x}} \right]  & \left[ G': C_{\underline{y}} \star_{\underline{y}} \underline{y} \to E_{\underline{y}} \right] \\
  \left[ H: C_{\underline{x}} \star_{\underline{x}} \underline{x} \to E_{\underline{x}} \right] & \left[ H': C_{\underline{y}} \star_{\underline{y}} \underline{y} \to E_{\underline{y}} \right] \\ };
\path[>=stealth,->,font=\scriptsize]
(m-1-1) edge (m-1-2)
edge  (m-2-1)
(m-1-2) edge  (m-2-2)
(m-2-1) edge (m-2-2);
\end{tikzpicture} \]

in $\widetilde{\Fun}_{D/S}(C \star_D D,E)_s$ lying over an edge $e: x \to y$ in $D_s$, if the horizontal edges are cartesian lifts over $e$ and the right vertical edge is $(\iota_C^\ast)_{s,y}$-cartesian, then the left vertical edge is $(\iota_C^\ast)_{s,x}$-cartesian. In other words, if we let $e_!: C_{\underline{x}} \star_{\underline{x}} {\underline{x}} \to C_{\underline{y}} \star_{\underline{y}} {\underline{y}}$ and $e^\ast: E_{\underline{y}} \to E_{\underline{x}}$ denote choices of pushforward and pullback functors, then we want to show that given $G \simeq e^\ast \circ G' \circ e_!$, $H \simeq e^\ast \circ H' \circ e_!$, and $G'|_{\underline{y}} \simeq H'|_{\underline{y}}$, we have that $G|_{\underline{x}} \simeq H|_{\underline{x}}$. But this is clear. We deduce that $(\pi_{(\phi,F)})_s$, being pulled back from $(\iota_C^\ast)_s$, is a cartesian fibration.

For the final verification, let us abbreviate objects
 \[ (x \in D, \left[ G: C_{\underline{x}} \star_{\underline{x}} \underline{x} \to E_{\underline{x}} \right] : G|_{C_{\underline{x}}} = F_{\underline{x}} ) \in E^{(\phi,F)/S} \]
as $\left[ G: C_{\underline{x}} \star_{\underline{x}} \underline{x} \to E_{\underline{x}} \right]$, the restriction to $C_{\underline{x}}$ equaling $F_{\underline{x}}$ being left implicit. We must check that given a square
\[ \begin{tikzpicture}[baseline]
\matrix(m)[matrix of math nodes,
row sep=6ex, column sep=4ex,
text height=1.5ex, text depth=0.25ex]
 {  x & x' \\
  y & y' \\ };
\path[>=stealth,->,font=\scriptsize]
(m-1-1) edge node[above]{$\widetilde{\alpha}_x$} (m-1-2)
edge node[right]{$e$} (m-2-1)
(m-1-2) edge node[right]{$e'$} (m-2-2)
(m-2-1) edge node[above]{$\widetilde{\alpha}_y$} (m-2-2);
\end{tikzpicture} \]
in $D$ lying over $\alpha: s \to t$ with the vertical edges in the fiber and the horizontal edges cocartesian lifts of $\alpha$, and given a lift of that square to a square
\[ \begin{tikzpicture}[baseline]
\matrix(m)[matrix of math nodes,
row sep=6ex, column sep=4ex,
text height=1.5ex, text depth=0.25ex]
 { \left[ G: C_{\underline{x}} \star_{\underline{x}} \underline{x} \to E_{\underline{x}} \right]  & \left[G' : C_{\underline{x'}} \star_{\underline{x'}} \underline{x'} \to E_{\underline{x'}} \right] \\
  \left[ H: C_{\underline{y}} \star_{\underline{y}} \underline{y} \to E_{\underline{y}} \right] & \left[ H': C_{\underline{y'}} \star_{\underline{y'}} \underline{y'} \to E_{\underline{y'}} \right] \\ };
\path[>=stealth,->,font=\scriptsize]
(m-1-1) edge (m-1-2)
edge  (m-2-1)
(m-1-2) edge  (m-2-2)
(m-2-1) edge (m-2-2);
\end{tikzpicture} \]
in $E^{(\phi,F)/S}$ with the horizontal edges cocartesian lifts of $\alpha$ and the left vertical edge $(\pi_{(\phi,F)})_s$-cartesian, then the right vertical edge is $(\pi_{(\phi,F)})_t$-cartesian. We will once more translate this compatibility statement into a more obvious looking one so as to conclude. Let $e_!, e^\ast, e'_!, e'^\ast$ be defined as above. Let $\alpha^\ast: \underline{x'} \to \underline{x}$, $\alpha^\ast: \underline{y'} \to \underline{y}$ be choices of pullback functors (e.g. the first sends a cocartesian edge $f:x' \to z$ to $f \circ \widetilde{\alpha}_x: x \to z$), and also label related functors by $\alpha^\ast$. Then the cocartesianness of the horizontal edges amounts to the equivalences $G' \simeq G \circ \alpha^\ast$ and $H' \simeq H \circ \alpha^\ast$, and the cartesianness of the left vertical edge amounts to the equivalence $G|_{\underline{x}} \simeq (e^\ast \circ H \circ e_!)|_{\underline{x}}$. Our desired assertion now is implied by the homotopy commutativity of the diagram
\[ \begin{tikzpicture}[baseline]
\matrix(m)[matrix of math nodes,
row sep=6ex, column sep=4ex,
text height=1.5ex, text depth=0.25ex]
 { \underline{x'} & \underline{x} & E_{\underline{x}} \\
  \underline{y'} & \underline{y} & E_{\underline{y}} \\ };
\path[>=stealth,->,font=\scriptsize]
(m-1-1) edge node[above]{$\alpha^\ast$} (m-1-2)
edge node[right]{$e'_!$} (m-2-1)
(m-1-2) edge node[right]{$e_!$} (m-2-2)
edge node[above]{$G|_{\underline{x}}$} (m-1-3)
(m-2-1) edge node[above]{$\alpha^\ast$} (m-2-2)
(m-2-2) edge node[above]{$H|_{\underline{y}}$} (m-2-3)
(m-2-3) edge node[right]{$e^\ast$} (m-1-3);
\end{tikzpicture} \]
(the content being in the commutativity of the first square), for this demonstrates that $G'|_{\underline{x'}} \simeq (e'^\ast \circ H' \circ e'_!)|_{\underline{x'}}$.
\end{proof}

\begin{lem} \label{lm:fiberwiseInitialToGlobalInitialSection} Let $p: W \to S$, $q: D \to S$ be $S$-categories and let $\pi: W \to D$ be a $S$-fibration such that for every object $s \in S$, $\pi_s$ is a cartesian fibration.
\begin{enumerate}
\item Suppose that:
	\begin{enumerate}
	\item For every object $x \in D$, there exists an initial object in $W_x$.
	\item For every $p$-cocartesian edge $w \rightarrow w'$ in $W$, if $w$ is an initial object in $W_{\pi(w)}$, then $w'$ is an initial object in $W_{\pi(w')}$.
	\end{enumerate}
	Let $W' \subset W$ be the full simplicial subset of $W$ spanned by those objects $w \in W$ which are initial in $W_{\pi(w)}$ and let $\pi' = \pi|_{W'}$. Then $W'$ is a full $S$-subcategory of $W$ and $\pi'$ is a trivial fibration.	
\item Let $\sigma: D \to W$ be a $S$-functor which is a section of $\pi$. Then $\sigma$ is a left adjoint of $\pi$ relative to $D$ if and only if, for every object $x \in D$, $\sigma(x)$ is an initial object of $W_x$.
\end{enumerate}
\end{lem}
\begin{proof} 
\begin{enumerate}[leftmargin=*]
\item Condition (b) ensures that $W'$ is a $S$-subcategory of $W$. By \cite[Proposition~2.4.4.9]{HTT}, for every object $s \in S$, $\pi'_s$ is a trivial fibration. In particular, $\pi'$ is $S$-cocartesian fibration (the compatibility condition being vacuous since all edges in $W'_s$ are $\pi'_s$-cocartesian). By Remark~\ref{relCocartFibIsCocartFib}, $\pi'$ is a cocartesian fibration. As a cocartesian fibration with contractible fibers, $\pi'$ is a trivial fibration.

\item Since relative adjunctions are stable under base change, if $\sigma$ is a left adjoint of $\pi$ relative to $D$, passage to the fiber over $x \in D$ shows that $\sigma(x)$ is an initial object of $W_x$. Conversely, if for all $x \in D$, $\sigma(x)$ is an initial object of $W_x$, then by \cite[Proposition~5.2.4.3]{HTT}, $\sigma_s$ is left adjoint to $\pi_s$ for all $s \in S$. Since $\sigma$ is already given as a $S$-functor, this implies that $\sigma$ is $S$-left adjoint to $\pi$; in particular, $\sigma$ is left adjoint to $\pi$. The existence of $\sigma$ implies the hypotheses of (1), so $\sigma$ is fully faithful. Now by definition, $\sigma$ is left adjoint to $\pi$ relative to $D$.
\end{enumerate}
\end{proof}

We now connect the construction $\widetilde{\Fun}_{D/S}(-,-)$ with $\underline{\Fun}_S(-,-)$. To this end, consider the commutative diagram

\[ \begin{tikzpicture}[baseline]
\matrix(m)[matrix of math nodes,
row sep=6ex, column sep=4ex,
text height=1.5ex, text depth=0.25ex]
 { \sO(S)^\sharp \times_S \leftnat{C} \\
 & \sO(S)^\sharp \times_S (\sO^\cocart(D) \times_D C, \sE) & (\sO^\cocart(D) \times_D C, \sE) & S^\sharp \\
   & \sO(S)^\sharp \times_S \leftnat{D} & \leftnat{D} \\
   & S^\sharp, \\ };
\path[>=stealth,->,font=\scriptsize]
(m-1-1) edge node[above]{$i$} (m-2-2)
edge (m-4-2)
edge[out=0, in=165] (m-2-4)
(m-2-2) edge (m-2-3)
edge (m-3-2)
(m-2-3) edge (m-2-4)
edge (m-3-3)
(m-3-2) edge node[above]{$\pr_D$} (m-3-3)
edge node[right]{$\ev_0$} (m-4-2);
\end{tikzpicture} \]
where the map $i$ is induced by the identity section $D \to \sO^\cocart(D)$.

\begin{lem} \label{pairingToFunctorLem} $i$ is a homotopy equivalence in $s\Set^+_{/S}$ (considered over $S$ via $p: C \to S$).
\end{lem}
\begin{proof} Define a map $h': \sO(S) \times_S \sO^\cocart(D) \to \Fun(\Delta^1, \sO(S) \times_S \sO^\cocart(D))$ to be the product of the following three maps:
\begin{enumerate}
\item Choose a lift $\sigma$
\[ \begin{tikzpicture}[baseline]
\matrix(m)[matrix of math nodes,
row sep=6ex, column sep=4ex,
text height=1.5ex, text depth=0.25ex]
 {  \Fun(\Delta^{\{0,1\}} ,S) & \Fun(\Delta^2,S) \\
  \Fun(\Lambda^2_1,S) & \Fun(\Lambda^2_1,S) \\ };
\path[>=stealth,->,font=\scriptsize]
(m-1-1) edge node[above]{$s_1$} (m-1-2)
edge (m-2-1)
(m-1-2) edge[->>] node[left]{$\sim$} (m-2-2)
(m-2-1) edge node[above]{$=$} (m-2-2)
edge[dotted] node[above]{$\sigma$} (m-1-2);
\end{tikzpicture} \]
and let $\Delta^1 \times \Delta^1 \to \Delta^2$ be the unique map so that the induced map $\Fun(\Delta^2,S) \to \Fun(\Delta^1 \times \Delta^1, S) \cong \Fun(\Delta^1,\sO(S))$ sends $(s \to t \to u)$ to $[s \to t] \to [s \to u]$. Use these two maps to define
\[ \sO(S) \times_S \sO^\cocart(D) \times_D C \to \sO(S) \times_S \sO(S) \cong \Fun(\Lambda^2_1,S) \to \Fun(\Delta^1,\sO(S)). \]
\item Use the unique map $\Delta^1 \times \Delta^1 \to \Delta^1$ which sends $(0,0)$ to $0$ and all other vertices to $1$ to define
\[ \sO(S) \times_S \sO^\cocart(D) \times_D C \to \sO^\cocart(D) \to \Fun(\Delta^1, \sO^\cocart(D)). \]
\item The degeneracy map $s_0: C \to \Fun(\Delta^1,C)$ defines
\[ \sO(S) \times_S \sO^\cocart(D) \times_D C \to C \to \Fun(\Delta^1,C). \]
\end{enumerate}
Then $h'$ is adjoint to a map of marked simplicial sets over $S$
\[ h: (\Delta^1)^\sharp \times \sO(S)^\sharp \times_S (\sO^\cocart(D) \times_D C, \sE) \to \sO(S)^\sharp \times_S (\sO^\cocart(D) \times_D C, \sE) \]
such that $h_0 = \id$ and $h_1$ factors as a composition
\[ \sO(S)^\sharp \times_S (\sO^\cocart(D) \times_D C, \sE) \xrightarrow{r} \sO(S)^\sharp \times_S \leftnat{C} \xrightarrow{i} \sO(S)^\sharp \times_S (\sO^\cocart(D) \times_D C, \sE) \]
where $r$ is defined by
\[ \sO(S)^\sharp \times_S (\sO^\cocart(D) \times_D C, \sE) \to \Fun(\Lambda^2_1,S)^\sharp \times_S \leftnat{C} \xrightarrow{d_1 \circ \sigma} \sO(S)^\sharp \times_S \leftnat{C}. \]
Our choice of $\sigma$ ensures that $r \circ i = \id$, completing the proof.
\end{proof}

Note that for any $S$-fibration $\pi: X \to D$, the $S$-category $\underline{\Sect}_{D/S}(\pi)$ defined in \ref{dfn:sectionsOfSFibration} may be identified with $(\ev_0)_\ast (\pr_D)^\ast (\leftnat{X} \overset{\pi}{\to} \leftnat{D})$. Combining Lemma~\ref{pairingToFunctorLem}, Lemma~\ref{lm:SpanComposition}, and Lemma~\ref{lm:SpanHomotopyInvariance}, we see that if $E$ is a $S$-category and $C \to D$ is $S$-cocartesian or $S$-cartesian, then the map induced by $i$
\[ i^\ast: \underline{\Sect}_{D/S}(\widetilde{\Fun}_{D/S}(C,E \times_S D)) \to \underline{\Fun}_S(C,E) \]
 is a equivalence of $S$-categories. Moreover, a chase of the definitions reveals that for every $S$-functor $F: C \to E$, we have an identification
\[ i^\ast \circ \underline{\Sect}_{D/S}(\tau_{F \times \phi}) = \sigma_{F}: S \to \underline{\Fun}_S(C,E). \]

We thus have a morphism of spans
\[ \begin{tikzcd}[row sep=2em, column sep=4em]
S \ar{r}{\underline{\Sect}_{D/S}(\tau_{F \times \phi})} \ar{d}{=} &[4em] \underline{\Sect}_{D/S}(\widetilde{\Fun}_{D/S}(C, E \times_S D)) \ar{d}{\simeq} & \underline{\Sect}_{D/S}(\widetilde{\Fun}_{D/S}(C \star_D D, E \times_S D)) \ar{l} \ar{d}{\simeq} \\
S \ar{r}{\sigma_F} & \underline{\Fun}_S(C,E) & \underline{\Fun}_S(C \star_D D, E). \ar{l}
\end{tikzcd} \]

The right horizontal maps are $S$-fibrations by Lemma~\ref{lm:paramPairFirstVariableCofibrationGivesCategoricalFibration} and \cite[Proposition~9.11(2)]{BDGNS1}, so taking pullbacks yields an equivalence
\begin{equation} \label{eqn:rightKanExtOfParamSlice} \underline{\Sect}_{D/S}((E \times_S D)^{(\phi, F \times \phi)/S}) \overset{\simeq}{\to} S \times_{\sigma_F,\underline{\Fun}_S(C,E)} \underline{\Fun}_S(C \star_D D, E).
\end{equation}

We are now prepared to introduce the main definition of this section.

\begin{dfn} \label{dfn:paramColimit} Let $\phi: C \to D$ be a $S$-cocartesian fibration. A $S$-functor $\overline{F}: C \star_D D \to E$ is a \emph{$D$-parametrized $S$-colimit diagram} if for every object $x \in D$, the $\underline{x}$-functor $\overline{F}|_{C_{\underline{x}} \star_{\underline{x}} \underline{x}}: C_{\underline{x}} \star_{\underline{x}} \underline{x} \to E_{\underline{s}}$ is a $\underline{s}$-colimit diagram.
\end{dfn}

\begin{prp} \label{prp:paramColimitIsInitialObject} Let $\phi: C \to D$ be a $S$-cocartesian fibration, let $F: C \to E$ be a $S$-functor, and let $\overline{F}: C \star_D D \to E$ be a $D$-parametrized $S$-colimit diagram extending $F$. Then the section
\[ \id_S \times \sigma_{\overline{F}}: S \to S \times_{\sigma_F, \underline{\Fun}_S(C,E)} \underline{\Fun}_S(C \star_D D, E) \]
is a $S$-initial object.
\end{prp}
\begin{proof} Combine Equation~\ref{eqn:rightKanExtOfParamSlice}, Lemma~\ref{lm:fiberwiseInitialToGlobalInitialSection}(2), and Corollary~\ref{cor:sectionsOfRelativeAdjIsAdj}.
\end{proof}

We have the following existence and uniqueness result for $D$-parametrized $S$-colimits.

\begin{thm} \label{thm:ExistenceAndUqnessOfParamColimit} Let $\phi: C \to D$ be a $S$-cocartesian fibration and let $F: C \to E$ be a $S$-functor. Suppose that for every object $x \in D$, the $\underline{s}$-functor $F|_{C_{\underline{x}}}: C_{\underline{x}} \to E_{\underline{s}}$ admits a $\underline{s}$-colimit. Then there exists a $D$-parametrized $S$-colimit diagram $\overline{F}: C \star_D D \to E$ extending $F$. Moreover, the full subcategory of $\{F\} \times_{\Fun_S(C, E)} \Fun_S(C \star_D D, E)$ spanned by the $D$-parametrized $S$-colimit diagrams coincides with that spanned by the initial objects.
\end{thm}
\begin{proof} By Proposition~\ref{prp:paramSliceIsCartesianFibration} and Corollary~\ref{cor:fibersOfParamSliceAreSliceCats}, the functor
\[ \pi_{(\phi, F \times \phi)}: (E \times_S D)^{(\phi, F \times \phi)/S} \to D \]
is a $S$-cartesian fibration with $\underline{x}$-fibers equivalent to $(E_{\underline{s}})^{(F|_{C_{\underline{x}}}, \underline{s}) /}$. Our hypothesis ensures that the conditions of Lemma~\ref{lm:fiberwiseInitialToGlobalInitialSection}(1) are satisfied, so $\pi_{(\phi, F \times \phi)}$ admits a section $\sigma$ which is a $S$-functor that selects an initial object in each fiber. The resulting $S$-functor $D \to \widetilde{\Fun}_{D/S}(C \star_D D, E \times_S D)$ covering $\tau_{F \times \phi}$ is adjoint to a $S$-functor $\overline{F}: C \star_D D \to E$ extending $F$, which is a $D$-parametrized $S$-colimit diagram. Having proven existence, the second statement now follows from Proposition~\ref{prp:paramColimitIsInitialObject}.
\end{proof}

Theorem~\ref{thm:ExistenceAndUqnessOfParamColimit} also admits the following `global' consequence.

\begin{cor} \label{cor:leftAdjointToRestrictionIsColimitFunctor} Let $\phi: C \to D$ be a $S$-cocartesian fibration and $E$ be an $S$-category. Suppose that for every $s \in S$ and $x \in D_s$, $E_{\underline{s}}$ admits all $S^{s/}$-colimits of shape $C_{\underline{x}}$. Then $U: \underline{\Fun}_S(C \star_D D, E) \to \underline{\Fun}_S(C,E)$ admits a left $S$-adjoint $L$ which is a section of $U$ such that for every object $F: C_{\underline{s}} \to E_{\underline{s}}$, $L(F)$ is a $D_{\underline{s}}$-parametrized $S^{s/}$-colimit diagram.
\end{cor}
\begin{proof} By Example~\ref{exm:bifibrationFunctor}, Theorem~\ref{thm:ExistenceAndUqnessOfParamColimit} and the stability of parametrized colimit diagrams under base change, the conditions of Lemma~\ref{lm:fiberwiseInitialToGlobalInitialSection}(1) are satisfied for $U$. Thus $U$ admits a section $L$ which selects an initial object in each fiber, necessarily a parametrized colimit diagram. By Lemma~\ref{lm:fiberwiseInitialToGlobalInitialSection}(2), $L$ is a left adjoint of $U$ relative to $\underline{\Fun}_S(C,E)$; in particular, $L$ is $S$-left adjoint to $U$.
\end{proof}

\subsection*{Application: Functor categories}

\begin{prp} \label{prp:colimitsInFunctorCategories}  Let $K$, $I$, and $C$ be $S$-categories.
\begin{enumerate}
\item Suppose that for all $s \in S$, $C_{\underline{s}}$ admits all $K_{\underline{s}}$-indexed colimits. $\overline{p}: K \star_S S \to \underline{\Fun}_S(I,C)$ is a $S$-colimit diagram if and only if, for every object $x \in I$ over $s$,
\[ K_{\underline{s}} \star_{\underline{s}} \underline{s} \xrightarrow{\overline{p}_{\underline{s}}} \underline{\Fun}_{\underline{s}}(I_{\underline{s}}, C_{\underline{s}}) \xrightarrow{\ev_x} C_{\underline{s}} \]
 is a $S^{s/}$-colimit diagram.
\item A $S$-functor $p: K \to \underline{\Fun}_S(I,C)$ admits an extension to a $S$-colimit diagram $\overline{p}$ if for all $x \in I$, $\ev_x \circ p_{\underline{s}}$ admits an extension to a $S^{s/}$-colimit diagram.
\end{enumerate}
\end{prp}
\begin{proof} We prove (1), the proof for (2) being similar. Let
$$\overline{p'}: (K \times_S I) \star_I I \cong (K \star_S S) \times_S I \to C$$
be a choice of adjoint of $p$ under the equivalence
$$\Fun_S(K \star_S S, \underline{\Fun}_S(I,C)) \simeq \Fun_S((K \star_S S) \times_S I, C).$$
By Theorem~\ref{thm:ExistenceAndUqnessOfParamColimit} applied to the $S$-cocartesian fibration $K \times_S I \to I$ and the hypothesis on $C$, there exists an $I$-parametrized $S$-colimit diagram $p''$ extending $p' = \overline{p'}|_{K \times_S I}$. By Proposition~\ref{prp:paramColimitIsInitialObject}, $p''$ defines an $S$-initial object in
\[ S \times_{\underline{\Fun}_S(K \times_S I,C)} \underline{\Fun}_S((K \times_S I) \star_I I,C) \simeq \underline{\Fun}_S(I,C)^{(p,S)/} \] 
so \emph{its} adjoint is a $S$-colimit diagram. For the `if' direction, supposing that $\overline{p}$ is a $S$-colimit diagram, then by the uniqueness of $S$-initial objects, $p''$ is equivalent to $\overline{p'}$. Then $\ev_x \circ \overline{p}_{\underline{s}}$ is equivalent to $p''_{\underline{x}}$, which is a $S^{s/}$-colimit diagram by definition of $I$-parametrized $S$-colimit diagram. For the `only if' direction, supposing that all the $\ev_x \overline{p}_{\underline{s}}$ are $S^{s/}$-colimit diagrams, we get that $\overline{p'}$ is a $I$-parametrized $S$-colimit diagram, so is equivalent to $p''$.
\end{proof}

\begin{cor} \label{cor:FunctorCategoryCocomplete} Suppose $C$ is $S$-cocomplete and $I$ is a $S$-category. Then $\underline{\Fun}_S(I,C)$ is $S$-cocomplete.
\end{cor}

\section{Kan extensions}

We now combine the theory of $S$-colimits parametrized by a base $S$-category $D$ and that of free $S$-cocartesian fibrations to establish the theory of left $S$-Kan extensions.

\begin{dfn} \label{dfn:SLKE} Suppose a diagram of $S$-categories
\[ \begin{tikzcd}[row sep=2em, column sep=3em]
C \ar{r}{F}[name=F,below,near start]{} \ar{d}[swap]{\phi} & E \\
D \ar{ru}[swap]{G}[name=G,above]{}
\arrow[Rightarrow,to path={(F) -- node[left]{$\scriptstyle\eta$} (G)}]
\end{tikzcd} \]
where by the `$2$-cell' $\eta$ we mean exactly the datum of a $S$-functor $\eta: C \times \Delta^1 \to E$ restricting to $F$ on $0$ and $G \circ \phi$ on $1$. Let
\[  G': (C \times_D \sO_S(D)) \star_D D \overset{\pi_D}{\to} D \overset{G}{\to} E, \]
let
\[ \theta: (C \times_D \sO_S(D)) \times \Delta^1 \to E \]
be the natural transformation adjoint to $G_\ast: C \times_D \sO_S(D) \to \sO_S(E)$, let
\[ \eta': (C \times_D \sO_S(D)) \times \Delta^1 \to C \times \Delta^1 \overset{\eta}{\to} E \] 
be the natural transformation obtained from $\eta$, and let $\theta' = \theta \circ \eta'$ be a choice of composition in $\Fun_S(C \times_D \sO_S(D),E)$. Let
\[ r: \Fun_S((C \times_D \sO_S(D)) \star_D D,E) \to \Fun_S(C \times_D \sO_S(D),E)\]
denote the restriction functor. By Lemma~\ref{bifibration}, we may select a $r$-cartesian edge $e$ in $\Fun_S((C \times_D \sO_S(D)) \star_D D,E)$ with $d_0(e) = G'$ covering $\theta'$, chosen so that $e|_D$ is degenerate. Let $G'' = d_1(e)$.

We say that $G$ is a \emph{left $S$-Kan extension} of $F$ along $\phi$ if $G''$ is a $D$-parametrized $S$-colimit diagram.
\end{dfn}

\begin{rem} The following are equivalent:
\begin{enumerate}
\item $G$ is a left $S$-Kan extension of $F$ along $\phi$.
\item For all $s \in S$, $G_{\underline{s}}$ is a left $S^{s/}$-Kan extension of $F_{\underline{s}}$ along $\phi_{\underline{s}}$. 
\item For all $s \in S$ and $x \in D_s$, $G|_{\underline{x}} : \underline{x} \to E_{\underline{s}}$ is a left $S^{s/}$-Kan extension of $F|_{C_{\underline{x}}}: C_{\underline{x}} \to E_{\underline{s}}$ along $\phi_{\underline{x}}: C_{\underline{x}} \to \underline{x}$.
\end{enumerate}
In other words, our notion of $S$-Kan extension generalizes the concept of \emph{pointwise} Kan extensions.
\end{rem}

We can bootstrap Theorem~\ref{thm:ExistenceAndUqnessOfParamColimit} to prove existence and uniqueness of left $S$-Kan extensions.

\begin{thm} \label{thm:existenceOfKanExtensions} Let $\phi: C \to D$ and $F: C \to E$ be $S$-functors. Suppose that for every object $x \in D$, the $S^{s/}$-functor
\[ C \times_D D^{/\underline{x}} \to C_{\underline{s}} \overset{F_{\underline{s}}}{\to} E_{\underline{s}} \]
admits a $S^{s/}$-colimit. Then there exists a left $S$-Kan extension $G: D \to E$ of $F$ along $\phi$, uniquely specified up to contractible choice.
\end{thm}
\begin{proof} We spell out the details of existence and leave the proof of uniqueness to the reader. By Theorem~\ref{thm:ExistenceAndUqnessOfParamColimit}, there exists a $D$-parametrized $S$-colimit diagram
\[ \overline{F}: (C \times_D \sO_S(D)) \star_D D \to E \] 
extending $C \times_D \sO_S(D) \to C \overset{F}{\to} E$. Let $G = \overline{F}|_D$. Define a map
\[ h: C \times \Delta^1 \to (C \times_D \sO_S(D)) \star_D D \]
over $D \times \Delta^1$ as adjoint to $(C \overset{(\id,\iota \phi)}{\to} C \times_D \sO_S(D), C \overset{\phi}{\to} D)$ and let $\eta = \overline{F} \circ h$, so that $\eta$ is a natural transformation from $F$ to $G \circ \phi$.

We claim that $\eta$ exhibits $G$ as a left Kan extension of $F$ along $\phi$. To show this, we will exhibit a $r$-cartesian edge $e$ from $\overline{F}$ to $G'$ such that the restriction $r(e)$ of $e$ to $C \times_D \sO_S(D)$ is a choice of composition $\theta \circ \eta'$. Define
\[ e': (C \times_D \sO_S(D)) \star_D D \times \Delta^1 \to (C \times_D \sO_S(D)) \star_D D \]
over $D \times \Delta^1$ as adjoint to $(\id, \pi_D)$, and let $e = \overline{F} \circ e'$, so that $e$ is an edge from $\overline{F}$ to $G'$. Since $(\pi_D)|_D = \id_D$, $e|_{D}$ is a degenerate edge in $\Fun_S(D,E)$, so $e$ is $r$-cartesian.

To finish the proof, we need to introduce a few more maps. Define
\[ \alpha = (\pr_C, \alpha'): C \times_D \sO_S(D) \times \Delta^1 \to C \times_D \sO_S(D) \]
where $\alpha'$ is adjoint to
\[ C \times_D \sO_S(D) \to \sO_S(D) = \widetilde{\Fun}_S(S \times \Delta^1,D) \overset{\text{min}^\ast}{\to} \widetilde{\Fun}_S(S \times \Delta^1 \times \Delta^1,D). \]
Here $\text{min}: \Delta^1 \times \Delta^1 \to \Delta^1$ is the functor which takes the minimum. Define
\[ \beta: C \times_D \sO_S(D) \times \Delta^1 \to \sO_S(D) \times \Delta^1 \overset{\ev}{\to} D. \]
Use $\alpha$ and $\beta$ to define
\[ \gamma: C \times_D \sO_S(D) \times \Delta^1 \times \Delta^1 \to (C \times_D \sO_S(D)) \star_D D \]
so that on objects $(c, \phi c \overset{f}{\rightarrow} d)$, $\gamma$ sends $\Delta^1 \times \Delta^1$ to the square
\[ \begin{tikzcd}[row sep=2em,column sep=2em]
(c, \phi c = \phi c) \ar{r} \ar{d}{(\id,f)} & \phi c \ar{d}{f} \\
(c,\phi c \overset{f}{\rightarrow} d) \ar{r} & d.
\end{tikzcd} \]
Then $\overline{F} \circ \gamma$ defines a square
\[ \begin{tikzcd}[row sep=2em,column sep = 2em]
F \circ \pr_C \ar{r}{\eta'} \ar{d}{=} & G \circ \phi \circ \pr_C \ar{d}{\theta} \\
F \circ \pr_C \ar{r}{r(e)} & G'.
\end{tikzcd} \]
in $\Fun_S(C \times_D \sO_S(D),E)$, which proves that $r(e) \simeq \theta \circ \eta'$.
\end{proof}

We also have the Kan extension counterpart to Corollary~\ref{cor:leftAdjointToRestrictionIsColimitFunctor}.

\begin{dfn}
Let $\phi: C \to D$ be a $S$-functor and $E$ a $S$-category. We say that $E$ \emph{admits the relevant $S$-colimits for $\phi$} if for every $s \in S$ and $x \in D_s$, $E_{\underline{s}}$ admits all $S^{s/}$-colimits of shape $C \times_D D^{/\underline{x}}$. 
\end{dfn}

\begin{thm} \label{thm:leftKanExtensionIsLeftAdjointToRestriction} Let $\phi: C \to D$ be a $S$-functor and $E$ a $S$-category. Suppose that $E$ admits the relevant $S$-colimits for $\phi$. Then the $S$-functor
\[ \phi^\ast: \underline{\Fun}_S(D,E) \to \underline{\Fun}_S(C,E) \]
given by restriction along $\phi$ admits a left $S$-adjoint $\phi_!$ such that for every $S$-functor $F: C \to E$, the unit map $F \to \phi^\ast \phi_! F$ exhibits $\phi_! F$ as a left $S$-Kan extension of $F$ along $\phi$.
\end{thm}
\begin{proof} Factor $\phi$ as the composition
\[ C \overset{\iota_C}{\to} C \times_D \sO_S(D) \overset{i}{\to} (C \times_D \sO_S(D)) \star_D D \overset{\pi_D}{\to} D. \]
Then $\phi^\ast$ factors as the composition
\[ \underline{\Fun}_S(D,E) \overset{\pi_D^\ast}{\to} \underline{\Fun}_S((C \times_D \sO_S(D)) \star_D D,E) \overset{i^\ast}{\to} \underline{\Fun}_S(C \times_D \sO_S(D),E) \overset{\iota_C^\ast}{\to} \underline{\Fun}_S(C,E). \]
By Proposition~\ref{prp:inclusionToFreeFibrationIsAdjoint} and Corollary~\ref{cor:RelAdjFromRelAdj}, $\pr_C^\ast$ is left $S$-adjoint to $\iota_C^\ast$. Since $i_D$ is right $S$-adjoint to $\pi_D$, by Corollary~\ref{cor:RelAdjFromRelAdj} again $i_D^\ast$ is left $S$-adjoint to $\pi_D^\ast$. By Corollary~\ref{cor:leftAdjointToRestrictionIsColimitFunctor}, $i^\ast$ admits a left $S$-adjoint $L$ which extends functors to $D$-parametrized $S$-colimit diagrams. Let $\phi_!$ be the composite of these three functors. The proof of Theorem~\ref{thm:existenceOfKanExtensions} shows that $\phi_!(F)$ is as asserted.
\end{proof}

The next proposition permits us to eliminate the datum of the natural transformation $\eta$ from the definition of a left $S$-Kan extension when $\phi$ is fully faithful.

\begin{prp} \label{prp:KanExtensionAlongFullyFaithfulInclusionIsExtension} Suppose $\phi: C \to D$ is the inclusion of a full $S$-subcategory. Then for any left $S$-Kan extension $G$ of $F: C \to E$ along $\phi$, $\eta$ is a natural transformation through equivalences. Consequently, $G$ is homotopic to a functor $\overline{F}: D \to E$ which is both an extension of $F$ and a left $S$-Kan extension (with the natural transformation $F \to \overline{F} \circ \phi = F$ chosen to be the identity).
\end{prp}
\begin{proof} Let $G'': (C \times_D \sO_S(D)) \star_D D \to E$ be as in the definition of a left $S$-Kan extension. Because $D$-parametrized $S$-colimit diagrams are stable under restriction to $S$-subcategories,
\[ (G'')_C: (C \times_D \sO_S(D) \times_D C) \star_C C \to E \]
is a $C$-parametrized $S$-colimit diagram. The additional assumption that $C$ is a full $S$-subcategory has the consequence that $(C \times_D \sO_S(D) \times_D C) \cong \sO_S(C)$. Also, for any object $x \in C$, the inclusion $\underline{x}$-functor $i_x: \underline{x} \to C^{/ \underline{x}}$ is $\underline{x}$-final, using the first criterion of Theorem~\ref{thm:cofinality}. Therefore, $\sO_S(C) \star_C C \overset{\pi_C}{\to} C \overset{F}{\to} E$ is a $C$-parametrized $S$-colimit diagram extending $\sO_S(C) \overset{\ev_0}{\to} C \overset{F}{\to} E$, so $(G'')_C \simeq F \circ \pi_C$.

The map $h$ in the proof of Theorem~\ref{thm:existenceOfKanExtensions} factors as
\[ C \times \Delta^1 \overset{h'}{\to} \sO_S(C) \star_C C \to (C \times_D \sO_S(D)) \star_D D. \]
We have the chain of equivalences
\[ \eta \simeq G'' \circ h \simeq F \circ \pi_C \circ h' = F \circ \pr_C, \]
proving the first assertion. For the second assertion, use that
\[ (\leftnat{D} \times \{1\}) \cup_{\leftnat{C} \times \{1\}} (\leftnat{C} \times (\Delta^1)^\sharp) \to \leftnat{D} \times (\Delta^1)^\sharp \]
is a cocartesian equivalence in $s\Set^+_{/S}$ to extend $(G,\eta)$ to a homotopy between $G$ and an extension $\overline{F}$, which is then necessarily a left $S$-Kan extension.
\end{proof}

\begin{cor} Suppose $\phi: C \to D$ a fully faithful $S$-functor and $E$ a $S$-cocomplete $S$-category. Then the left $S$-adjoint $\phi_!$ to the restriction $S$-functor $\phi^\ast$ exists and is fully faithful.
\end{cor}
\begin{proof} Combine Theorem~\ref{thm:leftKanExtensionIsLeftAdjointToRestriction} and Proposition~\ref{prp:KanExtensionAlongFullyFaithfulInclusionIsExtension}.
\end{proof}

As expected, $S$-colimit diagrams are examples of $S$-left Kan extensions.

\begin{prp} \label{prp:ColimitsAreLKEs} Suppose $\phi: C \to D$ a $S$-cocartesian fibration and $\overline{F}: C \star_D D \to E$ a $S$-functor extending $F: C \to E$. Then $\overline{F}$ is a $D$-parametrized $S$-colimit diagram if and only if $\overline{F}$ is a $S$-left Kan extension of $F$.
\end{prp}
\begin{proof} We may check the assertion objectwise on $D$, so let $x \in D_s$. Consider the commutative diagram
\[ \begin{tikzcd}[row sep=2em, column sep=2em]
C_{\underline{x}} \ar[hookrightarrow]{r} \ar[hookrightarrow]{d}[swap]{\theta} & C_{\underline{s}} \ar{d}{F_{\underline{s}}} \\
C \times_{C \star_D D} (C \star_D D)^{/ \underline{x}} \ar{r} \ar{ru}{\pr_C} & E_{\underline{s}}.
\end{tikzcd} \]
The value of a $D$-parametrized colimit of $F$ on $x$ is computed as the $S^{s/}$-colimit of $(F_{\underline{s}})|_{C_{\underline{x}}}$, and that of a $S$-left Kan extension of $F$ as the $S^{s/}$-colimit of $F_{\underline{s}} \circ \pr_C$. Therefore, it suffices to prove that $\theta$ is $\underline{x}$-final. Let $f: x \rightarrow y$ be an object in $\underline{x}$, i.e. a cocartesian edge in $D$, which lies over $s \to t$. Then $\theta_f$ is equivalent to the inclusion
\[ C_y \cong C_y \times_{(C_y)^\rhd} ((C_y)^\rhd)^{/ \{\infty\}} \to C_t \times_{C_t \star_{D_t} D_t} (C_t \star_{D_t} D_t)^{/y}. \]
Applying Lemma~\ref{lm:miscCofinalityLemma} to the map $C_t \to C_t \star_{D_t} D_t$ of cocartesian fibrations over $D_t$, we deduce that $\theta_f$ is final.
\end{proof}

\begin{lem} \label{lm:miscCofinalityLemma} Let $X \to Y$ be a map of cocartesian fibrations over $Z$ and let $y \in Y$ be an object over $z \in S$. Then the inclusion $X_z \times_{Y_z} Y_z^{/y} \to X \times_Y Y^{/y}$ is final.
\end{lem}
\begin{proof} By the dual of \cite[Lemma~3.4.1.10]{HA}, $X \times_Y Y^{/y} \to Z^{/z}$ is a cocartesian fibration. We have a pullback square
\[ \begin{tikzcd}[row sep=2em, column sep=2em]
X_z \times_{Y_z} Y_z^{/y} \ar{r} \ar{d} & X \times_Y Y^{/y} \ar{d}  \\
\{z\} \ar{r}{\id_z} & Z^{/z}.
\end{tikzcd} \]
Since the bottom horizontal map is final and cocartesian fibrations are smooth, the top horizontal map is final.
\end{proof}

As with $S$-colimits, $S$-left Kan extensions reduce to the usual notion of left Kan extension when taken in a $S$-category of objects.

\begin{prp} \label{prp:LKEsInCategoryOfObjects} Suppose a diagram of $S$-categories
\[ \begin{tikzcd}[row sep=2em, column sep=3em]
C \ar{r}{F}[name=F,below,near start]{} \ar{d}[swap]{\phi} & \underline{E}_S \\
D \ar{ru}[swap]{G}[name=G,above]{} & \quad \cdot
\arrow[Rightarrow,to path={(F) -- node[left]{$\scriptstyle\eta$} (G)}]
\end{tikzcd} \]
The following are equivalent:
\begin{enumerate}
	\item $G$ is a left $S$-Kan extension of $F$ along $\phi$.
	\item $G^\dagger$ is a left Kan extension of $F^\dagger$ along $\phi$.
	\item For all objects $s \in S$, $G^\dagger|_{D_s}$ is a left Kan extension of $F^\dagger|_{C_s}$ along $\phi_s$.
\end{enumerate}
\end{prp}
\begin{proof} We first prove that (1) and (2) are equivalent. Factor $\phi: C \to D$ through the free $S$-cocartesian fibration on $\phi$:
\[ \phi: C \xrightarrow{\iota_C} C \times_D \sO_S{D} \xrightarrow{\Fr^\cocart(\phi)} D. \]
Since $\iota_C$ is $S$-left adjoint to $\pr_C$, it is also left adjoint. Therefore, the $S$-left Kan extension resp. the left Kan extension of $F$ resp. $F^\dagger$ along $\iota_C$ is computed by $F \circ \pr_C$ resp. $F^\dagger \circ \pr_C$. By transitivity of Kan extensions, we thereby reduce to the case that $\phi$ is $S$-cocartesian. The claim now follows easily by combining Proposition~\ref{prp:identifyingColimitsInCatOfObjects} and Proposition~\ref{prp:ColimitsAreLKEs}.

We next prove that (2) and (3) are equivalent. For this, it suffices to observe that for all objects $d \in D$ over some $s \in S$, $C_s \times_{D_s} D_s^{/d} \to C \times_D D^{/d}$ is final by Lemma~\ref{lm:miscCofinalityLemma} applied to $C \to D$.
\end{proof}


\section{Yoneda lemma}

By Proposition~\ref{prp:identifyingColimitsInCatOfObjects}, $\underline{\Top}_S$ is $S$-cocomplete, so by Corollary~\ref{cor:FunctorCategoryCocomplete}, the $S$-category of presheaves
$$\PP_S(C) \coloneq \underline{\Fun}_S(C^\vop,\underline{\Top}_S)$$
is $S$-cocomplete. The $S$-Yoneda embedding $j: C \to \PP_S(C)$ was constructed in \cite[\S 10]{BDGNS1} via $S$-straightening the left fibration $\widetilde{\sO}_S(C) \to C^\vop \times_S C$ given fiberwise by twisted arrows. It was shown there that $j$ is fully faithful \cite[Theorem~10.4]{BDGNS}. In this section, we first prove the $S$-Yoneda lemma and then establish the universal property of $\PP_S(C)$ as the free $S$-cocompletion of $C$.

\nomenclature[presheaves]{$\PP_S(C)$}{Parametrized presheaves}

\begin{lem}[$S$-Yoneda lemma] \label{lm:Yoneda} Let $j: C \to \PP_S(C)$ denote the $S$-Yoneda embedding. Then the identity on $\PP_S(C)$ is a $S$-left Kan extension of $j$ along itself.
\end{lem}
\begin{proof} By Proposition~\ref{prp:colimitsInFunctorCategories}, it suffices to show that for every $s \in S$ and object $x \in C_s$, $\ev_x: \PP_{\underline{s}}(C_{\underline{s}}) \to \underline{\Top}_{\underline{s}}$ is a $S^{s/}$-left Kan extension of $\ev_x j_{\underline{s}}$. To ease notation, let us replace $S^{s/}$ by $S$ and suppose that $s \in S$ is an initial object.

We claim that $(\ev_x j)^\dagger: C \to \Top$ is homotopic to $\Map_C(x,-)$. By definition of the $S$-Yoneda embedding, $(\ev_x j)^\dagger$ classifies the left fibration
$$\ev_1: \widetilde{\sO}_S(C)_{x \rightarrow } \to C$$
pulled back from $\widetilde{\sO}_S(C) \to C^\vop \times_S C$ via the cocartesian section $\sigma: S \to C^\vop$ defined by $\sigma(s)=x$. By \cite[Proposition~4.4.4.5]{HTT}, it suffices to show that $\id_x$ is an initial object in $\widetilde{\sO}_S(C)_{x \rightarrow }$. For this, because $s \in S$ is an initial object we reduce to checking that for all edges $\alpha: s \rightarrow t$, the pushforward of $\id_x$ by $\alpha$ is an initial object in the fiber $(\widetilde{\sO}_S(C)_{x \rightarrow })_t$. But this fiber is equivalent to $\widetilde{\sO}(C_t)_{\alpha_! x \rightarrow} \simeq (C_t)^{\alpha_! x /}$.

Applying Proposition~\ref{prp:LKEsInCategoryOfObjects}, we reduce to showing that for all $t \in S$, $(\ev_x)^\dagger|_{\PP_S(C)_t}$ is a left Kan extension of $(\ev_x j)^\dagger|_{C_t}$. Note that for $y$ any cocartesian pushforward of $x$ over the essentially unique edge $s \rightarrow t$, we have both that $(\ev_x j)^\dagger|_{C_t}$ is homotopic to $\Map_{C_t}(y,-)$ and $(\ev_x)^\dagger|_{\PP_S(C)_t}$ is homotopic to $\ev_y$ (regarding $y$ as an object in $C_{\underline{t}}^\vop$). The inclusion
$$C_t \to \PP_S(C)_t \simeq \Fun(C^\vop_{\underline{t}}, \Top)$$
factors through $\PP(C_t)$ with $\PP(C_t) \to \Fun(C^\vop_{\underline{t}}, \Top)$ left adjoint to precomposition by the inclusion $i: C_t^\op \to C_{\underline{t}}^\vop$. By the usual Yoneda lemma for $\infty$-categories, $\ev_y: \PP(C_t) \to \Top$ is the left Kan extension of $\Map_{C_t}(y,-)$. The left Kan extension of $\ev_y$ to $\PP_S(C)_t$ is then given by precomposition by $i$, so is again $\ev_y$.
\end{proof}

To state the universal property of $\PP_S(C)$, we need to introduce a bit of terminology.

\begin{dfn} \label{def:strong-preservation} Let $F:C \to D$ be a $S$-functor. We say that $F$ \emph{strongly preserves $S$-(co)limits} if for all $s \in S$, $F_{\underline{s}}$ preserves $S^{s/}$-(co)limits.
\end{dfn}

\begin{rem} If $F$ strongly preserves $S$-colimits then $F$ preserves $S$-colimits. However, the converse is not necessarily true.
\end{rem}

\begin{ntn} Suppose that $C$ and $D$ are $S$-cocomplete $S$-categories. Let $\Fun^L_S(C,D)$ denote the full subcategory of $\Fun_S(C,D)$ on the $S$-functors $F$ which strongly preserve $S$-colimits. Let $\underline{\Fun}_S^L(C,D)$ denote the full $S$-subcategory of $\underline{\Fun}_S(C,D)$ with fibers $\Fun^L_{S^{s/}}(C,D)$ over $s \in S$.
\end{ntn}

\begin{thm} \label{thm:UniversalPropertyOfPresheaves} Let $E$ be a $S$-cocomplete $S$-category. Then restriction along the $S$-Yoneda embedding defines equivalences
\begin{align*} \Fun^L_S(\PP_S(C),E) &\xrightarrow{\simeq} \Fun_S(C,E) \\
\underline{\Fun}^L_S(\PP_S(C),E) &\xrightarrow{\simeq} \underline{\Fun}_S(C,E)
\end{align*}
with the inverse given by $S$-left Kan extension.
\end{thm}

We prepare for the proof of Theorem~\ref{thm:UniversalPropertyOfPresheaves} with some necessary results concerning $S$-mapping spaces. Recall that given an $\infty$-category $C$, we have a number of equivalent options for describing mapping spaces in $C$. The relevant ones to consider for us are:

\begin{enumerate}
	\item Straightening the left fibration $\widetilde{\sO}(C) \to C^\op \times C$, we obtain the mapping space functor
	$$\Map_C(-,-): C^\op \times C \to \Top;$$
	\item Fixing an object $x \in C$, straightening the left fibration $C^{x/} \to C$ also yields the functor
	$$\Map_C(x,-): C \to \Top;$$
	\item Fixing objects $x, y \in C$, we have that the space $\Map_C(x,y)$ is given by
	$$\{x\} \times_C \sO(C) \times_C \{y\}.$$
\end{enumerate}

Likewise, given a $S$-category $C$, we have these possibilities:

\begin{enumerate}
	\item The $S$-functor
	$$\underline{\Map}_C(-,-): C^\vop \times_S C \to \underline{\Spc}_S$$
	given by the $S$-straightening of $\widetilde{\sO}_S(C) \to C^\vop \times_S C$;
	\item Fixing an object $x \in C$, we have the left fibration $C^{\underline{x}/} = \underline{x} \times_C \sO_S(C) \to C$, which $S$-straightens to
	$$\underline{\Map}_C(x,-): C \to \underline{\Spc}_S;$$
	\item Fixing an object $x \in C$, we have the left fibration $C^{x/} \to C$, which $S$-straightens to
	$$\underline{\Map}_C(x,-): C \to \underline{\Spc}_S;$$
	\item Fixing objects $x \in C$ and $y \in C_s$, we have the $S^{s/}$-space
	$$\underline{\Map}_C(x,y) = \underline{x} \times_C \sO_S(C) \times_C \underline{y} \to \underline{y} \xto{\simeq} S^{s/}.$$
\end{enumerate}

In the proof of Lemma~\ref{lm:Yoneda}, we showed that (1) and (3) were equivalent, and by Proposition~\ref{SlicingUnderObject}, (2) and (3) are equivalent. Finally, (2) specializes to (4) by definition. We are thus justified in our abuse of notation when we interchangeably refer to any of these options by $\underline{\Map}_C(-,-)$.

\nomenclature[mappingSpace]{$\underline{\Map}_C(-,-)$}{Parametrized mapping space}

Our next goal is to prove that $\underline{\Map}_C(-,-)$ preserves $S$-limits in the second variable, and dually, takes $S$-colimits in the first variable to $S$-limits. For this, we need a few lemmas.

\begin{lem} \label{lm:EquivalenceOfSFibrations} Let $F:X \to Y$ be a map of $S$-cocartesian or $S$-cartesian fibrations over an $S$-category $C$. The following are equivalent:
\begin{enumerate} \item $F$ is an equivalence.
\item For all $s \in S$ and $S^{s/}$-functors $Z \to C_{\underline{s}}$,
$$\underline{\Fun}_{/C_{\underline{s}}, S^{s/}} (Z, X_{\underline{s}}) \to \underline{\Fun}_{/C_{\underline{s}}, S^{s/}} (Z, Y_{\underline{s}})$$
is an equivalence.
\item For all $s \in S$ and $c \in C_s$,
$$\underline{\Fun}_{/C_{\underline{s}}, S^{s/}} (\underline{c}, X_{\underline{s}}) \to \underline{\Fun}_{/C_{\underline{s}}, S^{s/} } (\underline{c}, Y_{\underline{s}})$$
is an equivalence.
\item For all $c \in C$, $F_c: X_c \to Y_c$ is an equivalence.
\end{enumerate}
If $X$ and $Y$ are $S$-left or $S$-right fibrations over $C$, then all instances of $\underline{\Fun}$ can be replaced by $\underline{\Map}$.\footnote{$\underline{\Map}$ refers here to the maximal sub-left fibration of $\underline{\Fun}$ and not the $S$-mapping space functor.}
\end{lem}
\begin{proof} (1) $\Rightarrow$ (2): If $F$ is an equivalence, so is $F_{\underline{s}}$ for all $s \in S$. The map in question is then induced by a map of pullbacks through equivalences in which two matching legs are $S$-fibrations, so is an equivalence.

\noindent (2) $\Rightarrow$ (3) is obvious.

\noindent (3) $\Rightarrow$ (4): Given $c \in C_s$, take fibers over $\{s\} \in \underline{s}$ and note that
$$\underline{\Fun}_{/C_{\underline{s}}, S^{s/} } (\underline{c}, X_{\underline{s}})_{s} \simeq \Fun_{/C_c}(\{c\},X_s) \simeq X_c.$$

\noindent (4) $\Rightarrow$ (1): We must check that $F_s$ is an equivalence for all $s \in S$, for which it suffices to check fiberwise over $C_s$ by the hypothesis.
\end{proof}

\begin{lem} \label{lm:LimitsInSCategoryOfSpaces} Let $\overline{q}: S \star_S K \to \underline{\Top}_S$ be a $S$-functor which extends $q: K \to \underline{\Top}_S$. Let $\overline{X} \to S \star_S K$ be a left fibration which is an unstraightening of $\overline{q}^\dagger$, and let $X = \overline{X} \times_{S \star_S K} K$. Then $\overline{q}$ is a $S$-limit diagram if and only if the restriction $S$-functor 
\[ R: \underline{\Map}_{/S \star_S K, S}(S \star_S K, \overline{X}) \to \underline{\Map}_{/S \star_S K, S}(K, \overline{X}) \cong \underline{\Map}_{/K, S}(K,X) \]
is an equivalence.
\end{lem}
\begin{proof} In view of \cite[Corollary~3.3.3.4]{HTT}, $R_s$ is a map from the limit of $\overline{q}^\dagger|_{\underline{s} \star_{\underline{s}} K_{\underline{s}}}$ to the limit of $q^\dagger|_{K_{\underline{s}}}$ induced by precomposition on the diagram. But by Proposition~\ref{prp:LimitsInCatOfObjects}, $\overline{q}$ is a $S$-limit diagram if and only if $\overline{q}^\dagger$ is a right Kan extension of $q^\dagger$, in which case both of the limits in question are equivalent to $\overline{q}^\dagger(s)$. The assertion now follows.
\end{proof}

\begin{prp} \label{prp:MappingSpacePreservesLimits} Let $\overline{p}: S \star_S K \to C$ be a $S$-functor. The following are equivalent:
\begin{enumerate}
	\item $\overline{p}$ is a $S$-limit diagram.
	\item For all $s \in S$ and $c \in C_s$,
	$$\underline{\Map}_{C_{\underline{s}}}(c,\overline{p}_{\underline{s}}(-)): \underline{s} \star_{\underline{s}} K_{\underline{s}}  \to \underline{\Top}_{S^{s/}}$$
	is a $S^{s/}$-limit diagram.
	\item For all $s \in S$ and $c \in C_s$,
	$$\underline{\Map}_{/C_{\underline{s}}, S^{s/}}(\underline{c}, C_{\underline{s}}^{/(\overline{p}_{\underline{s}},S^{s/})}) \to \underline{\Map}_{/C_{\underline{s}}, S^{s/} }(\underline{c}, C_{\underline{s}}^{/(p_{\underline{s}},S^{s/})})$$
	is an equivalence.
\end{enumerate}
Moreover, if the above conditions obtain, then
\[ \underline{\Map}_{/C_{\underline{s}}, S^{s/}}(\underline{c}, C_{\underline{s}}^{/(p_{\underline{s}},S^{s/})}) \simeq \underline{\Map}_{C_{\underline{s}}}(c,\overline{p}_{\underline{s}}(v)) \]
where $v$ is the cone point $\{s\} \in \underline{s} \star_{\underline{s}} K_{\underline{s}}$.
\end{prp}
\begin{proof}
(2) $\Leftrightarrow$ (3): We will show that the statements match after fixing $c \in C_s$. To ease notation, let us replace $S^{s/}$ by $S$ and suppose that $s \in S$ is an initial object. By Lemma~\ref{lm:LimitsInSCategoryOfSpaces} and using that $C^{\underline{c}/}$ is the $S$-unstraightening of $\underline{\Map}_C(c,-)$, $\underline{\Map}_C(c,\overline{p}(-))$ is a $S$-limit diagram if and only if
\[ \underline{\Map}_{/C,S}(S \star_S K, C^{\underline{c}/}) \to \underline{\Map}_{/C,S}(K, C^{\underline{c}/}) \]
is an equivalence. By Corollary~\ref{cor:SliceCompare}, this map is equivalent by a zig-zag to the map
\[ \underline{\Map}_{/C,S}(\underline{c}, C^{/(\overline{p},S)}) \to \underline{\Map}_{/C,S}(\underline{c}, C^{/(p,S)}). \]
The assertion now follows. The last assertion also follows in view of the equivalence $C^{/(\overline{p},S)} \simeq C^{/\underline{\overline{p}(v)}}$ and $\underline{\Map}_{/C,S}(\underline{c}, C^{/\underline{\overline{p}(v)}}) \simeq \underline{c} \times_{C} C^{/\underline{\overline{p}(v)}} \simeq \underline{\Map}_{C}(c,\overline{p}(v))$.

\noindent (1) $\Leftrightarrow$ (3): This follows from Lemma~\ref{lm:EquivalenceOfSFibrations} applied to $C^{/(\overline{p},S)} \to C^{/(p,S)}$, which is a map of $S$-right fibrations over $C$.
\end{proof}

\begin{cor} \label{cor:MappingSpaceDetectsIfFunctorPreservesColimits} Let $F: C \to D$ be a $S$-functor. Then
\begin{enumerate}
\item $F$ strongly preserves $S$-limits if and only if for all $s \in S$ and $d \in D_s$,
$$\underline{\Map}_{D_{\underline{s}}}(d,F_{\underline{s}}(-)): C_{\underline{s}} \to \underline{\Top}_{S^{s/}}$$
preserves $S^{s/}$-limits. 
\item $F$ strongly preserves $S$-colimits if and only if for all $s \in S$ and $d \in D_s$,
$$\underline{\Map}_{D_{\underline{s}}}(F_{\underline{s}}(-),d) = \underline{\Map}_{D^\vop_{\underline{s}}}(d,F^\vop_{\underline{s}}(-)): C^\vop_{\underline{s}} \to \underline{\Top}_{S^{s/}}$$
preserves $S^{s/}$-limits.
\end{enumerate}
\end{cor}

\begin{cor} Let $C$ be a $S$-category. The Yoneda embedding $j:C \to \PP_S(C)$ strongly preserves and detects $S$-limits.
\end{cor}
\begin{proof} Combine Proposition~\ref{prp:MappingSpacePreservesLimits} and Proposition~\ref{prp:colimitsInFunctorCategories}.
\end{proof}

\begin{proof}[Proof of Theorem~\ref{thm:UniversalPropertyOfPresheaves}] By Theorem~\ref{thm:leftKanExtensionIsLeftAdjointToRestriction}, we have a $S$-adjunction
\[ \adjunct{j_!}{\underline{\Fun}_S(C,E)}{\underline{\Fun}_S(\PP_S(C),E)}{j^\ast} \]
with $j^\ast j_! \simeq \id$ and the essential image of $j_!$ spanned by the left $S^{s/}$-Kan extensions ranging over all $s \in S$. By Proposition~\ref{adjFromRelAdj}, taking cocartesian sections yields an adjunction
\[ \adjunct{j_!}{\Fun_S(C,E)}{\Fun_S(\PP_S(C),E)}{j^\ast} \]
again with $j^\ast j_! \simeq \id$ and the essential image of $j_!$ spanned by the left $S$-Kan extensions. Both assertions will therefore follow if we prove that for a $S$-functor $F: \PP_S(C) \to E$, $F$ strongly preserves $S$-colimits if and only if $F$ is a left $S$-Kan extension of its restriction $f = F|_C$.

For the `only if' direction, because $\id_{\PP_S(C)}$ is a $S$-left Kan extension of $j$ by the $S$-Yoneda lemma \ref{lm:Yoneda}, $F =  F \circ \id_{\PP_S(C)}$ is a left $S$-Kan extension as it is the postcomposition of $\id_{\PP_S(C)}$ with a strongly $S$-colimit preserving functor.

For the `if' direction, we use the criterion of Corollary~\ref{cor:MappingSpaceDetectsIfFunctorPreservesColimits}. Replacing $S^{s/}$ by $S$ and supposing that $s \in S$ is an initial object, we reduce to showing that for all $x \in E_s$, $\underline{\Map}_E(F(-),x): \PP_S(C)^\vop \to \underline{\Top}_S$ preserves $S$-limits. We first observe that $F^\vop$ is a $S$-right Kan extension (of $f^\vop$), hence so is $\underline{\Map}_E(F(-),x) = \underline{\Map}_{E^\vop}(x,-) \circ F^\vop$ as the postcomposition of a $S$-right Kan extension with a strongly $S$-limit preserving functor. However, by the vertical opposite of the $S$-Yoneda lemma, for any $S$-functor $G: C^\vop \to \underline{\Top}_S$, the strongly $S$-limit preserving $S$-functor $\underline{\Map}_{\PP_S(C)}(-,G)$ is a $S$-right Kan extension of $G$. Applying this for $G = \underline{\Map}_E(f(-),x)$, we conclude.
\end{proof}

\section{Bousfield--Kan formula}

In this section, we prove two decomposition formulas for $S$-colimits which resemble the classical Bousfield--Kan formula for computing homotopy colimits. We first study the situation when $S = \Delta^0$.

\begin{ntn} Let $K$ be a simplicial set and let $\Delta_{/K}$ be the nerve of the category of simplices of $K$. We denote the first vertex map by $\upsilon_K: \Delta_{/K}^\op \to K$ and the last vertex map by $\mu_K: \Delta_{/K} \to K$.
\end{ntn}

\nomenclature[vertexFirst]{$\upsilon_K$}{First vertex map}
\nomenclature[vertexLast]{$\mu_K$}{Last vertex map}

By \cite[Proposition~4.2.3.14]{HTT}, $\mu_K$ is final. Unfortunately, this is the wrong direction for the purposes of obtaining a Bousfield--Kan type formula, since $\Delta_{/K}$ is a \emph{cartesian} fibration over $\Delta$. To rectify this state of affairs, we prove that $\upsilon_K$ is in fact final.

\begin{prp} \label{prp:cofinalityFirstVertex} Let $K$ be a simplicial set. Then the first vertex map $\upsilon_K: \Delta_{/K}^\op \to K$ is final. Equivalently, the last vertex map $\mu_{K^\op}$ is initial.
\end{prp}
\begin{proof} Note that $\upsilon_K$ is natural in $K$ and that $\Delta_{/(-)}^\op: s\Set \to s\Set$ preserves colimits. Recall from \cite[Proposition~4.1.2.5]{HTT} that a map $f: X \to Y$ is final if and only if it is a contravariant equivalence in $s\Set_{/Y}$. It follows that the class of final maps is stable under filtered colimits, so we may suppose that $K$ has finitely many nondegenerate simplices. Using left properness of the contravariant model structure, by induction we reduce to the assertion for $K = \Delta^n$. But in this case $\upsilon_K$ is final by the proof of \cite[Variant~4.2.3.15]{HTT} (which proves the result when $K$ is the nerve of a category).

For the second assertion, we note that the reversal isomorphism $\Delta_{/K^\op} \cong \Delta_{/K}$ interchanges $\mu_{K^\op}$ and $(\upsilon_K)^\op$.
\end{proof}

\begin{cor}[Bousfield--Kan formula] \label{cor:ordinaryBKformula} Suppose that $C$ admits (finite) coproducts. Then for a (finite) simplicial set $K$ and a map $p: K \to C$, the colimit of $p$ exists if and only if the geometric realization 
\[ \left| \begin{tikzcd}[row sep = 2em, column sep = 2em]
 \bigsqcup_{x \in K_0} p(x) & \bigsqcup_{\alpha \in K_1} p(\alpha(0)) \ar[shift right]{l} \ar[shift left]{l} & \bigsqcup_{\sigma \in K_2} p(\sigma(0)) \ar{l} \ar[shift left = 2]{l} \ar[shift right = 2]{l}
\end{tikzcd} \dots \right| \]
exists, in which case the colimit of $p$ is computed by the geometric realization.
\end{cor}
\begin{proof} The fibers of the cocartesian fibration $\pi_K: \Delta^\op_{/K} \to \Delta^\op$ are the discrete sets $K_n$. Therefore, the left Kan extension of $p \circ \upsilon_K$ along $\pi_K$ exists. By Proposition~\ref{prp:cofinalityFirstVertex}, $\colim p \simeq \colim p \circ \upsilon_K$, and the latter is computed as the colimit of $(\pi_K)_! (p \circ \upsilon_K)$ by the transitivity of left Kan extensions.
\end{proof}

We also have a variant of Cor \ref{cor:ordinaryBKformula} where the coproducts over $K_n$ are replaced by colimits indexed by the spaces $\Map(\Delta^n,K)$. To formulate this, we need to introduce some auxiliary constructions. Let $\xi: W \to \Delta^\op$ be the opposite of the relative nerve of the inclusion $\Delta \to s\Set$; this is a cartesian fibration which is an explicit model for the tautological cartesian fibration over $\Delta^\op$ pulled back from the universal cartesian fibration over $\Cat_\infty^\op$. Let $\lambda: \Delta^\op \to W$ be the `first vertex' section of $\xi$ which sends an $n$-simplex $\Delta^{a_0} \leftarrow ... \leftarrow \Delta^{a_n}$ to the $n$-simplex
\[ \begin{tikzcd}[row sep=2em, column sep=2em]
\Delta^n \ar{d}{(\lambda a)_0} & \ar{l} ... & \Delta^{\{n-1,n\}} \ar{l} \ar{d}{(\lambda a)_{n-1}} & \Delta^{\{n\}} \ar{l} \ar{d}{(\lambda a)_n} \\
\Delta^{a_0} & \ar{l} ... & \Delta^{a_{n-1}}  \ar{l} & \Delta^{a_n} \ar{l}
\end{tikzcd} \]
of $W$ specified by $(\lambda a)_i(0) = 0$ for all $0 \leq i \leq n$.

For an $\infty$-category $C$, let $Z_C = \widetilde{\Fun}_{\Delta^\op}(W,C \times \Delta^\op)$ and let $Z'_C \subset Z_C$ be the sub-simplicial set on the simplices $\sigma$ such that every edge of $\sigma$ is cocartesian (with respect to the structure map to $\Delta^\op$), so that $Z'_C \to \Delta^\op$ is the maximal sub-left fibration in $Z_C \to \Delta^\op$. Define a $\Delta^\op$-functor $\Delta_{/C}^\op \to Z_C$ as adjoint to the map $\Delta_C^\op \times_{\Delta^\op} W \to C$ which sends an $n$-simplex 
\[ \begin{tikzcd}[row sep=2em, column sep=2em]
\Delta^n \ar{d}{(\lambda a)_0} & \ar{l} ... & \Delta^{\{n-1,n\}} \ar{l} \ar{d}{(\lambda a)_{n-1}} & \Delta^{\{n\}} \ar{l} \ar{d}{(\lambda a)_n} \\
\Delta^{a_0} \ar{d}{\tau} & \ar{l} ... & \Delta^{a_{n-1}}  \ar{l} \ar{dll} & \Delta^{a_n} \ar{l} \ar{dlll} \\
C
\end{tikzcd} \]
to $\tau \circ (\lambda a)_0 \in C_n$. Note that since $\Delta_{/C}^\op \to \Delta^\op$ is a left fibration, this functor factors through $Z'_C$.

Define a `first vertex' functor $\Upsilon_C: Z_C \to C$ by precomposition with $\iota$ (using the isomorphism $\widetilde{\Fun}_{\Delta^\op}(\Delta^\op, C \times \Delta^\op) \cong C \times \Delta^\op$). We then have a factorization of the first vertex map as
\[ \begin{tikzcd}[row sep=2em, column sep=2em]
\Delta^\op_{/C} \ar{r} & Z'_C \ar{r} & Z_C \ar{r}{\Upsilon_C} & C.
\end{tikzcd} \]
\nomenclature[vertexFirstSpace]{$\Upsilon_K$}{First vertex functor, space variant}
\begin{prp} \label{prp:spaceBousfieldKan} The functors $\Upsilon_C$ and $\Upsilon'_C = (\Upsilon_C)|_{Z'_C}$ are final.
\end{prp}
\begin{proof} We first prove that $\Upsilon_C$ is final by verifying the hypotheses of \cite[Theorem~4.1.3.1]{HTT}. Let $c \in C$. The map $Z_C \to C$ is functorial in $C$, so we have a map $Z_{C_{c/}} \to Z_C \times_C C_{c/}$. We claim that this map is a trivial Kan fibration. Unwinding the definitions, this amounts to showing that for every cofibration $A \to B$ of simplicial sets over $\Delta^\op$, we can solve the lifting problem
\[ \begin{tikzcd}[row sep=2em, column sep=2em]
B \cup_A A \times_{\Delta^\op} W \ar{r} \ar{d} & C_{c/} \ar{d} \\
B \times_{\Delta^\op} W \ar{r} \ar[dotted]{ru} & C.
\end{tikzcd} \]

Since the class of left anodyne morphisms is right cancellative, we may suppose $A = \varnothing$. It thus suffices to prove that $\lambda_B = B \times_{\Delta^\op} \lambda: B \to B \times_{\Delta^\op} W$ is left anodyne for any map of simplicial sets $B \to \Delta^\op$. Observe that even though $\lambda$ is not a cartesian section, it is a left adjoint relative to $\Delta^\op$ to $\xi$ by \cite[Proposition~7.3.2.6]{HA} and the uniqueness of adjoints, since on the fibers it restricts to the adjunction $\adjunctb{\{0\}}{\Delta^n}$. Consequently, for any $\infty$-category $B$ and functor $B \to \Delta^\op$, by \cite[Proposition~7.3.2.5]{HA} $\lambda_B$ is a left adjoint, hence left anodyne. From this, we deduce the general case by using the characterization in \cite[Proposition~4.1.2.1]{HTT} of the left anodyne maps $X \to Y$ as the trivial cofibrations in $s\Set_{/Y}$ equipped with the covariant model structure. Indeed, arguing as in the proof of Proposition~\ref{prp:cofinalityFirstVertex}, by induction on the nondegenerate simplices of $B$ we reduce to the known case $B = \Delta^n$.

We next prove that $Z_C$ is weakly contractible if $C$ is, which will conclude the proof for $\Upsilon_C$. For this, another application of (the opposite of) \cite[Proposition~7.3.2.6]{HA} shows that the $\Delta^\op$-functor $C \times \Delta^\op \to Z_C$ defined by precomposition by $\xi$ is a left adjoint relative to $\Delta^\op$ to the functor $(\Upsilon_C,\id_{\Delta^\op})$, because it restricts to the adjunction $\adjunct{\iota}{C}{\Fun(\Delta^n,C)}{\ev_0}$ on the fibers. Hence, $|Z_C| \simeq |C \times \Delta^\op| \simeq |C|$, and the latter is contractible by hypothesis.

We employ the same strategy to show that $\Upsilon'_C$ is final. Since $C_{c/} \to C$ is conservative, the trivial Kan fibration above restricts to yield a trivial Kan fibration $Z'_{C_{c/}} \to Z'_C \times_C C_{c/}$. Thus it suffices to show that $Z'_C$ is weakly contractible if $C$ is. By (the opposite of) \cite[Proposition~7.3]{GHN}, the cocartesian fibration $Z'_C \to \Delta^\op$ is classified by the functor
$$\Delta^\op \xrightarrow{i^\op} \Cat_\infty \xrightarrow{\Map(-,C)} \Top.$$
Let $R$ denote the right adjoint to the colimit-preserving functor $L: \Fun(\Delta^\op,\Top) \to \Cat_\infty$ left Kan extended from the inclusion $i: \Delta \subset \Cat_\infty$; $R$ sends an $\infty$-category to its corresponding complete Segal space. Then $R(C) \simeq \Map(-,C) \circ i^\op$. For any $X_\bullet \in \Fun(\Delta^\op,\Top)$, we have $\colim X \simeq |L(X_\bullet)|$, hence
$$\colim R(C) \simeq |(L \circ R)(C)| \simeq |C|,$$
where $L \circ R \simeq \id$ by \cite[Corollary~4.3.16]{G}. By \cite[Corollary~3.3.4.6]{HTT}, $|Z'_C| \simeq \colim R(C)$, so we conclude that $|Z'_C|$ is contractible.
\end{proof}

The following corollary was previously proven by Mazel-Gee in \cite{MAZELGEE20194602}.

\begin{cor}[Bousfield--Kan formula, `simplicial' variant] \label{cor:BousfieldKanFormulaHomotopyInvariantVersion} Suppose that $C$ admits colimits indexed by spaces. Then for any $\infty$-category $K$ and functor $p:K \to C$, the colimit of $p$ exists if and only if the geometric realization
\[ \left| \begin{tikzcd}[row sep = 2em, column sep = 2em]
 \colim\limits_{x \in \Map(\Delta^0,K)} p(x) & \colim\limits_{\alpha \in \Map(\Delta^1,K)} p(\alpha(0)) \ar[shift right]{l} \ar[shift left]{l} & \colim\limits_{\sigma \in \Map(\Delta^2,K)} p(\sigma(0)) \ar{l} \ar[shift left = 2]{l} \ar[shift right = 2]{l}
\end{tikzcd} \dots \right| \]
exists, in which case the colimit of $p$ is computed by the geometric realization.
\end{cor}
\begin{proof} Using Proposition~\ref{prp:spaceBousfieldKan}, we may repeat the proof of Corollary~\ref{cor:ordinaryBKformula}, now using the span
\[ \Delta^\op \leftarrow Z'_K \xrightarrow{\Upsilon'_K} K.\]
\end{proof}

We now proceed to relativize the above picture, starting with the map $\Upsilon_C$. Let $C \to S$ be a $S$-category. Define the map
\[ \Upsilon_{C,S}: \widetilde{\Fun}_{\Delta^\op \times S/S} (W \times S, \Delta^\op \times C) \to C \]
to be the composition of the map to $\widetilde{\Fun}_{\Delta^\op \times S/S} (\Delta^\op \times S, \Delta^\op \times C)$ given by precomposition by $\lambda \times \id_S$, together with the equivalence of Lemma~\ref{lm:pairingFirstVariableTrivial} of this to $\Delta^\op \times C$ and the projection to $C$. Define $\Upsilon'_{C,S}$ to be the restriction of $\Upsilon_{C,S}$ to the maximal sub-left fibration (with respect to $\Delta^\op \times S$).

\nomenclature[firstVertexParamSpace]{$\Upsilon_{C,S}$}{Parametrized first vertex functor, space variant}

\begin{thm} \label{thm:relativeBKwithSpaces} The $S$-functors $\Upsilon_{C,S}$ and $\Upsilon'_{C,S}$ are $S$-final.
\end{thm}
\begin{proof} For every object $s \in S$, we have a commutative diagram
\[ \begin{tikzcd}[row sep=2em, column sep=2em]
\widetilde{\Fun}_{\Delta^\op \times S/S} (W \times S, \Delta^\op \times C)_s \ar{r}{(\lambda \times \id_S)^\ast_s} \ar[bend left=15]{rr}{(\Upsilon_{C,S})_s} \ar{d}{\simeq} & \widetilde{\Fun}_{\Delta^\op \times S/S} (\Delta^\op \times S, \Delta^\op \times C)_s \ar{d}{\simeq} \ar{r} & C_s \ar{d}{=} \\
\widetilde{\Fun}_{\Delta^\op}(W, \Delta^\op \times C_s) \ar{r}{\lambda^\ast} \ar[bend right=15]{rr}{\Upsilon_{C_s}} & \widetilde{\Fun}_{\Delta^\op}(\Delta^\op, \Delta^\op \times C_s) \cong \Delta^\op \times C_s \ar{r}{\pr_{C_s}} & C_s
\end{tikzcd} \]
where the left two vertical maps are given by the natural categorical equivalences of Lemma~\ref{lm:pairingProducts}; the only point to note is that the equivalences of Lemma~\ref{lm:pairingFirstVariableTrivial} and Lemma~\ref{lm:pairingProducts} coincide when the first variable is trivial. By Proposition~\ref{prp:spaceBousfieldKan}, $\Upsilon_{C_s}$ is final, so $(\Upsilon_{C,S})_s$ is final. By the $S$-cofinality Theorem~\ref{thm:cofinality}, $\Upsilon_{C,S}$ is $S$-final. A similar argument shows that $\Upsilon'_{C,S}$ is $S$-final.
\end{proof}

The process of relativizing $\upsilon_C$ is considerably more involved. We begin with some preliminaries on the relative nerve construction. Let $J$ be a category.

\begin{lem} The adjunctions
\begin{align*}
\adjunct{\mathfrak{F}_J}{s\Set_{/N(J)}}{\Fun(J,s\Set)}{N_J} \\
\adjunct{\mathfrak{F}^+_J}{s\Set^+_{/N(J)}}{\Fun(J,s\Set^+)}{N^+_J}
\end{align*}
of \cite[\S 3.2.5]{HTT} are simplicial.
\end{lem}
\begin{proof} Let $\underline{K}: J \to s\Set$ denote the constant functor at a simplical set $K$. We have an obvious map $\chi_K: N(J) \times K \to N_J(\underline{K})$ natural in $K$ and hence a map
\[ (\eta_X,\chi_K \circ \pr): X \times K \to N_J (\mathfrak{F}_J X \times \underline{K}) \cong N_J \mathfrak{F}_J X \times N_J(\underline{K}) \]
natural in $X$ and $K$. We want to show the adjoint
\[ \theta_{X,K}: \mathfrak{F}_J(X \times K) \to \mathfrak{F}_J(X) \times \underline{K} \]
is an isomorphism. Both sides preserve colimits separately in each variable, so we may suppose $X = \Delta^n \to J$ and $K = \Delta^m$. By \cite[Example~3.2.5.6]{HTT}, $\mathfrak{F}_I(I)(-) \cong N(I_{/-})$, and by \cite[Remark~3.2.5.8]{HTT}, for any functor $f: I \to J$, the square
\[ \begin{tikzcd}[row sep=2em, column sep=2em]
s\Set_{/N(I)} \ar{r}{f_!} \ar{d}{\mathfrak{F}_I} & s\Set_{/N(J)} \ar{d}{\mathfrak{F}_J} \\
\Fun(I,s\Set) \ar{r}{f_!} & \Fun(J,s\Set)
\end{tikzcd} \]
commutes. Letting $I = \Delta^n \times \Delta^m$ and $f: I \to J$ be the structure map, we have
\[ \mathfrak{F}_I(\Delta^n \times \Delta^m)(k,l) \cong (\Delta^n)_{/k} \times (\Delta^m)_{/l} \cong \Delta^k \times \Delta^l. \]
Factoring $f$ as $\Delta^n \times \Delta^m \xrightarrow{g} \Delta^n \xrightarrow{h} J$, we then have
\[ g_! \mathfrak{F}_I(\Delta^n \times \Delta^m)(k) \cong \Delta^i \times \Delta^m. \]
Let $G = g_! \mathfrak{F}_I(\Delta^n \times \Delta^m)$, so that $\mathfrak{F}_J(\Delta^n \times \Delta^m)(j) \cong (h_! G)(j)$. Then
\[(h_! G)(j) \cong \colim\limits_{\Delta^n \times_J J_{/j}} \left( (k,h(k) \rightarrow j) \mapsto \Delta^k \right) \times \Delta^m \cong \mathfrak{F}_J(\Delta^n)(j) \times \Delta^m \]
and one can verify that $\theta_{X,K}$ implements this isomorphism. For the assertion about $\mathfrak{F}^+_J \dashv N^+_J$, recall that the simplicial tensor $s\Set \times s\Set^+ \to s\Set^+$ is given by $(K,X) \mapsto K^\sharp \times X$. Consequently, in the above argument we may simply replace $\Delta^m$ by $(\Delta^m)^\sharp$ to conclude.
\end{proof}

Since $N^+_J(\underline{S^\sharp}) = N(J) \times S^\sharp$, the adjunction $\mathfrak{F}^+_J \dashv N^+_J$ lifts to an adjunction
\[ \adjunct{\mathfrak{F}^+_{J,S}}{s\Set^+_{/N(J) \times S}}{\Fun(J,s\Set^+_{/S})}{N^+_{J,S}} \]
between the overcategories. Moreover, for any functor $f:T \to S$, the square
\[ \begin{tikzcd}[row sep=2em, column sep=2em]
\Fun(J,s\Set^+_{/S}) \ar{r}{N^+_{J,S}} \ar{d}{f^\ast} & s\Set^+_{/N(J) \times S} \ar{d}{(\id \times f)^\ast} \\
\Fun(J,s\Set^+_{/T}) \ar{r}{N^+_{J,T}}  & s\Set^+_{/N(J) \times T},
\end{tikzcd} \]
commutes.

\begin{prp} \label{prp:SrelativeNerve} Equip $s\Set^+_{/N(J) \times S}$ with the cocartesian model structure and $\Fun(J,s\Set^+_{/S})$ with the projective model structure, where $s\Set^+_{/S}$ has the cocartesian model structure. Then the adjunction
\[ \adjunct{\mathfrak{F}^+_{J,S}}{s\Set^+_{/N(J) \times S}}{\Fun(J,s\Set^+_{/S})}{N^+_{J,S}} \]
is a Quillen equivalence.
\end{prp}
\begin{proof} We first prove that the adjunction is Quillen. Because this is a simplicial adjunction between left proper simplicial model categories, it suffices to show that $\mathfrak{F}^+_{J,S}$ preserves cofibrations and $N^+_{J,S}$ preserves fibrant objects. Observe that the slice model structure on $$s\Set^+_{/N(J) \times S} \cong (s\Set^+_{/N(J)})_{/(N(J) \times S)^\sharp}$$ is a localization of the cocartesian model structure. Similarly, the slice model structure on $$\Fun(J,s\Set^+_{/S}) \cong \Fun(J,s\Set^+)_{/\underline{S^\sharp}}$$ is a localization of the projective model structure, since the trivial fibrations for the two model structures coincide and postcomposition by $\pi_!: s\Set^+_{/S} \to s\Set^+$ gives a Quillen left adjoint between the projective model structures. Since the lift of a Quillen adjunction $\adjunct{L}{M}{N}{R}$ to the adjunction $\adjunct{\widetilde{L}}{M_{/R(x)}}{N_{/x}}{\widetilde{R}}$ is Quillen for the slice model structures, we deduce that $\mathfrak{F}^+_{J,S}$ preserves cofibrations.

Now suppose $F: J \to s\Set^+_{/S}$ is fibrant. Since $S$ is an $\infty$-category, $F \rightarrow \underline{S}$ is a fibration in $\Fun(J,s\Set)$. Hence $N_{J,S}(F) \to N(J) \times S$ is a categorical fibration. We verify that it is a cocartesian fibration (with every marked edge cocartesian) by solving the lifting problem ($n \geq 1$)
\[ \begin{tikzcd}[row sep=2em, column sep=2em]
\leftnat{\Lambda^n_0} \ar{r} \ar{d} & N^+_{J,S}(F) \ar{d} \\
\leftnat{\Delta^n} \ar{r}[swap]{(j_\bullet, s_\bullet)} \ar[dotted]{ru} & (N(J) \times S)^\sharp.
\end{tikzcd} \]
Unwinding the definitions, this amounts to solving the lifting problem
\[ \begin{tikzcd}[row sep=2em, column sep=2em]
\leftnat{\Lambda^n_0} \ar{r} \ar{d} & F(j_n) \ar{d} \\
\leftnat{\Delta^n} \ar{r}[swap]{s_\bullet} \ar[dotted]{ru} & S^\sharp,
\end{tikzcd} \]
and the dotted lift exists because $F(j_n)$ is cocartesian over $S$ with the cocartesian edges marked. Finally, it is easy to see that marked edges compose and are stable under equivalence. We conclude that $N^+_{J,S}(F)$ is fibrant in $s\Set^+_{/N(J) \times S}$.

To prove that the Quillen adjunction is a Quillen equivalence, we will show that the induced adjunction of $\infty$-categories
\[ \adjunct{\mathfrak{F}'^+_{J,S}}{N((s\Set^+_{/N(J) \times S})^\circ)}{N(\Fun(J,s\Set^+_{/S})^\circ)}{N'^+_{J,S}} \]
is an adjoint equivalence, where $N'^+_{J,S}$ is the simplicial nerve of $N^+_{J,S}$ and $\mathfrak{F}'^+_{J,S}$ is any left adjoint to $N'^+_{J,S}$. We first check that $N'^+_{J,S}$ is conservative. Indeed, for this we may work in the model category: for a natural transformation $\alpha: F \rightarrow G$ in $\Fun(J,s\Set^+_{/S})$, $N^+_{J,S}(F) \to N^+_{J,S}(G)$ on fibers is given by $F(j)_s \to G(j)_s$, hence if $F,G$ are fibrant and $N^+_{J,S}(\alpha)$ is an equivalence then $\alpha$ is as well. It now suffices to show that the unit transformation $\eta: \id \to N'^+_{J,S} \mathfrak{F}'^+_{J,S}$ is an equivalence. We have the known equivalence $N((s\Set^+_{/N(J) \times S})^\circ) \simeq \Fun(N(J) \times S, \Cat_\infty)$ so it further suffices to check that the map
\[ (\id \times i_s)^\ast \to (\id \times i_s)^\ast N'^+_{J,S} \mathfrak{F}'^+_{J,S} \simeq N'^+_{J} i_s^\ast \mathfrak{F}'^+_{J,S} \]
is an equivalence for all $s \in S$, $i_s: \{s\} \to S$ the inclusion. Equivalently, since $\mathfrak{F}^+_J \dashv N^+_J$ is a Quillen equivalence by \cite[Proposition~3.2.5.18]{HTT}, we must show that the adjoint map
\[ \mathfrak{F}'^+_{J} i_s^\ast \to (\id \times i_s)^\ast \mathfrak{F}'^+_{J,S} \]
is an equivalence. This statement is in turn equivalent to the adjoint map
\[ \theta: N'^+_{J,S} (i_s)_\ast \to (\id \times i_s)_\ast N'^+_{J} \]
being an equivalence. Recall that for a functor $f: T \to S$, $f_\ast: \Fun(T,\Cat_\infty) \to \Fun(S,\Cat_\infty)$ is induced by $\pi_\ast \rho^\ast: s\Set^+_{/T} \to s\Set^+_{/S}$ for the span
\[ \begin{tikzcd}[row sep=2em, column sep=2em]
S^\sharp & (\sO(S) \times_S T)^\sharp \ar{r}{\rho} \ar{l}[swap]{\pi} & T^\sharp
\end{tikzcd} \]
with $\pi$ given by evaluation at $0$ and $\rho$ projection to $T$. Moreover, for a functor $\id \times f: U \times T \to U \times S$, we may elect to use the span 
\[ \begin{tikzcd}[row sep=2em, column sep=2em]
(U \times S)^\sharp & (U \times \sO(S) \times_S T)^\sharp \ar{r}{\id \times \rho} \ar{l}[swap]{\id \times \pi} & (U \times T)^\sharp
\end{tikzcd} \]
to model $(\id \times f)_\ast$. Letting $f = i_s$, we see that $\theta$ is induced by the map
\[ N^+_{J,S} \pi_\ast \rho^\ast \to (\id \times \pi)_\ast N^+_{J, S^{s/}} \rho^\ast \cong (\id \times \pi)_\ast (\id \times \rho)^\ast N^+_{J}.  \]
where the first map is adjoint to the isomorphism $(\id \times \pi)^\ast N^+_{J,S} \cong N^+_{J, S^{s/}} \pi^\ast$. Direct computation reveals that this map is an equivalence on fibrant $F: J \to s\Set^+$.
\end{proof}

We now return to the situation of interest. Let $C$ be a $S$-category with structure map $\pi: C \to S$. We first extend our existing notation $\underline{x}$ for objects $x \in C$.

\begin{ntn} For an $n$-simplex $\sigma$ of $C$, define
 \[ \underline{\sigma} = \{ \sigma \} \times _{\Fun(\Delta^n \times \{0\},C)} \Fun((\Delta^n)^\flat \times (\Delta^1)^\sharp,\leftnat{C}) \times_{\Fun(\Delta^n \times \{1\},S)} S. \]
\end{ntn}

\begin{lem} \label{lm:fattenedSlice} There exists a map $b_\sigma: \underline{\sigma} \to \{ \pi \sigma(n) \} \times_S \sO(S) = S^{\pi \sigma(n)/}$ which is a trivial Kan fibration.
\end{lem}
\begin{proof} First define a map $b'_\sigma: \underline{\sigma} \to \underline{\pi \sigma}$ to be the pullback of the map
\[ (e_0, \sO(\pi))_\ast: \Fun(\Delta^n, \sO^\cocart(C)) \to C^{\Delta^n} \times_{S^{\Delta^n}} \Fun(\Delta^n, \sO(S))\]
over $\{ \sigma \}$ and $S$. Since $(e_0, \sO(\pi))$ is a trivial Kan fibration, so is $b'_\sigma$. Next, let $K$ be the pushout $\Delta^n \times\{0\} \cup_{\{n\} \times \{0\}} \{n\} \times \Delta^1$. We claim that the map $\Fun(\Delta^n, \sO(S)) \times_{S^{\Delta^n}} S \to \Fun(K,S)$ induced by $K \subset \Delta^n \times \Delta^1$ is a trivial Kan fibration. For a monomorphism $A \to B$, we need to solve the lifting problem
\[ \begin{tikzcd}[row sep=2em, column sep = 2em]
A \ar{r} \ar{d} & \Fun(\Delta^n,\sO(S)) \times_{S^{\Delta^n}} S \ar{d} \\
B \ar[dotted]{ru} \ar{r} & \Fun(K,S).
\end{tikzcd} \]
This transposes to
\[ \begin{tikzcd}[row sep=2em, column sep = 2em]
A \times \Delta^n \bigcup_{A \times \{n\}} B \times \{n\} \ar{r} \ar{d} & \sO(S) \ar{d}{\ev_0} \\
B \times \Delta^n \ar[dotted]{ru} \ar{r} & S
\end{tikzcd} \]
and the lefthand map is right anodyne by \cite[Corollary~2.1.2.7]{HTT}, hence the dotted lift exists as $\ev_0$ is a cartesian fibration. Now define $b''_{\sigma}$ to be the pullback
\[ \underline{\pi \sigma} = \{ \pi \sigma \} \times_{S^{\Delta^n}} \Fun(\Delta^n, \sO(S)) \times_{S^{\Delta^n}} S \to \{ \pi \sigma\} \times_{S^{\Delta^n}} \Fun(K,S) \cong S^{\pi \sigma(n)/}; \]
this is also a trivial Kan fibration. Finally, let $b_\sigma = b''_\sigma \circ b'_\sigma$.
\end{proof}

We will regard $\underline{\sigma}$ as a $S^{\pi \sigma(n)/}$ or $S$-category via $b_\sigma$. We also have a target map $\underline{\sigma} \to C^{\Delta^n}$ induced by $\Delta^n \times \{1\} \subset \Delta^n \times \Delta^1$. This covers the target map $S^{\pi \sigma(n)/} \to S$ and is a $S$-functor.

Define a functor $F_C: \Delta^\op \to s\Set^+_{/S}$ on objects $[n]$ by
\[ F_C([n]) = \bigsqcup\limits_{\sigma \in C_n} \underline{\sigma}^\sharp \]
and on morphisms $\alpha: [m] \rightarrow [n]$ by the map $\underline{\sigma} \to \underline{\sigma \alpha}$ induced by precomposition by $\alpha: \Delta^m \to \Delta^n$.

\begin{rem} The map $\underline{\sigma} \to \underline{\sigma(n)}$ is compatible with the maps $b_\sigma$ and $b_{\sigma(n)}$ of Lemma~\ref{lm:fattenedSlice}, hence is a categorical equivalence (in fact, a trivial Kan fibration). Consequently, given a morphism $f: x \rightarrow y$ in $C$, by choosing an inverse to $\underline{f} \xrightarrow{\simeq} \underline{y}$ we obtain a map $f^\ast: \underline{y} \to \underline{x}$, unique up to contractible choice. Moreover, if $f$ lies over an equivalence, then $\underline{f} \to \underline{x}$ is a trivial Kan fibration, so we also obtain a map $f_!: \underline{x} \to \underline{y}$.
\end{rem}

In order to define the $S$-first vertex map $N^+_{\Delta^\op,S}(F_C) \to C$, we need to introduce a few preliminary constructions. Let $A_n \subset \sO(\Delta^n)$ be the sub-simplicial set where a $k$-simplex $x_0 y_0 \rightarrow ... \rightarrow x_k y_k$ is in $A_n$ if and only if $x_k \leq y_0$. For the reader's aid we draw a picture of the inclusion $A_n \subset \sO(\Delta^n)$ for $n=2$, where dashed edges are not in $A_2$:
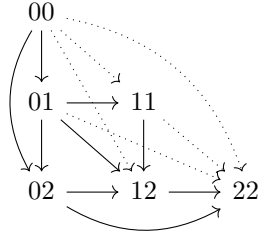
\[ \begin{tikzcd}[row sep=2em, column sep=2em]
0 0 \ar[dotted]{rd} \ar[dotted,bend left]{rrdd} \ar{d} \ar[bend right]{dd} \ar[dotted]{rdd} &  \\
0 1 \ar{r} \ar{d} \ar{rd} \ar[dotted]{rrd} & 1 1 \ar{d} \ar[dotted]{rd} \\
0 2 \ar{r} \ar[bend right]{rr} & 1 2 \ar{r} & 2 2.
\end{tikzcd} \]

\begin{lem} \label{lem:innerAnodyneMap} The inclusion $A_n \to \sO(\Delta^n)$ is inner anodyne.
\end{lem}
\begin{proof} In this proof we adopt the notation $[x_0 y_0,...,x_k y_k]$ for a $k$-simplex of $\sO(\Delta^n)$. Let $E$ be the collection of edges $[a b ,x y]$ in $\sO(\Delta^n)$ where $x > b$, and choose a total ordering $\leq$ on $E$ such that if we have a factorization
\[ \begin{tikzcd}[row sep=2em, column sep=2em]
a b \ar{r} \ar{d} & x y \ar{d} \\
a' b' \ar{r} & x' y'
\end{tikzcd} \]
then $[a' b' , x' y'] \leq [a b , x y]$. Index edges in $E$ by $I = \{0,...,N\}$. Define simplicial subsets $A_{n,i}$ of $\sO(\Delta^n)$ such that $A_{n,i}$ is obtained by expanding $A_n$ to contain every $k$-simplex $[x_0 y_0, ... ,x_k y_k]$ with $[x_0 y_0, x_k y_k]$ in $E_{< i}$. We will show that each inclusion $A_{n,i} \to A_{n,i+1}$ is inner anodyne. We may divide the nondegenerate $k$-simplices $[x_0 y_0, x_1 y_1,...,x_k y_k]$ in $A_{n,i+1}$ but not in $A_{n,i}$ into six classes:
\begin{itemize}	
	\item $A1$: $x_1 y_1 \neq x_0 (y_0+1)$ and $y_1 > y_0$.
	\item $A2$: $x_1 y_1 = x_0 (y_0+1)$.	
	\item $B1$: $x_1 y_1 = (x_0+1) y_0$, $y_2 > y_0$, and $x_2 y_2 \neq (x_0+1) (y_0+1)$.
	\item $B2$: $x_1 y_1 = (x_0+1) y_0$ and $x_2 y_2 = (x_0+1) (y_0+1)$.
	\item $C1$: $x_1 y_1 \neq (x_0+1) y_0$ and $y_1 = y_0$.
	\item $C2$: $x_1 y_1 = (x_0+1) y_0$ and $y_2 = y_0$.
\end{itemize}

We have bijections between classes of form 1 and classes of form 2 given by
\begin{itemize}
	\item $A$: $[x_0 y_0, x_1 y_1,...,x_k y_k] \mapsto [x_0 y_0, x_0 (y_0+1) , x_1 y_1,...,x_k y_k]$.	
	\item $B$: $[x_0 y_0, x_0 +1 y_1, x_2 y_2,...,x_k y_k] \mapsto [x_0 y_0, (x_0+1) y_0, (x_0+1) (y_0+1), x_2 y_2,...,x_k y_k]$.
	\item $C$: $[x_0 y_0, x_1 y_1,...,x_k y_k] \mapsto [x_0 y_0, (x_0+1) y_0 ,x_1 y_1,...,x_k y_k]$.
\end{itemize}

Moreover, this identifies simplices in a class of form 1 as inner faces of simplices in the corresponding class of form 2. Let $P$ be the collection of pairs $\tau \subset \tau'$ of nondegenerate $k-1$ and $k$-simplices matched by this bijection. Choose a total ordering on $P$ where pairs are ordered first by the dimension of the smaller simplex, and then by $A<B<C$, and then randomly. Let $J = \{0,...,M\}$ be the indexing set for $P$. We define a sequence of inner anodyne maps
\[ A_{n,i} = A_{n,i,0} \to A_{n,i,1} \to ... \to A_{n,i,M+1} = A_{n,i+1} \]
such that $A_{n,i,j+1}$ is obtained from $A_{n,i,j}$ by attaching the $j$th pair $\tau \subset \tau'$ along an inner horn. For this to be valid, we need the other faces of $\tau'$ to already be in $A_{n,i,j}$. The ordering on $E$ was chosen so that the outer faces of $\tau'$ are in $A_{n,i}$. The argument for the inner faces proceeds by cases:
\begin{itemize}
	\item $\tau'$ is in class $A2$: The other inner faces are also in class $A2$ since they contain $x_0 (y_0+1)$, hence were added at some earlier stage.
	\item $\tau'$ is in class $B2$: The other inner faces of $[x_0 y_0, (x_0+1) y_0, (x_0+1) (y_0+1), x_2 y_2,...,x_k y_k]$ are all in class $B2$, except for $[x_0 y_0, (x_0+1) (y_0+1), x_2 y_2,...,x_k y_k]$, which is in class $A1$. Both of these were added at an earlier stage.
	\item $\tau'$ is in class $C2$: The other inner faces are in class $C2$ or $B1$ since they contain $(x_0 +1) y_0$, hence were added at some earlier stage.
\end{itemize}
\end{proof}

Let $E_n \subset (A_n)_1 \subset \sO(\Delta^n)_1$ be the subset of edges $x_0 y_0 \rightarrow x_1 y_1$ where $y_0 = y_1$. Define simplicial sets $C'$ and $C''$ to be the pullbacks
\[ \begin{tikzcd}[row sep=2em, column sep=2em]
C'_\bullet \ar{r} \ar{d} & \Hom((\sO(\Delta^\bullet),E_\bullet),\leftnat{C}) \ar{d} \\
\Hom(\Delta^\bullet,S) \ar{r}{\ev_0^\ast} & \Hom(\sO(\Delta^\bullet),S)
\end{tikzcd},
\begin{tikzcd}[row sep=2em, column sep=2em]
C''_\bullet \ar{r} \ar{d} & \Hom((A_\bullet,E_\bullet),\leftnat{C}) \ar{d} \\
\Hom(\Delta^\bullet,S) \ar{r}{\ev_0^\ast} & \Hom(A_\bullet,S).
\end{tikzcd} \]

We now show that the map $C' \to C''$ induced by precomposition by $A_\bullet \to \sO(\Delta^\bullet)$ is a trivial Kan fibration. Indeed, in order to solve the lifting problem
\[ \begin{tikzcd}[row sep=2em, column sep=2em]
\partial \Delta^n \ar{r} \ar{d} & C' \ar{d} \\
\Delta^n \ar{r} \ar[dotted]{ru} & C''
\end{tikzcd} \]
we must supply a lift
\[ \begin{tikzcd}[row sep=2em, column sep=2em]
A_n \bigcup\limits_{\cup A_{n-1}} ( \bigcup \sO(\Delta^{n-1})) \ar{r} \ar{d} & C \ar{d} \\
\sO(\Delta^n) \ar{r} \ar[dotted]{ru} & S
\end{tikzcd} \]
and the left vertical map is a trivial cofibration by Lemma~\ref{lem:innerAnodyneMap}. Let $\sigma: C'' \to C'$ be any section. Also let $\delta: C' \to C$ be the map induced by precomposition by the identity section $\Delta^\bullet \to \sO(\Delta^\bullet)$.

\nomenclature[firstVertexParam]{$\upsilon_{C,S}$}{Parametrized first vertex functor}

Define a map $\upsilon_{C,S}: N^+_{\Delta^\op,S}(F_C) \to C$ over $S$ as follows: the data of an $n$-simplex of $N^+_{\Delta^\op,S}(F_C)$ consists of
\begin{itemize} \item an $n$-simplex $\Delta^{a_0} \leftarrow ... \leftarrow \Delta^{a_n}$ in $\Delta^\op$ (so we have maps $f_{ij}: \Delta^{a_j} \to \Delta^{a_i}$ for $i \leq j$);
\item an $n$-simplex $s_{\bullet}: \Delta^n \to S$;
\item a choice of $a_0$-simplex $\sigma_0 \in C_{a_0}$;
\item for $0 \leq i \leq n$, a map $\gamma_i: \Delta^i \to \underline{\sigma_i}$, where $\sigma_i = \sigma_0 \circ f_{0i}$
\end{itemize} 
such that for all $0 \leq i \leq j \leq n$, the diagram
\[ \begin{tikzcd}[row sep = 2em, column sep = 2em]
\Delta^i \ar{r}{\gamma_i} \ar{d}[swap]{\{0,...,i\} \subset [j]} & \underline{\sigma_i} \ar{d}{f_{ij}^\ast} \\
\Delta^j \ar{r}{\gamma_j} \ar{rd}[swap]{(s_\bullet)|_{ \{0,...,j\} }} & \underline{\sigma_j} \ar{d} \\
& S
\end{tikzcd} \]
commutes. Let $\overline{\gamma_i}: \Delta^i \times \Delta^{a_i} \times \Delta^1 \to C$ denote the adjoint map. 

We now define a map $A_n \to C$ to be that uniquely specified by sending for all $0 \leq k \leq n$ the rectangle $\Delta^k \times \Delta^{n-k} \subset A_n$ given by $00 \mapsto 0 k$ and $k (n-k) \mapsto k n$ to 
\[ \Delta^k \times \Delta^{n-k} \xrightarrow{\id \times (\lambda a)_k} \Delta^k \times \Delta^{a_k} \times \{1\} \xrightarrow{\overline{\gamma_i}|_{ \{1\} }} C \]

where the maps $(\lambda a)_k$ are obtained from the first vertex section of $W \to \Delta^\op$ restricted to $a_\bullet$ as before. One may check that the composite $A_n \to C \to S$ factors as $A_n \to \Delta^n \xrightarrow{s_\bullet} S$, so this defines a $n$-simplex of $C''$. This procedure is natural in $\Delta^n \in \Delta$, so yields a map $N^+_{\Delta^\op,S}(F_C) \to C''$. Finally, postcomposition by $\delta \circ \sigma: C'' \to C$ define our desired map $\upsilon_{C,S}$. By Proposition~\ref{prp:SrelativeNerve}, $N^+_{\Delta^\op,S}(F_C) \xrightarrow{\pi'} S$ is an $S$-category with an edge $\pi'$-cocartesian if and only if it is degenerate when projected to $\Delta^\op$. These edges are evidently sent to $\pi$-cocartesian edges in $C$, so $\upsilon_C$ is a $S$-functor.

\begin{thm} \label{thm:relativeBKwithCoproducts} The $S$-first vertex map $\upsilon_{C,S}: N^+_{\Delta^\op,S}(F_C) \to C$ is fiberwise a weak homotopy equivalence. Moreover, $\upsilon_{C,S}$ is $S$-final if either $C \to S$ is a left fibration, or $S$ is equivalent to the nerve of a $1$-category.
\end{thm}
\begin{proof} 
Let $t \in S$ be an object and $i_t: \{t\} \to S$ the inclusion. Then $N^+_{\Delta^\op,S}(F_C)_t \cong N^+_{\Delta^\op}(i_t^\ast F_C)$. We have a map $$N^+_{\Delta^\op}(i_t^\ast F_C) \to \Delta^\op_{/C} \cong N^+_{\Delta^\op}(C_\bullet)$$ of left fibrations over $\Delta^\op$ induced by the natural transformation $i_t^\ast F_C \to C_\bullet$ which collapses each $\underline{\sigma} \times_S \{t\}$ to a point. Moreover, this natural transformation is objectwise a Kan fibration, so the map itself is a left fibration. Also define a map $$N^+_{\Delta^\op}(i_t^\ast F_C) \to (S^{/t})^\op$$ as follows: in the above notation, the $\gamma_0$ map in the data of an $n$-simplex $(a_\bullet, \gamma_i: \Delta^i \to \underline{\sigma_i} \times_S \{t\})$ yields a map $\pi \gamma_0: \Delta^{a_0} \to \sO(S) \times_S \{t\} = S^{/t}$, and we send the $n$-simplex to
 \[ \Delta^n \xrightarrow{(\lambda a^{\rev})_0} (\Delta^{a_0})^\op \xrightarrow{(\pi \gamma_0)^\op} (S^{/t})^\op \]
where $a^\rev_\bullet$ is $(\Delta^{a_0})^\op \leftarrow ... \leftarrow (\Delta^{a_n})^\op$. Using these maps we obtain a commutative square
\[ \begin{tikzcd}[row sep=2em, column sep = 2em]
N^+_{\Delta^\op}(i_t^\ast F_C) \ar{r} \ar{d} & C^\op \times_{S^\op} (S^{/t})^\op \ar{d} \\
\Delta^\op_{/C} \ar{r}{\mu_{C}^\op} & C^\op.
\end{tikzcd} \]

We claim that the map
 \[ \theta_{C,t}: N^+_{\Delta^\op}(i_t^\ast F_C) \to (\Delta^\op_{/C}) \times_{C^\op} (C \times_S S^{/t})^\op \] 
is a categorical equivalence. Since $\theta_{C,t}$ is a map of left fibrations over $\Delta^\op_{/C}$, it suffices to check that for every object $\sigma \in \Delta^\op_{/C}$, the map on fibers
\[ \underline{\sigma} \times_S \{t\} \to (S^\op)^{t/} \times_{S^\op} \{\pi \sigma(n) \} \simeq \{ \pi \sigma (n)\} \times_S S^{/t} \]
is a homotopy equivalence. But this is the pullback of the trivial Kan fibration of Lemma~\ref{lm:fattenedSlice} over $\{t\}$.

We next define a map $N^+_{\Delta^\op}(i_t^\ast F_C) \to S^{/t}$ by sending $(a_\bullet,\gamma_i)$ to $\pi\gamma_0 \circ (\lambda a)_0$. Then the outer rectangle
\[ \begin{tikzcd}[row sep=2em, column sep = 2em]
N^+_{\Delta^\op}(i_t^\ast F_C) \ar[dotted]{r}[swap]{\upsilon'_{C,t}} \ar{d} \ar[bend left=15]{rr} & C \times_S S^{/t} \ar{d} \ar{r} & S^{/t} \ar{d} \\
\Delta^\op_{/C} \ar{r}{\upsilon_C} & C \ar{r}{\pi} & S
\end{tikzcd} \]
commutes so we obtain the dotted map $\upsilon'_{C,t}$.

Next, we choose a section $P$ of the trivial Kan fibration $\sO^\cocart(C) \to C \times_S \sO(S)$ which restricts to the identity section on $C$. $P$ restricts to a map $P_t: C \times_S S^{/t} \to \sO^\cocart(C) \times_S \{t\}$, and it is tedious but straightforward to construct a homotopy between the composition $(\ev_1 P_t) \circ \upsilon'_{C,t}$ and $(\upsilon_{C,S})_t$. Finally, we define a map $\upsilon''_{C,t}: \Delta_{/C \times_S S^{/t}}^\op \to N^+_{\Delta^\op}(i_t^\ast F_C)$ as follows: given an $n$-simplex
\[ \begin{tikzcd}[row sep = 2em, column sep = 2em]
\Delta^{a_0} \ar{d}{\tau_0} & \dots \ar{l} & \Delta^{a_n} \ar{lld}{\tau_n} \ar{l} \\
C \times_S S^{/t}
\end{tikzcd} \]
let $\sigma_i = \pr_C \circ \tau_i$, and define $\gamma_i: \Delta^i \to \underline{\sigma_i} \times_S \{t\}$ as the composition of the projection to $\Delta^0$ and the adjoint of the map $P_t \circ \tau_i$. Then $(a_\bullet, \gamma_i)$ assembles to yield an $n$-simplex of $N^+_{\Delta^\op}(i_t^\ast F_C)$.

Unwinding the definitions of the various maps, we identify the composition $\upsilon'_{C,t} \circ \upsilon''_{C,t}$ as given by $\upsilon_{C \times_S S^{/t}}$, and the composition $\theta_{C,t} \circ \upsilon''_{C,t}$ as given by the map $\Delta^\op_{/\pr_C}$ to the factor $\Delta^\op_{/C}$ and the map $(\mu_{C \times_S S^{/t}})^\op$ to the factor $(C \times_S S^{/t})^\op$. By Proposition~\ref{prp:cofinalityFirstVertex} and the fact that final maps pull back along cocartesian fibrations, we deduce that in
\[  \begin{tikzcd}[row sep=2em, column sep = 2em]
 \Delta^\op_{/C \times_S S^{/t}} \ar{r} \ar[bend left=15]{rr} & \Delta^\op_{/C} \times_{C^\op} (C \times_S S^{/t})^\op \ar{r} & (C \times_S S^{/t})^\op 
 \end{tikzcd} \]

the long composition and the second map are both final. Consequently, $\theta_{C,t} \circ \upsilon''_{C,t}$ is a weak homotopy equivalence. Moreover, if $S$ is equivalent to the nerve of a $1$-category then $\theta_{C,t} \circ \upsilon''_{C,t}$ is a categorical equivalence, as may be verified by checking that the map is a fiberwise equivalence over $\Delta^\op_{/C}$. Since $\theta_{C,t}$ is a categorical equivalence, $\upsilon''_{C,t}$ is then a weak homotopy equivalence resp. a categorical equivalence. Since $\upsilon_{C \times_S S^{/t}}$ is final, $\upsilon'_{C,t}$ is then a weak homotopy equivalence resp. final.

For the last step, let $j_t: C_t \to C \times_S S^{/t}$ denote the inclusion. As the inclusion of the fiber over a final object into a cocartesian fibration, $j_t$ is final. $(\ev_1 P_t) \circ j_t = \id_{C_t}$, so by right cancellativity of final maps, $\ev_1 P_t$ is final. We conclude that $(\upsilon_{C,S})_t$ is a weak homotopy equivalence resp. final. In addition, if $C \to S$ is a left fibration, $(\upsilon_{C,S})_t$ has target a Kan complex, so is final by \cite[Lemma~2.3.4.6]{HA}. Invoking the $S$-cofinality Theorem~\ref{thm:cofinality}, we conclude the proof.
\end{proof}

\begin{rem} The above proof that the $S$-first vertex map $\upsilon_{C,S}$ is final in special cases hinges upon the finality of the map $\theta_{C,t} \circ \upsilon''_{C,t}$. We believe, but are currently unable to prove, that this map is always final.
\end{rem}

We conclude this section with our main application to decomposing $S$-colimits.

\begin{cor} \label{cor:DecomposingColimits} Suppose that $S^\op$ admits multipullbacks. Then $C$ is $S$-cocomplete if and only $C$ admits all $S$-coproducts and geometric realizations.
\end{cor}
\begin{proof} We prove the if direction, the only if direction being obvious. Let $K$ be a $S^{s/}$-category and $p: K \to C_{\underline{s}}$ a $S^{s/}$-diagram. First suppose that $K \to S^{s/}$ is a left fibration. Consider the diagram
\[ \begin{tikzcd}[row sep=2em, column sep=2em]
N^+_{\Delta^\op,S^{s/}}(F_K) \ar{r}{\upsilon_{K,S^{s/}}} \ar{d}{\rho} & K \ar{r}{p} & C_{\underline{s}} \\
\Delta^\op \times S^{s/}.
\end{tikzcd} \]
By Theorem~\ref{thm:relativeBKwithCoproducts}, the $S^{s/}$-colimit of $p$ is equivalent to that of $p \circ \upsilon_{K,S^{s/}}$. Since $\rho$ is $S$-cocartesian, by Theorem~\ref{thm:ExistenceAndUqnessOfParamColimit} the $S^{s/}$-left Kan extension of $p \circ \upsilon_{K,S^{s/}}$ along $\rho$ exists provided that for all $n \in \Delta^\op$ and $f: s \rightarrow t$, the $S^{t/}$-colimit exists for $(p \circ \upsilon_{K,S^{s/}})_{\underline{(n,f)}}$.  To understand the domain of this map, note that because the pullback of $\rho$ along $f^\ast: \Delta^\op \times S^{t/} \to \Delta^\op \times S^{s/}$ is given by $N^+_{\Delta^\op,S^{t/}}(f^\ast F_K)$, the assumption that $S^\op$ admits multipullbacks ensures that the $\underline{(n,f)}$-fibers of $\rho$ decompose as coproducts of representable left fibrations. Therefore, these colimits exist since $C$ is assumed to admit $S$-coproducts. Now by transitivity of left $S^{s/}$-Kan extensions, the $S^{s/}$-colimit of $p \circ \upsilon_{K,S^{s/}}$ is equivalent to that of $\rho_! (p \circ \upsilon_{K,S^{s/}})$, and this exists since $C$ is assumed to admit geometric realizations.

Now suppose that $K \to S^{s/}$ is any cocartesian fibration. Consider the diagram
\[ \begin{tikzcd}[row sep=2em, column sep=2em]
\iota \widetilde{\Fun}_{\Delta^\op \times S^{s/}}(W \times S^{s/}, \Delta^\op \times K) \ar{r}{\Upsilon'_{K,S^{s/}}} \ar{d}{\rho'} & K \ar{r}{p} & C_{\underline{s}} \\
\Delta^\op \times S^{s/}.
\end{tikzcd} \]
By Theorem~\ref{thm:relativeBKwithSpaces}, the $S^{s/}$-colimit of $p$ is equivalent to that of $p \circ \Upsilon'_{K,S^{s/}}$. By Proposition~\ref{prp:identifyFibersOfPairing}, the $\underline{(n,f)}$-fiber of $\rho'$ is equivalent to $\iota \underline{\Fun}_{S^{t/}}(\Delta^n \times S^{t/}, K \times_{S^{s/}} S^{t/})$, which in any case remains a left fibration. We just showed that for all $t \in S$, $C_{\underline{t}}$ admits $S^{t/}$-colimits indexed by left fibrations. We are thereby able to repeat the above proof in order to show that the $S^{s/}$-colimit of $p$ exists.
\end{proof}


\section{Appendix: Fiberwise fibrant replacement}

In this appendix, we formulate a result (Proposition~\ref{prp:fiberwiseFibrantReplacement}) which will allow us to recognize a map as a cocartesian equivalence if it is a marked equivalence on the fibers. We begin by introducing a marked variant of Lurie's mapping simplex construction.

\begin{dfn} Suppose a functor $\phi: [n] \to s\Set^+$, $A_0 \to ... \to A_n$. Define $M(\phi)$ to be the simplicial set which is the opposite of the mapping simplex construction of \cite[\S 3.2.2]{HTT}, so that a $m$-simplex of $M(\phi)$ is given by the data of a map $\alpha: \Delta^m \to \Delta^n$ together with a map $\beta: \Delta^m \to A_{\alpha(0)}$. Endow $M(\phi)$ with a marking by declaring an edge $e = (\alpha,\beta)$ of $M(\phi)$ to be marked if and only if $\beta$ is a marked edge of $A_{\alpha(0)}$. Note that if each $A_i$ is given the degenerate marking, then the marking on $M(\phi)$ is that of \cite[Notation~3.2.2.3]{HTT}.
\end{dfn}

\begin{lem} \label{MappingSimplexEqv} Suppose $\eta: \phi \to \psi$ is a natural transformation between functors $[n] \to s\Set^+$ such that for all $0 \leq i \leq n$, $\eta_i: A_i \to B_i$ is a cocartesian equivalence. Then $M(\eta): M(\phi) \to M(\psi)$ is a cocartesian equivalence in $s\Set^+_{/ \Delta^n}$.
\end{lem}
\begin{proof} Using the decomposition of $M(\phi)$ as the pushout $$M(\phi') \cup_{A_0 \times \Delta^{n-1}} A_0 \times \Delta^n$$ for $\phi': A_1 \to ... \to A_n$, this follows by an inductive argument in view of the left properness of $s\Set^+_{/ \Delta^n}$.
\end{proof}

\begin{cnstr} \label{MappingSimplex} Let $X \to \Delta^n$ be a cocartesian fibration, let $\sigma$ be a section of the trivial Kan fibration $\sO^\cocart(X) \to X \times_{\Delta^n} \sO(\Delta^n)$ which restricts to the identity section on $X$, and let $P = \ev_1 \circ \sigma$ be the corresponding choice of pushforward functor. For $0 \leq i < n$, define $f_i: X_i \times \Delta^1 \to X$ by $P \circ (\id_{X_i} \times f'_i)$ where $f'_i: \Delta^1 \to \sO(\Delta^n)$ is the edge $(i = i) \to (i \to i+1)$, and let $\phi: X_0^\sim \to ... \to X_n^\sim$ be the sequence obtained from the $f_i \times \{1\}$. We will explain how to produce a map $M(\phi) \to X$ over $\Delta^n$ via an inductive procedure. Begin by defining the map $M(\phi)_n = X_n \to X_n$ to be the identity. Proceeding, observe that $M(\phi)$ is the pushout
\[ \begin{tikzpicture}[baseline]
\matrix(m)[matrix of math nodes,
row sep=6ex, column sep=4ex,
text height=1.5ex, text depth=0.25ex]
 { X_0 \times \Delta^{ \{1,...,n\}} & X_0 \times \Delta^n \\
   M(\phi') & M(\phi)  \\ };
\path[>=stealth,->,font=\scriptsize]
(m-1-1) edge (m-1-2)
edge node[right]{$\gamma$} (m-2-1)
(m-1-2) edge (m-2-2)
(m-2-1) edge  (m-2-2);
\end{tikzpicture} \]
with $\phi'$ the composable sequence $X_1 \to ... \to X_n$ and the map $\gamma$ given by $X_0 \times \Delta^{n-1} \to X_1 \times \Delta^{n-1} \to M(\phi')$. Given a map $g': M(\phi') \to X$ over $\Delta^{n-1}$, we have a commutative square
\[ \begin{tikzpicture}[baseline]
\matrix(m)[matrix of math nodes,
row sep=6ex, column sep=4ex,
text height=1.5ex, text depth=0.25ex]
 {  X_0 \times \Delta^1 \cup_{X_0 \times \Delta^{ \{1\} }} X_0 \times \Delta^{\{1,...,n\}} & X \\
   X_0 \times \Delta^n & \Delta^n,  \\ };
\path[>=stealth,->,font=\scriptsize]
(m-1-1) edge node[above]{$(f_0,g' \circ \gamma)$} (m-1-2)
edge (m-2-1)
(m-1-2) edge (m-2-2)
(m-2-1) edge (m-2-2)
edge[dotted] (m-1-2);
\end{tikzpicture} \]
and the left vertical map is inner anodyne by \cite[Lemma~2.1.2.3]{HTT} and \cite[Corollary~2.3.2.4]{HTT}. Thus a dotted lift exists and we may extend $g'$ to $g: M(\phi) \to X$.

Note that $g_i$ is the identity for all $0 \leq i \leq n$. Therefore, if we instead take the marking on $M(\phi)$ which arises from the degenerate marking on the $X_i$, then $g$ is (the opposite of) a quasi-equivalence in the terminology of \cite[Definition~3.2.2.6]{HTT}, hence a cocartesian equivalence in $s\Set^+_{/ \Delta^n}$ by \cite[Proposition~3.2.2.14]{HTT}. Now by Lemma~\ref{MappingSimplexEqv}, $g$ with the given marking is a cocartesian equivalence.

This construction of $M(\phi) \to X$ enjoys a convenient functoriality property: given a cofibration $F: X \to Y$ between cocartesian fibrations over $\Delta^n$, we may first choose $\sigma_X$ as above, and then define $\sigma_Y$to be a lift in the diagram
\[ \begin{tikzpicture}[baseline]
\matrix(m)[matrix of math nodes,
row sep=6ex, column sep=4ex,
text height=1.5ex, text depth=0.25ex]
 { (X \times_{\Delta^n} \sO(\Delta^n)) \cup_X Y & \sO^\cocart(Y) \\
   Y \times_{\Delta^n} \sO(\Delta^n) & Y \times_{\Delta^n} \sO(\Delta^n).  \\ };
\path[>=stealth,->,font=\scriptsize]
(m-1-1) edge node[above]{$(F \circ \sigma_X, \iota)$} (m-1-2)
edge (m-2-1)
(m-1-2) edge[->>] node[left]{$\sim$} (m-2-2)
(m-2-1) edge node[above]{$=$} (m-2-2)
edge[dotted] node[above]{$\sigma_Y$} (m-1-2);
\end{tikzpicture} \]

Consequently, we obtain compatible pushforward functors and a natural transformation $\eta: \phi_X \to \phi_Y$, which yields, by a similar argument, a commutative square
\[ \begin{tikzpicture}[baseline]
\matrix(m)[matrix of math nodes,
row sep=6ex, column sep=4ex,
text height=1.5ex, text depth=0.25ex]
 { M(\phi_X) & M(\phi_Y) \\
   X & Y.  \\ };
\path[>=stealth,->,font=\scriptsize]
(m-1-1) edge node[above]{$M(\eta)$} (m-1-2)
edge (m-2-1)
(m-1-2) edge (m-2-2)
(m-2-1) edge node[above]{$F$} (m-2-2);
\end{tikzpicture} \]
where the vertical maps are cocartesian equivalences in $s\Set^+_{/\Delta^n}$.
\end{cnstr}

\begin{prp} \label{prp:fiberwiseFibrantReplacement} Let $p: X \to S$ and $q:Y \to S$ be cocartesian fibrations over $S$ and let $F: X \to Y$ be a $S$-functor. Suppose collections of edges $\sE_X$, $\sE_Y$ of $X$, $Y$ such that
\begin{enumerate} \item $\sE_X$ resp. $\sE_Y$ contains the $p$ resp. $q$-cocartesian edges;
\item For $\sE^0_X \subset \sE_X$ the subset of edges which are either $p$-cocartesian or lie in a fiber, we have that $(X, \sE^0_X) \subset (X,\sE_X)$ is a cocartesian equivalence in $s\Set^+_{/S}$, and ditto for $Y$;
\item $F(\sE_X) \subset \sE_Y$;
\item For all $s \in S$, $F_s: (X_s, (\sE_X)_s) \to (Y_s, (\sE_Y)_s)$ is a cocartesian equivalence in $s\Set^+$.
\end{enumerate}
Let $X' = (X,\sE_X)$, $Y' = (Y,\sE_Y)$, and $F': X' \to Y'$ be the map given on underlying simplicial sets by $F$.
Then for all simplicial sets $U$ and maps $U \to S$, $F'_U$ is a cocartesian equivalence in $s\Set^+_{/U}$.
\end{prp}
\begin{proof} Without loss of generality, we may assume that an edge $e$ is in $\sE_X$ if and only if either $e$ is $p$-cocartesian or $p(e)$ is degenerate, and ditto for $\sE_Y$. First suppose that $F$ is a trivial fibration in $s\Set^+_{/S}$ and for all $s \in S$, $F'_s$ reflects marked edges. Then $F'$ is again a trivial fibration because $F'$ has the right lifting property against all cofibrations. For the general case, factor $F$ as $X \xrightarrow{G} Z \xrightarrow{H} Y$ where $G$ is a cofibration and $H$ is a trivial fibration, and let $Z' = (Z,\sE_Z)$ for $\sE_Z$ the collection of edges $e$ where $e$ is in $\sE_Z$ if and only if $H(e)$ is in $\sE_Y$. Then for all $s \in S$, $Z'_s \to Y'_s$ is a trivial fibration in $s\Set^+$, so as we just showed $H': Z' \to Y'$ is a trivial fibration. We thereby reduce to the case that $F$ is a cofibration.

Let $\sU$ denote the collection of simplicial sets $U$ such that for every map $U \to S$, $F'_U$ is a cocartesian equivalence in $s\Set^+_{/U}$. We need to prove that every simplicial set belongs to $\sU$. For this, we will verify the hypotheses of \cite[Lemma~2.2.3.5]{HTT}. Conditions (i) and (ii) are obvious, condition (iv) follows from left properness of the cocartesian model structure and \cite[Proposition~B.2.9]{HA}, and condition (v) follows from the stability of cocartesian equivalences under filtered colimits and \cite[Proposition~B.2.9]{HA}. It remains to check that every $n$-simplex belongs to $\sU$, so suppose $S = \Delta^n$. Let
\[ \begin{tikzcd}[row sep=2em, column sep=2em]
M(\phi_X) \ar{r}{M(\eta)} \ar{d} & M(\phi_Y) \ar{d} \\ 
X \ar{r}{F} & Y
\end{tikzcd} \]
be as in Construction~\ref{MappingSimplex}. Let $\phi'_X$ be the sequence $X'_0 \to ... \to X'_n$, where the maps are the same as in $\phi_X$, and similarly define $\phi'_Y$ and $\eta'$. Then we have pushout squares
\[ \begin{tikzcd}[row sep=2em, column sep=2em]
M(\phi_X) \ar{r} \ar{d} & M(\phi'_X) \ar{d} \\ 
X \ar{r} & X''
\end{tikzcd}, \quad
\begin{tikzcd}[row sep=2em, column sep=2em]
M(\phi_Y) \ar{r} \ar{d} & M(\phi'_Y) \ar{d} \\ 
Y \ar{r} & Y''
\end{tikzcd} \]
with all four vertical maps cocartesian equivalences in $s\Set^+_{/\Delta^n}$. Here we replace $X'$ by $X''$, which has the same underlying simplicial set $X$ but more edges marked with $X' \subset X''$ left marked anodyne, so that the vertical maps $M(\phi'_X) \to X''$ are defined and the squares are pushout squares (again, ditto for $Y''$). Note that $F$ defines a map $F'': X'' \to Y''$.

Finally, we have the commutative square
\[ \begin{tikzcd}[row sep=2em, column sep=2em]
M(\phi_X') \ar{r}{M(\eta')} \ar{d} & M(\phi'_Y) \ar{d} \\ 
X'' \ar{r}{F''} & Y''.
\end{tikzcd} \]
By assumption, $\eta': \phi'_X \to \phi'_Y$ is a natural transformation through cocartesian equivalences in $s\Set^+$. By Lemma~\ref{MappingSimplexEqv}, $M(\eta')$ is a cocartesian equivalence in $s\Set^+_{/\Delta^n}$. We deduce that $F''$, hence $F'$, is as well.
\end{proof}

\begin{rem} By a simple modification of the above arguments, we may further prove that for any marked simplicial set $A \to S$, $F'_A$ is a cocartesian equivalence in $s\Set^+_{/A}$. We leave the details of this to the reader.
\end{rem}


\clearpage
\printnomenclature[1in]

\clearpage

\bibliographystyle{amsplain}
\bibliography{Gcats}

\end{document}